\documentclass{amsart}
\usepackage[english]{babel}
\usepackage[utf8]{inputenc}

\usepackage[bookmarks,colorlinks,breaklinks,hypertexnames=false]{hyperref}
\hypersetup{linkcolor=blue,citecolor=blue,filecolor=blue,urlcolor=blue}

\usepackage{amssymb,graphicx,mathrsfs,cleveref,latexsym}

\usepackage{paper}
\usepackage{environments}
\usepackage{macros}

\providecommand{\assign}{:=}
\providecommand{\cdummy}{\cdot}
\providecommand{\nobracket}{.}
\providecommand{\of}{:}
\providecommand{\tmem}[1]{{\em #1\/}}
\providecommand{\tmmathbf}[1]{\ensuremath{\boldsymbol{#1}}}
\providecommand{\tmop}[1]{\ensuremath{\operatorname{#1}}}
\providecommand{\tmscript}[1]{\text{\scriptsize{\(#1\)}}}

\providecommand{\tmtextbf}[1]{\text{{\bfseries{#1}}}}
\providecommand{\tmtextit}[1]{\text{{\itshape{#1}}}}

\providecommand{\comment}[1]{}

\AtEveryBibitem{
    \clearfield{urldate}
}

\providecommand{\subsubsectiontoc}[1]{}

{\comment{* Local to this file/project}}

{\comment{** Operators and forms related to BHT}}

\providecommand{\BHT}{\mathrm{BHT}}

\providecommand{\Hil}{\mathrm{H}}

\providecommand{\HLM}{\mathrm{M}}
{\comment{parameter}}

\providecommand{\gr}{\mf{r}}
{\comment{** Time frequency notions}}

{\comment{Wave packet embedding}}

\providecommand{\Emb}{\ensuremath{\mathrm{Emb}}}
{\comment{Model tree}}

\providecommand{\mT}{\mf{T}}
{\comment{Trees}}

\providecommand{\TT}{\mbb{T}}
{\comment{Strips}}

\providecommand{\DD}{\mbb{D}}
{\comment{Full size}}

\providecommand{\SF}{\mbb{F}}
{\comment{Bilinear size}}

{\comment{Integral size}}

\providecommand{\SI}{\mbb{I}}
{\comment{Maximal truncation size}}

\providecommand{\SJ}{\mbb{J}}
{\comment{Lebesgue size}}

\providecommand{\SL}{\mbb{L}}
{\comment{Some other size}}

\providecommand{\SO}{\mbb{S}}
{\comment{Some other space}}

\providecommand{\XO}{\mbb{X}}
{\comment{Some other generating collections}}

\providecommand{\BO}{\mbb{B}}
{\comment{Local outer Lebesgue}}

\providecommand{\fL}{\not{L}}
{\comment{Wave packet operators}}

\providecommand{\wpD}{\mathfrak{d}}
{\comment{Defect}}

\providecommand{\dfct}{\Box}

\begin{document}

\title{The full range of uniform bounds for the bilinear Hilbert transform}

\author[G. Uraltsev]{Gennady Uraltsev}
\address{\noindent Department of Mathematics \newline \indent University of Virginia, Charlottesville, VA, USA}
\email{gennady.uraltesv@gmail.com}

\author[M.Warchalski]{Michał Warchalski}
\email{waral91@gmail.com }


\begin{abstract}
We prove  uniform \(L^{p}\) bounds for the family of bilinear Hilbert transforms
\(\mathrm{BHT}_{\beta} [f_1, f_2] (x) := \mathrm{p.v.} \int_{\R} f_1 (x - t) f_2 (x + \beta t)  \frac{\mathrm{d} t}{t} \). We show that the operator \(\mathrm{BHT}_{\beta}\) maps \(L^{p_{1}}\times L^{p_{2}}\) into \(L^{p}\)  as long as \(p_1 \in (1, \infty)\), \(p_2 \in (1, \infty)\), and \(p > \frac{2}{3}\) with a bound independent of \(\beta\in(0,1]\). This is the full open range of exponents where the modulation invariant class of bilinear operators containing \(\mathrm{BHT}_{\beta}\)  can be bounded uniformly. This is done by proving boundedness of certain affine transformations of the frequency-time-scale space \(\mathbb{R}^{3}_{+}\) in terms of iterated outer Lebesgue spaces. This results in new linear and bilinear wave packet embedding bounds well suited to study uniform bounds. 
\end{abstract}

{\maketitle}

{\tableofcontents}

\section{Introduction}\label{sec:intro}

The bilinear Hilbert transform with parameter \(\beta \in (0, 1]\) is a bilinear
singular integral operator given by
\begin{equation}
  \BHT_{\beta} [f_1, f_2] (x) \eqd \pv \int_{\R} f_1 (x - t) f_2 (x + \beta t)
  \frac{\dd t}{t} \label{eq:BHT}
\end{equation}
where \(f_1, f_2 \in \Sch (\R)\) are Schwartz functions on the real
line. It arises in diverse context: averages of \eqref{eq:BHT} over \(\beta\)
can yield relevant correction terms for evaluating Cauchy integrals along
Lipschitz domains in \(\C\) and for inverting elliptic operators with
non-constant coefficients; fiberwise applications of bounds on \eqref{eq:BHT}
can be used to obtain bounds on multilinear singular integral operators with
full \(\tmop{Gl}_2\) symmetry. We elaborate and provide references for these
applications further on.

In this paper we are interested in showing the a-priori bounds
\begin{equation}
  \| \BHT_{\beta} [f_1, f_2] \|_{L^p (\R)} \leq C _{p_1,
  p_2, \beta} \| f_1 \|_{L^{p_1} (\R)} \| f_2 \|_{L^{p_2} (\R)}, \qquad p^{- 1} = p_1^{- 1} + p_2^{- 1} \label{eq:BHT-bounds}
\end{equation}
for some constant \(C_{p_1, p_2, \beta} > 0\) that does not depend on \(f_1\) and
\(f_2\). A scaling argument shows that the constraints on the exponents \((p_1,
p_2, p)\) in \eqref{eq:BHT-bounds} must hold. In
{\cite{laceyEstimatesBilinearHilbert1997}}, Lacey and Thiele proved the first
estimates of the type \eqref{eq:BHT-bounds}: they showed that the bound holds
in the range \(2 < p_1, p_2 < \infty\), \(p < 2\). This corresponds to the open
triangle \(c\) in \Cref{fig:triangle}. The range of exponents for
\eqref{eq:BHT-bounds} was later extended by Lacey and Thiele in
{\cite{laceyCalderonConjectureBilinear1998}} to the convex hull of the open
triangles \(a_1, a_2, a_3\) in \Cref{fig:triangle}, where bounds below the line
\((0, 1, 0) - (0, 0, 1)\) and to the upper right side of the line \((0, 1, 0) -
(0, 0, 1)\) are to be intended in the restricted weak sense, discussed in
{\cite{thieleWavePacketAnalysis2006}}. Inspired by the works of Carleson
{\cite{carlesonConvergenceGrowthPartial1966}} and Fefferman
{\cite{feffermanPointwiseConvergenceFourier1973}} the authors of the above
results used the paradigm of time scale frequency analysis.

When \(\beta = 0\) the operator \(\BHT_{\beta} [f_1, f_2] (x)\) becomes simpler
and is given by the pointwise product of the Hilbert transform of \(f_1\) with
\(f_2\):
\[ \BHT_0 [f_1, f_2] (x) =  \Hil f_1 (x) \,f_2 (x) ; \]
this immediately implies boundedness for any \(p_1, p_2 \in (1, \infty)\) by the
boundedness of \(f \mapsto \Hil f\) on \(L^{p_1} (\R)\) for any \(p_1
\in (1, \infty)\) and by Hölder's inequality. For any fixed \(p_1, p_2 \in (1,
\infty)\), the constant \(C_{p_1, p_2, \beta}\) obtained in
{\cite{laceyEstimatesBilinearHilbert1997}} and
{\cite{laceyCalderonConjectureBilinear1998}}, while this is not stated
explicitly, grows linearly in \(| \beta |^{- 1}\). Naturally, in
{\cite{laceyCalderonConjectureBilinear1998}} the authors posed the question of
whether bounds \eqref{eq:BHT-bounds} can be shown with a constant \(C_{p_1,
p_2}\) independent of \(\beta \in (0, 1]\), in place of \(C _{p_1, p_2, \beta}\),
and if so, for what range of exponents \(p_1, p_2\). This is in accordance with
the observation that in the limit \(\beta = 0\) we are still dealing with a
bounded operator. These bounds are referred to as ``uniform bounds''; any
bound that holds only for \(\beta \gtrsim 1\) or with a constant that is allowed
to depend on \(\beta\), will be referred to in this paper as a ``non-uniform''
bound.

There is a rich collection of partial results concerning uniform variants of
\eqref{eq:BHT-bounds}. In {\cite{thieleUniformEstimate2002}} Thiele
established weak uniform bounds at the two upper corners of the triangle \(c\)
in \Cref{fig:triangle}. Grafakos and Li in
{\cite{grafakosUniformBoundsBilinear2004}} showed the uniform inequality in
the open triangle \(c\), and Li {\cite{liUniformBoundsBilinear2006}} proved the
uniform bounds in the open triangles \(a_1\) \(a_2\). By interpolation one obtains
the uniform variant of \eqref{eq:BHT-bounds} in the range corresponding to the
convex hull of the open triangles \(a_2\), \(a_3\), and \(c\). What, however, was
not known up to date, is whether the uniform bounds hold in the vicinity of
\((p_1, p_2) = (\infty, 1)\) and \((p_1, p_2) = (\infty, \infty)\) i.e. in the
corners of triangles \(b_2\) and \(b_3\). This is precisely the main result of
this work.

\begin{theorem}
  \label{thm:main-1}Bounds \eqref{eq:BHT-bounds} holds for all functions \(f_1,
  f_2 \in \Sch (\R)\) as long as \(p_1 \in (1, \infty)\), \(p_2 \in
  (1, \infty)\), and \(p > \frac{2}{3}\) with a constant dependent only on \(p_1\)
  and \(p_2\) but not on \(\beta \in (0, 1]\).
\end{theorem}

Our methods cover directly the regions \(b_1, b_2, b_3, c\) in
\Cref{fig:triangle}. The extension to the full shaded region in
\Cref{fig:triangle} can be obtained by interpolation with the result in
{\cite{liUniformBoundsBilinear2006}}. As noted before, bounds on the same
side of the triangle \(a_3\) of the segment \((0, 1, 0) - (0, 0, 1)\) are to be
intended in the restricted weak type sense. Otherwise, we believe our methods
can be used to obtain a uniform sparse domination of the trilinear form
associated to \(\BHT_{\beta}\) that would imply \Cref{thm:main-1} in the full
range. However, we leave this for possible subsequent works.

Since the Hilbert transform is not bounded on \(L^{\infty}\), bounds
\eqref{eq:BHT-bounds} cannot hold uniformly below the \((0, 1, 0) - (0, 0, 1)\)
segment. Thus, \Cref{thm:main-1} covers the full range where bounds
\eqref{eq:BHT-bounds} can hold uniformly in \(\beta \in (0, 1]\), inside the
range where bounds \eqref{eq:BHT-bounds} are known to hold at all
(non-uniformly). Furthermore Lacey in
{\cite{laceyBilinearMaximalFunctions2000a}} shows that the above range of
exponents is sharp, modulo endpoints, for the non-uniform modulation invariant
class of bilinear operators.

Oberlin and Thiele, showed \Cref{thm:main-1} for a Walsh model of the bilinear
Hilbert transform. The techniques there rely crucially on the perfect time
frequency analysis of the Walsh phase space and on the possibility of
constructing phase space tree projections, tools that we were unable to
generalize for obtaining the current result in the real case.

\begin{figure}[h]
  \raisebox{-0.00133654242739077\height}{\includegraphics[width=6.29001049455595cm,height=6.28161484979667cm]{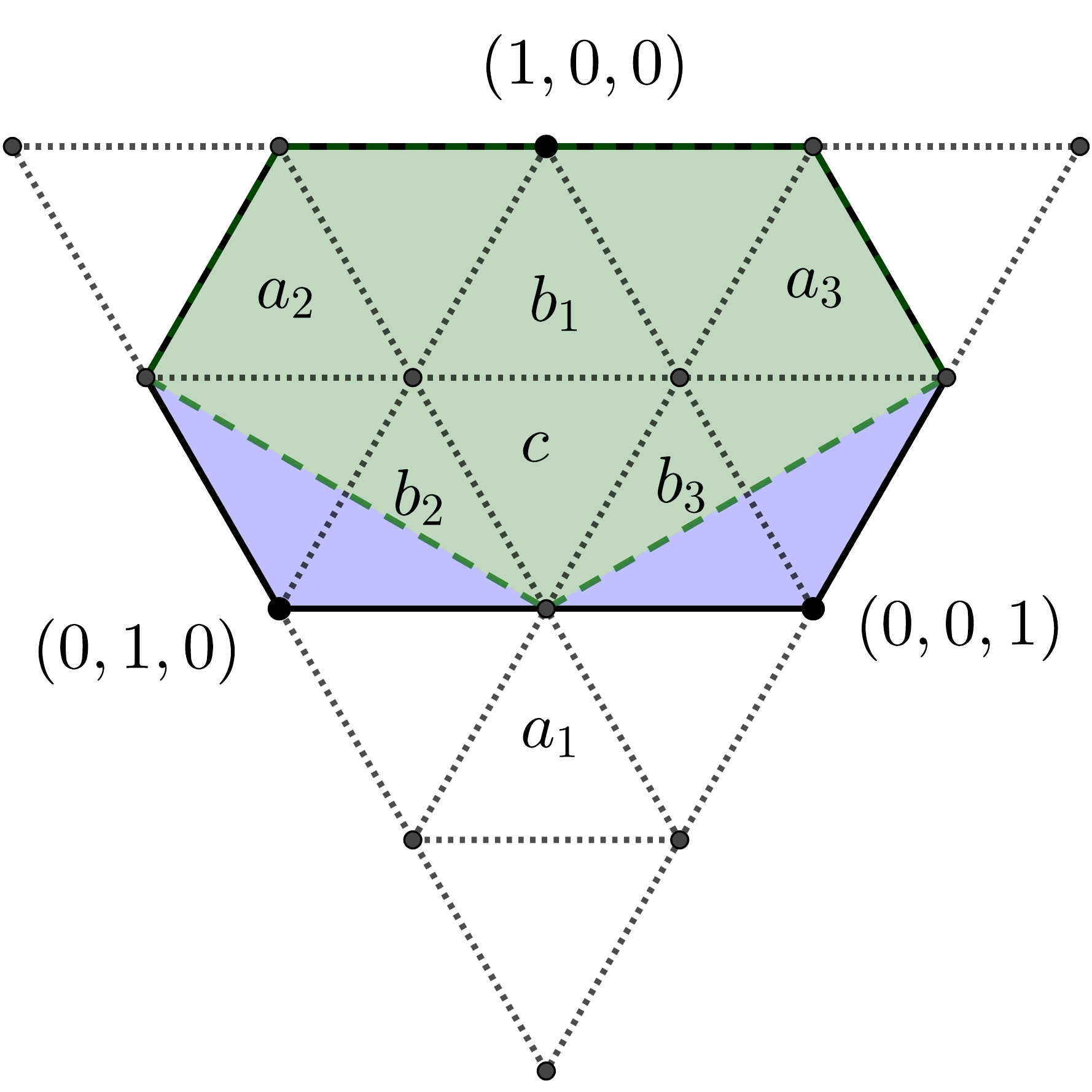}}
  \caption{\label{fig:triangle}
  Points \((p_1^{- 1}, p_2^{- 1}, 1 - p^{- 1}) \in (- 1, 1)^3\) with \(p^{- 1} =
  p_1^{- 1} + p_2^{- 1}\) for which \eqref{eq:BHT-bounds} may hold, represented
  as affine combinations.}
\end{figure}

\subsection{Applications}

The bilinear Hilbert transform has shown up in multiple circumstances. One of
the original motivations for the interest in \(\BHT_{\beta}\) was Calderón's
first commutator \(C_1 (f, A)\), that arises naturally as a correction term when
studying solutions for elliptic operators with non-smooth coefficients. In the
one-dimensional case one has
\[ C_1 (f, A) (x) \eqd \int_{\R} \frac{A (x) - A (y)}{(x - y)^2} f (y) \dd y.
\]
In {\cite{calderonCommutatorsSingularIntegral1965}} Calderón showed that \(f
\mapsto C_1 (f, A)\) is bounded on \(L^p (\R)\) for all \(p \in (1,
\infty)\) when \(\| A \|_{\tmop{Lip}} < + \infty\). Formally rewriting \(C_1 (f,
A)\) using the fundamental theorem of calculus yields
\[ C_1 (f, A) (x) \eqd \int_0^1 \int_{\R} f (y) A' (y + \beta (x - y))
   \frac{\dd y}{x - y} \dd \beta = \int_0^1 \BHT_{\beta} [f, A'] \dd \beta .
\]
Dualizing and enacting a change of variables allows one to recover Calderón's
result from {\cite{liUniformBoundsBilinear2006}}. An alternative proof of the
boundedness of Calderón's first commutator was given in
{\cite{muscaluCalderonCommutatorsCauchy2014}}. For a more detailed discussion
of Calderón's commutators see Chapter 4 of
{\cite{muscaluClassicalMultilinearHarmonic2013a}}.

More recently in {\cite{gressmanTrilinearSingularIntegral2016}}, bounds on a
trilinear form with full \(\tmop{GL} \left( \R^2 \right)\) symmetry acting on
functions on \(\R^2\) was obtained by a fiber-wise application of uniform bounds
in the range covered by {\cite{grafakosUniformBoundsBilinear2004}}.

Generalizations of the bilinear Hilbert transform to functions on \(\R^2\) have
been systematically studied in
{\cite{demeterTwodimensionalBilinearHilbert2010}}. Instead of just one
parameter \(\beta\), the corresponding trilinear form on \(\R^2\) is determined by
three matrixes \(\vec{B} = (B_j)_{j \in \{ 1, 2, 3 \}} \in \left( \R^{2 \times
2} \right)^3\) and is given by
\[ \tmop{BHF}^K_{\vec{B}} (f_1, f_2, f_3) \eqd \int_{\R^2} \int_{\R^2}
   \prod_{j = 1}^3 f_j \left(\begin{pmatrix}
     x\\
     y
   \end{pmatrix} + B_j \begin{pmatrix}
     s\\
     t
     \end{pmatrix} \right) K (s, t) \dd s \dd t \dd x \dd y, \]
where \(K\) is a two-dimensional smooth Calderón-Zygmund kernel. If \(B\) is
non-invertible, one can obtain completely novel multilinear integral
operators, some of which have bee   n studied independently (see Kovač's work on
the twisted paraproduct in {\cite{kovacBoundednessTwistedParaproduct2012}}),
while others remain beyond the current level understanding, like the
triangular Hilbert transform (see {\cite{kovacDyadicTriangularHilbert2015}}).
A partial investigation of uniform bounds in the context of the 2-dimensional
bilinear Hilbert transform has been carried by the second author in his PhD
thesis {\cite{warchalskiUniformEstimatesOneand2018}}. A preliminary version of
this joint work also appeared there.

\subsection{Wave packet representation and the space of symmetries}

The bilinear Hilbert transform is symmetric under translations, modulations,
and dilations, given by
\begin{equation}
  \begin{aligned}[c]
  & \Tr_y \phi (z) \eqd \phi (z - y),
  \\ & \Mod _{\eta} \phi (z) \eqd e^{2 \pi i \eta z} \phi (z),
  \\ & \Dil_t \phi (z) \eqd t^{- 1} \phi \left(
    \frac{z}{t} \right),
  \end{aligned}\quad    (\eta, y, t) \in \R^3_+ \eqd \R^2 \times \R_+, \label{eq:symmetries}
\end{equation}
that act on the functions \(f_1\) and \(f_2\). In particular, up to a linear
combination with the pointwise product, \(\BHT_{\beta}\) is the unique bilinear
operator that satisfies
\[ \begin{aligned}[t]
&\BHT_{\beta} \left[ \Tr_y f_1, \Tr_y f_2 \right] = \Tr_y \BHT_{\beta}
[f_1, f_2] \qquad \forall y \in \R,
\\ &
\BHT_{\beta} \left[ \Mod _{\beta \eta} f_1, \Mod _{\eta} f_2 \right] =
\Mod _{(1 + \beta) \eta} \BHT_{\beta} [f_1, f_2] \qquad \forall \eta \in
\R,
\\ &
t^2 \BHT_{\beta} \left[ \Dil_t f_1, \Dil_t f_2 \right] = t \Dil_t
\BHT_{\beta} [f_1, f_2] \qquad \forall t \in \R^+ .
\end{aligned} \]
Based on this we show that \(\BHT_{\beta}\) admits the following wave packet
representation
\begin{equation}
  \begin{aligned}[t]
  & \frac{1}{\pi i} \BHT_{\beta} [f_1, f_2] (x) = f_1 (x) f_2 (x)
  \\ & 
    \qquad \begin{aligned}[t]& - C_{\beta} 
    \int_{\R^3_+} \Emb [f_1] (\eta, y, t) [\phi_0]  \Emb [f_2] (\beta^{- 1}
      \eta - (\beta t)^{- 1}, y, \beta t) [\phi_0] 
      \\ &
      \qquad \times \Tr_y \Mod_{\frac{1 + \beta}{\beta} \eta - (\beta t)^{- 1}}
      \Dil_{\beta t} \phi_0 (x) \dd \eta \dd y \dd t ;
    \end{aligned}
  \end{aligned} \label{eq:BHT-wave-packet-representation}
\end{equation}
the wave packet \(\phi_0 \in \Sch (\R)\) can be chosen
independently of \(\beta\), and to satisfy \(\spt \FT{\phi_0} \subseteq
B_{\gr_0}\) for any fixed \(\gr_0 < 2^{- 20}\). The constant \(C_{\beta} > 0\) is
uniformly bounded above and away from \(0\) for all \(\beta \in (0, 1]\).

The model wave packet \(\phi\) is allowed to be any function in
\(\Phi^{\infty}_{\gr} \subset \Sch (\R)\), the subset of Schwartz
functions with compact Fourier support in \(\overline{B_{\gr}} \eqd \left[ -
\gr, \gr \right]\) for some \(\gr > 0\). Hölder's inequality shows that the
pointwise product \((f_1, f_2) \mapsto f_1 (x) f_2 (x)\) satisfies the bounds of
\Cref{thm:main-1}; thus we concern ourselves only with the second summand on
{\RHS{\eqref{eq:BHT-wave-packet-representation}}}.

When \(p > 1\), we prove \Cref{thm:main-1} by using identity
\(\eqref{eq:BHT-wave-packet-representation}\) and dualizing: we reduce to showing
that
\begin{equation}
  \left| \int_{\R^3_+} \Emb [f_1] (\eta, y, t) [\phi_0] \prod_{j = 2}^3 \Emb
  [f_j] \circ \Gamma_j (\eta, y, t) [\phi_0] \dd \eta \dd y \dd t \right|
  \lesssim  \prod_{j = 1}^3 \| f_j \|_{L^{p_j} (\R)},
  \label{eq:BHT-dual-bounds}
\end{equation}
holds as long as \(p_1, p_2, p_3 \in (1, \infty)\) and \(\sum_{j = 1}^3 p_j^{- 1}
= 1\). The maps \(\Gamma_j \of \R^3_+ \rightarrow \R^3_+\) are given by
\begin{equation}
  \begin{aligned}[t]
  & \Gamma_j  (\eta, y, t) \eqd \Gamma_{(\alpha_j, \beta_j, \gamma_j)}  (\eta, y, t) \eqd \Big(\alpha_j (\eta + \gamma_j t^{- 1}), y, \beta_j t\Big),
  \\ &
    \begin{aligned}[t]
     & (\alpha_2, \beta_2, \gamma_2) \eqd \left( \frac{1}{\beta}, \beta, - 1
     \right),
     \\ &  
      (\alpha_3, \beta_3, \gamma_3) \eqd \left( - \frac{1 + \beta}{\beta},
      \beta, \frac{- 1}{1 + \beta} \right) ;
    \end{aligned}
  \end{aligned} \label{eq:gamma:params}
\end{equation}
for convenience of notation we set \(\Gamma_1 (\eta, y, t) = (\eta, y, t)\). One
should think of \(\Gamma_j\) as encoding the geometric datum of the degeneration
parameter \(\beta \in (0, 1]\) in terms of the group of symmetries generated by
\eqref{eq:symmetries}.

The embedding coefficients \(\Emb [f] (\eta, y, t) [\phi]\) appearing above
provide a description of a function \(f\) in terms of the defining symmetries of
\(\BHT_{\beta}\). They are given by
\begin{equation}
  \Emb [f] (\eta, y, t) [\phi] \eqd \int_{\R} f (z) \Tr_y \Mod_{- \eta} \Dil_t
  \phi (z) \dd z. \label{eq:embedding}
\end{equation}
An appropriate functional analytic framework to understand the embedded
functions \(F = \Emb [f] : \R^3_+ \times \Phi_{\gr}^{\infty} \rightarrow \C\) is
given by outer Lebesgue spaces introduced in
{\cite{doTheoryOuterMeasures2015}}.

Non-uniform (\(\beta \gtrsim 1\)) variants of the bounds
\eqref{eq:BHT-dual-bounds} can be obtained in this framework from the chain of
inequalities
\begin{equation}
  \begin{aligned}[t]&
    \Big| \int_{\R^3_+} \prod_{j = 1}^3 \Emb [f_j] \circ \Gamma_j (\eta, y,
    t) [\phi_0] \dd \eta \dd y \dd t \Big|
    \\ &    \qquad
    \lesssim \Big\| \prod_{j = 1}^3 \Emb [f_j] \circ \Gamma_j
    [\phi_0] \Big\|_{L^1_{\nu_1} \fL^1_{\mu^1_{\Theta}, \nu_1} \SL^{(1, 1)}_{\Theta}}
    \\ &     \qquad
    \lesssim \prod_{j = 1}^3 \Big\| \Emb [f_j] \circ \Gamma_j [\phi_0]
    \Big\|_{L^{p_j}_{\nu_1} \fL^{q_j, +}_{\mu^1_{\Theta}, \nu_1}
    \SO_{\tmop{en}, j}},
  \end{aligned} \label{eq:outer-domination:non-uniform}
\end{equation}
and from the embedding bounds
\begin{equation}
  \left\| \Emb [f_j] \circ \Gamma_j \right\|_{L^{p_j}_{\nu_1} \fL^{q_j,
  +}_{\mu^1_{\Theta}, \nu_1} \SO_{\tmop{en}, j}} \lesssim \| f_j \|_{L^{p_j}
  (\R)} . \label{eq:iterated-embedding-bounds:non-uniform}
\end{equation}
for some appropriately chosen exponents \(q_j \in (2, \infty]\). The terms
appearing on \(\RHS{\eqref{eq:outer-domination:non-uniform}}\) are
\tmtextit{iterated outer Lebesgue quasi-norms} \(\| \cdot \|_{L^{p_j}_{\nu_1}
\fL^{q_j, +}_{\mu^1_{\Theta}, \nu_1} \SO_{\tmop{en}, j}}\). One should think of
these quasi-norms as appropriate generalizations of product Lebesgue space
norms defined on \(\R^3_+\), the parameterization of the symmetries
\eqref{eq:symmetries} of the problem at hand. Unlike in the case of classical
Lebesgue integral norms, the outer measures \(\nu_1\) and \(\mu^1_{\Theta}\)
appearing in \(\eqref{eq:iterated-embedding-bounds:non-uniform}\) fail to
satisfy Carathéodory measurability conditions on any non-trivial sets: thus
the need for a generalized integration theory. The symbols \(\SO_{\tmop{en},
j}\) and \(\SL^{(1, 1)}_{\Theta}\) appearing in the definition of the outer
Lebesgue quasi-norms are \tmtextit{sizes}, quantities that allow us to control
the magnitude of \(\Emb [f_j] \circ \Gamma_j\) locally.

The two inequalities in \eqref{eq:outer-domination:non-uniform} are abstract
functional analytic statements about (iterated) outer Lebesgue quasi-norms: a
Radon-Nikodym -type domination lemma and an outer Hölder inequality.
Inequalities \eqref{eq:iterated-embedding-bounds:non-uniform} are bounds on
wave packet embeddings \(f \mapsto \Emb [f] \circ \Gamma\). In
{\cite{uraltsevVariationalCarlesonEmbeddings2016}}, with minor notational
differences, one of the authors formulated inequalities
\eqref{eq:iterated-embedding-bounds:non-uniform} using this framework in the
context of the variation norm Carleson operator, deducing them from Di Plinio
and Ou's ideas in {\cite{diplinioModulationInvariantCarleson2018}}. The
iterated variants of outer Lebesgue spaces (e.g. bounds
\eqref{eq:outer-domination:non-uniform}) were initially introduced therein. A
complete study of non-uniform bounds for \(\BHT_{\beta}\) in this framework is
carried out with suitable generalizations in
{\cite{amentaBilinearHilbertTransform2020}} for Banach-valued functions; the
result there subsumes the non-uniform scalar results above. Originally, the
approach to non-uniform bounds for \(\BHT_{\beta}\) using outer Lebesgue spaces
was carried out in {\cite{doTheoryOuterMeasures2015}}, where the non-iterated
variants of outer Lebesgue quasi-norms have been introduced and were
sufficient to obtain bounds \eqref{eq:BHT-dual-bounds} for \(p_j \in (2,
\infty]\), \(j \in \{ 1, 2, 3 \}\).

The {\cite{doTheoryOuterMeasures2015}} proof of the non-iterated variant of
bound \eqref{eq:iterated-embedding-bounds:non-uniform} relies on
Calderón-Zygmund techniques when \(p = \infty\) and on
orthogonality considerations together with a covering lemma for sets on
\(\R^3_+\) when \(p = 2\). The range \(p \in (2, \infty]\) is then obtained by
interpolation. To go to \(p < 2\),
{\cite{diplinioModulationInvariantCarleson2018}} made use of a multi-frequency
Calderón-Zygmund decomposition. In
{\cite{amentaBilinearHilbertTransform2020}} provided an alternative proof of
\eqref{eq:iterated-embedding-bounds:non-uniform} in the context of Banach
space-valued spaces by interpolating the result with \(p > 2\) with an endpoint
that does not rely on cancellation at all, thus avoiding Hilbert space
techniques. On the other hand, in the recent result
{\cite{diplinioWeaktypeCarlesonTheorem2022}}, Di Plinio and Fragos improve on
{\cite{diplinioModulationInvariantCarleson2018}} and obtain an endpoint
embedding result close to \(L^1\) for the embedding bound
\eqref{eq:iterated-embedding-bounds:non-uniform} in the context of endpoint
estimates for the Carleson operator. Their result do not follow by
interpolation; they make use Gabor frame techniques and work in a localized
outer Lebesgue endpoint space formally reminiscent of the space
\(X_{\mu_{\Theta}^1, \nu_{\beta }}^{q, r, +}\) we introduce in
\Cref{sec:localized-outer-lebesgue}.

In the current paper, we generalize the framework of outer Lebesgue spaces
and the bounds \eqref{eq:iterated-embedding-bounds:non-uniform}. However, we
do not prove the analogue of \eqref{eq:iterated-embedding-bounds:non-uniform}
directly. With few exceptions, arguments are carried out purely in terms of
properties of embedded functions \(F \of \R^3_+ \times \Phi_{\gr}^{\infty}
\rightarrow \C\) and corresponding outer Lebesgue quasi-norms; we mainly rely
on a extensions of ideas related to the covering lemma from
{\cite{doTheoryOuterMeasures2015}}.

\subsection{Outline of the argument}

In this paper, we require a significant extensions of the framework of outer
Lebesgue spaces. A Radon-Nikodym domination type result, as before, allows us
to write
\begin{equation}
  \begin{aligned}[t]&
    \Big| \int_{\R^3_+} \prod_{j = 1}^3 \Emb [f_j] \circ \Gamma_j (\eta, y,
    t) [\phi_0] \dd \eta \dd y \dd t \Big| \lesssim \Big\| \prod_{j = 1}^3
    \Emb [f_j] \circ \Gamma_j \Big\|_{L^1_{\nu_1} \fL^1_{\mu^1_{\Theta}, \nu_1} \SI _{\Theta} \Phi_{\gr}^N} .
  \end{aligned} \label{eq:RN-domination-BHT}
\end{equation}
First of all, the sizes we use no longer evaluate the function \(\Emb [f] \circ
\Gamma\) at a fixed \(\phi_0 \in \Phi^{\infty}_{\mf{r}}\) but they track the
dependence of \(\Emb [f] \circ \Gamma\) on all possible generating wave packets
\(\phi \in \Phi^{\infty}_{\mf{r}}\). This is emphasized by the notation
\(\Phi_{\gr}^N\) appearing in the size above; the parameter \(N \in \N\) controls
the rate of spatial decay of \(\phi\) while \(\gr > 0\) governs the support of
\(\FT{\phi}\). In the subsequent discussion and in the entirety of this paper
the specific values of \(N \in \N\) and \(\gr > 0\) are inessential: inequalities
hold with constants dependent on \(N\) and \(\gr\) for all \(N \in \N\) large enough
and all \(\gr > 0\) small enough independently. However, the choice of \(N\) and
\(\gr\) can be made independently of \(\beta \in (0, 1]\). Actually, all implicit
constants henceforth are uniformly bounded in \(\beta \in (0, 1]\).

Continuing with the difference with respect to the non-uniform bound
\eqref{eq:outer-domination:non-uniform}, we use a much more involved size
\begin{equation}
  \SI _{\Theta} \Phi_{\gr}^N \eqd \begin{aligned}[t]&
    {\SI _{\Theta}}^{(l, l, l)} \Phi_{\gr}^N + {\SI _{\Theta}}^{(o, l, l)}
    \Phi_{\gr}^N\\ &
    \qquad + {\SI _{\Theta}}^{(l, l, o)} \Phi_{\gr}^N + {\SI _{\Theta}}^{(l,
    o, l)} \Phi_{\gr}^N + {\SI _{\Theta}}^{(l, o, o)} \Phi_{\gr}^N
  \end{aligned} \label{eq:I-size-sum}
\end{equation}
that replaces the simpler size \(\SL^{(1, 1)}_{\Theta}\) modeled locally on a
classical \(L^1\) Lebesgue norm.

The outer Lebesgue quasi-norm on the {\RHS{\eqref{eq:RN-domination-BHT}}} is
then divided into two parts. A part with all but one size forming \(\SI
_{\Theta} \Phi_{\gr}^N\) in \eqref{eq:I-size-sum}, can be bound, similarly to
\eqref{eq:outer-domination:non-uniform}, by a product of outer Lebesgue norms
of the embedded functions \(\Emb [f_j]\). It holds that
\begin{equation}
  \begin{aligned}[t]&
  \Big\| \prod_{j = 1}^3 \Emb [f_j] \circ \Gamma_j \Big\|_{L^1_{\nu_1} \fL^1_{\mu_{\Theta}^1, \nu_1} \left( {\SI _{\Theta}}^{(l, l, l)} \Phi_{\gr}^N + \SI^{(o, l, l)}_{\Theta} \Phi^N_{\gr} + \SI^{(l, o, l)}_{\Theta} \Phi^N_{\gr} + \SI^{(l, l, o)}_{\Theta} \Phi^N_{\gr}  \right)}
  \\ & \qquad
  \lesssim \Big\| \Emb [f_1] \Big\|_{L^{p_1}_{\nu_1} \fL^{q_1,+}_{\mu_{\Theta}^{\infty}, \nu_1} \SF^{u_1}_{\Theta} \Phi_{4 \mf{r}}^{N - 3}} \prod_{j = 2}^3 \Big\| \Emb [f_j] \circ \Gamma_j
  \Big\|_{L^{p_j}_{\nu_{\beta }} X_{\mu_{\Theta}^1, \nu_{\beta }}^{q_j, r_j, +} \widetilde{\SF}^{u_j}_{\Gamma_j} \Phi_{4 \mf{r}}^{N - 3}}
  \end{aligned} \label{eq:iterated-Holder-type-bounds:linear}
\end{equation}
for any \(p_j \in (1, \infty)\) with \(\sum_{j = 1}^3 p_j^{- 1} = 1\) with the
powers \(q_j\), \(r_j\), and \(u_j\) chosen appropriately depending on \(p_j\). We
introduce a new outer measures \(\nu_{\beta_j}\) on \(\R^3_+\) to deal with the
functions \(\Emb [f_j] \circ \Gamma_j\) with \(j \in \{ 2, 3 \}\), for which
\(\beta_j = \beta \ll 1\). Bounds \eqref{eq:iterated-Holder-type-bounds:linear}
are deduced from an ``atomic estimate'' using restricted weak type
interpolation for outer measure Lebesgue spaces
(\Cref{prop:outer-restricted-interpolation}). This generalizes the simpler
Hölder inequality for outer Lebesgue spaces used in the non-uniform case of
bounds \eqref{eq:outer-domination:non-uniform}. We then obtain bounds
\[
\begin{aligned}[t]
& \Big\| \prod_{j = 1}^3 \Emb [f_j] \circ \Gamma_j \Big\|_{L^1_{\nu_1} \fL^1_{\mu_{\Theta}^1, \nu_1} \left( {\SI _{\Theta}}^{(l, l, l)}
   \Phi_{\gr}^N + \SI^{(o, l, l)}_{\Theta} \Phi^N_{\gr} + \SI^{(l, o, l)}_{\Theta} \Phi^N_{\gr} + \SI^{(l, l, o)}_{\Theta} \Phi^N_{\gr} \right)}
\\ & \qquad
\lesssim \prod_{j = 1}^3 \| f_j \|_{L^{p_j} (\R)}
\end{aligned}
\]
by using \eqref{eq:iterated-Holder-type-bounds:linear} and by showing the
non-uniform (\(\beta = 1\)) embedding bounds
\begin{equation}
\begin{aligned}[t]
&\Big\| \Emb [f_1] \Big\|_{L^{p_1}_{\nu_1} \fL^{q_1,+}_{\mu_{\Theta}^{\infty}, \nu_1} \SF^{u_1}_{\Theta} \Phi_{\mf{r}}^N}
\lesssim \| f_1 \|_{L^{p_1} (\R)}
\end{aligned}
  \label{eq:embedding-bounds:non-uniform-iterated}
\end{equation}
and the uniform embedding bounds
\begin{equation}
  \Big\| \Emb [f_j] \circ \Gamma_j \Big\|_{L^{p_j}_{\nu_{\beta }}
  X_{\mu_{\Theta}^1, \nu_{\beta }}^{q_j, r_j, +}
  \widetilde{\SF}^{u_j}_{\Gamma_j} \Phi_{\mf{r}}^N} \lesssim \| f_j  \|_{L^{p_j} (\R)} .
  \label{eq:embedding-bounds:uniform-iterated}
\end{equation}
Bounds \eqref{eq:embedding-bounds:non-uniform-iterated} mostly extend the
available results from {\cite{amentaBilinearHilbertTransform2020}} to account
for additional terms in the size \(\SF_{\Theta}^{u_1}\) and a larger outer
measure \(\mu^{\infty}_{\Theta}\) (cfr. \eqref{eq:iterated-embedding-bounds:non-uniform}). Bounds
\eqref{eq:embedding-bounds:uniform-iterated} are an innovation of the current
work. Unlike \eqref{eq:embedding-bounds:non-uniform-iterated}, they are not
proven directly as in {\cite{amentaBilinearHilbertTransform2020}} and
{\cite{doTheoryOuterMeasures2015}} by analyzing the function \(f_j\). Rather, we
obtain \eqref{eq:embedding-bounds:uniform-iterated} from geometric properties
of the diffeomorphisms \(\Gamma_j\) on the time-frequency-scale space and by
proving a covering lemma directly on \(\R^3_+\).

To deal with the outer Lebesgue quasi-norms with the size \(\SI^{(l, o,
o)}_{\Theta} \Phi_{\gr}^N\) we prove the bound
\begin{equation}
  \begin{aligned}[t]&
    \Big\| \prod_{j = 1}^3 \Emb [f_j] \circ \Gamma_j \Big\|_{L^1_{\nu_1}
    \fL^1_{\mu_{\Theta}^1, \nu_1} \SI^{(l, o, o)}_{\Theta} \Phi^N_{\gr}}\\ &
    \qquad \lesssim \Big\| \Emb [f_1] \Big\|_{L^{p_1}_{\nu_1} \fL^{q_1,
    +}_{\mu_{\Theta}^{\infty}, \nu_1} \SF_{\Theta}^{u_1} \Phi_{4 \mf{r}}^{N -
    3}} \Big\| \prod_{j = 2}^3 \Emb [f_j] \circ \Gamma_j
    \Big\|_{L^{p_{\times}}_{\nu_{\beta }} X_{\mu_{\Theta}^1, \nu_{\beta
    }}^{q_{\times}, r_{\times}, +}
    \widetilde{\SF}^{u_{\times}}_{\Gamma_{\times}} \Phi_{4 \mf{r}}^{N - 3}}
  \end{aligned} \label{eq:iterated-Holder-type-bounds:bilinear}
\end{equation}
for any \(p_1, p_{\times} \in (1, \infty)\) with \(p_1^{- 1} + p_{\times}^{- 1} =
1\) and for an appropriate choice of \(q_1, q_{\times}\), \(r_{\times}\), and \(u_1,
u_{\times}\). To show
\[ \Big\| \prod_{j = 1}^3 \Emb [f_j] \circ \Gamma_j \Big\|_{L^1_{\nu_1}
   \fL^1_{\mu_{\Theta}^1, \nu_1} \SI^{(l, o, o)}_{\Theta} \Phi^N_{\gr}}
   \lesssim \prod_{j = 1}^3 \| f_j \|_{L^{p_j} ( \R )} \]
requires us, in addition to using
\eqref{eq:embedding-bounds:non-uniform-iterated}, to prove a bilinear uniform
embedding bound
\begin{equation}
  \Big\| \Emb [f_2] \circ \Gamma_2  \Emb [f_3] \circ \Gamma_3 \Big\|_{L^{p_{\times}}_{\nu_{\beta}} X^{q_{\times}, r_{\times},  +}_{\mu_{\Theta}^1, \nu_{\beta}}
  \widetilde{\SF}^{u_{\times}}_{\Gamma_{\times}} \Phi_{\mf{r}}^N} \lesssim \| f_2 \|_{L^{p_2} (\R)} \| f_2 \|_{L^{p_3} (\R)} .
  \label{eq:embedding-bound:bilinear}
\end{equation}
The product term in {\LHS{\eqref{eq:embedding-bound:bilinear}}} is another new
result of this work. We believe it to be of crucial importance since for any
fixed \(f_{1, 2, 3} \in \Sch (\R)\) the term
\[ \Big\| \prod_{j = 1}^3 \Emb [f_j] \circ \Gamma_j \Big\|_{L^1_{\nu_1} \fL^1_{\mu_{\Theta}^1, \nu_1} \SI^{(l, o, o)}_{\Theta} \Phi^N_{\gr}} \]
becomes the most relevant one when \(\beta \rightarrow 0^+\). Furthermore, we
believe that bounds \eqref{eq:embedding-bound:bilinear} are intrinsically
bilinear and they cannot be deduced from linear embedding bounds similar to
\eqref{eq:embedding-bounds:uniform-iterated}.

After setting up the appropriate extension of the outer measure Lebesgue space
framework, we prove the following results.

\begin{theorem}
  \label{thm:STE} Let \(\Gamma_j = \Gamma_{(\alpha_j, \beta_j, \gamma_j)}\), \(j
  \in \{ 2, 3 \}\) with \((\alpha_j, \beta_j, \gamma_j)\) as in
  \eqref{eq:gamma:params}. In particular suppose \(1 + \alpha_2 + \alpha_3 = 0\)
  and \(| \gamma_j | \geq 1 / 4\), \(j \in \{ 2, 3 \}\). The bound
  \begin{equation}
  \begin{aligned}[t]&
      \| F_1 F_2 F_3 \|_{\left( \SI^{(l, l, l)}_{\Theta} \Phi_{\gr}^N +
      \SI^{(o, l, l)}_{\Theta} \Phi_{\gr}^N + \SI^{(l, l, o)}_{\Theta}
      \Phi_{\gr}^N + \SI^{(l, o, l)}_{\Theta} \Phi_{\gr}^N \right)}
    \\ &\qquad
      \lesssim  \| F_1 \|_{\SF_{\Theta}^{u_1} \Phi_{4 \gr}^{N - 3}} \| F_2
      \|_{\widetilde{\SF }_{\Gamma_2}^{u_2} \Phi_{4 \gr}^{N - 3}} \| F_3
      \|_{\widetilde{\SF }_{\Gamma_3}^{u_3} \Phi_{4 \gr}^{N - 3}}
    \end{aligned} \label{eq:STE:l-l-o}
  \end{equation}
  holds for any three functions \(F_1, F_2, F_3 \in L^{\infty}_{\tmop{loc}}
  (\R^{3}_{+}) {\otimes \Phi_{4 \gr}^{(N - 3)}}'\) as long as \(u_{j
  \in \{ 1, 2, 3 \}} \in [1, \infty]\) with \(\sum_{j = 1}^3 u_j^{- 1} \leq 1\).
  Furthermore the bound
  \begin{equation}
    \| F_1 F_2 F_3 \|_{\SI^{(l, o, o)}_{\Theta} \Phi_{\gr}^N} \lesssim \| F_1
    \|_{\SF_{\Gamma_1}^{u_1} \Phi_{4 \gr}^{(N - 5)}} \| F_2 F_3
    \|_{\widetilde{\SF}_{\Gamma_{\times}}^{u_{\times}} \Phi_{4 \gr}^{(N - 5)}}
    \label{eq:STE:l-o-o}
  \end{equation}
  holds for any three functions \(F_1, F_2, F_3 \in L^{\infty}_{\tmop{loc}}
  (\R^{3}_{+}) {\otimes \Phi_{4 \gr}^{(N - 3)}}'\) as long as \(u_1,
  u_{\times} \in [1, \infty]\) with \(u_1^{- 1} + u_{\times}^{- 1} \leq 1\).
\end{theorem}

The bounds \eqref{eq:STE:l-l-o} and \eqref{eq:STE:l-o-o} are referred to as
the ``single tree estimate'' and they serve as the \(L^{\infty}\) endpoints for
the interpolation procedure to obtain
\eqref{eq:iterated-Holder-type-bounds:linear} and
\eqref{eq:iterated-Holder-type-bounds:bilinear}, respectively.

The second main result of this work consists of proving uniform embedding
bounds.

\begin{theorem}
  \label{thm:uniform-embedding-bounds:linear} Let \(\Gamma = \Gamma_{(\alpha,
  \beta, \gamma)}\) be as in \eqref{eq:gamma}. Let \(\Theta\) be a open interval
  with \(B_4 \subsetneq \Theta \subsetneq B_{2^8}\), and let
  \(\Theta^{\tmop{in}}\) be a bounded open interval with \(B_{2^{- 5}} (- \gamma)
  \subsetneq \Theta^{\tmop{in}} \subsetneq B_{2^{- 3}} (- \gamma)\). Let the
  size \(\widetilde{\SF}^u_{\Gamma} \Phi_{\mf{r}}^N\) be given by
  \begin{equation}
    \widetilde{\SF}^u_{\Gamma} \Phi_{\mf{r}}^N \eqd {\beta^{\frac{1}{u}}} 
    \Gamma^{\ast} \SL^{(u, \infty)}_{\Theta} \Phi_{\gr}^N + \Gamma^{\ast}
    \SL^{(u, 2)}_{\Theta} \wpD_{\gr}^N + \Gamma^{\ast} \SL^{(u, 1)}_{(\Theta,
    \Theta^{\tmop{in}})} \dfct^N_{\gr}
    \label{eq:uniform-embedding-full-size:linear}
  \end{equation}
  The bounds \eqref{eq:embedding-bounds:uniform-iterated} hold with \(\Gamma_j
  = \Gamma\) and with \((p_j, q_j, r_j, u_j) = (p, q, r, u)\) for all \(f \in \Sch
  (\R)\) as long as \(p \in (1, \infty]\), \(q \in (\max (p', 2),
  \infty]\), and \(1 \leq u < r < q\). The implicit constant is independent of
  \(\Gamma_{(\alpha, \beta, \gamma)}\) and \(f\). 
\end{theorem}

\begin{theorem}
  \label{thm:uniform-embedding-bounds:bilinear}Let \(\Gamma_j =
  \Gamma_{(\alpha_j, \beta_j, \gamma_j)}\), \(j \in \{ 2, 3 \}\) with \((\alpha_j,
  \beta_j, \gamma_j)\) as in \eqref{eq:gamma:params}. Let \(\Theta\) and
  \(\Theta^{\tmop{in}}\) be as in \Cref{thm:uniform-embedding-bounds:linear}.
  Let the size \(\widetilde{\SF}^{u }_{\Gamma_{\times}} \Phi_{\mf{r}}^N\) is
  given by
  \begin{equation}
    \begin{aligned}[t]&
      \widetilde{\SF}^{u }_{\Gamma_{\times}} \Phi_{\mf{r}}^N \eqd
      \begin{aligned}[t]&
      \beta^{\frac{1}{u }} \Big( \Gamma_{\times}^{\ast} \SL_{(\Theta,\Theta^{\tmop{in}}) }^{(u , \infty)} (\Phi \otimes \Phi)_{\gr}^N +
      \Gamma_{\times}^{\ast} \SL_{(\Theta, \Theta^{\tmop{in}}) }^{(u , 2)}( \Phi \otimes \wpD)_{\gr}^N \Big\nobracket
      \\ & \qquad
      \Big\nobracket + \Gamma_{\times}^{\ast} \SL_{(\Theta, \Theta^{\tmop{in}}) }^{(u , 2)} ( \wpD \otimes \Phi )_{\gr}^N \Big)
      + \Gamma_{\times}^{\ast} {\SL  }^{(u, 1)}_{\Theta} ( \wpD \otimes \wpD )_{\gr}^{N } +
      \Gamma_{\times}^{\ast} \SL_{(\Theta , \Theta^{\tmop{in}})}^{(u , 1)}\dfct_{\gr}^{N }
        \end{aligned}
    \end{aligned} \label{eq:uniform-embedding-full-size:bilinear}
  \end{equation}
  The bounds \eqref{eq:embedding-bound:bilinear} hold for all \(f_2, f_3 \in
  \Sch (\R)\) as long as \((p_2, p_3) \in (1, \infty]\),
  \(p_{\times}^{- 1} = p_2^{- 1} + p_3^{- 1}\), \(q_{\times}^{- 1} < \min \left(
  1 - \frac{1}{p_2}, \frac{1}{2} \right)^{- 1} + \min \left( 1 -
  \frac{1}{p_3}, \frac{1}{2} \right)^{- 1}\), and \(1 \leq u_{\times} <
  r_{\times} < q_{\times}\). The implicit constant is independent of \(\beta \in
  (0, 1]\), that determines \((\alpha_j, \beta_j, \gamma_j)\) and \(\Gamma_j\), \(j
  \in \{ 2, 3 \}\), and of \(f_2, f_3\). 
\end{theorem}

As mentioned before, we do not prove these results directly but rather we
deduce them from the non-uniform bounds
\eqref{eq:embedding-bounds:non-uniform-iterated}, that in turn follow from
{\cite{amentaBilinearHilbertTransform2020}}. To do so we make use of covering
lemmata depending on the geometry of the outer measure space \(\R^3_+\). The
arguments involved in the proof of the above results are purely positive and
can be interpreted as boundedness properties of the map \(F \mapsto F \circ
\Gamma\) in terms of outer Lebesgue quasi-norms. As a matter of fact, with few
exceptions, the functions \(F \in \Rad (\R^{3}_{+}) {\otimes
\Phi^{\infty}_{\gr}}'\) in play need not be of the form \(F = \Emb [f]\) for some
\(f \in \Sch (\R)\). We are hopeful that appropriate refinements
can lead to completely removing such assumptions in the intermediate lemmata.

Finally, we carry out some improvements to the non-uniform embedding bounds
in {\cite{amentaBilinearHilbertTransform2020}} to obtain the following
theorem. The improvement from using the outer measure \(\mu^1_{\Theta}\) to
\(\mu^{\infty}_{\Theta}\) is common knowledge in time-scale-frequency analysis
and had to merely be codified using the outer Lebesgue framework. We are not
aware of a prior use of the quantity (size) \(\SJ_{\Theta}^u \Phi_{\gr}^N\),
appearing as a part of \(\SF^u_{\Theta} \Phi_{\gr}^N\); it represents a maximal
truncation of a singular integral operator.

\begin{theorem}
  \label{thm:non-uniform-embedding-bounds}Let \(\Theta\) be an bounded open real
  interval with \(B_1 \subsetneq \Theta \subsetneq B_{2^{10}}\). Let the size
  \(\SF^u_{\Theta} \Phi_{\gr}^N\) be given by
  \begin{equation}
    \SF^u_{\Theta} \Phi_{\gr}^N \eqd \SL^{(\infty, \infty)}_{\Theta}
    \Phi_{\gr}^N + \SL^{(u, 2)}_{\Theta} \wpD_{\gr}^N + \SL^{(u, 1)}_{\Theta}
    \dfct^N_{\gr} + \SJ_{\Theta}^u \Phi_{\gr}^N .
    \label{eq:non-uniform-embedding-full-size}
  \end{equation}
  The bounds \eqref{eq:embedding-bounds:non-uniform-iterated} hold for all \(f
  \in \Sch (\R)\) as long as \(p \in (1, \infty)\), \(q \in (\max
  (p', 2), \infty)\), and \(u \in (1, \infty)\). The implicit constant is
  independent of \(f\). 
\end{theorem}

We summarize the steps necessary to prove \Cref{thm:main-1}.

\begin{proof}{Proof of \Cref{thm:main-1}}
  Given the identity \eqref{eq:BHT-wave-packet-representation}, the claim
  follows by interpolation of bound \eqref{eq:BHT-dual-bounds} for any fixed
  \((p_1, p_2, p_3) \in (1, \infty)^3\) with \(\sum_{j = 1}^{\infty} p_{j }^{- 1}
  = 1\) boundedness result {\cite[Theorem 2]{liUniformBoundsBilinear2006}}.
  
  To show that the bound \eqref{eq:BHT-dual-bounds} holds we first apply the
  bound \eqref{eq:RN-domination-BHT} to reduce ourselves to showing the bounds
  \[ \RHS{\eqref{eq:RN-domination-BHT}} =
     \LHS{\eqref{eq:iterated-Holder-type-bounds:linear}} +
     \LHS{\eqref{eq:iterated-Holder-type-bounds:bilinear}} \lesssim \prod_{j =
     1}^3 \| f_j \|_{L^{p_j} (\R)} =
     \RHS{\eqref{eq:BHT-dual-bounds}} . \]
  Let us fix \((p_1, p_2, p_3) \in (1, \infty)^3\).
  
  Bounds \eqref{eq:embedding-bounds:uniform-iterated} hold for \(j \in \{ 2, 3
  \}\) as long as \(q_j > \max (2, p_j')\), and \(1 < u_j < r_j < q_j\). Bounds
  \eqref{eq:embedding-bounds:non-uniform-iterated} hold as long as \(q_1 > \max
  (2, p_1')\) and \(u_1 < \infty\). Bounds \eqref{eq:embedding-bound:bilinear}
  hold with \(p_{\times}^{- 1} = p_2^{- 1} + p_3^{- 2} = 1 - p_1^{- 1}\), with
  \(q_{\times}^{- 1} = q_2^{- 1} + q_3^{- 1}\), with \(r_{\times}^{- 1} = r_2^{-
  1} + r_3^{- 1} > q_{\times}^{- 1}\). From this we get that
  \[ \RHS{\eqref{eq:iterated-Holder-type-bounds:linear}} +
     \RHS{\eqref{eq:iterated-Holder-type-bounds:bilinear}} \lesssim \prod_{j =
     1}^3 \| f_j \|_{L^{p_j} (\R)} . \]
  It remains to show that bounds \eqref{eq:iterated-Holder-type-bounds:linear}
  and \eqref{eq:iterated-Holder-type-bounds:bilinear} hold for an appropriate
  choice of \(1 < u_j < r_j < q_j\). Let \(q_j = \max (2, p_j') + 4 \varepsilon\)
  and let \(r_j = 2 + 2 \varepsilon'\) and \(u_j = 2 + \varepsilon''\), \(j \in \{
  2, 3 \}\), for some small \(\varepsilon > \varepsilon' > \varepsilon'' > 0\) to
  be determined later. It then holds that \(\sum_{j = 1}^3 q_j^{- 1} > 1\) if
  \(\varepsilon > 0\) is chosen small enough. We also have \(2 < u_j < r_j < q_j\)
  for \(j \in \{ 2, 3 \}\) for all sufficiently small choices of \(\varepsilon' >
  \varepsilon'' > 0\). Fix \(u_1 = (1 - u_2^{- 1} - u_3^{- 1})^{- 1} < \infty\).
  
  Note that \(p_2^{- 1} + p_3^{- 2} = 1 - p_1^{- 1} < 1\) while \(r_2^{- 1} +
  r_3^{- 1} = \frac{1}{1 + \varepsilon'} > 1 - p_1^{- 1}\) if \(\varepsilon' >
  0\) is chosen small enough. Under these conditions the bound
  \eqref{eq:iterated-Holder-type-bounds:linear} and the bound
  \eqref{eq:iterated-Holder-type-bounds:bilinear} hold (see
  \Cref{prop:holder-bound-trilinear}).
\end{proof}

\subsection{Structure of the paper}

In \Cref{sec:wave-packet-representation} we show that the identity
\eqref{eq:BHT-wave-packet-representation} holds.

In \Cref{sec:outer-space} we discuss the functional framework used to deal
with embedded functions \(F \of \R^3_+ \times \Phi_{\gr}^{\infty} \rightarrow
\C\) for which an important example are functions of the form \(F = \Emb [f]\).
In \Cref{sec:outer-generalities} we discuss the general outer Lebesgue space
framework and prove abstract functional analytic results. We recall the
necessary facts and prove a restricted weak type interpolation theorem for
these spaces. In \Cref{sec:time-frequency-outer-lebesgue} we discuss the outer
Lebesgue space structure specifically on \(\R^3_+\) that we will be using
throughout the paper. In particular, we introduce the outer measures
\(\mu^1_{\Theta}\) and \(\mu^{\infty}_{\Theta}\), trees, and sizes. We first
discuss the sizes used to locally control the magnitude of embedded functions
\(F \of \R^3_+ \times \Phi_{\gr}^{\infty} \rightarrow \C\) and then we introduce
the size \(\SI_{\Theta}\) appearing in \eqref{eq:RN-domination-BHT} adapted to
controlling the term \(\prod_{j = 1}^3 \Emb [f_j] \circ \Gamma_j\) i.e. the
product of several functions \(F_j \of \R^3_+ \times \Phi_{\gr}^{\infty}
\rightarrow \C\). In \Cref{sec:derived-sizes} we discuss how the former sizes
must be modified when dealing with compositions with diffeomorphisms: \(F
\mapsto F \circ \Gamma\). We also define the product sizes used to deal with
functions \(F_{\times} = F_2 F_3\) appearing in
\eqref{eq:iterated-Holder-type-bounds:bilinear} and
\eqref{eq:embedding-bound:bilinear}. In \Cref{sec:localized-outer-lebesgue} we
introduce the strip outer measures \(\nu_1\) and \(\nu_{\beta}\), and we construct
\tmtextit{localized outer Lebesgue quasi-norms} based on a localization
procedure of the already defined outer Lebesgue quasi-norms. We thus give
meaning to the symbols \(\fL^{q, +}_{\mu_{\Theta}^{\infty}, \nu_1}\), \(\fL^{q ,
+}_{\mu_{\Theta}^1, \nu_{\beta}}\), and \(X^{q, r, +}_{\mu^1_{\Theta},
\nu_{\beta}}\).

In \Cref{sec:single-tree} we prove \Cref{thm:STE}, the single tree estimate.
This is the starting point for showing bounds
\eqref{eq:iterated-Holder-type-bounds:linear} and
\eqref{eq:iterated-Holder-type-bounds:bilinear} in the following section. The
technically involved proof can be seen as the motivation for the specific
choice of sizes introduced in \(\Cref{sec:outer-space}\).

We begin \Cref{sec:forest} by showing that \eqref{eq:RN-domination-BHT} holds
and by doing so we justify the choice of the size \(\SI_{\Theta}\) introduced in
\Cref{sec:time-frequency-outer-lebesgue}. We then prove bounds
\eqref{eq:iterated-Holder-type-bounds:linear} and
\eqref{eq:iterated-Holder-type-bounds:bilinear} by using the restricted weak
type interpolation statement from \Cref{sec:time-frequency-outer-lebesgue}.
The bulk of this part consists of deducing the atomic bound necessary to apply
interpolation (\Cref{prop:outer-restricted-interpolation}) from the single
tree estimate (\Cref{thm:STE}).

In \Cref{sec:non-uniform} we prove \Cref{thm:non-uniform-embedding-bounds}.
We begin by referencing available non-uniform results that apply to \(\Emb
[f_1]\). We then prove the necessary improvements consisting of controlling the
new term \(\SJ^u_{\Theta}\), that is part of the size \(\SF^u_{\Theta},\) and
improving the localized outer Lebesgue norm \(\fL^{q, +}_{\mu_{\Theta}^1,
\nu_1}\) to \(\fL^{q, +}_{\mu_{\Theta}^{\infty}, \nu_1}\). In doing so we loose
the \(p = \infty\) endpoint for the embedding.

In \Cref{sec:uniform-embeddings:linear} we prove
\Cref{thm:uniform-embedding-bounds:linear}. We begin by proving two
boundedness results for the map \(F \mapsto F \circ \Gamma\) in terms of outer
Lebesgue quasi-norms. The two results, encoded in
\Cref{prop:unif-gamma-bounds:Sex} and
\Cref{prop:unif-gamma-bounds:singular-size}, are expressed in terms of two
auxiliary Lebesgue quasi-norms, that do not appear in
\Cref{thm:uniform-embedding-bounds:linear}. Next, in
\Cref{lem:unif-derived-size-bound} we show that as long as \(F = \Emb [f]\) the
two auxiliary Lebesgue quasi-norms control the ones appearing in
\Cref{thm:uniform-embedding-bounds:linear}. Then, by appealing to
\Cref{thm:non-uniform-embedding-bounds} we deduce that \
\Cref{thm:uniform-embedding-bounds:linear} holds.
\Cref{sec:bound-exterior-size} is dedicated to proving
\Cref{prop:unif-gamma-bounds:Sex}, \Cref{sec:unif-bounds:singular} is
dedicated to proving \Cref{prop:unif-gamma-bounds:singular-size}, and
\Cref{sec:unif-derived-size-bound} is dedicated to proving
\Cref{lem:unif-derived-size-bound}. \Cref{prop:unif-gamma-bounds:Sex} and
\Cref{prop:unif-gamma-bounds:singular-size} are a consequence of two purely
geometric covering lemmata on \(\R^3_+\) that can be found in the respective
sections.

In \Cref{sec:uniform-embeddings:bilinear} we prove
\Cref{thm:uniform-embedding-bounds:bilinear} that follows along the same lines
of the proof \Cref{thm:uniform-embedding-bounds:linear}. Most of the arguments
will rely on the results of \Cref{sec:uniform-embeddings:linear}.

\subsection{Notation}

We denote by \(B_r (c)\) the Euclidean open ball of radius \(r\). The dimension
can usually be extrapolated from context. If \(c = 0\) then we use the shorthand
\(B_r\). We use the bracket notation:
\[ \langle x \rangle \eqd \left( {1 + x^2}  \right)^{\frac{1}{2}} . \]
Given two subsets \(A, B\) of a metric space we set \(\dist (A, B) \eqd \inf_{a
\in A, b \in B} \dist (a, b)\) and \(\dist (a, B) \eqd \inf_{b \in B} \dist (a,
b)\). Given \(A \subset \R\) we denote by \(| A |\) its Lebesgue measure. The
Schwartz space, denoted by \(\Sch \left( \R^d \right)\), is the space of
\(C^{\infty} \left( \R^d \right)\) functions for which
\[ \Big\| \langle x \rangle^N f \Big\|_{C^N} < + \infty \qquad \forall N \in \N . \]
We denote \(\R_+ \eqd (0, + \infty)\) and \(\R^d_+ \eqd \R^{d - 1} \times \R_+\).
Given a set \(E \subset \R^d\) we set
\[ \begin{aligned}[t]&
     E + x = \{ y + x : y \in E \}, \qquad \lambda E = \{ \lambda y \suchthat
     y \in E \} .
   \end{aligned} \]
We denote the total derivative in a variable  (\(z\))  by \(\dd_z\); the partial
derivative in the \(i^{\tmop{th}}\) coordinate is denoted by \(\partial_i\). We
use the convention on the order of variables that allows us to use more
descriptive notation:
\[ \begin{aligned}[t]&
     \partial_{\eta} \equiv \partial_{\theta} \equiv \partial_1, \qquad
     \partial_y \equiv \partial_{\zeta} \equiv \partial_2, \qquad \partial_t
     \equiv \partial_{\sigma} \equiv \partial_3 .
   \end{aligned} \]
We do not use \(\phi'\) to denote the derivative of a function \(\phi\): we prefer
\(\partial \phi\) and leave the notation \(\phi'\) to denote another function that
is not necessarily related to \(\phi\).

The Fourier transform of a function \(f\) on \(\R^d\) is given by
\[ \Fourier f \left( \FT{z} \right) \equiv \FT{f} \left( \FT{z} \right) \eqd
   \int_{\R^d} f (z) e^{- 2 \pi i z \FT{z}} \dd z. \]
If \(f\) depends on multiple variables we will denote the variable in which we
take the Fourier transform by writing \(\Fourier_z\) in place of \(\Fourier\). We
often denote the ``dual'' variable to \(z\) by \(\FT{z}\). Given a functions \(f\)
we set \(f^{\vee} (z) = f (- z)\).

We write \(A \lesssim B\) if there exists a universal constant \(C > 0\) such that
\(A \leq C B\). The constant may implicitly depend on some parameters, as judged
from context, but it will always be independent of \(\beta \in (0, 1]\) in our
problem. If we want to emphasize that the constant \(C\) may depend on a
parameter \(p\), we write \(A \lesssim_p B\). We generally allow the constant \(C\)
in our bounds to change value from one line to the next. We say that a
quantity \(A\) is bounded if \(| A | \lesssim 1\) and we say it is bounded away
from \(0\) if \(| A |^{- 1} \lesssim 1\).

Given an equality or inequality \((\star)\) we denote by \(\LHS{} (\star)\) and
\(\RHS{} (\star)\) the expressions on the left and right hand sides of
\((\star)\), respectively.

If \(K \in D' \left( \R^d \right)\) we abuse the integral to denote duality
pairing by writing \(\int_{\R^d} f K\) or \(\int_{\R^d} f (z) K (z) \dd z\). The
distribution \(\delta\) is given by
\[ \int f (z) \delta (z) \dd z = f (0) . \]
Given \(\phi \in C^1 \left( \R^d ; \R \right)\) we set
\[ \int_{\R^d} f (z) \delta (\phi (z)) \dd z = \int_{\phi^{- 1} (0) \subset
   \R^d} \frac{f (z)}{| \nabla \phi (z) |} \dd \mc{H}^{d - 1} (z) \]
as long as \(\nabla \phi \neq 0\) on \(\phi^{- 1} (0)\). The above can be
interpreted using the limiting procedure if \(f \in C^0\) by setting
\[ \int_{\R^d} f (z) \delta (\phi (z)) \dd z = \int_{\R^d} f (z)
   \delta_{\varepsilon} (\phi (z)) \dd z \]
where \(\delta_{\varepsilon} (z) = \varepsilon^{- 1} \delta_1 (z /
\varepsilon)\) for some \(\delta_1 \in C^{\infty}_c (\R)\) with
\(\int \delta_1 (x) \dd x = 1\). We do not elaborate on the formal theory of
such integrals; the interested reader is directed to
{\cite{federerGeometricMeasureTheory1996}}.

\subsection{Acknowledgments}We thank Christoph Thiele, Pavel Zorin-Kranich, and Marco Fraccaroli
for their encouragement, suggestions, and insightful questions.

We thank Walker Stern and Filippo Mazzoli for helpful stylistic comments.

This paper has been written with the GNU TeXmacs system \cite{vanderhoevenGNUTeXmacsScientific1998}

\section{Wave packet representation.}\label{sec:wave-packet-representation}

\begin{proposition}
  \label{prop:wave-packet-represetion}For any \(\gr_0 < 2^{- 20}\) there exists
  \(\phi \in \Phi^{\infty}_{\gr_0}\) i.e. \(\phi_0 \in \Sch (\R)\)
  with \(\spt \FT{\phi}_0 \subset \left[ - \gr_0, \gr_0 \right]\), such that
  identity \eqref{eq:BHT-wave-packet-representation} holds for all \(\beta \in
  (0, 1]\) and all \(f_j \in \Sch (\R)\), \(j \in \{ 1, 2, 3 \}\) with
  a real constant \(C_{\beta} > 0\) uniformly bounded from above and away from
  \(0\). The integral in \eqref{eq:BHT-wave-packet-representation} converges
  absolutely and locally uniformly in \(x\). In particular, for any \(R \gg 1\)
  and \(\varepsilon > 0\) there exists a compact set \(\mathbb{K} \subset \R^3_+\)
  such that
  \[ \begin{aligned}[t]&
       \int_{\R^3_+ \setminus \mathbb{K}} \Big| \Emb [f_1] (\eta, y, t)
       [\phi_0]  \Emb [f_2] (\beta^{- 1} \eta - (\beta t)^{- 1}, y, \beta t)
       [\phi_0] \Big. \\ &\qquad
       \times \Big. \Tr_y \Mod_{\frac{1 + \beta}{\beta} \eta - (\beta t)^{-
       1}} \Dil_{\beta t} \phi_0 (x) \Big| \dd \eta \dd y \dd t < \varepsilon
     \end{aligned} \]
  for all \(x \in B_R\).
\end{proposition}

\begin{proof}
  Using the Fourier inversion formula we can write that
  \[ \BHT_{\beta} [f_1, f_2] (x) \begin{aligned}[t]&
       = \lim_{\varepsilon \rightarrow 0} \int_{\varepsilon < | t | < \varepsilon^{- 1}} \int_{\R^2} \FT{f_1} (\xi_1) \FT{f_2} (\xi_2) e^{2
       \pi i (\xi_1 (x - t) + \xi_2 (x + \beta t))} \dd \xi_1 \dd \xi_2
       \frac{\dd t}{t}\\ &
       = - \pi i \int_{\R^2} \FT{f_1} (\xi_1) \FT{f_2} (\xi_2) e^{2 \pi i
       (\xi_1 + \xi_2) x} \sgn (\xi_1 - \beta \xi_2) \dd \xi_1 \dd \xi_2 .
     \end{aligned} \]
  We used that
  \[ \begin{aligned}[t]&
       \lim_{\varepsilon \rightarrow 0} \int_{| t | > \varepsilon} e^{- 2 \pi
       i \xi t} \frac{\dd t}{t} = - \pi i \sgn (\xi) \qquad \forall \xi \neq
       0,\\ &
       \sup_{\varepsilon > 0} \sup_{\xi} \Big| \int_{| t | > \varepsilon}
       e^{- 2 \pi i \xi t} \frac{\dd t}{t} \Big| < + \infty .
     \end{aligned} \]
  Thus
  \[
  \frac{1}{2} f_1 (x) f_2 (x) - \frac{1}{2 \pi i} \BHT_{\beta} [f_1, f_2] (x) =
  \int_{\R^2} \FT{f_1} (\xi_1) \FT{f_2} (\xi_2) e^{2 \pi i (\xi_1 +
    \xi_2) x} {\1_{\xi_1 - \beta \xi_2 > 0}}  \dd \xi_1 \dd \xi_2 .
  \]
  Let \(\FT{\phi_0} \in C^{\infty}_c \left( B_{\gr_0} \right)\) be real-valued,
  even, with \(0 \leq \FT{\phi}_0 \left( \FT{z} \right) \leq 1\), and with
  \(\FT{\phi}_0 \left( \FT{z} \right) = 1\) for \(\FT{z} \in B_{\gr_0 / 2}\). We
  show that the multiplier \({\1_{\xi_1 - \beta \xi_2 > 0}} \) admits the
  representation
  \[ \1_{(\xi_1 - \beta  \xi_2) > 0} = \frac{1}{C_{\beta}} \int_{\R_+}
     \int_{\R} \FT{\phi}_0 (t \xi_1 - \theta) \FT{\phi}_0 ( \beta t \xi_2 -
     \theta + 1) \FT{\phi}_0 (\beta t (\xi_1 + \xi_2) - (1 + \beta) \theta +
     1) \dd \theta \frac{\dd t}{t} ; \]
  the integral converges uniformly for \((\xi_1, \xi_2)\) on any compact subset
  of\\ \(\R^2 \backslash \left\{ (\beta \xi, \xi), \xi \in \R \right\}\). A change
  of variables \(\theta' = \theta - t \xi_1\) gives
  \[ m_{t, \beta} (\xi_1, \xi_2) 
  \begin{aligned}[t]&
  \eqd
  \int_{\R} \FT{\phi}_0 (t \xi_1 - \theta) \FT{\phi}_0 ( \beta t \xi_2 - \theta + 1) \FT{\phi}_0 (\beta t (\xi_1 + \xi_2) - (1 + \beta) \theta + 1) \dd \theta\\ &
  = \int_{\R} \FT{\phi}_0 (- \theta) \FT{\phi}_0 (- \tilde{\xi} - \theta  + 1) \FT{\phi}_0 (- \tilde{\xi} - (1 + \beta) \theta + 1) \dd \theta =
  m (\tilde{\xi})
     \end{aligned} \]
  with \(\tilde{\xi} \eqd t (\xi_1 - \beta \xi_2)\). The integrand vanishes
  unless \(\theta \in B_{\mf{r}_0} \cap B_{\mf{r}_0} (1 - \tilde{\xi})\) and is
  \(1\) when \(\theta \in B_{\mf{r}_0 / 4} \cap B_{\mf{r}_0 / 4} (1 -
  \tilde{\xi})\). This shows that
  \[ \frac{\mf{r}_0}{4} \1_{B_{\frac{\mf{r}_0}{4}} (1)} (\tilde{\xi}) \leq 
     m (\tilde{\xi}) \leq \mf{r}_0 \1_{B_{2 \mf{r}_0} (1)} (\tilde{\xi}) . \]
  Using the change of variables \(t' = t \tilde{\xi}\) if \(\tilde{\xi} > 0\) and
  support considerations if \(\tilde{\xi} \leq 0\) we obtain that
  \[ \int_{\R_+} m_{t, \beta} (\xi_1, \xi_2) \frac{\dd t}{t} = \int_{\R_+} m
     (t (\xi_1 - \beta \xi_2)) \frac{\dd t}{t} = C_{\phi, \beta} \1_{(0 +
     \infty)} (t (\xi_1 - \beta \xi_2)) \]
  with \(C_{\phi, \beta} \in \left( \frac{\mf{r}_0^2}{8}, 8 \mf{r}_0^2
  \right)\). Using this identity and the Fourier inversion formula we can write

  \[ \begin{aligned}[t]&
       \int_{\R^2} \FT{f}_1 (\xi_1) \FT{f}_2 (\xi_2) {\1_{(\xi_1 - \beta 
       \xi_2) > 0}}  e^{2 \pi i (\xi_1 + \xi_2) x} \dd \xi_1 \dd \xi_2\\ &
       = \frac{1}{C_{\phi, \beta}} \int_{\R^2_+} \int_{\R^2} \FT{f}_1 (\xi_1)
       \FT{f}_2 (\xi_2) {\1_{(\xi_1 - \beta  \xi_2) > 0}}  e^{2 \pi i (\xi_1 +
       \xi_2) x} \dd \xi_1 \dd \xi_2 \dd \theta \frac{\dd t}{t}\\ &
       = \begin{aligned}[t]&
         \frac{1}{C_{\phi, \beta}} \int_{\R^2_+} \int_{\R^3} \int_{\R^2} f_1
         \ast \Dil_t \Mod_{\theta} \phi_0 (y_1) f_2 \ast \Dil_{\beta t}
         \Mod_{\theta - 1} \phi_0 (y_2)\\ &\qquad
         \times \Dil_{\beta t} \Mod_{(1 + \beta) \theta - 1} \phi (y_3) e^{2
         \pi i (\xi_1 (x - y_1 - y_3) + \xi_2 (x - y_2 - y_3))} \dd \xi_1 \dd
         \xi_2 \dd y_1 \dd y_2 \dd y_3 \dd \theta \frac{\dd t}{t}
       \end{aligned}\\ &
       = \begin{aligned}[t]&
         \frac{1}{C_{\phi, \beta}} \int_{\R^2_+} \int_{\R^3} f_1 \ast \Dil_t
         \Mod_{\theta} \phi_0 (y_1) f_2 \ast \Dil_{\beta t} \Mod_{\theta - 1}
         \phi_0 (y_2) \\ &
         \quad \times \Dil_{\beta t} \Mod_{(1 + \beta) \theta - 1} \phi_0
         (y_3) \delta (x - y_1 - y_3) \delta (x - y_2 - y_3) \dd y_1 \dd y_2
         \dd y_3 \dd \theta \frac{\dd t}{t}
       \end{aligned}\\ &
       = \begin{aligned}[t]&
         \frac{1}{C_{\phi, \beta}} \int_{\R^3_+} f_1 \ast \Dil_t \Mod_{\theta}
         \phi_0 (y) f_2 \ast \Dil_{\beta t} \Mod_{\theta - 1} \phi_0 (y )
         \\ &\qquad
         \times \Tr_y \Dil_{\beta t} \Mod_{(1 + \beta) \theta - 1} \phi_0 (x)
         \dd \theta \dd y \frac{\dd t}{t}
       \end{aligned}\\ &
       = \begin{aligned}[t]&
         \frac{1}{C_{\phi, \beta}} \int_{\R^3_+} f_1 \ast \Dil_t \Mod_{\theta}
         \phi_0 (y) f_2 \ast \Dil_{\beta t} \Mod_{\theta - 1} \phi_0 (y )
         \\ &\qquad
         \times \Tr_y \Mod_{\frac{1 + \beta}{\beta} \eta - (\beta t)^{- 1}}
         \Dil_{\beta t} \phi_0 (x) \dd \eta \dd y \frac{\dd t}{t} .
       \end{aligned}
     \end{aligned} \]
  The detailed verification of the validity of the above manipulations is
  omitted. To justify the convergence claim let us first suppose that \(\spt
  \FT{f_1}, \spt \FT{f_2} \subset B_S\) for some \(S \gg 1\). For \(t >
  \frac{1}{100 S}\) and \(x \in B_R\) it holds that
  \[ \begin{aligned}[t]&
       \left| f_1 \ast \Dil_t \Mod_{\theta} \phi_0 (y) \right| + \left| f_2
       \ast \Dil_{\beta t} \Mod_{\theta - 1} \phi_0 (y ) \right|
       \lesssim_{\phi, \beta} t^{- 1} (\| f_1 \|_{L^1} + \| f_2 \|_{L^1}),\\ &
       \left| \Tr_y \Mod_{\frac{1 + \beta}{\beta} \eta - (\beta t)^{- 1}}
       \Dil_{\beta t} \FT{\phi}_0 (x) \right| \lesssim_{\phi} t^{- 1}
       {\Big\langle \frac{| y | - R}{t} \Big\rangle^{- 100}} 
     \end{aligned} \]
  while for \(t < \frac{1}{100 S}\) either \(\FT{z} \rightarrow \FT{\phi_0}
  \left( \FT{z} - \theta \right)\) or \(\FT{z} \rightarrow \FT{\phi_0} \left(
  \FT{z} - \theta - 1 \right)\) vanishes on \(B_{\frac{1}{100}}\) so
  \[ f_1 \ast \Dil_t \Mod_{\theta} \phi_0 (y) = 0 \qquad \text{or} \qquad f_2
     \ast \Dil_{\beta t} \Mod_{\theta - 1} \phi_0 (y ) = 0. \]
  The claim for \(f_1, f_2 \in \Sch (\R)\) without restrictions on
  the support of the Fourier transform follows by standard approximation
  arguments. 
\end{proof}

\section{Outer Lebesgue spaces and the time-frequency-scale
space.}\label{sec:outer-space}

In this section we discuss outer Lebesgue spaces and quasi-norms in general
and, specifically, their details in the case of the time-scale-frequency space
\(\R^3_+\). We begin by discussion embedded functions i.e. functions of the form
\(\R^3_+ \times \Sch (\R) \rightarrow \Emb [f] (\eta, y, t)
[\phi]\) that are the main actors of our functional analytic framework.

We define the Banach space of band-limited wave packets of order \(N \in \N\)
and frequency radius \(\gr > 0\) by setting
\begin{equation}
  \Phi^N_{\gr} \eqd \left\{ \phi \in C^{\infty} (\R) \st \hspace{0.27em}
  \FT{\phi} \in C^N (\R), \spt \FT{\phi} \subseteq
  \overline{B_{\gr} } \right\} \label{eq:def:wave-packets}
\end{equation}
with an associated norm \(\| \phi \|_{\Phi_{\mf{r}}^N} = \left\| \FT{\phi}
\right\|_{C^N}\); we set \(\Phi^{\infty}_{\mf{r}} \eqd \bigcap_{N \in \N}
\Phi_{\mf{r}}^N\). We say that \(F \in \ms{R} (\R^3_+) \otimes^k
{\Phi_{\gr}^{\infty}}'\) if the map \(\left( \Phi^{\infty}_{\gr} \right)^k \ni
(\phi_1, \ldots, \phi_k) \mapsto F (\cdot) [\phi_1, \ldots \phi_k] \in \ms{R}
(\R^3_+)\) is \(k\)-linear and is valued in signed Radon measures on \(\R^3_+\).
Recall that a Radon measure is a measure that is locally finite and inner and
outer regular. In the special case of \(k = 0\), the space \(\ms{R} (\R^3_+)
\otimes^0 {\Phi_{\gr}^{\infty}}' = \ms{R} (\R^3_+)\) is simply to the space of
\(\C\)-valued Radon measures on \(\R^3_+\). When \(k = 1\) we omit \(k\) from the
notation and simply write \(\Rad (\R^{3}_{+}) {\otimes
\Phi_{\gr}^{\infty}} '\). We denote by \(L^{\infty}_{\tmop{loc}} (\R^{3}_{+}) \otimes^k {\Phi_{\gr}^{\infty}} '\) the subspaces of \(\Rad (\R^{3}_{+}) \otimes^k {\Phi_{\gr}^{\infty}} '\) given by functions such that
for every compact \(K \subset \R^3_+\) there exists \(N_K \in \N\) such that
\[ \begin{aligned}[t]&
     {\sup_{\tmscript{
       \phi_1, \ldots, \phi_k}}}
 \Big\| F (\eta, y, t) [\phi_1, \ldots, \phi_k]  \Big\|_{L^{\infty} (K)} < \infty,
   \end{aligned} \]
where the upper bound is taken over \(\phi_j \in \Phi_{\gr}^{\infty}\) with \(\|
\phi_j \|_{\Phi_{\gr}^{N_K}} \leq 1\), \(j \in \{ 1, \ldots, k \}\). For any
fixed \(\phi_1, \ldots, \phi_k \in \Phi^{\infty}_{\gr}\) we have identified the
function \(F (\eta, y, t) [\phi_1, \ldots, \phi_k]\) with the measure \(F (\eta,
y, t) [\phi_1, \ldots, \phi_k] \dd \eta \dd y \dd t\). According to this
notation, \(\Emb [f] \in L^{\infty}_{\tmop{loc}} (\R^{3}_{+}) {\otimes
\Phi_{\gr}^{\infty}} '\) for any \(f \in \Sch (\R)\). If \(\gr' >
\gr\) then by restricting we can identify a function \(F \in \Rad (\R^{3}_{+}) {\otimes \Phi^{\infty}_{\gr'}}'\) with a function \(F \in \Rad (\R^{3}_{+}) {\otimes \Phi^{\infty}_{\gr}}'\). Analogously to the above we
can define \(\Rad \left( \XO \right) {\otimes \Phi^{\infty}_{\gr'}}'\) for any
locally compact topological space \(\XO\).

\subsection{Generalities of outer measure
spaces}\label{sec:outer-generalities}

We now provide a with brief overview of the construction and properties of
general outer Lebesgue spaces. Specific instances for the space \(\R^3_+\) of
all the objects introduced here will be discussed in the next section. The
reader may want to always assume that \(\XO = \R^3_+\) as that is the context in
which we will be using the subsequent construction. Generally, let \(\XO\) be a
locally compact topological space. An outer measure \(\mu\) on \(\XO\) is a
function \(\mu : 2^{\XO} \rightarrow [0, + \infty]\) that is
\(\sigma\)-subadditive. A \(\sigma\)-generating collection on \(\XO\) is a
collection \(\BO \subset \Bor (\mathbb{X})\) of Borel sets \(B\) with \(\mu (B) <
\infty\) such that \(\XO\) can be written as a union of countably many elements
of \(\BO\) and such that \(\mu (E)\) for any \(E \subset \XO\) can be recovered by
countable coverings of elements of \(\BO\):
\begin{equation}
  \mu (E) = \inf \Big\{ \sum_{B \in \mc{B} \subset \BO} \mu (B) \suchthat E
  \supseteq \bigcup_{B \in \mc{B} \subset \BO} B \Big\} .
  \label{eq:measure-recovery}
\end{equation}
We denote countable unions of \(\BO\) by
\[ \BO^{\cup} \eqd \Big\{ \bigcup_{n \in \N} B_n \text{ with } B_n \in
   \mathbb{B} \Big\} . \]
Let \(\Rad \left( \XO \right) \otimes^k {\Phi_{\gr}^{\infty}} '\) be defined as
in the beginning of \Cref{sec:outer-space}, that we refer to, by a light abuse
of lexicon, as ``functions'' on \(\XO\). A size is a positive norm-like
functional \(\SO \of \Rad \left( \XO \right) \otimes^k {\Phi_{\gr}^{\infty}} '
\rightarrow [0, + \infty]\) with the following properties. First of all it
possesses the properties of a quasi-norm:
\begin{equation}
  \| \lambda F \|_{\SO} = | \lambda | \| F \|_{\SO} \qquad \forall F \in \Rad
  \left( \XO \right) \otimes^k {\Phi_{\mf{r}}^{\infty}}',
  \label{eq:size-homogeneity}
\end{equation}
and there exists \(\tmop{Tri}_{\SO} \geq 1\) such that
\begin{equation}
  \| F_1 + F_2 \|_{\SO} \leq \tmop{Tri}_{\SO} \left( \| F_1 \|_{\SO} + \| F_2
  \|_{\SO} \right), \quad \forall F_1, F_2 \in \Rad \left( \XO \right)
  \otimes^k {\Phi_{\mf{r}}^{\infty}}' . \label{eq:size-triangle}
\end{equation}
We do not require separation and allow functions \(F \neq 0\) with \(\| F
\|_{\SO} = 0\). We require monotonicity in the following form:
\begin{equation}
  \left\| \1_{\XO \backslash E} F \right\|_{\SO} \leq \| F \|_{\SO}, \quad
  \forall F \in \Rad \left( \XO \right) \otimes^k {\Phi_{\mf{r}}^{\infty}}',
  \quad \forall E \in \BO^{\cup} . \label{eq:size-monotonicity}
\end{equation}
Finally, we ask that \(\SO\) be regular: if \(E_n^+ \in \BO^{\cup}\) is an
increasing sequence of sets and \(E^-_n \in \BO^{\cup}\) is a decreasing
sequences of sets such that \(\bigcup_n E_n^+ \setminus E_n^- = \XO\) then
\begin{equation}
  \| F \|_{\SO} \lesssim \liminf_n \left\| \1_{E_n^+ \setminus E_n^-} F
  \right\|_{\SO} . \label{eq:size-regularity}
\end{equation}

We call the tuple \(\left( \XO, \BO, \mu, \SO \right)\) an outer Lebesgue space.
Sizes and measures can be combined via positive linear combinations. Two sizes
\(\SO\) and \(\SO'\) are equivalent if \(\|F\|_{\SO} \approx \|F\|_{\SO'}\) and
similarly two outer measures \(\mu\) and \(\mu'\) are equivalent if \(\mu (E)
\approx \mu' (E)\) for all \(E \in \Bor (\XO)\).

Given an outer Lebesgue space, we define outer Lebesgue quasi-norms by
combining the notions of size and outer measures. The \(\lambda\)-superlevel
measure of a function \(F \in \Rad \left( \XO \right) \otimes^k
{\Phi_{\mf{r}}^{\infty}}'\) is given by
\begin{equation}
  \mu (\|F\|_{\SO} > \lambda) \eqd \inf \left\{ \mu (E) : E \in \BO^{\cup},
  \hspace{0.27em} \| \1_{\XO \setminus E} F\|_{\SO} \leq \lambda \right\} .
  \label{eq:superlevel-measure}
\end{equation}
This quantity, in general, is not a measure of any specific set, but rather it
is an interpolating quantity between the outer measure \(\mu\) and the size
\(\SO\). The notion above encodes the optimal way of splitting \(F\) into two
parts: \(F = \1_E F + \1_{\XO \setminus E} F\) with \(E \in \BO^{\cup}\) so that
\(\1_{\XO \setminus E} F\) has controlled size and \(\1_E F\) is supported on a
set of controlled measure.

Given \(F \in \Rad \left( \XO \right) \otimes^k {\Phi_{\mf{r}}^{\infty}}'\),
its outer Lebesgue quasi-norms and their weak variants are given by
\begin{equation}
  \begin{aligned}[t]&
    \|F\|_{L_{\mu}^p \SO} \eqd \left( \int_0^{\infty} \lambda \mu (\|F\|_{\SO}
    > \lambda^{\frac{1}{p}}) \hspace{0.17em} \frac{\dd{} \lambda}{\lambda}
    \right)^{\frac{1}{p}},\\ &
    \|F\|_{L_{\mu}^{p, \infty} \SO} \eqd {\left( \sup_{\lambda > 0} \lambda
    \mu (\|F\|_{\SO} > \lambda^{\frac{1}{p}}) \right)^{\frac{1}{p}}} , \qquad
    \forall p \in (0, \infty)\\ &
    \|F\|_{L_{\mu}^{\infty} \SO} \eqd \|F\|_{\SO} .
  \end{aligned} \label{eq:superlevel-Lp}
\end{equation}
The quantities above vanish only if \(\mu (\|F\|_{\SO} > 0) = 0\) so we quotient
out such functions similarly to how one does in classical Lebesgue theory.
Using the monotonicity property \eqref{eq:size-monotonicity}, one can check
that the quantities above are indeed quasi-norms with
\[ \| F_1 + F_2 \|_{L^p_{\mu} \SO} \leq \tmop{Tri}_{L^p_{\mu} \SO} \left( \|
   F_1 \|_{L^p_{\mu} \SO} + \| F_2 \|_{L^p_{\mu} \SO} \right), \qquad
   \tmop{Tri}_{L^p_{\mu} \SO} \leq 2^{1 + \frac{1}{p}} \tmop{Tri}_{\SO} . \]
The quasi-norms \(L^p_{\mu} \SO\) are themselves regular for any \(p \in (0,
\infty)\) since
\[ \mu (\| \1_{E_n^+ \setminus E_n^-} F\|_{\SO} > 2 \lambda) \lesssim
   \liminf_n \mu (\| \1_{E_n^+ \setminus E_n^-} F\|_{\SO} > \lambda) \]
for any sequence of sets \(E_n^+, E^-_n \in \BO^{\cup}\) (respectively
increasing and decreasing) such that \(\bigcup_n E_n^+ \setminus E_n^- =
\XO\).

We write \(F \in L^p_{\mu} \SO\) if \(F \in \Rad \left( \XO \right) \otimes^k
{\Phi_{\gr}^{\infty}}'\) and \(\| F \|_{L^p_{\mu} \SO} < \infty\); the space
\(\XO\) and the values of \(k\) and \(\mf{r}\) are encoded in the domains for the
size \(\SO\) and the measure \(\XO\) and thus can be deduced without being
explicitly specified.

We now summarize some of the properties of outer Lebesgue quasi-norms relevant
for our work. First of all, outer Lebesgue spaces are a quasi-normed space; in
general, if \(\| \cdot \| : X \rightarrow [0, + \infty]\) is a function on a
vector space \(X\) that satisfies
\[ \| v_1 + v_2 \|  \leq \tmop{Tri}_{\| \cdot \|} (\| v_1 \| + \| v_2 \|) \]
then for any \(K \in \N\) one has
\begin{equation}
  \Big\| \sum_{k = 1}^K v_k \Big\| \leq C_{\| \cdot \|} \Big( \sum_{k =
  1}^K k^{C_{\| \cdot \|}} \| v_k \| \Big) \label{eq:multi-triangle}
\end{equation}
for some \(C_{\| \cdot \|} > 0\) depending only on \(\tmop{Tri}_{\| \cdot \|}\).
In particular, for any \(\varepsilon > 0\) there exists \(C_{\| \cdot \|,
\varepsilon}\) such that
\[ \left\| \sum_{k = 1}^K v_k \right\| \leq C_{\| \cdot \|, \varepsilon}
   \left( \sum_{k = 1}^K 2^{\varepsilon k} \| v_k \| \right) . \]
We omit the proof of this fact; it relies on a dyadic partitioning argument.

Outer Lebesgue quasi-norms satisfy many of the properties that hold for
classical Lebesgue spaces.

\begin{proposition}
  \label{prop:outer-properties} Let \((\XO, \BO, \mu, \SO)\) be an outer space.
  If \(\bar{\mu}\) is another outer measure then for any \(p \in (0, \infty)\) it
  holds that
  \begin{equation}
    \| F \|_{L^{p }_{\bar{\mu}} \SO} \leq \sup_{E \in \BO^{\cup}} \left(
    \frac{\bar{\mu} (E)}{\mu (E)} \right)^{\frac{1}{p}} \| F \|_{L^p_{\mu}
    \SO}  \label{eq:outer-measure-Lp-comparison}
  \end{equation}
  (the bound is non-trivial only if \(\mu (E) = 0 \Longrightarrow \bar{\mu} (E)
  = 0\)). The Chebyshev inequality holds: for any \(p \in (0, \infty)\) it holds
  that
  \begin{equation}
    \| F \|_{L^{p, \infty}_{\mu} \SO} \lesssim \| F \|_{L^p_{\mu} \SO} \qquad
    \forall F \in L^p_{\mu} \SO . \label{eq:outer-chebychev}
  \end{equation}
  On finite measure spaces, outer Lebesgue quasi-norms are nested: for any \
  \(p_-, p_+ \in (0, \infty]\) with \(p_- \leq p_+\) it holds that
  \begin{equation}
    \| F \|_{L^{p_-}_{\mu} \SO} \lesssim \mu \left( \| F \|_{\SO} > 0
    \right)^{\frac{1}{p_-} - \frac{1}{p_+}} \| F \|_{L^{p_+, \infty}_{\mu}
    \SO}  \label{eq:outer-Lp-finite-support}
  \end{equation}
  for any \(F \in L^p_{\mu} \SO \) with an implicit constant independent of \(F\).
  Outer Lebesgue quasi-norms are log-convex in the exponent \(p\). Given any any
  \(p_0, p_1 \in (0, \infty]\) and any \(\theta \in (0, 1)\) it holds that
  \begin{equation}
    \| F \|_{L^{p_{\theta}}_{\mu} \SO} \leq C_{p_0, p_1, \theta} \| F
    \|_{L^{p_0, \infty}_{\mu} \SO}^{1 - \theta} \| F \|_{L^{p_{1,
    \infty}}_{\mu} \SO}^{\theta}, \qquad \forall F \in L^{p_0, \infty}_{\mu}
    \SO \cap L^{p_1, \infty}_{\mu} \SO . \label{eq:outer-log-convexity}
  \end{equation}
  with \(p_{\theta}^{- 1} \eqd (1 - \theta) p_0^{- 1} + \theta p_1^{- 1}\).
\end{proposition}

We omit the proof of this statement as it the same as for classical Lebesgue
spaces and relies on the super-level set representation
\eqref{eq:superlevel-Lp} of the (outer) integral. The properties above
motivate the definition of a \(L^p_{\mu} \SO\)-atom and an associated
decomposition statement for functions \(F \in L^p_{\mu} \SO\).

\begin{proposition}[Atomic decomposition]
  \label{prop:atomic-decomposition} Let \((\XO, \BO, \mu, \SO)\) be an outer
  space. Given \(p \in (0, \infty]\), an \(L^p_{\mu} \SO\)-atom is a function \(F
  \in L^{\infty}_{\mu} \SO\) such that
  \[ \mu \left( \| F \|_{\SO} > 0 \right)^{\frac{1}{p}} \| F
     \|_{L^{\infty}_{\mu} \SO} \leq 1. \]
  Given any \(F \in L^p_{\mu} \SO\) there exists a decreasing sequence of sets
  \(A_k \subset A_{k - 1} \subset \XO\) such that
  \[ \lim_N \Big\| F - \sum_{k = - N}^N \1_{\Delta A^k} F
     \Big\|_{L^p_{\mu} \SO} = 0 \]
  with \(\Delta A_k \eqd A_{k - 1} \backslash A_k\). The sets \(\Delta A_k\) are
  such that
  \[ \left\| \1_{\Delta A_k} F \right\|_{\SO} \leq 2^{\frac{k}{p}}, \qquad
     \sum_{k \in \Z} \mu (\Delta A_k) 2^k \approx \| F \|_{L^p_{\mu} \SO}^p,
  \]
  and \(A_k = \emptyset\) if \(\| F \|_{\SO} < 2^{k / p} .\) In particular, for
  some \(C \gtrsim 1\), the functions \(\frac{1}{C} \frac{\1_{\Delta A_k} F}{\| F
  \|_{L^p_{\mu} \SO}}\) are \(L^p_{\mu} \SO\)-atoms. Furthermore, for any \(N \in
  \N\)
  \[ \begin{aligned}[t]&
       F - \Big( \sum_{k = - N}^N \1_{\Delta A^k} F \Big) = \1_{A_N} F +
       \1_{\XO \setminus A_{- N - 1}} F,
     \end{aligned} \]
  and the bounds
  \[ \left\| \1_{\XO \setminus A_{- N - 1}} F \right\|_{\SO} \leq 2^{-
     \frac{(N + 1)}{p}}, \qquad \mu \left( \left\| \1_{A_N} F \right\|_{\SO} >
     0 \right) \leq 2^{- N} \| F \|_{L^p_{\mu} \SO}^p, \]
  hold, as well as

  \[ \lim_N \left\| \1_{\XO \setminus A_{- N - 1}} F \right\|_{L^p_{\mu} \SO}
     = \lim_N \left\| \1_{A_N} F \right\|_{L^p_{\mu} \SO} = 0. \]
  \[  \]
  As a consequence, finite linear combination of \(L^p_{\mu} \SO\)-atoms and
  functions \(F \in L^{\infty}_{\mu} \SO\) with \(\mu \left( \| F \|_{\SO} > 0
  \right) < \infty\) are dense in \(L^p_{\mu} \SO\).
\end{proposition}

\begin{proof}
  Choose \(\tilde{A}_k \subset \XO\) such that
  \[ \left\| \1_{\XO \setminus \tilde{A}^k} F \right\|_{\SO} \leq
     2^{\frac{k}{p}}, \qquad \mu (\tilde{A}_k) \leq 2 \mu \left( \| F \|_{\SO}
     > 2^{\frac{k}{p}} \right) . \]
  Setting \(A_k \eqd \bigcup_{l \geq k} \tilde{A}_l\) we guarantee that \(A_k
  \subset A_{k - 1}\) and we have that \
  \[ \sum_{k \in \Z} \mu (A_k) 2^k \leq\begin{aligned}[t]&
        \sum_{k \in \Z} \sum_{l \geq k} \mu
       (\tilde{A}^l) 2^k \leq \sum_{l \in \Z} \mu (\tilde{A}_l) 2^l \sum_{k
       \leq l} 2^{k - l}\\ &
       \leq 2 \sum_{l \in \Z} \mu (\tilde{A}_l) 2^l \leq 4 \sum_{l \in \Z} \mu
       \left( \| F \|_{\SO} > 2^{\frac{l}{p}} \right) 2^l\\ &
       \leq 4 \int_0^{\infty} \lambda \mu \left( \| F \|_{\SO} >
       \lambda^{\frac{1}{p}} \right) \frac{\dd \lambda}{\lambda} \leq 4 \| F
       \|_{L^p_{\mu} \SO}^p .
     \end{aligned} \]
  Since \(\Delta A_k = A_{k - 1} \backslash A_k\), for any \(N \in \N\) it holds
  that
  \[ F = \sum_{k = - N}^N \1_{\Delta A_k} F + \1_{A_N} F + \1_{\XO \setminus
     A_{- N - 1}} F. \]
  To show the claims on outer Lebesgue norm convergence notice that
   \[ \Big\| \1_{A_N} F \Big\|_{L^p_{\mu} \SO}^p \leq \int_{2^N}^{\infty}
     \lambda \mu ( \| F \|_{\SO} > \lambda^{\frac{1}{p}} )
     \frac{\dd \lambda}{\lambda} + 2^N \mu (A_N) \]
  while
  \[ \Big\| \1_{\XO \setminus A_{- N - 1}} F \Big\|_{L^p_{\mu} \SO}^p \leq
     \int_0^{2^{- N - 1}} \hspace{-2em}\lambda \mu ( \| F \|_{\SO} >
     \lambda^{\frac{1}{p}} ) \frac{\dd \lambda}{\lambda} . \]
\end{proof}

\begin{proposition}[Restricted weak type interpolation for outer measure
spaces]
  \label{prop:outer-restricted-interpolation} Let \((\XO_j, \BO_j, \mu_j,
  \SO_j)\) for each \(j \in \{0, 1, \ldots, M\}\) be outer spaces and let
  \[ \Pi : \prod_{j = 1}^M L^{\infty}_{\mu_j} \SO_j \to \Rad \left( \XO_0
     \right) \otimes^k {\Phi^{\infty}_{\gr}}' \]
  be an \(M\)-linear map. Let \(\Omega_{\tmop{restr}} \subset [0, 1]^M\) be the
  set of points \((p_1^{- 1}, \ldots, p_M^{- 1}) \in [0, 1]^M\) such that the
  restricted bounds
  \begin{equation}
    \Big\| \Pi (F_1, \ldots, F_M)\Big\|_{L^{p_0, \infty}_{\mu_0} \SO_0} \leq
    C_{\tmop{restr}} (\vec{p}) \prod_{j = 1}^M \mu_j (\| F_j \|_{\SO_j} >
    0)^{\frac{1}{p_j}} \|F_j \|_{\SO_j} \qquad \forall F_j \in
    L^{\infty}_{\mu_j} \SO_j \label{eq:abstract-endpoint-holder}
  \end{equation}
  hold with \(p_0^{- 1} \eqd \sum_{j = 1}^M p_j^{- 1}\) for some constant
  \(C_{\tmop{restr}} (\vec{p}) > 0\) independent of \(F_j\). Then
  \(\Omega_{\tmop{restr}}\) is convex and if \((\bar{p}_1^{- 1}, \ldots,
  \bar{p}_M^{- 1}) \in [0, 1]^M\) is an interior point of
  \(\Omega_{\tmop{restr}}\) then the strong bounds
  \begin{equation}
    \| \Pi (F_1, \ldots, F_M)\|_{L_{\mu_0}^{\bar{p}_0} \SO_0} \leq C
    (\overrightarrow{\bar{p}}) \prod_{j = 1}^M \|F_j \|_{L_{\mu_j}^{\bar{p}_j}
    \SO_j} \qquad \forall F_j \in L_{\mu_j}^{\bar{p}_j} \SO_j
    \label{eq:abstract-Lq-holder}
  \end{equation}
  hold with \(\bar{p}_0^{- 1} \eqd \sum_{j = 1}^M \bar{p}_j^{- 1}\) and some \(C
  (\overrightarrow{\bar{p}}) \geq 0\) independent of \(F_j\).
\end{proposition}

The proof of this fact is a direct adaptation of the proof of restricted weak
type interpolation for multilinear forms on classical Lebesgue spaces.

\begin{proof}
  Let us first show that \(\Omega_{\tmop{restr}}\) is convex. Let \((p_{1, 0}^{-
  1}, \ldots, p_{M, 0}^{- 1}), (p_{1, 1}^{- 1}, \ldots, p_{M, 1}^{- 1}) \in
  \Omega_{\tmop{restr}}\) and let \(\bar{p}_{j }^{- 1} = (1 - \theta) p_{j,
  0}^{- 1} + \theta p_{j, 1}^{- 1}\) for a fixed \(\theta \in (0, 1)\). By log
  convexity of outer Lebesgue spaces (\eqref{eq:outer-log-convexity} in
  \Cref{prop:outer-properties}) it holds that
  \[
  \begin{aligned}[t]
  &\Big\| \Pi (F_1, \ldots, F_M)\Big\|_{L_{\mu_0}^{\bar{p}_0, \infty} \SO_0} \hspace{-2em}
  \begin{aligned}[t]& \hspace{2em}\leq
       \Big\| \Pi (F_1, \ldots, F_M)\Big\|_{L_{\mu_0}^{\bar{p}_0} \SO_0}\\ &
       \leq C_{\vec{p}_0, \vec{p}_1, \theta} \Big\| \Pi (F_1, \ldots,
       F_M)\Big\|_{L_{\mu_0}^{p_{0, 0}, \infty} \SO_0}^{(1 - \theta)} \Big\| \Pi (F_1,
       \ldots, F_M)\Big\|_{L_{\mu_0}^{p_{1, 0}, \infty} \SO_0}^{\theta}\\ &
       \leq C_{\vec{p}_0, \vec{p}_1, \theta} C_0^{(1 - \theta)} C_1^{\theta}
       \prod_{j = 1}^M \mu_j (\| F_j \|_{\SO_j} > 0)^{\frac{1 - \theta}{p_{j,
       0}} + \frac{\theta}{p_{j, 1}}} \|F_j \|_{\SO_j},\\ &
 \end{aligned}
  \\ &      C_0 \eqd C_{\tmop{restr}} ((p_{1, 0}, \ldots, p_{M, 0})) \qquad C_1
       \eqd C_{\tmop{restr}} ((p_{1, 1}, \ldots, p_{M, 1})),
  \end{aligned}
 \]
  thus showing that \((\bar{p}_1^{- 1}, \ldots, \bar{p}_M^{- 1}) \in
  \Omega_{\tmop{restr}}\), as required. We need to show that the bounds above
  hold with the strong outer Lebesgue norms instead.
  
  Assume that the factors on the right hand side of
  \eqref{eq:abstract-Lq-holder} are finite and non-zero, for otherwise there
  is nothing to prove. By homogeneity we may assume that \(\|F_j
  \|_{L^{\bar{p}_j}_{\mu_j} \SO_j} = 1\) for each \(j \in \{1, \ldots, M\}\).
  Using the atomic decomposition of \(F_j\) (proposition
  \ref{prop:atomic-decomposition}) let us represent
  \[ F_j = \sum_{k \in \Z} F_j^k \quad \text{with } F_j^k \eqd \1_{\Delta
     A_j^k} F_j \]
  and
  \[ \| F_j^k \|_{\SO_j} \leq 2^{\frac{k}{\bar{p}_j}} \qquad \sum_{k \in \Z}
     \mu_j (\Delta A^k_j) 2^k \approx 1. \]
  We reason under the assumptions that the above sums are finite; the full
  statement then follows by a limiting argument that we omit. We represent any
  \(\vec{k} = (k_1, \ldots, k_M) \in \Z^M\) as
  \[ k_i = n + d^{- 1} \sum_{j = 1}^{M - 1} l_j \mf{s}_j \mf{w}_{j, i}, \]
  with \(n \in \frac{1}{M} \Z\), for some \(\vec{l} = (l_1, \ldots, l_{M - 1})
  \in \N^{M - 1}\), \(\overrightarrow{\mf{s}} = \left( \mf{s}_1, \ldots,
  \mf{s}_{M - 1} \right) \in \{ - 1, 1 \}^{M - 1}\), and with \
  \(\overrightarrow{\mf{w}}_1, \ldots, \overrightarrow{\mf{w}}_{M - 1} \in
  \Z^M_0 \subset \Z^M\) a pairwise orthogonal collection of vectors that we
  denote by \(\overrightarrow{\mf{w}}_j \eqd \left( \mf{w}_{j, i} \right)_{i
  \in \{ 1, \ldots, M \}}\), \(j \in \{ 1, \ldots, M - 1 \}\), where
  \[ \Z^M_0 \eqd \Big\{ \overrightarrow{\mf{\vec{w}}} = ( \mf{w}_1,
     \ldots, \mf{w}_{M - 1} ) \in \Z^M \suchthat \sum_{j = 1}^M \mf{w}_j
     = 0  \Big\} . \]
  This is always possible by choosing \(n = M^{- 1} \sum_{i = 1}^M k_i\) and \(d
  \in \N\) large enough so that the map \(\Z^{M - 1} \ni (l_1, \ldots, l_{M -
  1}) \mapsto d^{- 1} \sum_{j = 1}^{M - 1} l_j \overrightarrow{\mf{w}}_j\) is
  surjective onto the set \(M^{- 1} {\Z^M_0} \). Using the above representation
  of \(\vec{k}\), we can decompose
  \[ \begin{aligned}[t]&
       \Pi (F_1, \ldots, F_M) = {\sum_{\overrightarrow{\mf{s}} \in \{ - 1, 1
       \}^{M - 1}}}_{\mf{\overrightarrow{}}} \Pi^{\overrightarrow{\mf{s}}}
       (F_1, \ldots, F_M),\\ &
       \begin{aligned}[t]&
         \Pi^{\overrightarrow{\mf{s}}} (F_1, \ldots, F_M) \eqd\\ &
         \qquad \sum_{n \in \frac{1}{M} \Z} \sum_{(l_1, \ldots, l_{M - 1}) \in
         \N^M} \Pi \Big(F_1^{n + d^{- 1} \sum_{j = 1}^{M - 1} l_j \mf{s}_j
         \mf{w}_{j, 1}}, \ldots, F_M^{n + d^{- 1} \sum_{j = 1}^{M - 1} l_j
         \mf{s}_j \mf{w}_{j, M}}\Big) .
       \end{aligned} 
     \end{aligned} \]
  In the above notation we set \(F_i^{k_i} = 0\) if \(k_i \notin \Z\). Let us fix
  \(\overrightarrow{\mf{s}}\) and show that \eqref{eq:abstract-Lq-holder} holds
  for \(\Pi^{\overrightarrow{\mf{s}}}\) in place of \(\Pi \). There exists a point
  \(\vec{p} = (p_1, \ldots, p_M) \in (0, \infty)\) with \((p_1^{- 1}, \ldots,
  p_M^{- 1})\) in the interior of \(\Omega_{\tmop{restr}}\) such that \(\sum_{i =
  1}^M {p_i} ^{- 1} = \bar{p}_0^{- 1}\) and such that
  \[ \mf{s}_j \sum_{i = 1}^M (p_i^{- 1} - \bar{p}_i^{- 1}) \mf{w}_{j, i} > 0
     \qquad \forall j \in \{ 1, \ldots, M - 1 \} . \]
  Let us choose an \(\varepsilon \in (0, 1)\) small enough to have
  \[ \begin{aligned}[t]&
       ((1 \pm \varepsilon)^{- 1} p_1^{- 1}, \ldots, (1 \pm \varepsilon)^{- 1}
       p_M^{- 1}) \in \Omega_{\tmop{restr}},\\ &
       \mf{s}_j \sum_{i = 1}^M (p_i^{- 1} - (1 \pm \varepsilon)^{- 1}
       \bar{p}_i^{- 1}) \mf{w}_{j, i} > \varepsilon \qquad \forall j \in \{ 1,
       \ldots, M - 1 \} .
     \end{aligned} \]
  Using the restricted weak boundedness assumptions with the exponents above,
  for any \(\lambda \in (0, \infty)\) it holds that
  \[ \begin{aligned}[t]&
       \mu_0 \Big( \Big\| \Pi (F_1^{n + d^{- 1} \sum_{j = 1}^{M - 1} l_j
       \mf{s}_j \mf{w}_{j, 1}}, \ldots, F_M^{n + d^{- 1} \sum_{j = 1}^{M - 1}
       l_j \mf{s}_j \mf{w}_{j, M}}) \Big\|_{\SO_0} >
     \lambda^{\frac{1}{\bar{p}_0}} \Big)\\ &\qquad
       \lesssim \lambda^{- (1 \pm \varepsilon)} \prod_{i = 1}^M \mu_i \Big(
       \| F_i \|_{\SO_i} \Big)^{\frac{\bar{p}_0 (1 \pm \varepsilon)}{p_i (1
       \pm \varepsilon)}} 2^{\frac{\bar{p}_0 (1 \pm \varepsilon)}{\bar{p}_i}
       \big( n + d^{- 1} \sum_{j = 1}^{M - 1} l_j \mf{s}_j \mf{w}_{j, i}
       \big)}\\ & \qquad
     =
     \begin{aligned}[t]
     & \lambda^{- (1 \pm \varepsilon)} \prod_{i = 1}^M \mu_i \Big( \Big\|
       F_i^{n + d^{- 1} \sum_{j = 1}^{M - 1} l_j \mf{s}_j \mf{w}_{j, i}}
       \Big\|_{\SO_i} \Big)^{\frac{\bar{p}_0}{p_i}}
       2^{\frac{\bar{p}_0}{p_i} \big( n + d^{- 1} \sum_{j = 1}^{M - 1} l_j
       \mf{s}_j \mf{w}_{j, i} \big)} \\ &\qquad
       \times 2^{\pm \varepsilon n} 2^{- \bar{p}_0 d^{- 1} \sum_{j = 1}^{M -
       1} l_j \mf{s}_j \sum_{i = 1}^M \mf{w}_{j, i} (p_i^{- 1} - (1 \pm
       \varepsilon) \bar{p}_i^{- 1})}     
     \end{aligned}
     \end{aligned} \]
  By the quasi-triangle inequality on \(\| \cdot \|_{\SO_0}\), for a
  sufficiently large \(C > 1\) it holds that
  \[ \begin{aligned}[t]&
       \begin{aligned}[t]
         \mu_0 \Big( \Big\| \sum_{n \in \frac{1}{M} \Z} \sum_{(l_1, \ldots,
         l_{M - 1}) \in \N^M} \Pi (&F_1^{n + d^{- 1} \sum_{j = 1}^{M - 1} l_j
         \mf{s}_j \mf{w}_{j, 1}}, \ldots \Big. \Big. \\ &\qquad 
         \Big. \Big. \ldots, F_M^{n + d^{- 1} \sum_{j = 1}^{M - 1} l_j
         \mf{s}_j \mf{w}_{j, M}}) \Big\|_{\SO_0} > C
         2^{\frac{n_0}{{\bar{p}_0} }} \Big)
       \end{aligned}\\ &
       \lesssim \begin{aligned}[t]
       \sum_{n \in \frac{1}{M} \Z} &\sum_{(l_1, \ldots, l_{M - 1}) \in
       \N^M} \mu_0 \Big( \Big\| \Pi (F_1^{n + d^{- 1} \sum_{j = 1}^{M - 1} l_j
         \mf{s}_j \mf{w}_{j, 1}}, \ldots \Big\nobracket  
         \Big\nobracket \\ & \qquad 
         \Big. \Big\nobracket  \ldots, F_M^{n + d^{- 1} \sum_{j = 1}^{M - 1}
         l_j \mf{s}_j \mf{w}_{j, M}}) \Big\|_{\SO_0} {>
         2^{\frac{n_0}{\bar{p}_0} - \frac{\varepsilon'}{\bar{p}_0} | n - n_0 |
         - \frac{\varepsilon'}{\bar{p}_0} \sum_{j = 1}^{M - 1} l_j}}  \Big)
       \end{aligned}\\ &
       \lesssim
       \begin{aligned}[t]
       2^{- n_0} &\sum_{(l_1, \ldots, l_{M - 1}) \in \N^M} 2^{-
       \bar{p}_0 \varepsilon d^{- 1} \sum_{j = 1}^{M - 1} l_j}
       \sum_{\tmscript{n \in \frac{1}{M} \Z}}  2^{- (\varepsilon - \varepsilon') | n - n_0 |}\\ &
       \qquad \times \prod_{i = 1}^M \mu_i \Big( \Big\| F_i^{n + d^{- 1} \sum_{j =
       1}^{M - 1} l_j \mf{s}_j \mf{w}_{j, i}} \Big\|_{\SO_i}
       \Big)^{\frac{\bar{p}_0}{p_i}} 2^{\frac{\bar{p}_0}{p_i} \Big( n +
       d^{- 1} \sum_{j = 1}^{M - 1} l_j \mf{s}_j \mf{w}_{j, i} \Big)} .
       \end{aligned}
     \end{aligned} \]
  Summing over \(n_0 \in \Z\) we get that
  \[ \begin{aligned}[t]&
       \sum_{n_0 \in \Z} \begin{aligned}[t]
         \mu_0 \Big( \Big\| \sum_{n \in \frac{1}{M} \Z} \sum_{(l_1, \ldots,
         l_{M - 1}) \in \N^M} \Pi (&F_1^{n + d^{- 1} \sum_{j = 1}^{M - 1} l_j
         \mf{s}_j \mf{w}_{j, 1}}, \ldots \Big.  \Big.\\ & \qquad
           \Big. \Big. \ldots, F_M^{n + d^{- 1} \sum_{j = 1}^{M - 1} l_j
           \mf{s}_j \mf{w}_{j, M}}) \Big\|_{\SO_0} > C
           2^{\frac{n_0}{{\bar{p}_0} }} \Big)
       \end{aligned}\\ &
       \lesssim \begin{aligned}[t]
         \sum_{(l_1, \ldots, l_{M - 1}) \in \N^M} 2^{- \bar{p}_0 \varepsilon
         d^{- 1} \sum_{j = 1}^{M - 1} l_j} \sum_{\tmscript{n \in \frac{1}{M} \Z}} &\prod_{i = 1}^M \mu_i \Big( \Big\| F_i^{n + d^{- 1}
         \sum_{j = 1}^{M - 1} l_j \mf{s}_j \mf{w}_{j, i}} \Big\|_{\SO_i}
         \Big)^{\frac{\bar{p}_0}{p_i}}\\ &\qquad
         \times 2^{\frac{\bar{p}_0}{p_i} \Big( n + d^{- 1} \sum_{j = 1}^{M -
         1} l_j \mf{s}_j \mf{w}_{j, i} \Big)}
       \end{aligned}\\ &
       \lesssim \begin{aligned}[t]
       \sum_{(l_1, \ldots, l_{M - 1}) \in \N^M} 2^{- \bar{p}_0 \varepsilon d^{- 1} \sum_{j = 1}^{M - 1} l_j}
       &\prod_{i = 1}^M \Big( \sum_{n \in \frac{1}{M} \Z} \mu_i \Big( \Big\| F_i^{n + d^{- 1} \sum_{j = 1}^{M - 1} l_j \mf{s}_j \mf{w}_{j, i}} \Big\|_{\SO_i} \Big)^{\bar{p}_0}
     \Big. \\ &\qquad
         \times \Big\nobracket 2^{\bar{p}_0 ( n + d^{- 1} \sum_{j =
         1}^{M - 1} l_j \mf{s}_j \mf{w}_{j, i} )}
         \Big)^{\frac{p_i}{\bar{p}_0}}
       \end{aligned}\\ &
       \lesssim \sum_{(l_1, \ldots, l_{M - 1}) \in \N^M} 2^{- \bar{p}_0
       \varepsilon d^{- 1} \sum_{j = 1}^{M - 1} l_j} \lesssim 1
     \end{aligned} \]
\end{proof}

As a result we deduce the outer Hölder inequality used extensively in the
context of outer Lebesgue spaces; variants include {\cite[Proposition
2.34]{amentaBilinearHilbertTransform2020}}, {\cite[Proposition
3.4]{doTheoryOuterMeasures2015}}.

\begin{corollary}[Hölder inequality for outer Lebesgue spaces]
  \label{cor:outer-holder-classical}For each\\ \(j \in \{0, 1, \ldots, M\}\) let
  \((\XO_j, \BO_j^{\cup}, \mu_j, \SO_j)\) be an outer space. Let
  \[ \Pi : \prod_{j = 1}^M L^{\infty}_{\mu_j} \SO_j \to L^{\infty}_{\mu_0}
     \SO_j \]
  be an \(M\)-linear map such that
  \begin{equation}
    \| \Pi (F_1, \ldots, F_M)\|_{\SO_0} \leq C_{\infty} \prod_{j = 1}^M \|F_j
    \|_{\SO_j} \label{eq:abstract-size-holder}
  \end{equation}
  and suppose that for all \(F_j \in L^{\infty}_{\mu_j} \SO_j\) it holds that
  \begin{equation}
    \mu_0 (\| \Pi (F_1, \ldots, F_M) \|_{\SO_0} > 0) \leq \min_{j = 1, \ldots,
    M} \mu_j \bigl( \| F_j \|_{\SO_j} > 0 \bigr) .
    \label{eq:abstract-measure-holder}
  \end{equation}
  Then for all \(\vec{p} = (p_1, \ldots, p_M) \in (0, \infty]^M\) it holds that
  \begin{equation}
    \| \Pi (F_1, \ldots, F_M)\|_{L_{\mu_0}^{p_0} \SO_0} \lesssim_{p_j} C
    \prod_{j = 1}^M \|F_j \|_{L_{\mu_j}^{p_j} \SO_j}
    \label{eq:abstract-Lp-holder}
  \end{equation}
  with \(p_0^{- 1} \eqd \sum_{j = 1}^M p_j^{- 1}\).
\end{corollary}

\begin{proof}
  The claim above follows by noticing that
  \[ \| \Pi (F_1, \ldots, F_M) \|_{L_{\mu_0}^{p_0} \SO_0} \begin{aligned}[t]&
       \leq \mu_0 \left( \| \Pi (F_1, \ldots, F_M) \|_{\SO_0} > 0
       \right)^{\frac{1}{p_0}} \| \Pi (F_1, \ldots, F_M)
       \|_{L_{\mu_0}^{\infty} \SO_0}\\ &
       \leq \prod_{j = 1}^M \mu_j (\| F_j \|_{\SO_j} > 0)^{\frac{1}{p_j}}
       \|F_j \|_{\SO_j}
     \end{aligned} \]
  and applying \Cref{prop:outer-restricted-interpolation}.
\end{proof}

The next statement is usually referred to as outer Radon-Nikodym. It has been
used extensively in results proved using the framework of outer Lebesgue
spaces e.g. in {\cite[Proposition 2.33]{amentaBilinearHilbertTransform2020}},
{\cite[Corollary 4.13]{amentaBanachvaluedModulationInvariant2020}},
{\cite[Lemma 2.2]{uraltsevVariationalCarlesonEmbeddings2016}}. We merely
sketch its proof. This result allows one to compare classical integrals with
\(L^1\)-type outer Lebesgue quasi-norms e.g. the one appearing on
{\LHS{\eqref{eq:RN-domination-BHT}}}, that represents the dual form to
\(\BHT_{\beta}\) in terms of embedded functions. \

\begin{proposition}[Radon-Nikodym -type domination]
  \label{prop:outer-RN}Let \((\XO, \BO^{\cup}, \mu, \SO)\) be an outer space and
  let \(\mf{m}\) be a Radon measure on \(\XO\) and let \(\phi_1, \ldots, \phi_k \in
  \Phi_{\gr}^{\infty}\). Suppose the bound
  \[
  \begin{aligned}[t]
  & \Big| \int_B F (x) [\phi_1, \ldots, \phi_k] \dd{\mf{m} (x)} \Big| \leq
     C_{\mf{m}} \mu (B) \| \1_B F\|_{\SO} \qquad \forall B \in \BO,\\ & \qquad
     \forall F \in L^1_{\dd \mf{m}} (E) \otimes^k
     {\Phi_{\gr}^N}' 
  \end{aligned}
  \]
  holds with some constant \(C_{\mf{m}} > 0\) and
  \[ \mu (A) = 0 \hspace{0.17em} \Rightarrow \mf{m} (A) = 0 \qquad \forall A
     \in \Bor (\XO) . \]
  Then we have
  \[ | \int_{\XO} F (x) [\phi_1, \ldots, \phi_k] \hspace{0.17em} \dd{\mf{m}
     (x)}| \lesssim C_{\mf{m}} \|F\|_{L^1_{\mu} \SO} \qquad
     \forall F \in L^1_{\dd \mf{m}} (E) \otimes^k {\Phi_{\gr}^N}' . \]
\end{proposition}

\begin{proof}
  If \(\|F\|_{L^1_{\mu} \SO} \leq 1\) and \(F \in L^1_{\dd \mf{m}} (E) \otimes^k
  {\Phi_{\gr}^N}'\) we enact the atomic decomposition of \(F\) as per
  \Cref{prop:atomic-decomposition} to obtain
  \[ F = \sum_{k \in \Z} F_k \]
  with
  \[ \sum_{k \in \Z} \mu ( \| F_k \|_{\SO} > 0) \| F_k \|_{\SO}
     \lesssim \|F\|_{L^1_{\mu} \SO} . \]
  We claim that
  \[ \Big| \int_{\XO} F_k (x) [\phi_1, \ldots, \phi_k] \dd{\mf{m} (x)}
    \Big| \leq C_{\mf{m}} \mu ( \| F_k \|_{\SO} > 0  )
     \| F_k \|_{\SO} . \]
  This would allow us to conclude by summing the above estimate over \(k \in
  \Z\) and using that \(F_k\) are supported on disjoint sets. The assumptions on
  \(\mu\), \(\mf{m}\), and on the integrability of \(F\) justify the convergence of
  appropriate series. To prove the claimed bound on \(F_k\) we choose \(E_k \in
  \BO^{\cup}\) such that \(\1_{\XO \setminus E_k} F_k = 0\) and \(\mu (E_k)
  \approx \mu \left( \| F_k \|_{\SO} > 0 \right)\). We then cover \(E_k\)
  optimally with sets from \(\BO\) by selecting \(B_{k, n} \in \BO\) such that \
  \(E_k \subset \bigcup_{n \in \N} B_{k, n}\) with \(\sum_{n \in \N} \mu (B_{k,
  n}) \lesssim \mu (E_k)\). It holds that
  \[ \begin{aligned}[t]&
       \left| \int_{\XO} F_k (x) [\phi_1, \ldots, \phi_k] \dd{\mf{m} (x)}
       \mf{m} (x) \right| \leq \sum_{n \in \N} \left| \int_{\Delta B_{k, n}}
       F_k (x) [\phi_1, \ldots, \phi_k] \dd{\mf{m} (x)} \mf{m} (x) \right|\\ &
       \lesssim C_{\mf{m}} \sum_{n \in \N} \mu (\Delta B_{k, n}) \left\|
       \1_{\Delta B_{k, n}} F_k \right\|_{\SO} \lesssim \mu \left( \| F_k
       \|_{\SO} > 0 \right) \| F_k \|_{\SO} .
     \end{aligned} \]
  We used the monotonicity property \eqref{eq:size-monotonicity} in the above
  inequalities, and \eqref{eq:measure-recovery} to find the sets \(B_{k, n} \in
  \BO\). Summing the above inequality over \(k\) yields the claim. 
\end{proof}

\subsection{The time-frequency-scale
space}\label{sec:time-frequency-outer-lebesgue}

In this section we discuss the outer \\ Lebesgue space structure we will be using
in this paper. First we introduce trees: distinguished subsets of \(\R^3_+\)
that represent frequency- and space-localized portions of \(\R^3_+\),
interpreted as the group of symmetries (as per \eqref{eq:symmetries}). Next we
introduce the multiple sizes that will come into play in our argument. We
first present all of them succinctly to allow for efficient consultation; we
postpone the discussion about the motivation and approximate meaning of these
sizes to the end of the corresponding subsection.

\subsubsection{The \(\sigma\)-generating collection: trees}

We begin our construction of the outer Lebesgue space on \(\R^3_+\) by
introducing the family of distinguished sets, \tmtextit{trees}, efficiently
describe the local geometry of the time-frequency-scale space and play the
role of a \(\sigma\)-generating collection.

\begin{definition}[Trees]
  \label{def:tree}A tree \(T_{\Theta} (\xi_T, x_T, s_T)\) with top \((\xi_T, x_T,
  s_T) \in \R^3_+\) is the set given by
  \begin{equation}
    \begin{aligned}[t]&
      T_{\Theta} (\xi_T, x_T, s_T) \eqd \pi_{(\xi_T, x_T, s_T)}
      (\mT_{\Theta}),\\ &
      \mT_{\Theta} \eqd \left\{ (\theta, \zeta, \sigma) \in \Theta \times B_1
      \times (0, 1) \st \sigma < 1 - | \zeta | \right\},\\ &
      \pi_{(\xi, x, s)} (\theta, \zeta, \sigma) \eqd (\xi + \theta (s
      \sigma)^{- 1}, x + s \zeta, s \sigma) .
    \end{aligned} \label{eq:def:tree}
  \end{equation}
  The set \(\Theta \subset \R\) can be any Borel, that mainly is assumed to be
  an open interval \(\Theta \supset B_4\). 
\end{definition}

There is a one-to-one correspondence between the datum \(\Theta\), \((\xi_T, x_T,
s_T)\) and a tree \(T\). Given a tree \(T\), we call \(\mT_{\Theta}\) its model tree,
\(\Theta\) is its model frequency band, the \(\xi_T\) is the tree's frequency,
\(B_{s_T} (x_T)\) is its spatial interval, \(\xi_T + s_T^{- 1} \Theta\) is its
frequency band, and \(\pi_T \eqd \pi_{(\xi_T, x_T, s_T)}\) is its local
coordinate map.  We denote by \(\mathbb{T}_{\Theta}\) the collection of all
trees \(T = T_{\Theta} (\xi_T, x_T, s_T)\). By \(\TT_{\Theta}^{\cup}\) we denote
countable unions of trees.

We can now use trees to generate an outer measures \(\mu_{\Theta}^{\mf{p}}\) on
\(\R^3_+\) that allow us to measure how large a collection of trees is.

\begin{definition}[Forest outer measure]
  \label{def:forest-outer-measure}Given a countable collection \(\mc{T} \subset
  \mathbb{T}_{\Theta}\) we denote its counting function by
  \begin{equation}
    \begin{aligned}[t]&
      N_{\mc{T}} (z) \eqd \sum_{T \in \mc{T}} \1_{B_{s_{T}} (x_T)} (z) .
    \end{aligned} \label{eq:def:counting-function}
  \end{equation}
  The outer measures \(\mu_{\Theta}^{\mf{p}}\) of a set \(E \subset \R^3_+\) is
  defined in terms of the counting function of a covering by trees:
  \[ \mu_{\Theta}^{\mf{p}} (E) \eqd \inf \Big\{ \|N_{\mc{T}}
     (z)\|_{L^{\mf{p}}_{\dd z} (\R)} \st \mc{T} \subset \TT_{\Theta},
     \; E \subset \bigcup_{T \in \mc{T}} T \Big\}, \qquad
     \mf{p} \in [1, + \infty] . \]
  We are principally interested in \(\mf{p} \in \{ 1, \infty \}\).
\end{definition}

Any tree \(T_{\Theta} (\xi_T, x_T, s_T)\) can be obtained from \(T_{\Theta} (0,
0, 1)\) as follows.
\begin{equation}
  T_{\Theta} (\xi_T, x_T, s_T) = \{ (\xi_T + s_T^{- 1} \eta, x_T + s_T y, s_T
  t) \suchthat (\eta, y, t) \in T_{\Theta} (0, 0, 1) \} .
  \label{eq:tree-symmetry}
\end{equation}
This is in agreement with the understanding that trees represent localized
subsets of \(\R^3_+\) seen as a group of symmetries. As a matter of fact, if
\((\eta, y, t) \in T_{\Theta} (\xi_T, x_T, s_T)\) it holds that the spatial
support of \(\Tr_y \Mod_{\eta} \Dil_t \phi (z)\) is essentially contained in
\(B_{s_T} (x_T)\); more precisely
\[ \left| \Tr_y \Mod_{\eta} \Dil_t \phi (z) \right| \lesssim \left\langle
   \frac{z - x_T}{s_T} \right\rangle^{- N} \| \phi \|_{\Phi^N_{\gr}} . \]
Furthermore, localization in Fourier space can be seen from the fact that
\[ \dist \left( \spt \left( \FT{\Tr_y \Mod_{\eta} \Dil_t \phi} \right), \xi_T
   \right) \lesssim \diam \left( \spt \left( \FT{\Tr_y \Mod_{\eta} \Dil_t
   \phi} \right) \right) \]
for any \((\eta, y, t) \in T_{\Theta} (\xi_T, x_T, s_T)\). Note however that
trees are not scale-localized e.g. if \((\eta, y, t) \in T_{\Theta} (\xi_T,
x_T, s_T)\) then \((\eta, y, t') \in T_{\Theta} (\xi_T, x_T, s_T)\) for all \(t' <
t\). Trees respects the symmetry property of the map \(f \mapsto \Emb  [f]\). It
follows from \eqref{eq:embedding} that
\begin{equation}
  \begin{aligned}[t]&
    \Emb [f] (\eta, y, t) [\phi]\\ &
    \qquad = e^{2 \pi i \xi (y - x)} \Emb \left[ s \Dil_{s^{- 1}} \Mod_{- \xi}
    \Tr_{- x } f \right] \Bigl( s (\eta - \xi), \frac{y - x}{s}, \frac{t}{s}
    \Bigr) [\phi]
  \end{aligned} \label{eq:emb-symmetry}
\end{equation}
Identities \eqref{eq:tree-symmetry} and \eqref{eq:emb-symmetry} together give
\begin{equation}
  \begin{aligned}[t]&
    \Emb [f] \circ \pi_T (\theta, \zeta, \sigma)\\ &
    \qquad = e^{2 \pi i \xi_T s_T \zeta} \Emb \left[ s_T \Dil_{s_T^{- 1}}
    \Mod_{- \xi_T} \Tr_{- x_T} f \right] \circ \pi_{(0, 0, 1)} (\theta, \zeta,
    \sigma) .
  \end{aligned} \label{eq:piT-symmetry}
\end{equation}

\subsubsection{Sizes}\label{sec:sizes}

The second ingredient of an outer Lebesgue spaces are sizes. To construct
sizes, we first introduce \tmtextit{local sizes} measure how ``large'' a
function is on any fixed tree \(T \in \TT_{\Theta}\). Formally, a
\tmtextit{local size} is a collection \(\left( \SO (T)  \right)_{T \in
\TT_{\Theta}}\) indexed by \(T \in \TT_{\Theta}\) of quasi-norms on \(\Rad (T)
\otimes^k {\Phi_{\mf{r}}^{\infty}}'\) for some fixed \(k \in \N\); we require
that the quasi-triangle constants \({\tmop{Tri}_{_{\SO (T)}}} \) be uniform i.e.
that there exist \(\tmop{Tri}_{\SO} \geq 1\) such that
\[ \| F_1 + F_2 \|_{\SO (T)} \leq \tmop{Tri}_{\SO} \left( \| F_1 \|_{\SO (T)}
   + \| F_1 \|_{\SO (T)} \right) \qquad \forall T \in \TT_{\Theta}, \; \forall
   F_1, F_2 \in \Rad (T) \otimes^k {\Phi_{\mf{r}}^{\infty}}' . \]
Local sizes generate a size by setting
\begin{equation}
  \|F\|_{\SO} \eqd \sup_{T \in \TT_{\Theta}} \sup_{W^- \in
  \TT^{\cup}_{\Theta}}  \| \1_{\R^3_+ \setminus W^-} F\|_{\SO (T)}
  \label{eq:generated-size:trees}
\end{equation}
where we identify \(\1_{\R^3_+ \setminus W^-} F\) appearing above with its
restriction to \(T\). The upper bound over \(W^-\) guarantees that monotonicity
\eqref{eq:size-monotonicity} holds.

We will now introduce several (local) sizes that will each control relevant
properties of functions \(F \in \Rad (T) {\otimes \Phi_{\mf{r}}^{\infty}}'\).
These sizes are combined to obtain the size \(\SF^u_{\Theta} \Phi_{\gr}^N\) in
\eqref{eq:non-uniform-embedding-full-size}. Most of the sizes below depend on
the datum \(N, \gr\) and \(\Theta\) that we track but do not explicitly discuss.

\begin{definition}[Lebesgue size]
  \label{def:lebesgue-size}Let \((u, v) \in [1, \infty]^2\). We define the local
  sizes \(\SL{}^{(u, v)}_{\Theta}\) as follows. For \(F \in \Rad (\R^{3}_{+})\) and \(T \in \mathbb{T}_{\Theta}\) set
  \begin{equation}
    \begin{aligned}[t]&
      \|F\|_{\SL{}^{(u, v)}_{\Theta} (T)} \eqd {\sup_G}  \left| \frac{1}{|
      \Theta | s_T} \int_T G (\eta, y, t)  \dd F (\eta, y, t) \right|
    \end{aligned}  \label{eq:lebesgue-size-scalar}
  \end{equation}
  where the upper bound gets taken over all \(G \in C^{\infty}_c (T)\) with \(\|
  G \circ \pi_T \|_{L^{u'}_{\frac{\dd \theta \dd \zeta}{| \Theta |} }
  L^{v'}_{\frac{\dd \sigma}{\sigma}}} \leq 1\), where \(u' = \frac{u}{u - 1},
  v' = \frac{v}{v - 1}\) and
  \[ \begin{aligned}[t]&
       \| G \circ \pi_T \|_{L^{u'}_{\frac{\dd \theta \dd \zeta}{| \Theta
       |} } L^{v'}_{\frac{\dd \sigma}{\sigma}}} \assign \Big( \int_{\R^2}
       \Big( \int_{\R_+} | G \circ \pi_T (\theta, \zeta, \sigma) |^{v'}
       \frac{\dd \sigma}{\sigma} \Big)^{\frac{u'}{v'}} \frac{\dd
       \theta \dd \zeta}{| \Theta |} \Big)^{\frac{1}{u'}} .
     \end{aligned} \]
  The integral in \eqref{eq:lebesgue-size-scalar} should be interpreted as the
  integration of the test function \(G (\eta, y, t)\) against the measure \(F\).
  For \(F \in \Rad (\R^{3}_{+}) \otimes^k {{\Phi^{\infty}_{\gr}}'} \)
  we define the local sizes
  \begin{equation}
    \|F\|_{\SL{}^{(u, v)}_{\Theta} \Phi_{\gr}^N (T)} {\assign \sup_{\phi_1,
    \ldots, \phi_k}}  \|F [\phi_1, \ldots, \phi_k] \|_{\SL{}^{(u, v)}_{\Theta}
    (T)} \label{eq:lebesgue-size}
  \end{equation}
  with the upper bound being taken over all wave packets \(\phi_1, \ldots,
  \phi_k \in \Phi_{\gr}^{\infty}\) with \(\| \phi_j \|_{\Phi_{\gr}^N} \leq 1\),
  \(j \in \{ 1, \ldots, k \}\). 
\end{definition}

In \eqref{eq:lebesgue-size-scalar} we defined the size by duality because we
need to deal with the eventuality that \(F\) is a Radon measure and is not given
by a locally integrable function w.r.t to the measure \(\dd \eta \dd y \dd t\).
If \(F \in L^1_{\tmop{loc}}\) then it simply holds that
\[ \| F \|_{\SL{}^{(u, v)}_{\Theta}} = \left\| \left( \1_{\mT_{\Theta}} F
   \circ \pi_T \right) (\theta, \zeta, \sigma) \right\|_{L^u_{\frac{\dd
   \theta \dd \zeta}{| \Theta |} } L^v_{\frac{\dd \sigma}{\sigma}}} . \]
Furthermore, if \(u \neq \{ 1, \infty \}\) and \(v \neq \{ 1, \infty \}\) then \(\|
F \|_{\SL{}^{(u, v)}_{\Theta}} < \infty\) only if \(F \in L^1_{\tmop{loc}}\).
However, we will come across functions of the form \(F (\eta, y, t) t \delta
\left( t - \mf{b} (\eta, y) \right) \dd \eta \dd y \dd t\) for some \(F \in
L^{\infty}_{\tmop{loc}} (\R^{3}_{+})\) and for some continuous
function \(\mf{b} \of \R^2 \rightarrow [0, + \infty]\). Integrating a test
function \(G\) against such a Radon measure amounts to computing the integral
\[ G \mapsto \int_{\R^2} G \left( \eta, y, \mf{b} (\eta, y) \right) F \left(
   \eta, y, \mf{b} (\eta, y) \right) \mf{b} (\eta, y) \dd \eta \dd y. \]
\begin{definition}[Lacunary size]
  \label{def:lacunary-size}Let \((u, v) \in [1, \infty]^2\); we define the local
  sizes \(\SL^{(u, v)}_{\Theta} \wpD_{\gr}^{N }\) as follows. For \(F \in \Rad
  (\R^3_+) {{\otimes \Phi^{\infty}_{\gr}}'} \) and \(T \in \mathbb{T}_{\Theta}\)
  set
  \[ \|F (\eta, y, t) \|_{{\SL  }^{(u, v)}_{\Theta} \wpD_{\gr}^{N } (T)}
     \assign \left\|  \wpD_{\zeta} (t (\eta - \xi_T)) F (\eta, y, t)
     \right\|_{{\SL  }^{(u, v)}_{\Theta} \Phi_{\mf{r}}^N (T)} . \]
  The linear operator \(\wpD_{\zeta} (\theta)\), defined on \(F \in \Rad (\R^3_+)
  \otimes \Phi_{\gr}^{\infty} (\R)'\) is given by
  \begin{equation}
    \wpD_{\zeta} (\theta) F (\theta, \zeta, \sigma) [\phi] \eqd F (\theta,
    \zeta, \sigma) [(- d_z + 2 \pi i \theta) \phi (z)] \label{eq:space-boost}
  \end{equation}
  where \(z\) in {\RHS{{\eqref{eq:space-boost}}}} plays the role of a dummy
  variable.
\end{definition}

\begin{definition}[Defect size]
  \label{def:defect-size}Let \((u, v) \in [1, \infty]^2\); we define the
  {\tmem{space defect}} local sizes \({\SL_{\Theta}^{(u, v)}}  \dfct_{\zeta,
  \gr}^N\) and the {\tmem{scale defect}} local sizes \(\SL^{(u, v)}_{\Theta}
  \dfct_{\sigma, \gr}^N\) as follows. For \(F \in \Rad (\R^{3}_{+})
  {{\otimes \Phi^{\infty}_{\gr}}'} \) and \(T \in \mathbb{T}_{\Theta}\) set
  \begin{equation}
    \begin{aligned}[t]&
      \|F\|_{{\SL_{\Theta}^{(u, v)}}  \dfct^N_{\zeta, \gr} (T)} \assign
      \left\| \left( \wpD_{\zeta} (t (\eta - \xi_T)) - t \dd_y + 2 \pi i
      \xi_T \right) F (\eta, y, t)  \right\|_{\SL_{\Theta}^{(u, v)}
      \Phi_{\mf{r}}^N (T) },\\ &
      \|F\|_{\SL^{(u, v)}_{\Theta} \dfct^N_{\sigma, \gr} (T)} \assign \left\|
      \left( \wpD_{\sigma} (t (\eta - \xi_T)) - t \dd_t + (\eta - \xi_T)
      \dd_{\eta} \right) F (\eta, y, t) \right\|_{\SL_{\Theta}^{(u, v)}
      \Phi_{\mf{r}}^N (T)} .
    \end{aligned}  \label{eq:defect-size}
\end{equation}

  The operator \(\wpD_{\sigma} (\theta)\), defined on \(\Rad (\R^3_+) {\otimes
  \Phi_{\gr}^{\infty}}'\), is given by
  \begin{equation}
    \begin{aligned}[t]&
      \wpD_{\sigma} (\theta) F (\theta, \zeta, \sigma) [\phi] \assign F
      (\theta, \zeta, \sigma) [(- d_z + 2 \pi i \theta) (z \phi (z))] .
    \end{aligned} \label{eq:scale-boost}
  \end{equation}
  The derivatives appearing in \eqref{eq:defect-size} are intended in the
  distributional sense.
  
  We combine the two defect sizes and use the shorthand
  \begin{equation}
    {\SL_{\Theta}^{(u, v)}}  \dfct^N_{\gr} (T) \eqd {\SL_{\Theta}^{(u, v)}} 
    \dfct^N_{\zeta, \gr} (T) + {\SL_{\Theta}^{(u, v)}}  \dfct^N_{\sigma, \gr}
    (T) . \label{eq:total-defect-size}
  \end{equation}
\end{definition}

We introduce an additional size that represents the maximal truncation of a
localized singular integral operator. This size has not been studied in prior
works that use outer Lebesgue spaces in the context of time frequency
analysis. Controlling this size for functions \(F = \Emb [f]\) is one of the
main objectives of \Cref{sec:non-uniform:SIO}.

\begin{definition}[SIO truncation size]
  \label{def:SIO-size}Let \(u \in [1, \infty)\); we define the {\tmem{SIO
  truncation}} local sizes \(\SJ_{\Theta}^u \Phi_{\gr}^N\) as follows. For \(F
  \in \Rad (\R^{3}_{+}) {{\otimes \Phi^{\infty}_{\gr}}'} \) and \(T \in
  \mathbb{T}_{\Theta}\) set
  \begin{equation}
    \begin{aligned}[t]&
      \|F\|_{\SJ_{\Theta}^u \Phi_{\gr}^N (T)} \eqd 
      {\sup_{\tmscript{\begin{aligned}[t]&
        \phi, G
      \end{aligned}}}}  \left| \frac{1}{s_T | \Theta |} \int_T \wpD_{\zeta} (t
      (\eta - \xi_T)) F (\eta, y, t) G (\eta, y, t) \dd \eta \dd y \dd t
      \right|
    \end{aligned} \label{eq:SIO-size}
  \end{equation}
  where the upper bound gets taken over all \(\phi \in \Phi_{\gr}^{\infty}\)
  with \(\| \phi \|_{\Phi_{\gr}^N} \leq 1\) and over all \(G \in C^{\infty}_c
  (\R^{3}_{+})\) such that
  \[ \| \sigma d_{\sigma} G \circ \pi_T (\theta, \zeta, \sigma)
     \|_{L^{u'}_{\frac{\dd \theta \dd \zeta}{| \Theta |} }
     L^{\infty}_{\frac{\dd \sigma}{\sigma}}} \leq 1. \]
\end{definition}

With these definitions in hand we have all the elements necessary to
understand \eqref{eq:non-uniform-embedding-full-size} that defines the size
\(\SF^u_{\Theta} \Phi_{\gr}^N\). The single summands in
\eqref{eq:non-uniform-embedding-full-size} are related to the local sizes
above through \eqref{eq:generated-size:trees}.

\subsubsection{Sizes: some technical remarks}

All the sizes above are regular in the sense of \eqref{eq:size-regularity}. On
a technical level, this allows us to prove most of our statements under the
assumptions that the functions \(F \in \Rad (\R^{3}_{+}) {\otimes
\Phi^{\infty}_{\gr}}'\) in play are compactly supported.

\begin{remark}
  \label{rmk:cpt-sets}There exists two sequences of forests \(W_n^+ \in
  \TT_{\Theta}^{\cup}\) and \(W_n^- \in \TT_{\Theta}^{\cup}\) such that \(W^+_n\)
  is increasing, \(W^-_n\) is decreasing, and \(\mathbb{K}_n \eqd W^+_n \setminus
  W^-_n\) are compact, increasing, and satisfy \(\bigcup_n \mathbb{K}_n =
  \R^3_+\).
\end{remark}

We simplify dealing with the dependence on \(\phi \in \Phi^{\infty}_{\gr}\)
appearing in the definition of all sizes above using the following wave packet
decomposition result.

\begin{lemma}
  \label{lem:wave-packet-decomposition} For any \(0 < N' < N\) there exists a
  sequence of test functions \(\tilde{\phi}_k \in \Phi_{2 \gr}^{N'}\), \(k \in
  \Z\) satisfying \(\| \tilde{\phi}_k \|_{\Phi^{N'}_{2 \gr}} \lesssim 1\) such
  that the following decomposition holds. For any \(\phi \in \Phi_{\gr}^N\)
  there exists a sequence \(a_k \in \C\) such that
  \[ \phi (z) = \sum_{k \in \Z} a_k \tilde{\phi}_k, \]
  and \(| a_k | \lesssim \langle k \rangle^{- (N - N')} \| \phi
  \|_{\Phi_{\gr}^N} .\) In particular, the above sum converges absolutely in
  \(\Phi^{N'}_{2 \gr}\) if \(N > N' + 1\). 
\end{lemma}

\begin{proof}
  The proof is based on the Fourier inversion formula. Fix \(0 < \varepsilon <
  1\) and set
  \[ \begin{aligned}[t]&
       \FT{\tilde{\phi}} \left( \FT{z} \right) \eqd \left( \left| \FT{z}
       \right| - 2 \gr \right)^{(N' + \varepsilon)} \1_{B_{2 \gr}} \left(
       \FT{z} \right),\\ &
       \tilde{\phi}_k (z) \eqd \langle k \rangle^{- N'} \tilde{\phi} \left( z
       + \frac{k}{4 \gr} \right)
     \end{aligned} \]
  so that \(\| \tilde{\phi}_k (z) \|_{\Phi^{N'}_{2 \gr}} \lesssim 1\). Let
  \[ \begin{aligned}[t]&
       a_k \eqd \frac{\langle k \rangle^{N'}}{\left( 4 \gr \right)^2}
       \int_{B_{2 \gr}} e^{- 2 \pi i \frac{k}{4 \gr} \FT{w}} \FT{\phi} \left(
       \FT{w} \right)  \left( \left| \FT{w} \right| - 2 \gr \right)^{- (N' +
       \varepsilon)} \dd \FT{w} .
     \end{aligned} \]
  Since \(\FT{\phi}\) is supported on \(B_{\gr}\) it holds that \(\left\|
  \FT{\phi} \left( \FT{w} \right)  \left( \left| \FT{w} \right| - 2 \gr
  \right)^{- (N' + \varepsilon)} \right\|_{C^N} \approx \left\| \FT{\phi}
  \left( \FT{w} \right) \right\|_{C^N}\) so \(| a_k | \lesssim \langle k
  \rangle^{- (N - N')} \| \phi \|_{\Phi_{\gr}^N}\). The Fourier inversion
  formula gives us that
  \[ \sum_{k \in \Z} a_k  \langle k \rangle^{- N'} e^{2 \pi i \frac{k}{4 \gr}
     \FT{z}} = \FT{\phi} \left( \FT{z} \right)  \left( \left| \FT{z} \right| -
     2 \gr \right)^{- (N' + \varepsilon)} \]
  with the series converging uniformly. It follows that
  \[ \begin{aligned}[t]&
       \FT{\phi} \left( \FT{z} \right) = \sum_{k \in \Z} a_k  \langle k
       \rangle^{- N'} e^{2 \pi i \frac{k}{4 \gr} \FT{z}} \FT{\tilde{\phi}}
       \left( \FT{z} \right)
     \end{aligned} \]
  i.e.
  \[ \phi (z) = \sum_{k \in \Z} a_k \tilde{\phi}_k (z) \]
  as required. 
\end{proof}

As a consequence one can compare \(F \in \Rad (\R^{3}_{+}) {\otimes
\Phi^{\infty}_{2 \gr}}'\) with scalar functions in the following sense.

\begin{corollary}
  \label{cor:wave-packet-decomposition-and-sizes}Let \((\eta, y, t) \mapsto
  \phi_{(\eta, y, t)} \in \Phi_{\gr}^N\) be a measurable function with \(\|
  \phi_{(\eta, y, t)} \|_{\Phi^N_{\gr}} \lesssim 1\). Then \((\eta, y, t)
  \mapsto F (\eta, y, t) [\phi_{(\eta, y, t)}]\) defines an element of \(\Rad
  (\R^{3}_{+})\) with
  \begin{equation}
    \| F (\eta, y, t) [\phi_{(\eta, y, t)}] \|_{\SL^{(u, v)}_{\Theta}}
    \lesssim \| F \|_{\SL^{(u, v)}_{\Theta} \Phi^{N'}_{2 \gr}}
    \label{eq:scalar-non-scalar-bound-size}
  \end{equation}
  as long as \(\| F \|_{\SL^{(u, v)}_{\Theta} \Phi^{N'}_{2 \gr}} < \infty\) and
  \(N \geq N' + 4\). Furthermore
  \begin{equation}
    \| F (\eta, y, t) [\phi_{(\eta, y, t)}] \|_{L^p_{\mu^{\mf{p}}_{\Theta}}
    \SL^{(u, v)}_{\Theta}} \lesssim \| F \|_{\SL^{(u, v)}_{\Theta}
    \Phi^{N'}_{2 \gr}} . \label{eq:scalar-non-scalar-bound-Lp}
  \end{equation}
  for any \(p \in (0, + \infty]\) and
  \begin{equation}
    \mu^{\mf{p}}_{\Theta} \Bigl( \| F (\eta, y, t) [\phi_{(\eta, y, t)}]
    \|_{\SL^{(u, v)}_{\Theta}} > \lambda \Bigr) \leq \mu^{\mf{p}}_{\Theta}
    \Bigl( \| F \|_{\SL^{(u, v)}_{\Theta} \Phi^{N'}_{2 \gr}} > C \lambda
    \Bigr) . \label{eq:scalar-non-scalar-bound-spt}
  \end{equation}
  Finally, for any \(p \geq 1\) it holds that that
  \begin{equation}
    \| F \|_{L^p_{\mu^{\mf{p}}_{\Theta}} \SL^{(u, v)}_{\Theta} \Phi^N_{\gr}}
    \lesssim \sup_{\phi} \| F (\eta, y, t) [\phi]
    \|_{L^p_{\mu^{\mf{p}}_{\Theta}} \SL^{(u, v)}_{\Theta}},
    \label{eq:non-scalar-scalar-bound-Lp}
  \end{equation}
  where the upper bound is taken over all \(\phi \in \Phi^{\infty}_{2 \gr}\)
  with \(\| \phi \|_{\Phi^{N'}_{2 \gr}} \lesssim 1\), as long as \(N \geq N' +
  4\). 
\end{corollary}

\begin{proof}
  Bounds \eqref{eq:scalar-non-scalar-bound-size} and
  \(\eqref{eq:scalar-non-scalar-bound-spt}\) follow directly from
  \Cref{lem:wave-packet-decomposition} noting that the quasi-triangle
  inequality holds with constant \(1\). From those two bounds one deduces
  \eqref{eq:scalar-non-scalar-bound-Lp}. For bound
  \eqref{eq:non-scalar-scalar-bound-Lp} note that
  \[ \| F \|_{\SL^{(u, v)}_{\Theta} \Phi^N_{\gr}} \lesssim \sum_{k \in \Z}
     \langle k \rangle^{- (N - N')} \| F \circ \Gamma [\tilde{\phi}_k]
     \|_{\SL^{(u, v)}_{\Theta}} \]
  with \(\tilde{\phi}_k\) given by \Cref{lem:wave-packet-decomposition}. Thus
  \[ \mu^{\mf{p}}_{\Theta} \left( \| F \|_{\SL^{(u, v)}_{\Theta}
     \Phi^N_{\gr}} > \lambda \right) \lesssim \sum_{k \in \Z}
     \mu^{\mf{p}}_{\Theta} \left( \| F [\tilde{\phi}_k] \|_{\SL^{(u,
     v)}_{\Theta}} > \lambda C \langle k \rangle^2 \right) \]
  and
  \[  \| F \|_{L^p_{\mu^{\mf{p}}_{\Theta}} \SL^{(u, v)}_{\Theta}
       \Phi^N_{\gr}}^p \begin{aligned}[t]&
      \lesssim \sum_{k \in \Z} \int_0^{\infty}
       \mu^{\mf{p}}_{\Theta} \left( \| F [\tilde{\phi}_k] \|_{\SL^{(u,
       v)}_{\Theta}} > \lambda^{\frac{1}{p}} C \langle k \rangle^2 \right) \dd
       \lambda\\ &
       \lesssim \sum_{k \in \Z} \langle k \rangle^{- 2 p} \int_0^{\infty}
       \mu^{\mf{p}}_{\Theta} \left( \| F [\tilde{\phi}_k] \|_{\SL^{(u,
       v)}_{\Theta}} > \lambda^{\frac{1}{p}} \right) \dd \lambda\\ &
       = \sum_{k \in \Z} \langle k \rangle^{- 2 p} \| F [\widetilde{\phi_k}]
       \|_{\SL^{(u, v)}_{\Theta}}^p .
     \end{aligned} \]
  This shows \eqref{eq:non-scalar-scalar-bound-Lp}.
\end{proof}

\subsubsection{Sizes: motivation}\label{sec:size-motivation}

We conclude this section by providing some details and motivation for the
definition of the sizes introduced so far. As already discussed for the trees
themselves, the sizes are adapted to treating \(\R^3_+\) as the group of
symmetries \eqref{eq:symmetries}. If \(\SO_{\Theta} \Phi^N_{\gr}\) is any of the
sizes \(\SL^{(u, v)}_{\Theta} \Phi^N_{\gr}\), \(\SL^{(u, v)}_{\Theta}
\wpD^N_{\gr}\), \(\SL^{(u, v)}_{\Theta} \dfct^N_{\gr}\), or \(\SJ^u_{\Theta}
\Phi^N_{\gr}\) above, then setting
\[ \tilde{F} (\eta, y, t) = e^{2 \pi i \xi (y - x)} F \Bigl( s (\eta - \xi),
   \frac{y - x}{s}, \frac{t}{s} \Bigr) \]
it holds that
\begin{equation}
  \| \tilde{F} \|_{\SO_{\Theta} \Phi_{\gr}^N (T_{\Theta} (\xi', x', s'))} = \|
  F \|_{\SO_{\Theta} \Phi_{\gr}^N \left( T_{\Theta} \left( s (\xi' - \xi),
  \frac{x' - x}{s}, \frac{s'}{s} \right) \right)} \label{eq:size-symmetry}
\end{equation}
and thus \(\| \tilde{F} \|_{\SO_{\Theta} \Phi_{\gr}^N} = \| F \|_{\SO_{\Theta}
\Phi_{\gr}^N}\). Thanks to this fact and identities \eqref{eq:piT-symmetry} and
\eqref{eq:tree-symmetry}, in our proofs we will often be able to restrict our
focus to \(T = T_{\Theta} (0, 0, 1)\). Let us make this assumption also for the
discussion below that motivates the sizes we introduce and relates them to
well-known quantities in harmonic analysis.

The Lebesgue size \(\SL{}^{(u, v)}_{\Theta}\) of \Cref{def:lebesgue-size}
controls the magnitude of the coefficients of the embedding on a tree. If \(v =
\infty\) this is closely related to bounding a space and frequency localized
version of a Hardy-Littlewood maximal function. Let \(\phi \in \Phi^N_{\gr}\) be
chosen to maximize {\RHS{\eqref{eq:lebesgue-size}}} for \(F = \Emb [f]\).
Setting \(\widetilde{\phi}_{\theta} = \Mod_{\theta} (\phi^{\vee})\) we have
\[ \bigl\| \Emb [f] \bigr\|_{\SL{}^{(u, \infty)}_{\Theta} \Phi_{\gr}^N (T)}
   \approx \int_{\Theta} \int_{B_1} \sup_{0 < \sigma < 1 - | \zeta |} \left| f
   \ast \Dil_{\sigma} \tilde{\phi}_{\theta} (\zeta)  \right|^u \frac{\dd \zeta
   \dd \theta}{\Theta} . \]
For every fixed \(\theta \in \Theta\) the inner integrand is a Hardy-Littlewood
maximal function over scales bounded from above. The lacunary size \(\SL^{(u,
2)}_{\Theta} \wpD_{\gr}^{N }\) encodes the magnitude of a localized square
function. Let \(\phi \in \Phi^N_{\gr}\) be chosen to extremize the upper bound
in the size \Cref{def:lacunary-size}. Setting \(\psi_{\theta} (z) = \dd_z
\left( \Mod_{\theta} (\phi^{\vee}) (z) \right)\) we have that
\[ \bigl\| \Emb [f] \bigr\|_{\SL{}^{(u, 2)}_{\Theta} \wpD_{\gr}^{N } (T)}
   \approx \int_{\Theta} \int_{B_1} \left( \int_0^{1 - | \zeta |} \left| f
   \ast \Dil_{\sigma} \psi_{\theta} (\zeta)  \right|^2 \frac{\dd
   \sigma}{\sigma} \right)^{u / 2} \frac{\dd \zeta \dd \theta}{\Theta} . \]
For every fixed \(\theta \in \Theta\) the innermost integral is the
Littlewood-Paley square function. Note that \(\psi_{\theta}\) satisfies \(\int
\psi_{\theta} (z) \dd z = 0\).

Next, the {\tmem{defect sizes}} quantify by how much a function \(F \in \Rad
(\R^3_+) {\otimes \Phi^{\infty}_{\gr}}'\) fails to be of the form \(F = \Emb
[f]\). The differential operators \(t \dd_y + 2 \pi i \xi_T\) and \(t \dd_t
- (\eta - \xi_T) \dd_{\eta}\) appearing in \Cref{def:defect-size} are the push
forwards of \(\sigma \dd_{\sigma}\) and \(\sigma \dd_{\zeta}\) using the map
\(\pi_T\). More explicitly
\[ \begin{aligned}[t]&
     \begin{aligned}[t]&
       \sigma \dd_{\zeta} (e^{2 \pi i \xi_T s_T \zeta} F \circ \pi_T (\theta,
       \zeta, \sigma)) \\ &
       \qquad = e^{2 \pi i \xi_T s_T \zeta} (s_T \sigma \partial_y + 2 \pi i
       \xi_T) F (\xi_T + \theta s_T^{- 1} \sigma^{- 1}, x_T + s_T \zeta, s_T
       \sigma)\\ &
       \qquad= e^{2 \pi i \xi_T s_T \zeta}  \left( \left( t \dd_y + 2 \pi i \xi_T
       \right) F (\eta, y, t) \right) \circ \pi_T (\theta, \zeta, \sigma)
     \end{aligned}\\ &
     \begin{aligned}[t]&
       \sigma \dd_{\sigma} (e^{2 \pi i \xi_T s_T \zeta} F \circ \pi_T (\theta,
       \zeta, \sigma))\\ &
       \qquad = e^{2 \pi i \xi_T s_T \zeta} (- \theta s_T^{- 1} \sigma^{- 1}
       \partial_{\eta} + s_T \sigma \partial_t) F (\xi_T + \theta s_T^{- 1}
       \sigma^{- 1}, x_T + s_T \zeta, s_T \sigma)\\ &
       \qquad = e^{2 \pi i \xi_T s_T \zeta} \left( \left( t \dd_t - (\eta - \xi_T)
       \dd_{\eta} \right) F (\eta, y, t) \right) \circ \pi_T (\theta, \zeta,
       \sigma)
     \end{aligned}
   \end{aligned} \]
A direct computation gives that for any \(f \in \Sch (\R)\) it
holds that
\begin{equation}
  \begin{aligned}[t]&
    \sigma \dd_{\zeta} \left( e^{2 \pi i \xi_T s_T \zeta} \Emb [f] \circ \pi_T
    (\theta, \zeta, \sigma) \right) = \wpD_{\zeta} (\theta) \left( e^{2 \pi i
    \xi_T s_T \zeta} \Emb [f] \circ \pi_T (\theta, \zeta, \sigma) \right)\\ &
    \sigma \dd_{\sigma} \left( e^{2 \pi i \xi_T s_T \zeta} \Emb [f] \circ
    \pi_T (\theta, \zeta, \sigma) \right) = \wpD_{\sigma} (\theta) \left( e^{2
    \pi i \xi_T s_T \zeta} \Emb [f] \circ \pi_T (\theta, \zeta, \sigma)
    \right) .
  \end{aligned} \label{eq:embedded-no-defect}
\end{equation}
It follows that \(\left\| \Emb [f] \right\|_{L^{(u, v)}_{\Theta} \dfct_{\zeta,
\gr}^N} = \left\| \Emb [f] \right\|_{L^{(u, v)}_{\Theta} \dfct_{\sigma,
\gr}^N} = 0\). We actually use identity \eqref{eq:embedded-no-defect} to
\tmtextit{define} the operators \(\wpD_{\zeta} (\theta)\) and \(\wpD_{\sigma}
(\theta)\) acting on general functions \(F \in \Rad (\R^3_+) {\otimes
\Phi^{\infty}_{\gr}}'\) not necessarily of the form \(F = \Emb [f]\). In
\eqref{eq:space-boost} and \eqref{eq:scale-boost} we postulate that the
operators \(\wpD_{\zeta} (\theta)\) and \(\wpD_{\sigma} (\theta)\) should act on
the argument \(\Phi^{\infty}_{\gr}\) of \(F \circ \pi_{T_{\Theta} (0, 0, 1)}\) in
the same way as the derivatives \(\sigma \dd_{\zeta}\) and \(\sigma \dd_{\sigma}\)
act on \(\Emb [f] \circ \pi_{T_{\Theta} (0, 0, 1)}\). The defect size then
encodes and quantifies the local failure of a function \(F \in \Rad (\R^3_+)
{\otimes \Phi^{\infty}_{\gr}}'\) to be of the form \(F = \Emb [f]\). We cannot
expect that the functions \(F \in \Rad (\R^3_+) {\otimes \Phi^{\infty}_{\gr}}'\)
in play in the arguments of this paper to always be of the form \(F = \Emb
[f]\). Even in the process of defining size itself, described by
\eqref{eq:generated-size:trees}, we need to consider at least functions of the
form
\[ F = \1_{\R^3_+ \setminus W^-} \Emb [f] \]
with W\(^- \in \TT^{\cup}_{\Theta}\). For the sake of discussion let us suppose
that \(T = T_{\Theta} (0, 0, 1)\): the quantity intervening in the definition of
\({\SL_{\Theta}^{(u, v)}}  \dfct^N_{\zeta, \gr} (T)\) then becomes
\[ \left( \wpD_{\zeta} (t \eta) - t \dd_y \right) F (\eta, y, t) = - \left( t
   \dd_y \1_{\R^3_+ \setminus W^-} (\eta, y, t) \right) \Emb [f] (\eta, y, t)
   . \]
Based on the geometry of the set \(W^-\), as explained in more detail further
down in \Cref{lem:geometry-of-boundary}, the term \(\left( t \dd_y \1_{\R^3_+
\setminus W^-} (\eta, y, t) \right)\) is a Radon measure supported on the
boundary of \(W^-\). A similar discussion can be carried out for the size
\({\SL_{\Theta}^{(u, v)}}  \dfct^N_{\sigma, \gr}\). The crucial theme when
dealing with the defect is showing that mostly this is that generally we can
guarantee that the defect of the functions \(F \in \Rad (\R^3_+) {\otimes
\Phi^{\infty}_{\gr}}'\) in play will resemble the one discussed above.

Finally, the \(\SJ_{\Theta}^u \Phi_{\gr}^N\) represents the maximal truncation
of a localized singular integral operator. To see this chose \(\phi \in
\Phi^N_{\gr}\) to extremize the upper bound in \eqref{eq:SIO-size}. Setting
\(\psi_{\theta} (z) = \dd_z \left( \Mod_{\theta} (\phi^{\vee}) (z) \right)\) and
computing the upper bound over \(G\) we get that
\[ \| \Emb [f] \|_{\SJ_{\Theta}^u \Phi_{\gr}^N (T)}^u \lesssim \int_{\Theta}
   \int_{B_1} \sup_{\sigma < 1 - | \zeta | } \left| \int_0^{\sigma} f \ast
   \Dil_{\rho} \psi_{\theta} (\zeta) \frac{\dd \rho}{\rho} \right|^u \frac{\dd
   \theta \dd \zeta}{| \Theta |} . \]
The inner integral is related and can be bound by the Hardy-Littlewood maximal
function of \(f \ast K\) where the kernel \(K\) is Calderón-Zygmund given by
\[ K (z) \eqd \int_0^{\infty}  \Dil_{\rho} \psi_{\theta} (\zeta) \frac{\dd
   \rho}{\rho} . \]

\subsection{Changes of variables and product sizes}\label{sec:derived-sizes}

The sizes that make up size \(\SF^u_{\Theta} \Phi_{\gr}^N\), introduced in the
previous section, are well suited to control functions of the form \(F = \Emb
[f] \in \Rad (\R^{3}_{+}) {\otimes \Phi_{\gr}^{\infty}}'\). To control
functions \(\Emb [f] \circ \Gamma\) appearing in
\eqref{eq:embedding-bounds:uniform-iterated} we pull back the sizes via the
change of variables \(\Gamma\). We consider changes of variables \(\Gamma =
\Gamma_{(\alpha, \beta, \gamma)}\) of the form
\begin{equation}
  \begin{aligned}[t]&
    \Gamma_{(\alpha, \beta, \gamma)} (\eta, y, t) \eqd (\alpha (\eta + \gamma
    t^{- 1}), y, \beta t),\\ &
    \beta \in (0, 1], \quad | \alpha \beta | \in \left( \frac{1}{2}, 2
    \right), \qquad | \gamma | < 1 \tmmathbf{} .
  \end{aligned} \label{eq:gamma}
\end{equation}
Let us introduce the notation \(\theta_{\Gamma} \eqd \alpha \beta (\theta +
\gamma)\), and let \(\Theta_{\Gamma} \eqd \{ \theta_{\Gamma} \suchthat \theta
\in \Theta \}\). The former is an implicit function of \(\theta \in \R\).

First we study how \(\Gamma\) acts on trees and forests. It holds that \\\(\Gamma
(T) \subset T_{\Theta_{\Gamma}} (\alpha \xi_T, x_T, s_T)\) and, in particular,
we have
\[ \Gamma \circ \pi_T (\theta, \zeta, \sigma) = \pi_{(\alpha \xi_T, x_T,
   s_T)} (\theta_{\Gamma}, \zeta, \beta \sigma) ; \]
furthermore, \(\Gamma (W) \in \mathbb{T}^{\cup}_{\Theta_{\Gamma}}\) for any \(W
\in \TT^{\cup}_{\Theta}\). This implies that \(\mu_{\Theta}^{\mf{p}} (W) =
\mu_{\Theta_{\Gamma}}^{\mf{p}} (\Gamma (W))\). For any \(f \in \Sch (\R)\) we have that (cfr. \eqref{eq:piT-symmetry})
\begin{equation}
  \begin{aligned}[t]&
    \Emb [f] \circ \Gamma \circ \pi_T (\theta, \zeta, \sigma)\\ &
    \qquad = e^{2 \pi i \alpha \xi_T s_T \zeta} \Emb \left[ s_T \Dil_{s_T^{-
    1}} \Mod_{- \alpha \xi_T} \Tr_{- x_T} f \right] \circ \pi_{(0, 0, 1)}
    (\theta_{\Gamma}, \zeta, \beta \sigma) .
  \end{aligned} \label{eq:piT-Gamma-symmetry}
\end{equation}
Next, we define pullbacks of some of the sizes introduced above in
\Cref{sec:sizes}: if \(\SO\) is any of the local sizes we have so far defined we
ensure that for any \(F \in \Rad (\R^{3}_{+}) {\otimes
\Phi_{\gr}^{\infty}}'\) supported in compactly in \(\Gamma (T)\) we have that
\[ {\| F \circ \Gamma \|_{\Gamma^{\ast} \SO_{ \Theta} (T)}}  = {\| F
   \|_{\SO_{\Theta_{\Gamma}} (T_{{\Theta_{\Gamma}} } (\alpha \xi_T, x_T,
   s_T))}}  . \]
The pullbacks of Lebesgue (local) sizes are then given by
\begin{equation}
  \begin{aligned}[t]&
    \|F\|_{\Gamma^{\ast} \SL{}^{(u, v)}_{\Theta} (T)} \eqd \|F\|_{\SL{}^{(u,
    v)}_{\Theta} (T)} ;
  \end{aligned} \label{eq:gamma-Lebesgue-size}
\end{equation}
the pullbacks of the lacunary (local) sizes are given by
\begin{equation}
  \begin{aligned}[t]&
    \|F\|_{\Gamma^{\ast} {\SL  }^{(u, v)}_{\Theta} \wpD_{\gr}^{N } (T)} \eqd
    \left\|  \Gamma^{\ast} \wpD_{\zeta} (t (\eta - \xi_T)) {F } 
    \right\|_{{\SL  }^{(u, v)}_{\Theta} \Phi_{\mf{r}}^N (T)},\\ &
    \Gamma^{\ast} \wpD_{\zeta} (\theta) \eqd \wpD_{\zeta} (\theta_{\Gamma}) ;
  \end{aligned} \label{eq:gamma-lacunary-size}
\end{equation}
the pullbacks of the defect (local) sizes are given by
\begin{equation}
  \begin{aligned}[t]&
    \|F\|_{\Gamma^{\ast} {\SL_{\Theta }^{(u, v)}}  \dfct_{\zeta, \gr}^{N }
    (T)} \assign {\left\| \left( \Gamma^{\ast} \wpD_{\zeta} (t (\eta -
    \xi_T)) - \beta t \partial_y + 2 \pi i \alpha \xi_T \right) F
    \right\|_{\SL_{\Theta}^{(u, v)} \Phi_{\mf{r}}^N (T)}} ,\\ &
    \|F\|_{\Gamma^{\ast} \SL^{(u, v)}_{\Theta } \dfct_{\sigma, \gr}^{N } (T)}
    \assign \left\| \left( \Gamma^{\ast} \wpD_{\sigma} (t (\eta - \xi_T)) - t
    \partial_t + (\eta - \xi_T) \partial_{\eta} \right) F
    \right\|_{\SL_{\Theta}^{(u, v)} \Phi_{\mf{r}}^N (T)},\\ &
    \Gamma^{\ast} \wpD_{\sigma} (\theta) \eqd \wpD_{\sigma} (\theta_{\Gamma})
    .
  \end{aligned}  \label{eq:gamma-defect-size}
\end{equation}
and we set \(\Gamma^{\ast} \SL^{(u, v)}_{\Theta } \dfct_{\gr}^{N } (T) =
\Gamma^{\ast} \SL^{(u, v)}_{\Theta } \dfct_{\sigma, \gr}^{N } (T) +
\Gamma^{\ast} \SL^{(u, v)}_{\Theta } \dfct_{\zeta, \gr}^{N } (T)\). Similarly
to the computation \eqref{eq:embedded-no-defect}, using
\eqref{eq:piT-Gamma-symmetry}, one can see that
\begin{equation}
  \begin{aligned}[t]&
    \left( \Gamma^{\ast} \wpD_{\zeta} (t (\eta - \xi_T)) - \beta t \partial_y
    + 2 \pi i \alpha \xi_T \right) F\\ &
    \qquad = \left( \Gamma^{\ast} \wpD_{\sigma} (t (\eta - \xi_T)) - t
    \partial_t + (\eta - \xi_T) \partial_{\eta} \right) F = 0
  \end{aligned} \label{eq:no-defect-gamma}
\end{equation}
if \(F = \Emb [f] \circ \Gamma\). We do not make use of the pullback of the
local sizes \(\SJ_{\Theta_{\Gamma}}^u \Phi_{\gr}^N (T)\) since these sizes only
appear in relation to the term \(\Emb [f_1] \circ \Gamma_1\), for which
\(\Gamma_1 = \Id\).

We introduce frequency restricted variants of all the sizes above. Let us fix
\(\tilde{\Theta} \subset \Theta\); denoting by \(\SO_{\Theta}\) any of the (local)
sizes above we set
\begin{equation}
  \|F\|_{\SO _{(\Theta, \tilde{\Theta})} (T)} \assign \| \1_{\pi_T \left(
  \mT_{\tilde{\Theta}} \right)} F\|_{\SO _{\Theta} (T)},
  \label{eq:in-out-size}
\end{equation}
with \(\mT{}_{\tilde{\Theta}} \eqd \left\{ (\theta, \zeta, \sigma) \in
\tilde{\Theta} \times B_1 \times (0, 1) \st \sigma < 1 - | \zeta | \right\}\).
Setting \(\tilde{\Theta}_{\Gamma} \eqd \{ \theta_{\Gamma} \suchthat \theta \in
\tilde{\Theta} \}\), we have that
\[ {\| F \circ \Gamma \|_{\Gamma^{\ast} \SO_{ (\Theta, \tilde{\Theta})} (T)}} 
   = {\| F \|_{\SO_{(\Theta_{\Gamma}, \tilde{\Theta}_{\Gamma})}
   (T_{{\Theta_{\Gamma}} } (\alpha \xi_T, x_T, s_T))}}  \]
for any \(F \in \Rad (\R^{3}_{+}) {\otimes \Phi_{\gr}^{\infty}}'\)
supported compactly in \(\Gamma (T)\). This concludes the definition of the
pullback sizes and provides us with all the elements necessary to understand
\eqref{eq:uniform-embedding-full-size:linear} that defines the size
\(\widetilde{\SF}^u_{\Gamma} \Phi_{\mf{r}}^N\). The single summands in
\eqref{eq:uniform-embedding-full-size:linear} are related to the local sizes
above through \eqref{eq:generated-size:trees}.

The sizes defined so far are well suited to control the single functions \(F_j
= \Emb [f_j] \circ \Gamma_j\). The size \(\SI_{\Theta} \Phi_{\gr}^N\), defined
below, controls the product term
\[ \prod_{j = 1}^3 \Emb [f_j] \circ \Gamma_j \]
appearing in {\RHS{\eqref{eq:BHT-dual-bounds}}}. We make sure that this size
is large enough to satisfy the hypothesis of \Cref{prop:outer-RN} to be able
to conclude that \eqref{eq:RN-domination-BHT} holds. On the other hand, the
size \(\SI_{\Theta} \Phi_{\gr}^N\) has to be small enough to be able to control
it by the sizes \(\SF^u_{\Theta} \Phi_{\gr}^N\) of \(\Emb [f_1]\) by the sizes
\(\widetilde{\SF}^{u_j}_{\Gamma_j} \Phi_{\gr}^N\) and
\(\widetilde{\SF}^{u_{\times}}_{\Gamma_{\times}} \Phi_{\gr}^N\), defined
subsequently, of the functions \(\Emb [f_j] \circ \Gamma_j\), \(j \in \{ 2, 3 \}\)
and of the function \(\Emb [f_2] \circ \Gamma_2  \Emb [f_3] \circ \Gamma_3\).
This is required to be able to show that
\eqref{eq:iterated-Holder-type-bounds:linear} and
\eqref{eq:iterated-Holder-type-bounds:bilinear} hold.

\begin{definition}[Integral size]
  \label{def:integral-sizes}Let us fix \(\omega \in \Phi_1^{\infty}\) such that
  \(\FT{\omega} \left( \FT{z} \right) = 1\) for \(\FT{z} \in B_{7 / 8}\). For any
  \(\phi \in \Phi_{\gr / 2}^{\infty}\) and \(j \in \{ 1, 2, 3 \}\), we define
  \begin{equation}
    \begin{aligned}[t]&
      \phi_{j, \theta}^o (z) \eqd \FT{\phi}  (\theta_{\Gamma_j}) \Dil_{\gr^{-
      1}} \omega (z),\\ &
      \phi_{j, \theta}^l (z) \eqd \phi (z) - \FT{\phi} (\theta_{\Gamma_j})
      \Dil_{\gr^{- 1}} \omega (z) .
    \end{aligned} \label{eq:lac-ov-wp-decomp}
  \end{equation}
  We define the local sizes \(\| \cdummy \|_{\SI^{(\star_1, \star_2,
  \star_3)}_{\Theta} \Phi_{\gr}^N (T)}\) with \((\star_1, \star_2, \star_3) \in
  \{ o, l \}^3\) as follows. For \(H \in \Rad (\R^{3}_{+}) \otimes^3
  {\Phi_{\gr}^{\infty}}'\) and \(T \in \mathbb{T}_{\Theta}\) set
  \[ \begin{aligned}[t]&
       \| H \|_{\SI^{(\star_1, \star_2, \star_3)}_{\Theta} \Phi_{\gr}^N (T)}
       \qquad\\ &
       \qquad \eqd \sup_{\tmscript{\begin{aligned}[t]&
         \phi
       \end{aligned}}} \liminf_{\epsilon \rightarrow 0} \int_{\Theta} \left|
       \int_{\epsilon}^{\infty} \left( \1_T H \right) \circ \pi_T (\theta,
       \zeta, \sigma) [\phi_{1, \theta}^{\star_1}, \phi_{2, \theta}^{\star_2},
       \phi_{3, \theta}^{\star_3}] \dd \zeta \frac{\dd \sigma}{\sigma} \right|
       \frac{\dd \theta}{| \Theta |},
     \end{aligned} \]
  where the upper bound is taken over \(\phi \in \Phi_{\gr / 2}^{\infty}\) with
  \(\| \phi \|_{\Phi_{\gr / 2}^N} \leq 1\). Using the local sizes above we have
  given meaning to the definition \eqref{eq:I-size-sum} of the size
  \(\SI_{\Theta} \Phi_{\gr}^N\). The single summands in \eqref{eq:I-size-sum} are
  related to the local sizes above through \eqref{eq:generated-size:trees}.
\end{definition}

We now define the (local) sizes used to control the product term \(\Emb [f_2]
\circ \Gamma_2  \Emb [f_3] \circ \Gamma_3\) appearing in
\eqref{eq:embedding-bound:bilinear}. Our procedure of bounding the term
\[ \left\| \Emb [f_1]  \Emb [f_2] \circ \Gamma_2 \Emb [f_3] \circ \Gamma_3
   \right\|_{\SI^{(l, o, o)}_{\Theta} \Phi_{\gr}^N} \]
that appears in \Cref{thm:STE} requires controlling quantities that cannot be
bounded uniformly by products of sizes of \(\Emb [f_2] \circ \Gamma_2\) and
\(\Emb [f_3] \circ \Gamma_3\). We believe that these additional quantities, that
arise in the proof of \Cref{thm:STE} and that we encoded in the size
\(\widetilde{\SF}^{u_{\times}}_{\Gamma_{\times}} \Phi_{\gr}^N\), is an
intrinsically bilinear feature dependent on the interaction between \(\Emb
[f_2] \circ \Gamma_2\) and of \(\Emb [f_3] \circ \Gamma_3\). For this reason, we
believe that our inability to control them using products of sizes of \(\Emb
[f_2] \circ \Gamma_2\) and \(\Emb [f_3] \circ \Gamma_3\) is intrinsic and not a
shortcoming of the current method.

In general, given two function \(G_2, G_3 \in L^{\infty}_{\tmop{loc}} (\R^{3}_{+}) {\otimes \Phi_{\gr}^{\infty}}'\) their pointwise product is a
function in \(G_{\times} \in L^{\infty}_{\tmop{loc}} (\R^{3}_{+})
\otimes^2 {\Phi_{\gr}^{\infty}}'\) given by
\[ G_{\times} (\eta, y, t) [\phi_2, \phi_3] \eqd G_2 (\eta, y, t) [\phi_2] G_3
   (\eta, y, t) [\phi_3] . \]
This brings us to define the following product (local) sizes indexed by \(T \in
\TT_{\Theta}\) on functions \(H \in \Rad (\R^{3}_{+}) \otimes^2
{\Phi_{\gr}^{\infty}}'\). In our notation we ``remember'' both diffeomorphisms
\(\Gamma_2\) and \(\Gamma_3\) by including the symbol \(\Gamma_{\times}^{\ast}\).
The product Lebesgue local size is given by
\[ \| H \|_{\Gamma_{\times}^{\ast} \SL{}^{(u, v)}_{\Theta} (\Phi \otimes
   \Phi)_{\gr}^N (T)} \eqd \sup_{\tmscript{\begin{aligned}[t]&
     \phi_2, \phi_3
   \end{aligned}}} \|H [\phi_2, \phi_3] \|_{\SL{}^{(u, v)}_{\Theta} (T)} \]
where the upper bound is taken over all \(\phi_j \in \Phi_{\gr}^{\infty}\) with
\(\| \phi_j \|_{\Phi_{\gr}^N} \leq 1\), \(j \in \{ 2, 3 \}\). The product lacunary
local sizes, with \(\theta_{\Gamma_j} \eqd \alpha_j \beta_j (\theta +
\gamma_j)\), \(j \in \{ 2, 3 \}\), are given by
\begin{equation}
  \begin{aligned}[t]&
    \|H\|_{\Gamma_{\times}^{\ast} {\SL  }^{(u, v)}_{\Theta} \left( \Phi
    \otimes \wpD \right)_{\gr}^{N } (T)} \eqd \left\|  \left( \Id \otimes 
    \Gamma^{\ast}_3 \wpD_{\zeta} (t (\eta - \xi_T)) \right) H
    \right\|_{\Gamma_{\times}^{\ast} \SL{}^{(u, v)}_{\Theta} (\Phi \otimes
    \Phi)_{\gr}^N (T)},\\ &
    \|H\|_{\Gamma_{\times}^{\ast} {\SL  }^{(u, v)}_{\Theta} \left( \wpD
    \otimes \Phi \right)_{\gr}^{N } (T)} \eqd \left\|  \left( \Gamma^{\ast}_2
    \wpD_{\zeta} (t (\eta - \xi_T)) \otimes \Id \right) H
    \right\|_{\Gamma_{\times}^{\ast} \SL{}^{(u, v)}_{\Theta} (\Phi \otimes
    \Phi)_{\gr}^N (T)},\\ &
    \|H\|_{\Gamma_{\times}^{\#} {\SL  }^{(u, v)}_{\Theta} \left( \wpD \otimes
    \wpD \right)_{\gr}^{N } (T)}\\ &
    \qquad \eqd \left\|  \left( \Gamma^{\ast}_2 \wpD_{\zeta} (t (\eta -
    \xi_T)) \otimes \Gamma^{\ast}_3 \wpD_{\zeta} (t (\eta - \xi_T)) \right) {H
    }  \right\|_{\Gamma_{\times}^{\ast} \SL{}^{(u, v)}_{\Theta} (\Phi \otimes
    \Phi)_{\gr}^N (T)} .
  \end{aligned} \label{eq:product-tensor-sizes}
\end{equation}
The linear operators \(\left( \Id \otimes  \Gamma^{\ast}_3 \wpD_{\zeta}
(\theta) \right)\), \(\left( \Gamma^{\ast}_2 \wpD_{\zeta} (\theta) \otimes \Id
\right)\), and \(\left( \Gamma^{\ast}_2 \wpD_{\zeta} (\theta) \otimes
\Gamma^{\ast}_3 \wpD_{\zeta} (\theta) \right)\), defined on functions \(H \in
\Rad (\R^3_+) \otimes^2 {\Phi_{\gr}^{\infty}}'\), are given by
\begin{equation}
  \begin{aligned}[t]&
    \left( \Id \otimes \Gamma^{\ast}_3 \wpD_{\zeta} (\theta) \right) H
    (\theta, \zeta, \sigma) [\phi_2, \phi_3] \eqd H (\theta, \zeta, \sigma)
    [\phi_2, (- d_z + 2 \pi i \theta_{\Gamma_3}) \phi_3 (z)],\\ &
    \left( \Gamma^{\ast}_2 \wpD_{\zeta} (\theta) \otimes \Id \right) H
    (\theta, \zeta, \sigma) [\phi_2, \phi_3] \eqd H (\theta, \zeta, \sigma)
    [(- d_z + 2 \pi i \theta_{\Gamma_2}) \phi_2 (z), \phi_3],\\ &
    \begin{aligned}[t]&
      \left( \Gamma^{\ast}_2 \wpD_{\zeta} (\theta) \otimes \Gamma^{\ast}_3
      \wpD_{\zeta} (\theta) \right) H (\theta, \zeta, \sigma) [\phi_2, \phi_3]
      \qquad\\ &
      \qquad \eqd H (\theta, \zeta, \sigma) [(- d_z + 2 \pi i
      \theta_{\Gamma_2}) \phi_2 (z), (- d_z + 2 \pi i \theta_{\Gamma_3})
      \phi_2 (z)] .
    \end{aligned}
  \end{aligned} \label{eq:tensor-boost}
\end{equation}
To define product defect sizes, we also introduce the operators \(\left( \Id
\otimes \Gamma^{\ast}_3 \wpD_{\sigma} (\theta) \right)\), \(\left(
\Gamma^{\ast}_2 \wpD_{\sigma} (\theta) \otimes \Id \right)\) and \(\left(
\Gamma^{\ast}_2 \wpD_{\sigma} (\theta) \otimes \Gamma^{\ast}_3 \wpD_{\sigma}
(\theta) \right)\) analogously to \eqref{eq:tensor-boost} with \(\wpD_{\sigma}\)
in place of \(\wpD_{\zeta}\). We set
\begin{equation}
  \begin{aligned}[t]&
    \|H\|_{\Gamma^{\ast}_{\times} {\SL_{\Theta}^{(u, v)}}  \dfct_{\zeta, \gr
    }^{N } (T)}\\ &
    \quad \eqd \left\| \left( \Gamma^{\ast}_2 \wpD_{\zeta} (\theta) \otimes
    \tmop{Id} + \tmop{Id} \otimes \Gamma^{\ast}_3 \wpD_{\zeta} (\theta) -
    \beta t \dd_y + 2 \pi i (\alpha_2 + \alpha_3) \xi_T \right) H
    \right\|_{\Gamma_{\times}^{\ast} \SL{}^{(u, v)}_{\Theta} (\Phi \otimes
    \Phi)_{\gr}^N (T)},
  \end{aligned} \label{eq:product-space-defect}
\end{equation}

\begin{equation}
  \begin{aligned}[t]&
    \|H\|_{\Gamma^{\ast}_{\times} {\SL_{\Theta}^{(u, v)}}  \dfct_{\sigma, \gr
    }^{N } (T)}\\ &
    \qquad \eqd \begin{aligned}[t]&
      \left\| \left( \Gamma^{\ast}_2 \wpD_{\sigma} (\theta ) \otimes
      \tmop{Id} + \tmop{Id} \otimes \Gamma^{\ast}_3 \wpD_{\sigma} (\theta ) -
      t \dd_t + (\eta - \xi_T) \dd_{\eta} \right) H
      \right\|_{\Gamma_{\times}^{\ast} \SL{}^{(u, v)}_{\Theta} (\Phi \otimes
      \Phi)_{\gr}^N (T)} .
    \end{aligned}
  \end{aligned} \label{eq:product-scale-defect}
\end{equation}

Similarly to the computation \eqref{eq:embedded-no-defect}, using
\eqref{eq:piT-Gamma-symmetry}, one can see that
\begin{equation}
  \begin{aligned}[t]&
    \left( \Gamma^{\ast}_2 \wpD_{\zeta} (\theta) \otimes \tmop{Id} + \tmop{Id}
    \otimes \Gamma^{\ast}_3 \wpD_{\zeta} (\theta) - \beta t \dd_y + 2 \pi i
    (\alpha_2 + \alpha_3) \xi_T \right) H\\ &
    \qquad = \left( \Gamma^{\ast}_2 \wpD_{\sigma} (\theta ) \otimes \tmop{Id}
    + \tmop{Id} \otimes \Gamma^{\ast}_3 \wpD_{\sigma} (\theta ) - t \dd_t +
    (\eta - \xi_T) \dd_{\eta} \right) H = 0
  \end{aligned} \label{eq:no-defect-product}
\end{equation}
if \(H = \Emb [f_2] \circ \Gamma_2  \Emb [f_3] \circ \Gamma_3\). We denote
\(\Gamma^{\ast}_{\times} {\SL_{\Theta}^{(u, v)}}  \dfct_{\gr }^{N } (T) \eqd
\Gamma^{\ast}_{\times} {\SL_{\Theta}^{(u, v)}}  \dfct_{\zeta, \gr }^{N } (T) +
\Gamma^{\ast}_{\times} {\SL_{\Theta}^{(u, v)}}  \dfct_{\sigma, \gr }^{N }
(T)\). This concludes the definition of the bilinear sizes and provides us with
all the elements necessary to understand
\eqref{eq:uniform-embedding-full-size:bilinear} that defines the size
\(\SF^u_{\Theta} \Phi_{\gr}^N\). The single summands in
\eqref{eq:uniform-embedding-full-size:bilinear} are related to the local sizes
above through \eqref{eq:generated-size:trees}.
\[  \]

\subsection{Localized outer Lebesgue
spaces}\label{sec:localized-outer-lebesgue}

In this section we localize the outer Lebesgue quasi-norm
\(L^p_{\mu^{\mf{p}}_{\Theta}} \SO\) to spatial regions. Iterated outer measure
spaces appear already in the context of non-uniform bounds in
\eqref{eq:BHT-dual-bounds}. In the range of exponents \(p_j \in (2, \infty]\),
\(j \in \{ 1, 2, 3 \}\) (triangle \(c\) in \Cref{fig:triangle}) it is possible to
reason at the level of non-iterated outer Lebesgue quasi-norms both for
non-uniform (see {\cite{doTheoryOuterMeasures2015}}) and uniform bounds (we do
not elaborate on this in this paper). Introducing localization is necessary to
prove \eqref{eq:BHT-dual-bounds} outside of the regime \(p_j \in (2, \infty]\).
This necessity is dictated by the fact that the embedding bounds
\[ \left\| \Emb [f_j] \circ \Gamma_j \right\|_{L^{p_j}_{\mu^1_{\Theta}}
   \SO_{\tmop{en}, j}} \lesssim \| f_j \|_{L^{p_j} (\R)} . \]
are known not to hold when \(p_j \leq 2\). Iterated outer Lebesgue norms can be
thought of as appropriate generalizations of product Lebesgue norms to the
outer Lebesgue setting. The symbols \(\fL^{q, +}_{\mu_{\Theta}^{\infty},
\nu_1}\), \(\fL^{q , +}_{\mu_{\Theta}^1, \nu_{\beta}}\), and \(X^{q, r,
+}_{\mu^1_{\Theta}, \nu_{\beta}}\), defined subsequently, play the role of
sizes on analogues of fibers (in our setup: strips). These sizes are
themselves outer Lebesgue quasi-norms on said fibers. This is completely
analogous to the classical fact that the product Lebesgue norm
\[ \| f (x, y) \|_{L^p_{x \in \R} L^q_{y \in \R}} \]
is defined by taking the Lebesgue norm \(L^p_{x \in \R}\) of the function \(x
\mapsto \| f (x, y) \|_{L^q_{y \in \R}}\) where, effectively, the magnitude of
the function on each fiber \(\{ x \} \times \R\) is measured by taking its \(L^q\)
norm. The difficulty in our setup is that different fibers, or ``strips'', are
not necessarily disjoint in largely the same way as different trees are not
necessarily disjoint either.

We begin by introducing a family of distinguished sets:
\(\beta\)-\tmtextit{strips}. Countable unions thereof will play the role of a
\(\sigma\)-generating collection As for trees (\Cref{def:tree}) and sizes in
\Cref{sec:sizes}, our localization procedure must be adapted to each map
\(\Gamma_j\) separately. This calls for having strips depend on the parameter
\(\beta \in (0, 1]\).

\begin{definition}[Strips and strip outer measure]
  Let \(\beta \in (0, 1]\). A strip \(D_{\beta} (x_D, s_D)\) with top \((x_D, s_D)
  \in \R^2_+\) is given by
  \[ D_{\beta^{}} (x_D, s_D) = \{ (\eta, y, t) : \beta t < \beta (s_T - | y -
     x_T |) \} . \]
  We denote by \(\mathbb{D}_{\beta }\) the collection of all strips \(D_{\beta}
  (x_T, s_T)\). The outer measure \(\nu_{\beta} (E)\) is given by
  \[ \begin{aligned}[t]&
       \nu_{\beta} (E) \eqd \inf \left\{ \|N_{\mc{D}} (z)\|_{L^1_{\dd z} (\R)}
       \st \mc{D} \subset \DD_{\beta}, \hspace{0.27em} E \subset \bigcup_{D
       \in \mc{D}} D \right\},\\ &
       N_{\mc{D}} (z) \eqd \sum_{D \in \mc{D}} \1_{B_{s_T} (x_T)} (z) .
     \end{aligned} \]
\end{definition}

There is a one-to-one correspondence between the datum \(\beta, (x_D, s_D)\) and
a given strip. We call \(B_{s_D} (x_D)\) the spatial interval of a strip. We
denote countable unions of strips in \(\DD_{\beta}\) by \(\mathbb{D}_{\beta
}^{\cup}\). We define two ways to localize the outer Lebesgue quasi-norms
\(L^q_{\mu^{\mf{p}}_{\Theta}} \SO\) to unions of strips; here \(\SO\) can be any
size on \(\Rad (\R^{3}_{+}) \otimes^k {\Phi_{\gr}^{\infty}}'\).

\begin{definition}[Localized outer Lebesgue quasi-norms]
  \label{def:localized-outer-lebesgue} Let \(\beta \in (0, 1]\) and let \(\SO\) be
  a size on \(\Rad (\R^{3}_{+}) \otimes^k {\Phi_{\gr}^{\infty}}'\); we
  define the localized outer Lebesgue (local) sizes
  \(\fL^q_{\mu^{\infty}_{\Theta}, \nu_{\beta}} \SO\) and \(\fL^{q,
  +}_{\mu^{\infty}_{\Theta}, \nu_{\beta}} \SO\) as follows. For \(\Rad (\R^{3}_{+}) \otimes^k {\Phi_{\gr}^{\infty}}'\) and \(V^+ \in
  \DD_{\beta}^{\cup}\) set
  \[ \begin{aligned}[t]&
       \| F \|_{\fL^q_{\mu^{\infty}_{\Theta}, \nu_{\beta}} \SO (V^+)} \eqd
       \left\| \1_{V^+} F \right\|_{L^q_{\mu^{\infty}_{\Theta}} \SO}\\ &
       \| F \|_{\fL^{q, +}_{\mu^{\infty}_{\Theta}, \nu_{\beta}} \SO (V^+)}
       \eqd \sup_{\bar{q} \in [q, \infty]} \| F
       \|_{\fL^{\bar{q}}_{\mu^{\infty}_{\Theta}, \nu_{\beta}} \SO (V^+)} .
     \end{aligned} \]
  Similarly, we define \(\fL^q_{\mu^1_{\Theta}, \nu_{\beta}} \SO\) and \(\fL^{q,
  +}_{\mu^1_{\Theta}, \nu_{\beta}} \SO\); for \(F \in \Rad (\R^{3}_{+})
  \otimes^k {\Phi_{\gr}^{\infty}}'\) and \(V^+ \in \DD_{\beta}^{\cup}\) set
  \[ \begin{aligned}[t]&
       \| F \|_{\fL^q_{\mu^1_{\Theta}, \nu_{\beta}} \SO (V^+)} \eqd
       \nu_{\beta} (V^+)^{- \frac{1}{q}} \left\| \1_{V^+} F
       \right\|_{L^q_{\mu_{\Theta}^1} \SO (V^+)},\\ &
       \| F \|_{\fL^{q, +}_{\mu^1_{\Theta}, \nu_{\beta}} \SO (V^+)} \eqd
       \sup_{\bar{q} \in [q, \infty]} \| F \|_{\fL^{\bar{q}}_{\mu^1_{\Theta},
       \nu_{\beta}} \SO (V^+)} .
     \end{aligned} \]
  We then set, analogously to \eqref{eq:generated-size:trees},
  \[ \begin{aligned}[t]&
       \| F \|_{\fL^q_{\mu^{\infty}_{\Theta}, \nu_{\beta}} \SO} \eqd \sup_{V^+
       \in \DD_{\beta}^{\cup}} \sup_{V^- \in \DD_{\beta}^{\cup}} \left\|
       \1_{\R^3_+ \setminus V^-} F \right\|_{\fL^q_{\mu^{\infty}_{\Theta},
       \nu_{\beta}} \SO (V^+)},\\ &
       \| F \|_{\fL^{q, +}_{\mu^{\infty}_{\Theta}, \nu_{\beta}} \SO} \eqd
       \sup_{V^+ \in \DD_{\beta}^{\cup}} \sup_{V^- \in \DD_{\beta}^{\cup}}
       \left\| \1_{\R^3_+ \setminus V^-} F \right\|_{\fL^{q,
       +}_{\mu^{\infty}_{\Theta}, \nu_{\beta}} \SO (V^+)},\\ &
       \| F \|_{\fL^q_{\mu^1_{\Theta}, \nu_{\beta}} \SO} \eqd \sup_{V^+ \in
       \DD_{\beta}^{\cup}} \sup_{V^- \in \DD_{\beta}^{\cup}} \left\|
       \1_{\R^3_+ \setminus V^-} F \right\|_{\fL^q_{\mu^1_{\Theta},
       \nu_{\beta}} \SO (V^+)},\\ &
       \| F \|_{\fL^{q, +}_{\mu^1_{\Theta}, \nu_{\beta}} \SO} \eqd \sup_{V^+
       \in \DD_{\beta}^{\cup}} \sup_{V^- \in \DD_{\beta}^{\cup}} \left\|
       \1_{\R^3_+ \setminus V^-} F \right\|_{\fL^{q, +}_{\mu^1_{\Theta},
       \nu_{\beta}} \SO (V^+)},
     \end{aligned} \]
  to obtain the corresponding sizes.
  
  We also define the local sizes \(X^{q, r, +}_{\mu^1_{\Theta}, \nu_{\beta}}
  \SO\) as follows. For \(F \in \Rad (\R^{3}_{+}) \otimes^k
  {\Phi_{\gr}^N}'\) and \(V^+ \in \DD_{\beta}^{\cup}\) set
  \begin{equation}
    \begin{aligned}[t]&
      \| F \|_{X^{q, r}_{\mu^1_{\Theta}, \nu_{\beta}} \SO (V^+)} \eqd
      \nu_{\beta} (V^+)^{- \frac{1}{r}} \sup_{W^+ \in \TT^{\cup}_{\Theta}}
      \left( \mu^{\infty}_{\Theta} (W^+)^{\frac{1}{q} - \frac{1}{r}} \left\|
      \1_{(V^+ \cap W^+)} F \right\|_{L^r_{\mu_{\Theta}^1} \SO} \right),\\ &
      \| F \|_{X^{q, r, +}_{\mu^1_{\Theta}, \nu_{\beta}} \SO (V^+)} \eqd
      \sup_{\bar{r} \in [r, q]} \| F \|_{X^{q, \bar{r}}_{\mu^1_{\Theta},
      \nu_{\beta}} \SO (V^+)}
    \end{aligned} \label{eq:X-local-size}
  \end{equation}
  and we get the corresponding sizes by setting
  \begin{equation}
    \begin{aligned}[t]&
      \| F \|_{X^{q, r}_{\mu^1_{\Theta}, \nu_{\beta}} \SO} \eqd \sup_{V^+ \in
      \DD_{\beta}^{\cup}} \sup_{V^- \in \DD_{\beta}^{\cup}} \left\| \1_{\R^3_+
      \setminus V^-} F \right\|_{X^{q, r}_{\mu^1_{\Theta}, \nu_{\beta}} \SO
      (V^+)},\\ &
      \| F \|_{X^{q, r, +}_{\mu^1_{\Theta}, \nu_{\beta}} \SO} \eqd \sup_{V^+
      \in \DD_{\beta}^{\cup}} \sup_{V^- \in \DD_{\beta}^{\cup}} \left\|
      \1_{\R^3_+ \setminus V^-} F \right\|_{X^{q, r, +}_{\mu^1_{\Theta},
      \nu_{\beta}} \SO (V^+)} .
    \end{aligned} \label{eq:X-size}
  \end{equation}
\end{definition}

Let us now discuss the meaning of the sizes above. The sizes \(\| F
\|_{\fL^q_{\mu^1_{\Theta}, \nu_{\beta}} \SO}\) were first introduced in
{\cite{uraltsevVariationalCarlesonEmbeddings2016}} (cfr.
{\cite[(2.14)]{uraltsevVariationalCarlesonEmbeddings2016}}) albeit their local
variants were indexed by strips \(D \in \DD_1\) rather than by countable unions
\(V_+\) thereof, as in \Cref{def:localized-outer-lebesgue}. For the case \(\beta
= 1\) these two approaches are functionally the same while for \(\beta \ll 1\)
they yield non uniformly equivalent sizes. Note that the sizes \(\fL^{q,
+}_{\mu^{\mf{p}}_{\Theta}, \nu_{\beta}} \SO\) are monotone in \(q\) in the sense
that \(\| F \|_{\fL^{q, +}_{\mu^{\mf{p}}_{\Theta}, \nu_{\beta}} \SO} \geq \| F
\|_{\fL^{\bar{q}, +}_{\mu^{\mf{p}}_{\Theta}, \nu_{\beta}} \SO}\) if \(q <
\bar{q}\).

The size \(X^{q, r}_{\mu^1_{\Theta}, \nu_{\beta}} \SO\) are weaker variants of
the sizes \(\| F \|_{\fL^q_{\mu^1_{\Theta}, \nu_{\beta}} \SO}\). Clearly
\[ \| F \|_{\left( X^{q, q}_{\mu^1_{\Theta}, \nu_{\beta }} \SO \right) (V)} =
   \nu_{\beta } (V)^{- \frac{1}{q}} \left\| \1_V F
   \right\|_{L^q_{\mu_{\Theta}^1} \SO} = \left\| \1_V F
   \right\|_{\fL^q_{\mu^1_{\Theta}, \nu_{\beta}} \SO (V^+)} . \]
On the other hand, for \(r \leq q\) we have that
\[ \begin{aligned}[t]&
     \| F \|_{\left( X^{q, r}_{\mu^1_{\Theta}, \nu_{\beta }} \SO \right) (V)}
     = \sup_{W \in \TT^{\cup}_{\Theta}} \nu_{\beta} (V)^{- \frac{1}{r}}
     \mu^{\infty}_{\Theta} (W)^{\frac{1}{q} - \frac{1}{r}} \left\| \1_V \1_W F
     \right\|_{L^r_{\mu_{\Theta}^1} \SO}\\ &
     \lesssim \sup_{W \in \TT^{\cup}_{\Theta}} \beta^{- \frac{1}{r}}
     \mu^{\infty}_{\Theta} (W)^{\frac{1}{q} - \frac{1}{r}} \left\| \1_V \1_W F
     \right\|_{L^r_{\mu_{\Theta}^{\infty}} \SO} \leq \beta^{- \frac{1}{r}}
     \left\| \1_V F \right\|_{L^q_{\mu_{\Theta}^{\infty}} \SO}
   \end{aligned} \]
showing that \(\| F \|_{\left( X^{q, 1, +}_{\mu^1_{\Theta}, \nu_{\beta}} \SO
\right) (V)} \leq \beta^{- 1} \| F \|_{\fL^q_{\mu^{\infty}_{\Theta},
\nu_{\beta}} \SO (V)}\). In practice, this means that the local sizes \(\| F
\|_{\left( X^{q, r, +}_{\mu^1_{\Theta}, \nu_{\beta}} \SO \right) (V)}\) are
weaker than \(\| F \|_{\fL^q_{\mu^{\infty}_{\Theta}, \nu_{\beta}} \SO (V)}\) for
\(\beta = 1\) but they may be better behaved when \(\beta\) is small.

The claimed bound above follows from the following lemma.

\begin{lemma}
  \label{lem:measure-comparison-1-infty}Given any \(V \in \DD_{\beta}^{\cup}\)
  and any \(W \in \TT^{\cup}_{\Theta}\) it holds that
  \[ \mu^1_{\Theta} (W \cap V) \leq 4 \beta^{- 1} \nu_{\beta} (V)
     \mu^{\infty}_{\Theta} (W \cap V) . \]
\end{lemma}

\begin{proof}
  Given any strip \(D \in \mathbb{D}_{\beta}\) and any tree \(T \in
  \mathbb{T}_{\Theta}\) it holds that
  \[ T (\xi_T, x_{D, T}, 2 \beta^{- 1} s_{D, T}) \supset T \cap D \]
  where \(B_{s_{D, T}} (x_{D, T}) \eqd I_T \cap I_D\). From this it follows that
  \[ \mu^1_{\Theta} (W \cap V) \leq 2 \beta^{- 1} \nu  (V)
     \mu^{\infty}_{\Theta} (W \cap V) \]
  for any \(V \in \DD_{\beta}^{\cup}\) and any \(W \in \TT^{\cup}_{\Theta}\). As a
  matter of fact, let \(\mc{D} \subset \DD_{\beta}\) be such that \(V \subset
  \bigcup_{D \in \mc{D}} D\) and \(\|N_{\mc{D}} \|_{L^1 } \leq 2 \nu_{\beta}
  (V)\) and let \(\mc{T} \subset \mathbb{T}_{\Theta}\) be such that
  \[ \left\| N_{\mc{T}} \right\|_{L^{\infty}} \leq 2 \mu_{\Theta}^{\infty} (W
     \cap V), \qquad W \cap V \subset \bigcup_{T \in \mc{T}} T. \]
  It then holds that
  \[ \sum_{D \in \mc{D}} \sum_{T \in \mc{T}} | I_D \cap I_T | = \int \sum_{D
     \in \mc{D}} \sum_{T \in \mc{T}} \1_{I_D} (z) \1_{I_T} (z) \dd z \leq
     \left\| N_{\mc{D}} \right\|_{L^1} \left\| N_{\mc{T}}
     \right\|_{L^{\infty}} . \]
  Since \(W \cap V \subset \bigcup_{T \in \mc{T}} \bigcup_{D \in \mc{D}} (D
  \cap T)\) we have that
  \[ W \cap V \subset \bigcup_{D \in \mc{D}} \bigcup_{T \in \mc{T}} T (\xi_T,
     x_{D, T}, 2 \beta^{- 1} s_{D, T}), \]
  so \(\mu^1_{\Theta} (W \cap V) \leq \sum_{D \in \mc{D}} \sum_{T \in \mc{T}} 4
  \beta^{- 1} s_{D, T} = \sum_{D \in \mc{D}} \sum_{T \in \mc{T}} 2 \beta^{- 1}
  | I_D \cap I_T |\), as required
\end{proof}

Finally, note that the sizes \(X^{q, r, +}_{\mu^{\mf{p}}_{\Theta}, \nu_{\beta}}
\SO\) are monotone in \(q\) and \(r\) in the sense that \(\| F \|_{X^{q, r,
+}_{\mu^1_{\Theta}, \nu_{\beta}} \SO} \geq \| F \|_{\fL^{\bar{q}, r,
+}_{\mu^1_{\Theta}, \nu_{\beta}} \SO}\) if \(q < \bar{q}\) and \(\| F \|_{X^{q, r,
+}_{\mu^1_{\Theta}, \nu_{\beta}} \SO} \geq \| F \|_{\fL^{q, \bar{r},
+}_{\mu^1_{\Theta}, \nu_{\beta}} \SO}\) if \(r < \bar{r}\).

When proving iterated bounds of the form
\eqref{eq:embedding-bounds:non-uniform-iterated} and
\eqref{eq:embedding-bounds:uniform-iterated} it is necessary to unwrap the
sequence of definitions above; since this is going to be a recurrent procedure
in this work we find it useful to explicitly go through this procedure in the
following remark.

\begin{remark}
  \label{rmk:iterated-outer-bounds-explicit} Let us unwrap
  \Cref{def:localized-outer-lebesgue}: in that setup suppose we need to show
  that
  \[ \| F \|_{L^p_{\nu_{\beta}} \fL^q_{\mu^{\mf{p}}_{\Theta}, \nu_{\beta}}
     \SO} \lesssim A \text{} \]
  with some \(A \in [0, + \infty)\) and an implicit constant independent of \(A\)
  and \(F \in \Rad (\R^{3}_{+}) \otimes^k {\Phi_{\gr}^N}'\). This is
  tantamount to exhibiting a choice of sets \({V^-}  (\lambda) \in
  \DD^{\cup}_{\beta}\) for any \(\lambda \in (0, + \infty)\) such that
  \begin{equation}
    \begin{aligned}[t]&
      \int_0^{\infty} \nu_{\beta} \left( {V^-}  (\lambda^{1 / p}) \right)
      \dd \lambda \lesssim A^p, \qquad \left\| \1_{\R^3_+ {\setminus V^-} 
      (\lambda)} F \right\|_{\fL^q_{\mu^{\mf{p}}_{\Theta}, \nu_{\beta}} \SO}
      \lesssim \lambda .
    \end{aligned} \label{eq:iterated-outer-bounds-explicit}
  \end{equation}
  The second condition amounts to exhibiting a choice of sets \(W_q (\tau, V^+,
  V^-) \in \TT_{\Theta}^{\cup}\) for any \(V^{\pm} \in \DD_{\beta}^{\cup}\) and
  \(\tau \in (0, + \infty)\) such that if \(V^- \supseteq \bar{V} (\lambda)\) then
  \begin{equation}
    \begin{aligned}[t]&
      \int_0^{\infty} \mu_{\Theta}^{\mf{p}} (W_q (\tau^{1 / q}, V^+, V^-))
      \dd \tau \lesssim \lambda^q \nu_{\beta} (V^+)^{\frac{1}{\mf{p}}},\\ &
      \left\| \1_{V^+ \setminus (V^- \cup W_q (\lambda, \tau, V^+, V^-))} F
      \right\|_{\SO} \lesssim \tau .
    \end{aligned} \label{eq:iterated-outer-bounds-explicit-full}
  \end{equation}
  When we are showing \(\| F \|_{L^p_{\nu_{\beta}} \fL^{q,
  +}_{\mu^{\mf{p}}_{\Theta}, \nu_{\beta}} \SO} \lesssim A\), then we need to
  show that \eqref{eq:iterated-outer-bounds-explicit} holds with \(\bar{q}\) in
  place of \(q\), for all \(\bar{q} \in \{ q, + \infty \}\). This follows by log
  convexity of outer Lebesgue quasi-norms (\eqref{eq:outer-log-convexity} in
  \Cref{prop:outer-properties}). All the implicit constants above cannot
  depend on any of the variable quantities \(\lambda, \tau, V^{\pm}\).
  
  Similarly, showing that
  \[ \| F \|_{L^p_{\nu_{\beta}} X^{q, r, +}_{\mu^1_{\Theta}, \nu_{\beta}}
     \SO} \lesssim A \]
  amounts to exhibiting a choice of sets \({V^-}  (\lambda) \in
  \DD^{\cup}_{\beta}\) for any \(\lambda \in (0, + \infty)\) such that
  \[ \int_0^{\infty} \nu_{\beta} \left( {V^-}  (\lambda^{1 / p}) \right)
     \dd \lambda \lesssim A^p, \qquad \left\| \1_{\R^3_+ {\setminus V^-} 
     (\lambda)} F \right\|_{X^{q, r, +}_{\mu^1_{\Theta}, \nu_{\beta}} \SO}
     \lesssim \lambda . \]
  According to \eqref{eq:X-size} and using log convexity of outer Lebesgue
  quasi-norms (\eqref{eq:outer-log-convexity} in
  \Cref{prop:outer-properties}). The second bound amount to showing that for
  any \(V^+ \in \DD_{\beta}^{\cup}\), \(W^+ \in \TT^{\cup}_{\Theta}\), it holds
  that
  \begin{equation}
    \left\| \1_{(V^+ \cap W^+) \setminus V^- (\lambda)} F
    \right\|_{L^{\bar{r}}_{\mu_{\Theta}^1} \SO} \lesssim \lambda \nu_{\beta}
    (V^+)^{\frac{1}{\bar{r}}} \mu^{\infty}_{\Theta} (W^+)^{\frac{1}{\bar{r}} -
    \frac{1}{q}} \label{eq:iterated-outer-bounds-explicit-X}
  \end{equation}
  with \(\bar{r} \in \{ r, q \}\).
\end{remark}

We conclude this section with a lemma that encodes how the boundary of sets \(V
\in \DD_{\beta}^{\cup}\) and \(W \in \TT_{\Theta}^{\cup}\) looks like. We showed
in \Cref{sec:size-motivation} that this is of crucial importance because the
defect size of a function of the form
\[ \1_{(V^+ \cap W^+) \setminus (V^- \cup W^-)} \Emb [f] \circ \Gamma \]
is going to be a bound on a singular measure supported on the boundary of
\((V^+ \cap W^+) \setminus (V^- \cup W^-)\).

\begin{lemma}[Geometry of the boundary of generating sets]
  \label{lem:geometry-of-boundary}Let \\\(V \in \DD_{\beta}^{\cup}\), \(W_1,
  \ldots, W_n \in \TT_{\Theta}^{\cup}\) and \(E\) a set of the form \(E = V \cap
  \left( \bigcap_{l = 1}^k W \right)\) or \(E = V \cup \left( \bigcup_{l = 1}^k
  W \right)\). The boundary of \(E\) is described by the graph of a function
  \(\mf{b}_E  \of \R^2 \rightarrow [0, \infty]\) i.e.
  \[ {E }  = \left\{ (\eta, y, t) \in \R^3_+ \st t \in \left( 0, \mf{b}_E 
     (\eta, y) \right) \right\} . \]
  The function \( \mf{b}_E \) is locally Lipschitz; if we denote \(\Theta =
  (\theta_-, \theta_+)\) then it holds that
  \begin{equation}
    \left| \partial_y \mf{b}_E  (\eta, y) \right| \leq \beta^{- 1}, \qquad
    \frac{\mf{b}_E  (\eta, y)}{- \theta_+} \leq \frac{\partial_{\eta} \mf{b}_E
    (\eta, y)}{\mf{b}_E  (\eta, y)} \leq \frac{\mf{b}_E  (\eta, y)}{-
    \theta_-} . \label{eq:boundary-regularity}
  \end{equation}

  The distributional derivative \(t \dd_t \1_E (\eta, y, t)\) is given by
  measure \(t \delta \left( t - \mf{b}_E  (\eta, y) \right)\). Given any \(F \in
  L^{\infty}_{\tmop{loc}} (\R^{3}_{+})\) and any \(\mf{b} \of \R^2
  \rightarrow [0, \infty]\) satisfying \eqref{eq:boundary-regularity} it holds
  that
  \begin{equation}
    \begin{aligned}[t]&
      \left\| F (\eta, y, t) t \delta \left( t - \mf{b} (\eta, y) \right)
      \right\|_{\SL^{(u, 1)}_{(\Theta, \tilde{\Theta})} (T)}\\ &
      \qquad \approx \left( \int_{\R^2} \1_{\pi_T \left( \mT_{\tilde{\Theta}}
      \right)} \left( \eta, y, \mf{b}  (\eta, y) \right) \left| F \left( \eta,
      y, \mf{b}  (\eta, y) \right) \right|^u \mf{b}  (\eta, y) \frac{\dd \eta
      \dd y}{| \Theta | s_T} \right)^{\frac{1}{u}} .
    \end{aligned} \label{eq:boundary-singular-size}
  \end{equation}
  The statement above has an analogue in local coordinates given by \(\pi_T\).
  Let \(T \in \TT_{\Theta}\) and let \(\mf{b} \of \R^2 \rightarrow [0, \infty]\)
  be a locally Lipschitz function that satisfies
  \(\eqref{eq:boundary-regularity}\). Then there exists a function
  \(\mf{b}^{\ast} : \Theta \times B_1 \rightarrow [0, + \infty)\) such that
  \begin{equation}
    \pi_T \left( \left\{ (\theta, \zeta, \sigma) \in \mT_{\Theta} : \sigma
    \leq \mf{b}^{\ast} (\theta, \zeta)  \right\} \right) = \left\{ (\eta, y,
    t) \in T : t < \mf{b} (\eta, y) \right\} .
    \label{eq:boundary-pullback-condition}
  \end{equation}
  This function \(\mf{b}^{\ast}\) satisfies
  \begin{equation}
    \left| \partial_{\zeta} \mf{b}^{\ast} (\theta, \zeta) \right| \leq
    \beta^{- 1}, \qquad \frac{1}{\theta - \theta_+} \leq
    \frac{\partial_{\theta} \mf{b}^{\ast} (\theta, \zeta)}{\mf{b}^{\ast}
    (\theta, \zeta)} \leq \frac{1}{\theta - \theta_-} .
    \label{eq:boundary-regularity-local}
  \end{equation}
  Given any \(F \in L^{\infty}_{\tmop{loc}} (\R^{3}_{+})\) it holds
  that

  \begin{equation}
    \begin{aligned}[t]&
      \left\| F (\eta, y, t) t \delta \left( t - \mf{b} (\eta, y) \right)
      \right\|_{\SL^{(u, 1)}_{(\Theta, \tilde{\Theta})} (T)} \\& \qquad \approx \left(
      \int_{\tilde{\Theta} \times B_1} \left| \1_{\mT_{\tilde{\Theta}}} (F
      \circ \pi_T) \left( \theta, \zeta, \mf{b}^{\ast} (\theta, \zeta) \right)
      \right|^u \frac{\dd \theta \dd \zeta}{| \Theta |} \right)^{\frac{1}{u}}
      .
    \end{aligned} \label{eq:boundary-singular-size-local}
\end{equation}

  All the implicit constants are uniformly bounded from above and away from
  \(0\) as long as \(\dist \left( \tilde{\Theta}, \R \setminus \Theta \right)\) is
  bounded away from \(0\). 
\end{lemma}

\begin{proof}
  The bound \eqref{eq:boundary-regularity} holds for \(E = D_{\beta} (x, s)\)
  and \(E = T_{\Theta} (\xi, x, s)\) since, explicitly,
  \[ \begin{aligned}[t]&
       \mf{b}_{D_{\beta} (x, s)} (\eta, y) = \max (\beta^{- 1} (s - | y - x
       |), 0),\\ &
       \mf{b}_{T_{\Theta} (\xi, x, s)} (\eta, y) = \min \left( \max (\beta^{-
       1} (s - | y - x |), 0), \1_{\eta > 0} \frac{\theta_+}{| \eta - \xi |} +
       \1_{\eta < 0} \frac{\theta_-}{| \eta - \xi |} \right),
     \end{aligned} \]
  where \(\Theta = (\theta_-, \theta_+)\). The bounds
  \eqref{eq:boundary-regularity} follow since for any collection \(E_i\) it
  holds that
  \[ \mf{b}_{\bigcup_i E_i} (\eta, y) = \sup_i  \mf{b}_{E_i} (\eta, y) \qquad
     \mf{b}_{\bigcap_i E_i} (\eta, y) = \inf_i  \mf{b}_{E_i} (\eta, y) \]
  and the conditions \eqref{eq:boundary-regularity} are stable with respect to
  the operations above.
  
  Now, given \(\mf{b}_E  \of \R^2 \rightarrow [0, \infty]\) let us construct
  \(\mf{b}^{\ast} : \Theta \times B_1 \rightarrow [0, + \infty)\). Let us define
  \[ \mf{b}^{\ast} (\theta, \zeta) \eqd \sup \left\{ \sigma : s_T \sigma \leq
     \mf{b} (\xi_T + \theta (s_T \sigma)^{- 1}, x_T + s_T \zeta) \right\} \]
  that gives \(\mf{b}^{\ast} (\theta, \zeta) \geq 0\) since \(\mf{b} (\eta, y)
  \geq 0\). Note that the function
  \[ \sigma \mapsto s_T \sigma - \mf{b}_E (\xi_T + \theta (s_T \sigma)^{- 1},
     x_T + s_T \zeta) \]
  is continuous and for any \(\theta \in \Theta\) and \(\zeta \in \R\) it holds
  that
  \[ \frac{\dd_{\sigma}}{s_T} \left( s_T \sigma - \mf{b}_E (\xi_T + \theta
     (s_T \sigma)^{- 1}, x_T + s_T \zeta) \right) = 1 + \theta
     \frac{\partial_{\eta} \mf{b}_E (\xi_T + \theta (s_T \sigma)^{- 1}, x_T +
     s_T \zeta)}{s_T \sigma^2} . \]
  Thus if \(s_T \sigma - \mf{b}_E (\xi_T + \theta (s_T \sigma)^{- 1}, x_T + s_T
  \zeta)\) is in a small neighborhood of \(0\) and \(\theta \in \Theta =
  (\theta_-, \theta_+)\) it holds, thank to \eqref{eq:boundary-regularity},
  that
  \[ \frac{\dd_{\sigma}}{s_T} \left( s_T \sigma - \mf{b} (\xi_T + \theta (s_T
     \sigma)^{- 1}, x_T + s_T \zeta) \right) > 0. \]
  It follows that \(s_T \sigma - \mf{b}_E (\xi_T + \theta (s_T \sigma)^{- 1},
  x_T + s_T \zeta)\) can change sign only once and from \(< 0\) to \(> 0\). This
  shows \eqref{eq:boundary-pullback-condition}. We can now differentiate
  \[ \begin{aligned}[t]&
       \mf{b} (\eta, y) = s_T \mf{b}^{\ast} (\theta, \zeta),\\ &
       \mf{b} (\eta, y) (\eta - \xi_T) = \theta,\\ &
       y = x_T + s_T \zeta
     \end{aligned} \]
  to obtain
  \[ \left| \partial_{\zeta} \mf{b}^{\ast} (\theta, \zeta) \right| \leq
     \beta^{- 1} \]
  and
  \[ \partial_{\eta} \mf{b} (\eta, y) = s_T \partial_{\theta} \mf{b}^{\ast}
     (\theta, \zeta) \left( \partial_{\eta} \mf{b} (\eta, y) (\eta - \xi_T) +
     \mf{b} (\eta, y) \right) \]
  giving that
  \[ \frac{\partial_{\theta} \mf{b}^{\ast} (\theta, \zeta)}{\mf{b}^{\ast}
     (\theta, \zeta)} = \frac{\partial_{\eta} \mf{b} (\eta, y)}{\mf{b} (\eta,
     y)^2} \left( \frac{\partial_{\eta} \mf{b} (\eta, y)}{\mf{b} (\eta, y)^2}
     \theta + 1 \right)^{- 1} \]
  and thus
  \[ \frac{1}{\theta - \theta_+} \leq \frac{\partial_{\theta} \mf{b}^{\ast}
     (\theta, \zeta)}{\mf{b}^{\ast} (\theta, \zeta)} \leq \frac{1}{\theta -
     \theta_-}, \]
  as required.
  
  Finally, by definition of the local size \(\SL^{(u, 1)}_{\Theta} (T)\) we
  dualize with \(G \in C^{\infty}_c (T)\) and obtain that
  \begin{equation}
    \begin{aligned}[t]&
      \int_{\R^3_+} F (\eta, y, t) G (\eta, y, t) t \delta \left( t - \mf{b}_E
      (\eta, y) \right) \dd \eta \dd y \dd t\\ &
      \qquad = \int_{\R^3_+} F \left( \eta, y, \mf{b}  (\eta, y) \right) G
      \left( \eta, y, {\mf{b}_E }  (\eta, y) \right) \mf{b}_E  (\eta, y) \dd
      \eta \dd y.
    \end{aligned} \label{eq:boundary-singular-size-local-duality}
  \end{equation}
  We enact the change of variables \(\theta = \mf{b}_E  (\eta, y) (\eta -
  \xi_T)\), \(\zeta = \frac{y - x_T}{s_T}\) with Jacobian it holds that
  \[ J (\eta, y) = \mf{b}_E  (\eta, y) \left( 1 + (\eta - \xi_T)
     \frac{\partial_{\eta} \mf{b}_E  (\eta, y)}{\mf{b}_E  (\eta, y)} \right) .
  \]
  From bound \eqref{eq:boundary-regularity} we obtain that when \((\eta, y, t)
  \in T\) and, in particular, \(t (\eta - \xi_T) \in \Theta\) it holds that
  \[ 0 < J (\eta, y) \lesssim 1 \]
  so the change of variables is valid. We set \(J^{\ast} (\theta, \zeta) \eqd J
  (\eta (\theta), y (\zeta))\) and from bound \eqref{eq:boundary-regularity} we
  obtain that
  \[ \dist \left( \theta, \R \setminus \Theta \right) \lesssim J^{\ast}
     (\theta, \zeta) \lesssim 1. \]
  Rewriting \eqref{eq:boundary-singular-size-local-duality} in new coordinates
  we get
  \[ \begin{aligned}[t]&
       \int_{\R^3_+} F \left( \eta, y, \mf{b}  (\eta, y) \right) G \left(
       \eta, y, \mf{b}  (\eta, y) \right) \mf{b} (\eta, y) \frac{\dd \eta \dd
       y}{| \Theta | s_T}\\ &
       \qquad = \int_{\R^3_+} F \circ \pi_T \left( \theta, \zeta,
       \mf{b}_E^{\ast} (\sigma, \zeta) \right) G \left( \theta, \zeta,
       \mf{b}_E^{\ast} (\sigma, \zeta) \right) \frac{1}{J^{\ast} (\theta,
       \zeta)} \frac{\dd \theta \dd \zeta}{| \Theta |} .
     \end{aligned} \]
  Taking the upper bound over \(G\) with \(\| G \circ \pi_T
  \|_{L^{u'}_{\frac{\dd \theta \dd \zeta}{| \Theta |} }
  L^{\infty}_{\frac{\dd \sigma}{\sigma}}} \leq 1\) we get that
  \[ \begin{aligned}[t]&
       \left\| F (\eta, y, t) t \dd_t \1_E (\eta, y, t) \right\|_{\SL^{(u,
       1)}_{(\Theta, \tilde{\Theta})} (T)}^u =\\ &
       \qquad \int_{\R^3_+} \left| \1_{\mT_{\tilde{\Theta}}} (F \circ \pi_T)
       \left( \theta, \zeta, \mf{b}_E^{\ast} (\sigma, \zeta) \right) \right|^u
       J^{\ast} (\theta, \zeta)^{- u} \frac{\dd \theta \dd \zeta}{|
       \Theta |},
     \end{aligned} \]
  that shows \eqref{eq:boundary-singular-size-local}. Finally, using the
  inverse change of variables\\ \((\eta, y) = \left( \xi_T + \theta \left( s_T
  \mf{b}_E^{\ast} (\theta, \zeta) \right)^{- 1}, x_T + s_T \zeta \right)\) we
  get
  \[ \begin{aligned}[t]&
       \left\| F (\eta, y, t) t \dd_t \1_E (\eta, y, t) \right\|_{\SL^{(u,
       1)}_{(\Theta, \tilde{\Theta})} (T)}^u\\ &
       \qquad = \begin{aligned}[t]&
         \int_{\R^2} \1_{\pi_T \left( \mT_{\tilde{\Theta}} \right)} \left(
         \eta, y, \mf{b}_E  (\eta, y) \right) \left| F \left( \eta, y,
         \mf{b}_E  (\eta, y) \right) \right|^u\\ &
         \qquad \times J (\eta, y)^{1 - u}  \mf{b}_E  (\eta, y) \frac{\dd \eta
         \dd y}{| \Theta | s_T} .
       \end{aligned}
     \end{aligned} \]
  This shows \eqref{eq:boundary-singular-size}, concluding the proof.
\end{proof}

\section{The single tree estimate}\label{sec:single-tree}

In this section we focus on proving \Cref{thm:STE} (bounds
\eqref{eq:STE:l-l-o} and \eqref{eq:STE:l-o-o}). These bounds correspond to the
\(L^{\infty}\) endpoint bound when applying
\Cref{prop:outer-restricted-interpolation} since, tautologically,
\(L_{\mu}^{\infty} \SO \equiv \SO\). We will use \Cref{thm:STE} recurrently in
\Cref{sec:forest} to prove to prove bounds
\eqref{eq:iterated-Holder-type-bounds:linear} and
\eqref{eq:iterated-Holder-type-bounds:bilinear}.

The bounds \eqref{eq:STE:l-l-o} and \eqref{eq:STE:l-o-o} are often referred to
as the single tree estimates since the sizes in play are generated using
identity \eqref{eq:generated-size:trees} starting from local sizes indexed by
single trees. Accordingly, it is sufficient to fix \(T \in \TT_{\Theta}\) and
show that that bounds
\begin{equation}
  \begin{aligned}[t]&
    \| F_1 F_2 F_3 \|_{\left( \SI^{(l, l, l)}_{\Theta} \Phi_{\gr}^N + \SI^{(o,
    l, l)}_{\Theta} \Phi_{\gr}^N \right) (T)}\\ &
    \qquad \lesssim \| F_1 \|_{\SF_{\Theta}^{u_1} \Phi_{2 \gr}^{N - 3} (T)} \|
    F_2 \|_{\widetilde{\SF}_{\Gamma_2}^{u_2} \Phi_{2 \gr}^{N - 3} (T)} \| F_3
    \|_{\widetilde{\SF}_{\Gamma_3}^{u_3} \Phi_{2 \gr}^{N - 3} (T)},
  \end{aligned} \label{eq:STE:o-l-l:tree}
\end{equation}
\begin{equation}
  \begin{aligned}[t]&
    \| F_1 F_2 F_3 \|_{\left( \SI^{(l, l, o)}_{\Theta} \Phi_{\gr}^N + \SI^{(l,
    o, l)}_{\Theta} \Phi_{\gr}^N \right) (T)}\\ &
    \qquad \lesssim \| F_1 \|_{\SF_{\Theta}^{u_1} \Phi_{4 \gr}^{N - 5} (T)} \|
    F_2 \|_{\widetilde{\SF}_{\Gamma_2}^{u_2} \Phi_{4 \gr}^{N - 5} (T)} \| F_3
    \|_{\widetilde{\SF}_{\Gamma_3}^{u_3} \Phi_{4 \gr}^{N - 5} (T)},
  \end{aligned} \label{eq:STE:l-l-o:tree}
\end{equation}
and
\begin{equation}
  \| F_1 F_2 F_3 \|_{\SI^{(l, o, o)}_{\Theta} \Phi_{\gr}^N (T)} \lesssim \|
  F_1 \|_{\SF_{\Theta}^{u_1} \Phi_{4 \gr}^{N - 5} (T)} \| F_2 F_3
  \|_{\widetilde{\SF}_{\Gamma_{\times}}^{u_{\times}} (\Phi \otimes \Phi)_{4
  \gr}^{N - 5} (T)} . \label{eq:STE:l-o-o:tree}
\end{equation}
hold for all \(F_j \in C^{\infty}_c (T) {\otimes \Phi_{4 \gr}^{\infty}}' .\)

Let us argue that this implies the claim of \Cref{thm:STE}. Assuming that the
bounds \eqref{eq:STE:o-l-l:tree}, \eqref{eq:STE:l-l-o:tree}, and
\eqref{eq:STE:l-o-o:tree} hold, using a standard mollification argument, the
same inequalities continue to hold for all functions \(F_j \in L^{\infty} (T)
{\otimes \Phi_{4 \gr}^{\infty}}'\) that are compactly supported in the interior
of \(T\). As a matter of fact all the sizes in play are lower-semicontinuous
when convolving with a compactly supported approximate identity. To
extend bounds \eqref{eq:STE:o-l-l:tree}, \eqref{eq:STE:l-l-o:tree}, and
\eqref{eq:STE:l-o-o:tree} from the case of functions \(F_j\) that are compactly
supported in \(T\) to arbitrary functions \(F_j \in L^{\infty}_{\tmop{loc}}
(\R^{3}_{+}) {\otimes \Phi_{4 \gr}^{\infty}}'\) let us set
\[ \mathbb{K}_n = \begin{aligned}[t]
     \Big(T \cap T_{\Theta} (\xi_T + 2^{- n}, x_T, (1 - 2^{- n})&
     s_T) \cap T_{\Theta} (\xi_T + 2^{- n}, x_T, (1 - 2^{- n}) s_T)\Big) \\ &
     \setminus \{ (\eta, y, t), t > 2^{- n} s_T \}
   \end{aligned} \]
and let us assume bounds \eqref{eq:STE:o-l-l:tree}, \eqref{eq:STE:l-l-o:tree},
and \eqref{eq:STE:l-o-o:tree} hold for \(\1_{\mathbb{K}_n} F_j\) in place of
\(F_j\). According to \Cref{def:integral-sizes} of sizes
\(\SI_{\Theta}^{(\star_1, \star_2, \star_3)} \Phi_{\gr}^N\), we have that
\[ \begin{aligned}[t]&
     \| F_1 F_2 F_3 \|_{\SI_{\Theta}^{(\star_1, \star_2, \star_3)}
     \Phi_{\gr}^N (T)} \leq \liminf_n \left\| \1_{\mathbb{K}_n} F_1 F_2 F_3
     \right\|_{\SI_{\Theta}^{(\star_1, \star_2, \star_3)} \Phi_{\gr}^N (T)}
   \end{aligned} \]
for any \((\star_1, \star_2, \star_3) \in \{ o, l \}^3\). On the other hand,
according to \eqref{eq:generated-size:trees}, the terms on
{\RHS{}}\eqref{eq:STE:l-l-o} and {\RHS{\eqref{eq:STE:l-o-o}}} increase at most
by a fixed factor when replacing \(F_j\) by \(\1_{W^+} {F_j} \) or by \(\1_{\R^3_+
\setminus W^-} F_j\) with \(W^{\pm} \in \TT_{\Theta}^{\cup}\); the transformation
\(F_j \mapsto \1_{\mathbb{K}_n} F_j\) can be expressed as at most \(4\) such
steps. This concludes the reduction to \(F_j \in C^{\infty}_c (T) {\otimes
\Phi_{4 \gr}^{\infty}}'\).

Thus, for the subsequent proof we assume that \(F_j \in C^{\infty}_c (T)
{\otimes \Phi_{4 \gr}^{\infty}}'\); let us introduce the notation
\begin{equation}
  F_j^{\ast} (\theta, \zeta, \sigma) \assign e^{- 2 \pi i \alpha_j \xi_T s_T
  \zeta} F_j \circ \pi_T (\theta, \zeta, \sigma), \qquad j \in \{ 1, 2, 3 \} ;
  \label{eq:local-F}
\end{equation}
According to \eqref{eq:gamma:params} it holds that \(\alpha_1 + \alpha_2 +
\alpha_3 = 1 + \alpha_2 + \alpha_3 = 0\) so
\[ \Big( \prod_{j = 1}^3 F_j \Big) \circ \pi_T (\theta, \zeta, \sigma) =
   \prod_{j = 1}^3 F_j^{\ast} (\theta, \zeta, \sigma) . \]
With the notation of \Cref{def:integral-sizes}, given a wave packet \(\phi \in
\Phi^{\infty}_{\gr / 2}\) we define the wave packets \(\tilde{\phi}_{j,
\theta}^l (z)\) and \(\tilde{\phi}_{j, \theta}^o (z)\) to satisfy
\[ \begin{aligned}[t]&
     (- d_z + 2 \pi i \theta_{\Gamma_j}) \tilde{\phi}_{j, \theta}^l (z) =
     \phi_{j, \theta}^l (z),\\ &
     (- d_z + 2 \pi i \theta_{\Gamma_j}) \tilde{\phi}_{j, \theta}^o (z) = z
     \phi_{j, \theta}^o (z)
   \end{aligned} \]
by setting
\begin{equation}
  \FT{\tilde{\phi}_{j, \theta}^l} \left( \FT{z} \right) \eqd
  \frac{\FT{\phi_{j, \theta}^l} \left( \FT{z} \right)}{- 2 \pi i \left( \FT{z}
  - \theta_{\Gamma_j} \right)} . \label{eq:tilde-phi-lac}
\end{equation}
and
\begin{equation}
  \FT{\tilde{\phi}_{j, \theta}^o} \left( \FT{z} \right) \eqd \frac{\partial
  \FT{\phi_{j, \theta}^o} \left( \FT{z} \right)}{- 4 \pi^2 \left( \FT{z} -
  \theta_{\Gamma_j} \right)} . \label{eq:tilde-phi-ov}
\end{equation}
These expressions give well defined wave packets \(\tilde{\phi}_{j, \theta}^l,
\tilde{\phi}_{j, \theta}^o \in \Phi^{\infty}_{\gr}\) with
\[ \| \tilde{\phi}_{j, \theta}^o \|_{\Phi_{\gr}^{N - 1}} + \|
   \tilde{\phi}_{j, \theta}^l \|_{\Phi_{\gr}^{N - 1}} \lesssim \| \phi_{j,
   \theta}^o \|_{\Phi_{\gr}^N} + \| \phi_{j, \theta}^l \|_{\Phi_{\gr}^N} \|
   \phi_j \|_{\Phi_{\gr / 2}^N} . \]
This holds since \(\FT{\phi_{j, \theta}^l} \left( \FT{z} \right) = \FT{\phi_j}
\left( \FT{z} \right) - \FT{\phi_j} (\theta_{\Gamma_j}) \FT{\omega} \left(
\frac{\FT{z}}{\gr} \right)\) and \(\partial \FT{\phi_{j, \theta}^o} \left(
\FT{z} \right) = \FT{\phi}_j (\theta_{\Gamma_j}) \gr^{- 1} \FT{\omega}' \left(
\frac{\FT{z}}{\gr} \right)\) and in particular, \(\FT{\phi_{j, \theta}^l}
(\theta_{\Gamma_j}) = \dd_{\FT{z}} \FT{\phi_{j, \theta}^o} (\theta_{\Gamma_j})
= 0\).

Finally, we mention that main motivation for introducing the specific sizes in
\Cref{sec:time-frequency-outer-lebesgue} is given by the quantities arising in
the proof of \Cref{thm:STE}.

We offer the following suggestions for a preliminary reading of these proofs.
First of all one can suppose that \(T = T_{\Theta} (0, 0, 1)\). Next, we
encourage the reader to suppose that \(F_j^{\ast} = \Emb [f_j] \circ \Gamma_j
\circ \pi_T\). This guarantees that
\[ \begin{aligned}[t]&
     F_1^{\ast} (\theta, \zeta, \sigma) [\phi] \eqd f_1 \ast \Dil_{\sigma}
     \tmop{Mod}_{\theta} \phi^{\vee},\\ &
     F_j^{\ast} (\theta, \zeta, \sigma) [\phi] \eqd f_j \ast \Dil_{\beta_j
     \sigma} \tmop{Mod}_{\theta_{\Gamma_j}} \phi^{\vee} .
   \end{aligned} \]
where \(\theta_{\Gamma_j} = \alpha_j \beta_j (\theta + \gamma_j)\). Furthermore,
this assumption guarantees that
\[ \begin{aligned}[t]&
     \wpD_{\sigma} (\theta) F_1^{\ast} (\theta, \zeta, \sigma) [\phi] = \sigma
     \dd_{\sigma} \left( f_1 \ast \Dil_{\sigma} \tmop{Mod}_{\theta}
     \phi^{\vee} \right),\\ &
     \wpD_{\sigma} (\theta_{\Gamma_j}) F_j^{\ast} (\theta, \zeta, \sigma)
     [\phi] = \sigma \dd_{\sigma} \left( f_j \ast \Dil_{\beta_j \sigma}
     \tmop{Mod}_{\theta_{\Gamma_j}} \phi^{\vee} \right),\\ &
     \wpD_{\zeta} (\theta) F_1^{\ast} (\theta, \zeta, \sigma) [\phi] = \sigma
     \dd_{\zeta} \left( f_1 \ast \Dil_{\sigma} \tmop{Mod}_{\theta} \phi^{\vee}
     \right),\\ &
     \wpD_{\zeta} (\theta_{\Gamma_j}) F_j^{\ast} (\theta, \zeta, \sigma)
     [\phi] = \beta_j \sigma \dd_{\zeta} \left( f_j \ast \Dil_{\beta_j \sigma}
     \tmop{Mod}_{\theta_{\Gamma_j}} \phi^{\vee} \right) .
   \end{aligned} \]
In particular, any term in the subsequent manipulations that contains the
difference between the {\LHS{}} and {\RHS{}} of the quantities above vanishes.
In the complete proof the defect sizes included in \(\SF_{\Theta}^{u_1}
\Phi^N_{\gr}\), \(\widetilde{\SF}_{\Gamma_j}^{u_j} \Phi^N_{\gr}\), and
\(\widetilde{\SF}_{\Gamma_{\times}}^{u_{\times}} (\Phi \otimes \Phi)^N_{\gr}\)
are used to control the discrepancies. Similarly, the reader is invited to
disregard boundary terms appearing when integrating by parts: they too are
controlled by the defect sizes. The reader can also suppose that \(\beta_2 =
\beta_3 = \beta\), \(\alpha_2 = \alpha_3 = \beta^{- 1},\) and \(\gamma_2 =
\gamma_3 = \frac{1}{2}\). The second assumption is formally not allowed since
we require \(1 + \alpha_2 + \alpha_3 = 0\). However, this fact is used only in
the reduction following \eqref{eq:local-F} and is not of relevance for the
remainder of the proof. Finally, in the proof of \eqref{eq:STE:o-l-l:tree} the
reader can assume simply that \(\theta = 0\) instead of having an integral over
\(\theta \in \Theta\) while in the proof of \eqref{eq:STE:l-l-o:tree} and of
\eqref{eq:STE:l-o-o:tree} the reader can assume that \(\theta = - \frac{1}{2}\)
so that \(\theta_{\Gamma_j} = 0\) for \(j \in \{ 2, 3 \}\).

\begin{proof}{Proof of bound \eqref{eq:STE:o-l-l:tree}}
  We deal only with the size \(\SI^{(o, l, l)}_{\Theta} \Phi_{\gr}^N\) appearing
  on {\LHS{\eqref{eq:STE:o-l-l:tree}}}; the term \(\SI^{(l, l, l)}_{\Theta}
  \Phi_{\gr}^N\) is dealt with in the same way. We begin by fixing \(T \in
  \TT_{\Theta}\) and \(\phi \in \Phi_{\gr}^N\) with \(\| \phi \|_{\Phi_{\gr}^N}
  \leq 1\) so that we now need to show bound
  \[ \int_{\Theta} \left| \int_0^{1 - | \zeta |} F_1 F_2 F_3 \circ \pi_T
     (\theta, \zeta, \sigma) [\phi_{1, \theta}^{o }, \phi_{2, \theta}^l,
     \phi_{3, \theta}^l] \dd \zeta \frac{\dd \sigma}{\sigma} \right| \frac{\dd
     \theta}{| \Theta |} . \]
  Using the notation from the discussion preceding this proof we can rewrite
  and bound this expression as follows:
  \[ \begin{aligned}[t]&
       \int_{\Theta} \left| \int_{\R^2_+} F_1 F_2 F_3 \circ \pi_T (\theta,
       \zeta, \sigma) [\phi_{1, \theta}^{o }, \phi_{2, \theta}^l, \phi_{3,
       \theta}^l] \dd \zeta \frac{\dd \sigma}{\sigma} \right| \frac{\dd
       \theta}{| \Theta |}\\ &
       \qquad \leq \int_{\R^3_+} \begin{aligned}[t]&
         | F_1^{\ast} (\theta, \zeta, \sigma) [\phi_{1, \theta}^{o }]
         \nobracket \; \wpD_{\zeta} (\theta_{\Gamma_2}) F_2^{\ast} (\theta,
         \zeta, \sigma) [\widetilde{\phi_{2, \theta}^l}]\\ &
         \qquad \left. \times \wpD_{\zeta} (\theta_{\Gamma_3}) F_3^{\ast}
         (\theta, \zeta, \sigma) [\widetilde{\phi_{2, \theta}^l}] \right| \dd
         \zeta \frac{\dd \sigma}{\sigma} \frac{\dd \theta}{| \Theta |}
       \end{aligned}\\ &
       \qquad \lesssim \sup_{\tilde{\phi}_1, \tilde{\phi}_2, \tilde{\phi}_3}
       \int_{\R^3_+} \begin{aligned}[t]&
         | F_1^{\ast} (\theta, \zeta, \sigma) [\widetilde{\phi_1}] | \left| \;
         \wpD_{\zeta} (\theta_{\Gamma_2}) F_2^{\ast} (\theta, \zeta, \sigma)
         [\widetilde{\phi_2}] \right|\\ &
         \qquad \times \left| \wpD_{\zeta} (\theta_{\Gamma_3}) F_3^{\ast}
         (\theta, \zeta, \sigma) [\widetilde{\phi_3}] \right| \dd \zeta
         \frac{\dd \sigma}{\sigma} \frac{\dd \theta}{| \Theta |},
       \end{aligned}
     \end{aligned} \]
  with the upper bound being taken over wave packets \(\widetilde{{\phi_j} }
  \in \Phi^{\infty}_{2 \gr}\) with \(\| \widetilde{{\phi_j} } \|_{\Phi_{2
  \gr}^{N - 3}} \leq 1\), \(j \in \{ 1, 2, 3 \}\). The second inequality follows
  from \Cref{lem:wave-packet-decomposition} according to which we can
  decompose each of the wave packets \(\tilde{\phi}_{1, \theta}^o,
  \tilde{\phi}_{2, \theta}^l, \tilde{\phi}_{3, \theta}^l\) into an absolutely
  convergent series e.g.
  \[ \tilde{\phi}_{1, \theta}^l \eqd \sum_k \tilde{a}^l_{1, k} (\theta)
     \widetilde{\tilde{\phi}}_k \]
  with \(\| \widetilde{\tilde{\phi}}_k \|_{\Phi_{2 \gr}^{N - 3}} \lesssim \|
  \tilde{\phi}_{1, \theta}^l \|_{\Phi_{\gr}^{N - 1}} \lesssim 1\). We bound the
  {\RHS{}} of the chain of inequalities above Cauchy-Schwartz inequality: we
  obtain the bound
  \[ \begin{aligned}[t]&
       \int_{\R^3_+} \begin{aligned}[t]&
         | F_1^{\ast} (\theta, \zeta, \sigma) [\widetilde{\phi_1}] | \left| \;
         \wpD_{\zeta} (\theta_{\Gamma_2}) F_2^{\ast} (\theta, \zeta, \sigma)
         [\widetilde{\phi_2}] \right| \qquad\\ &
         \qquad \times \left| \wpD_{\zeta} (\theta_{\Gamma_3}) F_3^{\ast}
         (\theta, \zeta, \sigma) [\widetilde{\phi_3}] \right| \dd \zeta
         \frac{\dd \sigma}{\sigma} \frac{\dd \theta}{| \Theta |}
       \end{aligned}\\ &
       \qquad \lesssim \begin{aligned}[t]&
         \| F_1^{\ast} (\theta, \zeta, \sigma) [\widetilde{\phi_1}]
         \|_{L^{u_1}_{\frac{\dd \zeta \dd \theta}{| \Theta |}}
         L^{\infty}_{\frac{\dd \sigma}{\sigma}}} \left\| \wpD_{\zeta}
         (\theta_{\Gamma_2})  F_2^{\ast} (\theta, \zeta, \sigma)
         [\widetilde{\phi_2}] \right\|_{L^{u_2}_{\frac{\dd \zeta \dd
         \theta}{| \Theta |}} L^2_{\frac{\dd \sigma}{\sigma}}}\\ &
         \times \left\| \wpD_{\zeta} (\theta_{\Gamma_3}) F^{\ast}_3 (\theta,
         \zeta, \rho) [\widetilde{\phi_3}] \right\|_{L_{\frac{\dd \zeta
         \dd \theta}{| \Theta |}}^{u_3} L_{\frac{\dd \rho}{\rho}}^2}
       \end{aligned}\\ &
       \qquad \lesssim \| F_1 \|_{\SL_{\Theta }^{(u_1, \infty)} \Phi_{2
       \gr}^{N - 2} (T)} \| F_2 \|_{\SL_{\Theta }^{(u_2, 2)} \Gamma_2^{\ast}
       \wpD_{2 \gr}^{N - 3} (T)} \| F_3 \|_{\SL_{\Theta }^{(u_3, 2)}
       \Gamma_3^{\ast} \wpD_{2 \gr}^{N - 3} (T)},
     \end{aligned} \]
  as required.
\end{proof}

\begin{proof}{Proof of bound \eqref{eq:STE:l-l-o:tree}}
  We deal only with the size \(\SI^{(l, l, o)}_{\Theta} \Phi_{\gr}^N\) appearing
  on {\LHS{\eqref{eq:STE:l-l-o:tree}}}; the term with \(\SI^{(l, o,
  l)}_{\Theta} \Phi_{\gr}^N\) is dealt with analogously, by exchanging the role
  of the indexes \(2\) and \(3\). We begin by fixing \(T \in \TT_{\Theta}\) and
  \(\phi \in \Phi_{\gr}^N\) with \(\| \phi \|_{\Phi_{\gr}^N} \leq 1\). Using the
  notation from the discussion preceding this proof, we need to bound the
  expression
  \begin{equation}
    \begin{aligned}[t]&
      \int_{\Theta} \left| \int_{\R^2_+} \wpD_{\zeta} (\theta) F_1^{\ast}
      (\theta, \zeta, \sigma) [\tilde{\phi}_{1, \theta}^l] \wpD_{\zeta}
      (\theta_{\Gamma_2}) F_2^{\ast} (\theta, \zeta, \sigma) [\tilde{\phi}_{2,
      \theta}^l] F_3^{\ast} (\theta, \zeta, \sigma) [\phi_{3, \theta}^o] \dd
      \zeta \frac{\dd \sigma}{\sigma} \right| \frac{\dd \theta}{| \Theta |} .
    \end{aligned} \label{eq:l-l-o-quantity}
\end{equation}

  We claim that
  \begin{equation}
    \begin{aligned}[t]&
      \eqref{eq:l-l-o-quantity} \leq \int_{\Theta} \begin{aligned}[t]&
        (| \textrm{M1} (\theta) [\tilde{\phi}_{1, \theta}^l, \tilde{\phi}_{2,
        \theta}^l, \tilde{\phi}_{3, \theta}^o] | + | \textrm{M2} (\theta)
        [\tilde{\phi}_{1, \theta}^l, \tilde{\phi}_{2, \theta}^l,
        \tilde{\phi}_{3, \theta}^o] | + \nobracket \qquad\\ &
        + | \textrm{B1} (\theta) [\tilde{\phi}_{1, \theta}^l, \tilde{\phi}_{2,
        \theta}^l, \phi_{3, \theta}^o] | + | \textrm{B2} (\theta)
        [\tilde{\phi}_{1, \theta}^l, \tilde{\phi}_{2, \theta}^l,
        \tilde{\phi}_{3, \theta}^o] |\\ &
        \qquad \left\nobracket + \left| \text{B3} (\theta) [\tilde{\phi}_{1,
        \theta}^l, \tilde{\phi}_{2, \theta}^l, \tilde{\phi}_{3, \theta}^o]
        \right| + \left| \textrm{B4} (\theta) [\tilde{\phi}_{1, \theta}^l,
        \tilde{\phi}_{2, \theta}^l, \tilde{\phi}_{3, \theta}^o] \right|
        \right) \frac{\dd \theta}{| \Theta |}
      \end{aligned}\\ &
      \lesssim \sup_{\tmscript{\begin{aligned}[t]&
        \left\| {\phi_j}  \right\|_{\Phi_{\gr}^N} \leq 1\\ &
        j \in \{ 1, 2, 3 \}
      \end{aligned}}} \int_{\Theta} \begin{aligned}[t]&
        \left( | \textrm{M1} (\theta) [\phi_1 , \phi_2 , \phi_3 ] | + \left|
        \textrm{M2\((\theta) [\phi_1 , \phi_2 , \phi_3 ]\)} \right| + \right.
        \qquad\\ &
        + | \textrm{B1} (\theta) [\phi_1 , \phi_2 , \phi_3 ] | + | \textrm{B2}
        (\theta) [\phi_1 , \phi_2 , \phi_3 ] |\\ &
        \qquad \left\nobracket + \left| \text{B3} (\theta) [\phi_1 , \phi_2 ,
        \phi_3 ] \right| + \left| \textrm{B4} (\theta) [\phi_1 , \phi_2 ,
        \phi_3 ] \right| \right) \frac{\dd \theta}{| \Theta |}
      \end{aligned}
    \end{aligned} \label{eq:single-tree:l-l-o:decomp}
  \end{equation}
  where
  \[ \textrm{M1} (\theta) [\phi_1 , \phi_2 , \phi_3 ] \assign \begin{aligned}[t]&
       \int_{\R^2_+} \wpD_{\zeta} (\theta) F_1^{\ast} (\theta, \zeta, \sigma)
       [\phi_1 ] \wpD_{\zeta} (\theta_{\Gamma_2}) ^2 F_2^{\ast} (\theta,
       \zeta, \sigma) [\phi_2 ] \qquad\\ &
       \times \int_0^{\sigma} \frac{\rho}{\sigma} \wpD_{\zeta}
       {(\theta_{\Gamma_3})  }  F_3^{\ast} (\theta, \zeta, \rho) [\phi_3 ]
       \frac{\dd \rho}{\rho} \frac{\dd \sigma}{\sigma} \dd \zeta,
     \end{aligned} \]
  \[ \textrm{M2} (\theta) [\phi_1 , \phi_2 , \phi_3 ] \assign \begin{aligned}[t]&
       \beta \int_{\R^2_+} \wpD_{\zeta} {(\theta) }^2 F_1^{\ast} (\theta,
       \zeta, \sigma) [\phi_1 ] \wpD_{\zeta} (\theta_{\Gamma_2})  F_2^{\ast}
       (\theta, \zeta, \sigma) [\phi_2 ] \qquad\\ &
       \times \int_0^{\sigma} \frac{\rho}{\sigma} \wpD_{\zeta}
       {(\theta_{\Gamma_3}) }  F_3^{\ast} (\theta, \zeta, \rho) [\phi_3 ]
       \frac{\dd \rho}{\rho} \frac{\dd \sigma}{\sigma} \dd \zeta,
     \end{aligned} \]
  and
  \[ \begin{aligned}[t]&
       \textrm{B1} (\theta) [\phi_1 , \phi_2 , \phi_3 ] \assign
       \begin{aligned}[t]&
         \int_{\R^2_+} \wpD_{\zeta} (\theta) F_1^{\ast} (\theta, \zeta,
         \sigma) [\phi_1 ] \hspace{0.17em} \wpD_{\zeta} (\theta_{\Gamma_2}) 
         F_2^{\ast} (\theta, \zeta, \sigma) [\phi_2 ]\\ &
         \qquad \times \int_0^{\sigma} \left( \wpD_{\sigma}
         (\theta_{\Gamma_3}) - \rho \dd_{\rho} \right) F_3^{\ast} (\theta,
         \zeta, \rho) [\phi_3 ] \frac{\dd \rho}{\rho} \frac{\dd
         \sigma}{\sigma} \dd \zeta,
       \end{aligned}
     \end{aligned} \]
  \[ \begin{aligned}[t]&
       \textrm{B2} (\theta) [\phi_1 , \phi_2 , \phi_3 ] \assign
       \begin{aligned}[t]&
         \int_{\R^2_+} \wpD_{\zeta} (\theta) F_1^{\ast} (\theta, \zeta,
         \sigma) [\phi_1 ] \hspace{0.17em} \wpD_{\zeta} (\theta_{\Gamma_2}) 
         F_2^{\ast} (\theta, \zeta, \sigma) [\phi_2 ]\\ &
         \qquad \times \int_0^{\sigma} \left( \wpD_{\zeta_3} - \beta \rho
         d_{\zeta} \right) \wpD_{\zeta_3} F^{\ast}_3 (\theta, \zeta, \rho)
         [\phi_3 ] \frac{\dd \rho}{\rho} \frac{\dd \sigma}{\sigma}
         \dd \zeta,
       \end{aligned}
     \end{aligned} \]
  \[ \begin{aligned}[t]&
       \textrm{B3} (\theta) [\phi_1 , \phi_2 , \phi_3 ]\\ &
       \assign \begin{aligned}[t]&
         \int_{\R^2_+} \wpD_{\zeta} (\theta) F_1^{\ast} (\theta, \zeta,
         \sigma) [\phi_1 ] \left( \wpD_{\zeta} (\theta_{\Gamma_2}) - \beta
         \sigma \dd_{\zeta} \right) \wpD_{\zeta} (\theta_{\Gamma_2}) 
         F_2^{\ast} (\theta, \zeta, \sigma) [\phi_2 ] \qquad\\ &
         \qquad \times \int_0^{\sigma} \frac{\rho}{\sigma} \wpD_{\zeta}
         {(\theta_{\Gamma_3}) }  F_3^{\ast} (\theta, \zeta, \rho) [\phi_3 ]
         \frac{\dd \rho}{\rho} \frac{\dd \sigma}{\sigma} \dd \zeta,
       \end{aligned}
     \end{aligned} \]
  \[ \begin{aligned}[t]&
       \textrm{B4} (\theta) [\phi_1 , \phi_2 , \phi_3 ]\\ &
       \assign \begin{aligned}[t]&
         \beta \int_{\R^2_+} \left( \wpD_{\zeta} (\theta_{\Gamma_2}) - \sigma
         \dd_{\zeta} \right) \wpD_{\zeta} (\theta) F_1^{\ast} (\theta,
         \zeta, \sigma) [\phi_1 ] \wpD_{\zeta} (\theta_{\Gamma_2})  F_2^{\ast}
         (\theta, \zeta, \sigma) [\phi_2 ] \qquad\\ &
         \qquad \times \int_0^{\sigma} \frac{\rho}{\sigma} \wpD_{\zeta}
         {(\theta_{\Gamma_3}) }  F_3^{\ast} (\theta, \zeta, \rho) [\phi_3 ]
         \frac{\dd \rho}{\rho} \frac{\dd \sigma}{\sigma} \dd \zeta .
       \end{aligned}
     \end{aligned} \]

  The second inequality in \eqref{eq:single-tree:l-l-o:decomp} follows from
  \Cref{lem:wave-packet-decomposition} according to which we can decompose
  each of the wave packets \(\tilde{\phi}_{1, \theta}^l, \tilde{\phi}_{2,
  \theta}^l, \tilde{\phi}_{3, \theta}^o, \phi_{3, \theta}^o\) into an
  absolutely convergent series e.g.
  \[ \tilde{\phi}_{1, \theta}^l \eqd \sum_k \tilde{a}^l_{1, k} (\theta)
     \widetilde{\tilde{\phi}}_k \]
  with \(\| \widetilde{\tilde{\phi}}_k \|_{\Phi_{\gr}^N} \leq 1\).
  
  Let us now show that each of the 6 terms in
  {\RHS{\eqref{eq:single-tree:l-l-o:decomp}}} contributes a quantity bounded
  by {\RHS{\eqref{eq:STE:l-l-o}}}. We record Young's convolution inequality in
  a form useful for us. A direct computation gives
  \begin{equation}
    \bigl\| \rho \1_{\rho < 1} \bigr\|_{L^1_{\frac{\dd \rho}{\rho}}
    (\R)} + \bigl\| \rho \1_{\rho < 1}
    \bigr\|_{L^{\infty}_{\frac{\dd \rho}{\rho}} (\R)}
    \lesssim 1, \label{eq:young-kernel-estimate}
  \end{equation}
  so it follows that for any \(q \in [1, \infty]\) and any \(h \in L^q_{\dd
  \rho / \rho} (\R_+)\) it holds that
  \begin{equation}
    \left\| \int_0^{\sigma} \frac{\rho}{\sigma} h (\rho) \frac{\dd
    \rho}{\rho} \right\|_{L^q_{\frac{\dd \rho}{\rho}} (\R) +
    L^{\infty}_{\frac{\dd \rho}{\rho}} (\R)} \lesssim
    \|h\|_{L^q_{\frac{\dd \rho}{\rho}} (\R)} .
    \label{eq:young-multiplicative}
  \end{equation}
  Note that the integral above is a multiplicative convolution with respect to
  the multiplicative Haar measure \(\frac{\dd \rho}{\rho}\). For \(\left\|
  {\phi_j}  \right\|_{\Phi_{\gr}^N} \leq 1\), \(j \in \{ 1, 2, 3 \}\), the
  following bounds thus hold:
  \[ \begin{aligned}[t]&
       \int_{\Theta} | \textrm{M1} (\theta) [\phi_1 , \phi_2 , \phi_3 ] |
       \dd \theta\\ &
       \qquad \lesssim \begin{aligned}[t]&
         \left\| \wpD_{\zeta} (\theta)  F_1^{\ast} (\theta, \zeta, \sigma)
         [\phi_1] \right\|_{L^{u_1}_{\frac{\dd \zeta \dd \theta}{|
         \Theta |}} L^{\infty}_{\frac{\dd \sigma}{\sigma}}} \left\|
         \wpD_{\zeta} (\theta_{\Gamma_2})^2 F_2^{\ast} (\theta, \zeta, \sigma)
         [\phi_2] \right\|_{L^{u_2}_{\frac{\dd \zeta \dd \theta}{|
         \Theta |}} L^2_{\frac{\dd \sigma}{\sigma}}} \qquad\\ &
         \times \left\| \wpD_{\zeta} (\theta_{\Gamma_3}) F_3^{\ast} (\theta,
         \zeta, \rho) [\phi_3] \right\|_{L_{\frac{\dd \zeta \dd
         \theta}{| \Theta |}}^{u_3} L_{\frac{\dd \sigma}{\sigma}}^2}
       \end{aligned} \\ &
       \qquad \lesssim \| F_1 \|_{\SL_{\Theta }^{(u_1, \infty)} \Phi_{\gr}^N
       (T)} \| F_2 \|_{\SL_{\Theta }^{(u_2, 2)} \Gamma_2^{\ast} \wpD_{\gr}^N
       (T)} \| F_3 \|_{\SL_{\Theta }^{(u_3, 2)} \Gamma_3^{\ast} \wpD_{\gr}^N
       (T)} ;
     \end{aligned} \]
  \[ \begin{aligned}[t]&
       \int_{\Theta} | \textrm{M2} (\theta) [\phi_1 , \phi_2 , \phi_3 ] |
       \dd \theta\\ &
       \qquad \lesssim \begin{aligned}[t]&
         \beta \left\| \wpD_{\zeta} (\theta) ^2 F_1^{\ast} (\theta, \zeta,
         \sigma) [\phi_1] \right\|_{L^{u_1}_{\frac{\dd \zeta \dd
         \theta}{| \Theta |}} L^{\infty}_{\frac{\dd \sigma}{\sigma}}}
         \left\| \wpD_{\zeta} (\theta_{\Gamma_2}) F_2^{\ast} (\theta, \zeta,
         \sigma) [\phi_2] \right\|_{L^{u_2}_{\frac{\dd \zeta \dd
         \theta}{| \Theta |}} L^2_{\frac{\dd \sigma}{\sigma}}} \qquad\\ &
         \times \left\| \wpD_{\zeta} (\theta_{\Gamma_3}) F_3^{\ast} (\theta,
         \zeta, \rho) [\phi_3] \right\|_{L_{\frac{\dd \zeta \dd
         \theta}{| \Theta |}}^{u_3} L_{\frac{\dd \sigma}{\sigma}}^2}
       \end{aligned} \\ &
       \qquad \lesssim \beta \| F_1 \|_{\SL_{\Theta }^{(u_1, \infty)}
       \Phi_{\gr}^N (T)} \| F_2 \|_{\SL_{\Theta }^{(u_2, 2)} \Gamma_2^{\ast}
       \wpD_{\gr}^N (T)} \| F_3 \|_{\SL_{\Theta }^{(u_3, 2)} \Gamma_3^{\ast}
       \wpD_{\gr}^N (T)} .
     \end{aligned} \]
  Bounding the boundary terms \(\textrm{B1}\), \(\textrm{B2}\), \(\textrm{B3}\), and
  \(\textrm{B4}\) requires the use of defect sizes:
  \[ \begin{aligned}[t]&
       \int_{\Theta} | \textrm{B1} (\theta) [\phi_1 , \phi_2 , \phi_3 ] |
       \dd \theta\\ &
       \qquad \lesssim \begin{aligned}[t]&
         \left\| \wpD_{\zeta} (\theta) F_1^{\ast} (\theta, \zeta, \sigma)
         [\phi_1 ] \right\|_{L^{u_1}_{\frac{\dd \zeta \dd \theta}{|
         \Theta |}} L^2_{\frac{\dd \sigma}{\sigma}}} \left\| \wpD_{\zeta}
         (\theta_{\Gamma_2})  F_2^{\ast} (\theta, \zeta, \sigma) [\phi_2]
         \right\|_{L^{u_2}_{\frac{\dd \theta \dd \zeta}{| \Theta |}}
         L^2_{\frac{\dd \sigma}{\sigma}}}\\ &
         \times \left\| \left( \wpD_{\sigma} (\theta_{\Gamma_3}) - \rho
         \dd_{\rho} \right)  F_3^{\ast} (\theta, \zeta, \rho) [\phi_3 ]
         \right\|_{L_{\frac{\dd \theta \dd \zeta}{| \Theta |}}^{u_3}
         L_{\frac{\dd \rho}{\rho}}^1}
       \end{aligned} \qquad\\ &
       \hspace{2.0em} \lesssim \| F_1 \|_{\SL_{\Theta }^{(u_1, \infty)}
       \wpD_{\gr}^N (T)} \| F_2 \|_{\SL_{\Theta }^{(u_2, 2)} \Gamma_2^{\ast}
       \wpD_{\gr}^N (T)} \| F_3 \|_{\SL_{\Theta }^{(u_3, 1)} \Gamma_3^{\ast}
       {\dfct }_{\gr}^N (T)},
     \end{aligned} \]
  \[ \begin{aligned}[t]&
       \int_{\Theta} \left| \textrm{B2} (\theta) [\phi_1 , \phi_2 , \phi_3 ]
       \right| \dd \theta\\ &
       \qquad \lesssim \begin{aligned}[t]&
         \left\| \wpD_{\zeta} (\theta) F_1^{\ast} (\theta, \zeta, \sigma)
         [\phi_1 ] \right\|_{L^{u_1}_{\frac{\dd \zeta \dd \theta}{|
         \Theta |}} L^2_{\frac{\dd \sigma}{\sigma}}} \left\| \wpD_{\zeta}
         (\theta_{\Gamma_2})  F_2^{\ast} (\theta, \zeta, \sigma) [\phi_2]
         \right\|_{L^{u_2}_{\frac{\dd \theta \dd \zeta}{| \Theta |}}
         L^2_{\frac{\dd \sigma}{\sigma}}}\\ &
         \times \left\| \left( \wpD_{\zeta_3} - \beta \rho d_{\zeta} \right)
         \wpD_{\zeta} (\theta_{\Gamma_3}) F^{\ast}_3 (\theta, \zeta, \rho)
         [\phi_3 ] \right\|_{L_{\frac{\dd \theta \dd \zeta}{| \Theta
         |}}^{u_3} L_{\frac{\dd \rho}{\rho}}^1}
       \end{aligned} \qquad\\ &
       \qquad \lesssim \| F_1 \|_{\SL_{\Theta }^{(u_1, \infty)} \wpD_{\gr}^N
       (T)} \| F_2 \|_{\SL_{\Theta }^{(u_2, 2)} \Gamma_2^{\ast} \wpD_{\gr}^N
       (T)} \| F_3 \|_{\SL_{\Theta }^{(u_3, 1)} \Gamma_3^{\ast} \dfct_{\gr}^N
       (T)},
     \end{aligned} \]
  \[ \begin{aligned}[t]&
       \int_{\Theta} | \textrm{B3} (\theta) [\phi_1 , \phi_2 , \phi_3 ] |
       \dd \theta\\ &
       \qquad \lesssim \begin{aligned}[t]&
         \left\| \wpD_{\zeta} (\theta) F_1^{\ast} (\theta, \zeta, \sigma)
         [\phi_1 ] \right\|_{L^{u_1}_{\frac{\dd \zeta \dd \theta}{|
         \Theta |}} L^{\infty}_{\frac{\dd \sigma}{\sigma}}} \qquad\\ &
         \qquad \times \left\| \left( \wpD_{\zeta} (\theta_{\Gamma_2}) -
         \beta \sigma \dd_{\zeta} \right) \wpD_{\zeta} (\theta_{\Gamma_2}) 
         F_2^{\ast} (\theta, \zeta, \sigma) [\phi_2]
         \right\|_{L^{u_2}_{\frac{\dd \theta \dd \zeta}{| \Theta |}}
         L^1_{\frac{\dd \sigma}{\sigma}}}\\ &
         \qquad \times \left\| \wpD_{\zeta } (\theta_{\Gamma_3}) F^{\ast}_3
         (\theta, \zeta, \rho) [\phi_3 ] \right\|_{L_{\frac{\dd \theta
         \dd \zeta}{| \Theta |}}^{u_3} L_{\frac{\dd \rho}{\rho}}^2}
       \end{aligned}\\ &
       \qquad \lesssim \| F_1 \|_{\SL_{\Theta }^{(u_1, \infty)} \Phi_{\gr}^N
       (T)} \| F_2 \|_{\Gamma_2^{\ast} \SL_{\Theta }^{(u_2, 1)} \dfct_{\gr}^N
       (T)} \| F_3 \|_{\Gamma_3^{\ast} \SL_{\Theta }^{(u_3, 2)} \wpD_{\gr}^N
       (T)},
     \end{aligned} \]
  
  \[ \begin{aligned}[t]&
       \int_{\Theta} \left| \textrm{B4} (\theta) [\phi_1 , \phi_2 , \phi_3 ]
       \right| \dd \theta\\ &
       \qquad \lesssim \begin{aligned}[t]&
         \beta \left\| \left( \wpD_{\zeta} (\theta_{\Gamma_2}) - \sigma
         \dd_{\zeta} \right) \wpD_{\zeta} (\theta) F_1^{\ast} (\theta,
         \zeta, \sigma) \right\|_{L^{u_1}_{\frac{\dd \zeta \dd
         \theta}{| \Theta |}} L^1_{\frac{\dd \sigma}{\sigma}}} \qquad\\ &
         \times \left\| \wpD_{\zeta} (\theta_{\Gamma_2})  F_2^{\ast} (\theta,
         \zeta, \sigma) [\phi_2] \right\|_{L^{u_2}_{\frac{\dd \zeta \dd
         \theta}{| \Theta |}} L^{\infty}_{\frac{\dd \sigma}{\sigma}}}
         \left\| \wpD_{\zeta } (\theta_{\Gamma_3}) F^{\ast}_3 (\theta, \zeta,
         \rho) [\phi_3 ] \right\|_{L_{\frac{\dd \zeta \dd \theta}{|
         \Theta |}}^{u_3} L_{\frac{\dd \rho}{\rho}}^2}
       \end{aligned}\\ &
       \qquad \lesssim \beta \| F_1 \|_{\SL_{\Theta }^{(u_1, 1)} \dfct_{\gr}^N
       (T)} \| F_2 \|_{\SL_{\Theta }^{(u_2, \infty)} \Phi_{\gr}^N (T)} \| F_3
       \|_{\SL_{\Theta }^{(u_3, 2)} \Gamma_3^{\ast} \wpD_{\gr}^N (T)}
     \end{aligned} \]

  We conclude the proof by showing that \eqref{eq:single-tree:l-l-o:decomp}
  holds. Our manipulations rely on using the fundamental theorem of calculus
  in the scale parameter on \(F_3^{\ast}\). To exploit certain cancellations we
  need to integrate by parts in the spatial variable. Every integration causes
  ``defect'' terms to appear that have to be accounted for; they are grouped
  into the terms \(\textrm{B1} - \textrm{B4}\).
  
  We crucially rely on the following algebraic identity that follows directly
  from definitions \eqref{eq:space-boost} and \eqref{eq:scale-boost}:
  \begin{equation}
    \wpD_{\sigma} (\theta) F (\theta, \zeta, \sigma) [\phi] = \wpD_{\zeta}
    (\theta) F (\theta, \zeta, \sigma) [z \phi (z)] .
    \label{eq:useful-boost-identities}
  \end{equation}
  In what follows we drop \(\theta\) from the notation and keep the dependence
  on this variable implicit: we write \(\wpD_{\zeta_j}\) in place of
  \(\wpD_{\zeta} (\theta_{\Gamma_j})\), \(\wpD_{\sigma_j}\) in place of
  \(\wpD_{\sigma} (\theta_{\Gamma_j})\) , \(F^{\ast}_j (\zeta, \sigma)\) in place
  of \(F^{\ast}_j (\theta, \zeta, \sigma)\), \(\tilde{\phi}_j^{l / o}\) in place
  of \(\tilde{\phi}_{j, \theta}^{l / o}\), and so on.
  
  It holds that (using the convention of keeping \(\theta\) implicit)
  \begin{equation}
    \begin{aligned}[t]&
      \rho \dd_{\rho} F^{\ast}_3 (\zeta, \rho) [\phi_3^o] = \wpD_{\sigma_3}
      F^{\ast}_3 (\zeta, \rho) [\phi_3^o] - \left( \wpD_{\sigma_3} - \rho
      \dd_{\rho} \right) F^{\ast}_3 (\zeta, \rho) [\phi_3^o]\\ &
      \qquad = \wpD_{\zeta_3}^2 F^{\ast}_3 (\zeta, \rho) [\tilde{\phi}_3^o] -
      \left( \wpD_{\sigma_3} - \rho \dd_{\rho} \right) F^{\ast}_3 (\zeta,
      \rho) [\phi_3^o]\\ &
      \qquad = \begin{aligned}[t]&
        \beta \rho d_{\zeta} \wpD_{\zeta_3} F^{\ast}_3 (\zeta, \rho)
        [\tilde{\phi}_3^o] + \left( \wpD_{\zeta_3} - \beta \rho d_{\zeta}
        \right) \wpD_{\zeta_3} F^{\ast}_3 (\zeta, \rho) [\tilde{\phi}_3^o]\\ &
        - \left( \wpD_{\sigma_3} - \rho \dd_{\rho} \right) F^{\ast}_3
        (\zeta, \rho) [\phi_3^o]
      \end{aligned}
    \end{aligned} \label{eq:scale-derivative-gain-of-cancellation}
  \end{equation}
  We now concentrate on the inner \(\R^2_+\) integral in
  {\LHS{\eqref{eq:single-tree:l-l-o:decomp}}}. Applying the fundamental
  theorem of calculus in the scale variable to term \(F^{\ast}_3 (\zeta,
  \sigma) [\phi_3^o]\) and using identity
  \eqref{eq:scale-derivative-gain-of-cancellation} we get
  \[ \begin{aligned}[t]&
       \int_{\R^2_+} {\wpD_{\zeta_1}}  F_1^{\ast} (\zeta, \sigma)
       [\tilde{\phi}_1^l] \wpD_{\zeta_2} F_2^{\ast} (\zeta, \sigma)
       [\tilde{\phi}_2^l] F_3^{\ast} (\theta, \zeta, \sigma) [\phi_2^o] \dd
       \zeta \frac{\dd \sigma}{\sigma}\\ &
       \qquad = \int_{\R^2_+} {\wpD_{\zeta_1}}  F^{\ast}_1 (\zeta, \sigma)
       [\tilde{\phi}_1^l] \hspace{0.17em} \wpD_{\zeta_2} F^{\ast}_2 (\zeta,
       \sigma) [\tilde{\phi}_2^l] \int_0^{\sigma} \rho \dd_{\rho}
       F^{\ast}_3 (\zeta, \rho) [\phi_3^o] \frac{\dd \rho}{\rho} \dd
       \zeta \frac{\dd \sigma}{\sigma} .\\ &
       \qquad = \begin{aligned}[t]&
         \int_{\R^2_+} \wpD_{\zeta_1} F^{\ast}_1 (\zeta, \sigma)
         [\tilde{\phi}_1^l] \wpD_{\zeta_2} F^{\ast}_2 (\zeta, \sigma)
         [\tilde{\phi}_2^l] \int_0^{\sigma} \beta \rho \dd_{\zeta}
         \wpD_{\zeta_3} F^{\ast}_3 (\zeta, \rho) [\tilde{\phi}_3^o]
         \frac{\dd \rho}{\rho} \dd \zeta \frac{\dd \sigma}{\sigma}\\ &
         - \textrm{B1} [\tilde{\phi}_1^l, \tilde{\phi}_2^l, \phi_3^o] +
         \textrm{B2} [\tilde{\phi}_1^l, \tilde{\phi}_2^l, \tilde{\phi}_3^o] .
       \end{aligned}
     \end{aligned} \]
  Integrating the first summand on the \(\RHS{}\) by parts in \(\zeta\) gives
  \[ \begin{aligned}[t]&
       \int_{\R^2_+} \wpD_{\zeta_1} F^{\ast}_1 (\zeta, \sigma)
       [\tilde{\phi}_1^l] \wpD_{\zeta_2} F^{\ast}_2 (\zeta, \sigma)
       [\tilde{\phi}_2^l] \int_0^{\sigma} \beta \rho \dd_{\zeta}
       \wpD_{\zeta_3} F^{\ast}_3 (\zeta, \rho) [\tilde{\phi}_3^o] \frac{\dd
       \rho}{\rho} \dd \zeta \frac{\dd \sigma}{\sigma}\\ &
       \qquad = \begin{aligned}[t]&
         \int_{\R^2_+} \wpD_{\zeta_1} F^{\ast}_1 (\zeta, \sigma)
         [\tilde{\phi}_1^l] \hspace{0.17em} \beta \sigma \dd_{\zeta}
         \wpD_{\zeta_2} F^{\ast}_2 (\zeta, \sigma) [\tilde{\phi}_2^l]
         \int_0^{\sigma} \frac{\rho}{\sigma} \wpD_{\zeta_3} F^{\ast}_3 (\zeta,
         \rho) [\tilde{\phi}_3^o] \frac{\dd \rho}{\rho} \frac{\dd
         \sigma}{\sigma} \dd \zeta\\ &
         \qquad + \beta \int_{\R^2_+} \sigma \dd_{\zeta} \wpD_{\zeta_1}
         F^{\ast}_1 (\zeta, \sigma) [\tilde{\phi}_1^l] \wpD_{\zeta_2}
         F^{\ast}_2 (\zeta, \sigma) [\tilde{\phi}_2^l] \int_0^{\sigma}
         \frac{\rho}{\sigma} \wpD_{\zeta_3} F^{\ast}_3 (\zeta, \rho)
         [\tilde{\phi}_3^o] \frac{\dd \rho}{\rho} \frac{\dd
         \sigma}{\sigma} \dd \zeta
       \end{aligned}\\ &
       \qquad = \textrm{M1} [\tilde{\phi}_1^l, \tilde{\phi}_2^l,
       \tilde{\phi}_3^o] + \textrm{B3} [\tilde{\phi}_1^l, \tilde{\phi}_2^l,
       \tilde{\phi}_3^o] + \textrm{M2} [\tilde{\phi}_1^l, \tilde{\phi}_2^l,
       \tilde{\phi}_3^o] + \textrm{B4} [\tilde{\phi}_1^l, \tilde{\phi}_2^l,
       \tilde{\phi}_3^o] .
     \end{aligned} \]
  This concludes the proof of the decomposition
  \eqref{eq:single-tree:l-l-o:decomp}.
\end{proof}

\begin{proof}{Proof of bound \eqref{eq:STE:l-o-o:tree}}
  Fix \(\phi_j \in \Phi_{\gr / 4}^{N + 3}\) with \(\| \phi_j \|_{\Phi_{\gr /
  4}^{N + 3}} \leq 1\) so that \(\| \phi_{j, \theta}^{o / l} \|_{\Phi^{N +
  3}_{\gr / 2}} \lesssim 1\). Using the notation from the discussion preceding
  this proof, we need to bound the expression
  \begin{equation}
    \int_{\Theta} \left| \int_{\R^2_+} \wpD_{\zeta} (\theta) F_1^{\ast}
    (\theta, \zeta, \sigma) [\tilde{\phi}_{1, \theta}^l] F_{\times }^{\ast}
    (\theta, \zeta, \rho) [\phi_{2, \theta}^o, \phi_{3, \theta}^o] \dd \zeta
    \frac{\dd \sigma}{\sigma} \right| \frac{\dd \theta}{| \Theta |}
    \label{eq:l-o-o-quantity}
  \end{equation}
  for every \(\phi_j \in \Phi_{\gr / 4}^{N + 3}\) with \(\| \phi_j \|_{\Phi_{\gr
  / 4}^{N + 3}} \leq 1\). We claim that
  \begin{equation}
    \eqref{eq:l-o-o-quantity} \begin{aligned}[t]&
      \leq \int_{\Theta} \begin{aligned}[t]&
        (| \textrm{M1} (\theta) [\tilde{\phi}_{1, \theta}^l, \tilde{\phi}_{2,
        \theta}^o, \phi_{3, \theta}^o] | + | \textrm{M1} (\theta)
        [\tilde{\phi}_{1, \theta}^l, \phi_{2, \theta}^o, \tilde{\phi}_3^o] |
        \nobracket \qquad\\ &
        + \left| \textrm{M2}_{\left( \wpD_{\zeta_2} \otimes \Id \right)}
        [\tilde{\phi}_{1, \theta}^l, \tilde{\phi}_{2, \theta}^o, \phi_{3,
        \theta}^o] \right| + \left| \textrm{M2}_{\left( \Id \otimes
        \wpD_{\zeta_3} \right)} [\tilde{\phi}_{1, \theta}^l, \phi_{2,
        \theta}^o, \tilde{\phi}_{3, \theta}^o] \right|\\ &
        + | \textrm{B1} (\theta) [\tilde{\phi}_{1, \theta}^l, \phi_{2,
        \theta}^o, \phi_{3, \theta}^o] |\\ &
        + \left| \textrm{B2}_{\left( \wpD_{\zeta_2} \otimes \Id \right)}
        (\theta) [\tilde{\phi}_{1, \theta}^l, \tilde{\phi}_{2, \theta}^o,
        \phi_{3, \theta}^o] \right| + \left| \textrm{B2}_{\left( \Id \otimes
        \wpD_{\zeta_3} \right)} (\theta) [\tilde{\phi}_{1, \theta}^l, \phi_{2,
        \theta}^o, \tilde{\phi}_{3, \theta}^o] \right|\\ &
        + \left| \textrm{B3}_{\left( \wpD_{\zeta_2} \otimes \Id \right)}
        (\theta) [\tilde{\phi}_{1, \theta}^l, \tilde{\phi}_{2, \theta}^o,
        \phi_{3, \theta}^o] \right|\\ &
        \qquad \left\nobracket + \left| \textrm{B3}_{\left( \Id \otimes
        \wpD_{\zeta_3} \right)} (\theta) [\tilde{\phi}_{1, \theta}^l, \phi_{2,
        \theta}^o, \tilde{\phi}_{3, \theta}^o] \right| \right) \frac{\dd
        \theta}{| \Theta |}
      \end{aligned}\\ &
      \lesssim_N \sup_{\tmscript{\begin{aligned}[t]&
        \left\| {\phi_j}  \right\|_{\Phi_{\gr}^N} \leq 1\\ &
        j \in \{ 1, 2, 3 \}
      \end{aligned}}} \int_{\Theta} \begin{aligned}[t]&
        \left( | \textrm{M1} (\theta) [\phi_1 , \phi_2 , \phi_3 ] | + \left|
        \textrm{M2\((\theta) [\phi_1 , \phi_2 , \phi_3 ]\)} \right| + \right.
        \qquad\\ &
        \qquad + | \textrm{B1} (\theta) [\phi_1 , \phi_2 , \phi_3 ] | + |
        \textrm{B2} (\theta) [\phi_1 , \phi_2 , \phi_3 ] |\\ &
        \qquad \left\nobracket + \left| \text{B3} (\theta) [\phi_1 , \phi_2 ,
        \phi_3 ] \right| \right) \frac{\dd \theta}{| \Theta |}
      \end{aligned}
    \end{aligned} \label{eq:single-tree:l-o-o:decomp}
\end{equation}

  where
  \[ \textrm{M1} (\theta) [\phi_1 , \phi_2 , \phi_3 ] \assign
     \begin{aligned}[t]&
       \int_{\R^2_+} \wpD_{\zeta} (\theta) F_1^{\ast} (\theta, \zeta, \sigma)
       [\phi_1 ] \qquad\\ &
       \qquad \times \int_0^{\sigma} \left( \wpD_{\zeta } (\theta_{\Gamma_2})
       \otimes \wpD_{\zeta } (\theta_{\Gamma_3}) \right) F_{\times }^{\ast}
       (\theta, \zeta, \sigma) [\phi_2, \phi_3] \frac{\dd \rho}{\rho}
       \frac{\dd \sigma}{\sigma} \dd \zeta,
     \end{aligned} \]
  \[ \textrm{\(\textrm{M2}_{\left( \wpD_{\zeta} \otimes \Id \right)}\)} (\theta)
     [\phi_1 , \phi_2 , \phi_3 ] \assign \begin{aligned}[t]&
       \beta \int_{\R^2_+} \wpD_{\zeta} (\theta)^2 F_1^{\ast} (\theta, \zeta,
       \sigma) [\phi_1 ] \qquad\\ &
       \qquad \times \int_0^{\sigma} \frac{\rho}{\sigma} \left( \wpD_{\zeta }
       (\theta_{\Gamma_2}) \otimes \tmop{Id} \right) F_{\times }^{\ast}
       (\theta, \zeta, \sigma) [\phi_2, \phi_3] \frac{\dd \rho}{\rho}
       \frac{\dd \sigma}{\sigma} \dd \zeta,
     \end{aligned} \]
  \[ \textrm{\(\textrm{M2}_{\left( \Id \otimes \wpD_{\zeta} \right)}\)} (\theta)
     [\phi_1 , \phi_2 , \phi_3 ] \assign \begin{aligned}[t]&
       \beta \int_{\R^2_+} \wpD_{\zeta} (\theta)^2 F_1^{\ast} (\theta, \zeta,
       \sigma) [\phi_1 ] \qquad\\ &
       \qquad \times \int_0^{\sigma} \frac{\rho}{\sigma} \left( \tmop{Id}
       \otimes \wpD_{\zeta } (\theta_{\Gamma_3}) \right) F_{\times }^{\ast}
       (\theta, \zeta, \sigma) [\phi_2, \phi_3] \frac{\dd \rho}{\rho}
       \frac{\dd \sigma}{\sigma} \dd \zeta,
     \end{aligned} \]
  and
  \[ \begin{aligned}[t]&
       \textrm{B1} (\theta) [\phi_1 , \phi_2 , \phi_3 ]\\ &
       \qquad \assign \begin{aligned}[t]&
         \int_{\R^2_+} \wpD_{\zeta} (\theta) F_1^{\ast} (\theta, \zeta,
         \sigma) [\phi_1 ]\\ &
         \qquad \times \int_0^{\sigma} \left( \wpD_{\sigma}
         (\theta_{\Gamma_2}) \otimes \Id + \Id \otimes \wpD_{\sigma }
         (\theta_{\Gamma_3}) - \rho \dd_{\rho} \right) F_{\times }^{\ast}
         (\theta, \zeta, \sigma) [\phi_2, \phi_3] \frac{\dd \rho}{\rho}
         \frac{\dd \sigma}{\sigma} \dd \zeta,
       \end{aligned}
     \end{aligned} \]
  \[ \begin{aligned}[t]&
       \textrm{B2}_{\left( \wpD_{\zeta} \otimes \Id \right)} [\phi_1 , \phi_2
       , \phi_3 ]\\ &
       \qquad \assign \begin{aligned}[t]&
         \int_{\R^2_+} \wpD_{\zeta} (\theta) F_1^{\ast} (\theta, \zeta,
         \sigma) [\phi_1 ] \int_0^{\sigma} \left( \wpD_{\zeta}
         (\theta_{\Gamma_2}) \otimes \Id + \Id \otimes \wpD_{\zeta }
         (\theta_{\Gamma_3}) - \beta \rho \dd_{\zeta} \right)\\ &
         \qquad \left( \wpD_{\zeta} (\theta_{\Gamma_2}) \otimes \Id \right)
         F_{\times }^{\ast} (\theta, \zeta, \sigma) [\phi_2, \phi_3]
         \frac{\dd \rho}{\rho} \frac{\dd \sigma}{\sigma} \dd \zeta,
       \end{aligned}
     \end{aligned} \]
  \[ \begin{aligned}[t]&
       \textrm{B2}_{\left( \Id \otimes \wpD_{\zeta} \right)} [\phi_1 , \phi_2
       , \phi_3 ]\\ &
       \qquad \assign \begin{aligned}[t]&
         \int_{\R^2_+} \wpD_{\zeta} (\theta) F_1^{\ast} (\theta, \zeta,
         \sigma) [\phi_1 ] \int_0^{\sigma} \left( \wpD_{\zeta}
         (\theta_{\Gamma_2}) \otimes \Id + \Id \otimes \wpD_{\zeta }
         (\theta_{\Gamma_3}) - \beta \rho \dd_{\zeta} \right)\\ &
         \qquad \left( \Id \otimes \wpD_{\zeta } (\theta_{\Gamma_3}) \right)
         F_{\times }^{\ast} (\theta, \zeta, \sigma) [\phi_2, \phi_3]
         \frac{\dd \rho}{\rho} \frac{\dd \sigma}{\sigma} \dd \zeta,
       \end{aligned}
     \end{aligned} \]
  \[ \textrm{B3}_{\left( \wpD_{\zeta} \otimes \Id \right)} [\phi_1 , \phi_2 ,
     \phi_3 ] \assign \begin{aligned}[t]&
       \beta \int_{\R^2_+} \left( \wpD_{\zeta} (\theta) - \sigma d_{\zeta}
       \right) F_1^{\ast} (\theta, \zeta, \sigma) [\phi_1 ]\\ &
       \qquad \times \int_0^{\sigma} \frac{\rho}{\sigma} \left( \wpD_{\zeta}
       (\theta_{\Gamma_2}) \otimes \Id \right) F_{\times }^{\ast} (\theta,
       \zeta, \sigma) [\phi_2, \phi_3] \frac{\dd \rho}{\rho} \frac{\dd
       \sigma}{\sigma} \dd \zeta,
     \end{aligned} \]
  \[ \textrm{B3}_{\left( \Id \otimes \wpD_{\zeta} \right)} [\phi_1 , \phi_2 ,
     \phi_3 ] \assign \begin{aligned}[t]&
       \beta \int_{\R^2_+} \left( \wpD_{\zeta} (\theta) - \sigma d_{\zeta}
       \right) F_1^{\ast} (\theta, \zeta, \sigma) [\phi_1 ] \qquad\\ &
       \qquad \times \int_0^{\sigma} \frac{\rho}{\sigma} \left( \Id \otimes
       \wpD_{\zeta } (\theta_{\Gamma_3}) \right) F_{\times }^{\ast} (\theta,
       \zeta, \sigma) [\phi_2, \phi_3] \frac{\dd \rho}{\rho} \frac{\dd
       \sigma}{\sigma} \dd \zeta,
     \end{aligned} \]

  Let us now show that each of the 6 terms in
  {\RHS{\eqref{eq:single-tree:l-o-o:decomp}}} contributes a quantity bounded
  by {\RHS{\eqref{eq:STE:l-o-o}}}. For \(\left\| {\phi_j} 
  \right\|_{\Phi_{\gr}^N} \leq 1\), \(j \in \{ 1, 2, 3 \}\), the following bounds
  thus hold:
  \[ \begin{aligned}[t]&
       \int_{\Theta^{\tmop{in}}} | \textrm{M1} (\theta) [\phi_1 , \phi_2 ,
       \phi_3 ] | \dd \theta\\ &
       \qquad\lesssim \begin{aligned}[t]&
         \left\| \wpD_{\zeta} (\theta)  F_1^{\ast} (\theta, \zeta, \sigma)
         [\phi_1] \right\|_{L^{u_1}_{\frac{\dd \zeta \dd \theta}{|
               \Theta |}} L^{\infty}_{\frac{\dd \sigma}{\sigma}}} \\ & \qquad \times
         \left\| \left(
         \wpD_{\zeta } (\theta_{\Gamma_2}) \otimes \wpD_{\zeta }
         (\theta_{\Gamma_3}) \right) F_{\times }^{\ast} (\theta, \zeta,
         \sigma) [\phi_2, \phi_3] \right\|_{L^{u_{\times}}_{\frac{\dd
         \zeta \dd \theta}{| \Theta |}} L^1_{\frac{\dd \sigma}{\sigma}}}
       \end{aligned} \\ &
       \qquad \lesssim \| F_1 \|_{\SL_{\Theta }^{(u_1, \infty)} \Phi_{\gr}^N (T)} \|
       F_2 F_3 \|_{\SL_{\Theta }^{(u_{\times}, 1)} \Gamma_{\times}^{\ast}
       \left( \wpD \otimes \wpD \right)_{\gr}^N (T)} ;
     \end{aligned} \]
  \[ \begin{aligned}[t]&
       \int_{\Theta^{\tmop{in}}} \left| \textrm{M2}_{\left( \wpD_{\zeta}
       \otimes \Id \right)} (\theta) [\phi_1 , \phi_2 , \phi_3 ] \right|
       \dd \theta\\ &
       \qquad \lesssim \begin{aligned}[t]&
         \beta \left\| \wpD_{\zeta} (\theta) ^2 F_1^{\ast} (\theta, \zeta,
         \sigma) [\phi_1] \right\|_{L^{u_1}_{\frac{\dd \zeta \dd
               \theta}{| \Theta |}} L^2_{\frac{\dd \sigma}{\sigma}}} \\ & \qquad \times\left\|
         \left( \wpD_{\zeta } (\theta_{\Gamma_2}) \otimes \Id \right)
         F_{\times }^{\ast} (\theta, \zeta, \sigma) [\phi_2, \phi_3]
         \right\|_{L^{u_{\times}}_{\frac{\dd \zeta \dd \theta}{| \Theta
         |}} L^2_{\frac{\dd \sigma}{\sigma}}}
       \end{aligned} \\ &\qquad
       \lesssim \beta \| F_1 \|_{\SL_{\Theta }^{(u_1, 2)} \wpD_{\gr}^N (T)} \|
       F_2 F_3 \|_{\Gamma_{\times}^{\ast} \SL_{\Theta }^{(u_{\times}, 2)}
       \left( \Id \otimes \wpD \right)_{\gr}^N (T)} ;
     \end{aligned} \]
  similarly
  \[ \begin{aligned}[t]&
       \int_{\Theta} \left| \textrm{M2}_{\left( \Id \otimes \wpD_{\zeta}
       \right)} (\theta) [\phi_1 , \phi_2 , \phi_3 ] \right| \dd \theta\\ &
       \qquad \lesssim \beta \| F_1 \|_{\SL_{\Theta }^{(u_1, 2)} \wpD_{\gr}^N
       (T)} \| F_2 F_3 \|_{\Gamma_{\times}^{\ast} \SL_{\Theta }^{(u_{\times},
       2)} \left( \wpD \otimes \Id \right)_{\gr}^N (T)} .
     \end{aligned} \]
  Bounding the boundary terms \(\textrm{B1}\), \(\textrm{B2}\), and \(\textrm{B3}\),
  requires the use of defect sizes:
  \[ \begin{aligned}[t]&
       \int_{\Theta^{\tmop{in}}} | \textrm{B1} (\theta) [\phi_1 , \phi_2 ,
       \phi_3 ] | \dd \theta\\ &
       \qquad \lesssim \begin{aligned}[t]&
         \left\| \int_{\rho}^{+ \infty} \wpD_{\zeta} (\theta)  F_1^{\ast}
         (\theta, \zeta, \sigma) [\phi_1] \frac{\dd \sigma}{\sigma}
         \right\|_{L^{u_1}_{\frac{\dd \zeta \dd \theta}{| \Theta |}}
         L^{\infty}_{\frac{\dd \rho}{\rho}}}\\ &
         \times \left\| \left( \wpD_{\sigma} (\theta_{\Gamma_2}) \otimes \Id
         + \Id \otimes \wpD_{\sigma } (\theta_{\Gamma_3}) - \rho \dd_{\rho}
         \right) F_{\times }^{\ast} (\theta, \zeta, \rho) [\phi_2, \phi_3]
         \right\|_{L^{u_{\times}}_{\frac{\dd \zeta \dd \theta}{| \Theta
         |}} L^1_{\frac{\dd \rho}{\rho}}}
       \end{aligned} \\ &
       \lesssim \| F_1 \|_{\SJ^{(u_1, \infty)}_{\Theta} \Phi_{\gr}^N (T)} \|
       F_2 F_3 \|_{\Gamma_{\times}^{\ast} \SL_{(\Theta ,
       \Theta^{\tmop{in}})}^{(u_{\times}, 1)} \dfct_{\gr}^{N } (T)} ;
     \end{aligned} \]
  \[ \begin{aligned}[t]&
       \int_{\Theta} \left| \textrm{B2}_{\left( \wpD_{\zeta_2} \otimes \Id
       \right)} [\phi_1 , \phi_2 , \phi_3 ] \right| \dd \theta\\ &
       \qquad \lesssim \begin{aligned}[t]&
         \left\| \wpD_{\zeta} (\theta)  F_1^{\ast} (\theta, \zeta, \sigma)
         [\phi_1] \frac{\dd \sigma}{\sigma} \right\|_{L^{u_1}_{\frac{\dd
         \zeta \dd \theta}{| \Theta |}} L^2_{\frac{\dd
         \sigma}{\sigma}}}\\ &
         \times \begin{aligned}[t]&
           \left\| \left( \wpD_{\zeta} (\theta_{\Gamma_2}) \otimes \Id + \Id
           \otimes \wpD_{\zeta } (\theta_{\Gamma_3}) - \beta \rho
           \dd_{\zeta} \right) \right\nobracket\\ &
           \qquad \times \left. \left( \wpD_{\zeta} (\theta_{\Gamma_2})
           \otimes \Id \right) F_{\times }^{\ast} (\theta, \zeta, \rho)
           [\phi_2, \phi_3] \right\|_{L^{u_{\times}}_{\frac{\dd \zeta
           \dd \theta}{| \Theta |}} L^1_{\frac{\dd \rho}{\rho}}}
         \end{aligned}
       \end{aligned} \\ &
       \qquad \lesssim \| F_1 \|_{\SL^{(u_1, 2)}_{\Theta} \wpD_{\gr}^N (T)} \| F_2
       F_3 \|_{\Gamma_{\times}^{\ast} \SL_{(\Theta ,
       \Theta^{\tmop{in}})}^{(u_{\times}, 1)} \dfct_{\gr}^{N } (T)}
     \end{aligned} \]
  and similarly
  \[ \int_{\Theta} \left| \textrm{B2}_{\left( \wpD_{\zeta_2} \otimes \Id
     \right)} [\phi_1 , \phi_2 , \phi_3 ] \right| \dd \theta \lesssim \|
     F_1 \|_{\SL^{(u_1, 2)}_{\Theta} {\wpD_{\gr}^N}' (T)} \| F_2 F_3
     \|_{\Gamma_{\times}^{\ast} \SL_{(\Theta , \Theta^{\tmop{in}})
     }^{(u_{\times}, 1)} \dfct_{\gr}^{N } (T)} ; \]
  \[ \begin{aligned}[t]&
       \int_{\Theta} \left| \textrm{B3}_{\left( \wpD_{\zeta_2} \otimes \Id
       \right)} [\phi_1 , \phi_2 , \phi_3 ] \right| \dd \theta
       \begin{aligned}[t]&
         \lesssim \begin{aligned}[t]&
           \beta \left\| \left( \wpD_{\zeta} (\theta) - \sigma d_{\zeta}
           \right)  F_1^{\ast} (\theta, \zeta, \sigma) [\phi_1] \frac{\dd
           \sigma}{\sigma} \right\|_{L^{u_1}_{\frac{\dd \zeta \dd
           \theta}{| \Theta |}} L^1_{\frac{\dd \sigma}{\sigma}}}\\ &
           \qquad \times \left\| \left( \wpD_{\zeta} (\theta_{\Gamma_2})
           \otimes \Id \right) F_{\times }^{\ast} (\theta, \zeta, \rho)
           [\phi_2, \phi_3] \right\|_{L^{u_{\times}}_{\frac{\dd \zeta
           \dd \theta}{| \Theta |}} L^2_{\frac{\dd \rho}{\rho}}}
         \end{aligned}\\ &
         \lesssim \beta \| F_1 \|_{\SL^{(u_1, 1)}_{\Theta} \dfct_{\gr}^{N }
         (T)} \| F_2 F_3 \|_{\Gamma_{\times}^{\ast} \SL_{\Theta
         }^{(u_{\times}, 1)} {\left( \wpD \otimes \Id \right) }_{\gr}^{N }
         (T)}
       \end{aligned} \\ &       
     \end{aligned} \]
  and similarly
  \[ \begin{aligned}[t]&
       \int_{\Theta} \left| \textrm{B3}_{\left( \Id \otimes \wpD_{\zeta_3}
       \right)} [\phi_1 , \phi_2 , \phi_3 ] \right| \dd \theta \lesssim
       \begin{aligned}[t]&
         \beta \left\| \left( \wpD_{\zeta} (\theta) - \sigma d_{\zeta} \right)
         F_1^{\ast} (\theta, \zeta, \sigma) [\phi_1] \frac{\dd
         \sigma}{\sigma} \right\|_{L^{u_1}_{\frac{\dd \zeta \dd
         \theta}{| \Theta |}} L^1_{\frac{\dd \sigma}{\sigma}}}\\ &
         \qquad \times \left\| \left( \Id \otimes \wpD_{\zeta }
         (\theta_{\Gamma_3}) \right) F_{\times }^{\ast} (\theta, \zeta, \rho)
         [\phi_2, \phi_3] \right\|_{L^{u_{\times}}_{\frac{\dd \zeta \dd
         \theta}{| \Theta |}} L^2_{\frac{\dd \rho}{\rho}}}
       \end{aligned} \\ &
       \lesssim \beta \| F_1 \|_{\SL^{(u_1, 1)}_{\Theta} \dfct_{\gr}^{N } (T)}
       \| F_2 F_3 \|_{\Gamma_{\times}^{\ast} \SL_{\Theta }^{(u_{\times}, 1)}
       {\left( \Id \otimes \wpD \right) }_{\gr}^{N } (T)} .
     \end{aligned} \]

  We conclude the proof by showing that \eqref{eq:single-tree:l-o-o:decomp}
  holds. Our manipulations rely on using the fundamental theorem of calculus
  in the scale parameter on the product \(F_2^{\ast} F_3^{\ast}\). To exploit
  certain cancellations we need to integrate by parts in the spatial variable.
  Every integration causes ``defect'' terms to appear that have to be
  accounted for; they are grouped into the terms \(\textrm{B1} - \textrm{B3}\).
  
  In what follows we drop \(\theta\) from the notation and keep the dependence
  on this variable implicit: we write \(\wpD_{\zeta_j}\) in place of
  \(\wpD_{\zeta} (\theta_{\Gamma_j})\), \(\wpD_{\sigma_j}\) in place of
  \(\wpD_{\sigma} (\theta_{\Gamma_j})\), \(\tilde{\phi}_j^{l / o}\) in place of
  \(\tilde{\phi}_{j, \theta}^{l / o}\), and so on. For conciseness, we also
  write \(F_{\times }^{\ast} (\zeta, \rho)\) for \(F^{\ast}_2 (\zeta, \rho)
  F^{\ast}_3 (\zeta, \rho)\). It holds that (using the convention of keeping
  \(\theta\) implicit)
  \begin{equation}
    \begin{aligned}[t]&
      \rho \dd_{\rho} (F^{\ast}_2 (\zeta, \rho) [\phi_2^o] F^{\ast}_3
      (\zeta, \rho) [\phi_3^o]) = \rho \dd_{\rho} F_{\times }^{\ast}
      (\zeta, \rho) [\phi_2^o, \phi_3^o]\\ &
      =
      \begin{aligned}[t]
      &+ \left( \wpD_{\sigma_2} \otimes \Id + \Id \otimes \wpD_{\sigma_3}
      \right) F_{\times }^{\ast} (\zeta, \rho) [\phi_2^o, \phi_3^o]\\ &
      - \left( \wpD_{\sigma_2} \otimes \Id + \Id \otimes \wpD_{\sigma_3} -
      \rho \dd_{\rho} \right) F_{\times }^{\ast} (\zeta, \rho) [\phi_2^o,
      \phi_3^o]
      \end{aligned}
      \\ &
      =
      \begin{aligned}[t]
      &+ \left( \wpD_{\zeta_2} \otimes \Id \right)^2 F_{\times }^{\ast} (\zeta,
      \rho) [\tilde{\phi}_2^o, \phi_3^o] + \left( \Id \otimes \wpD_{\zeta_3}
      \right)^2 F_{\times }^{\ast} (\zeta, \rho) [\phi_2^o,
      \tilde{\phi}_3^o]\\ &
      - \left( \wpD_{\sigma_2} \otimes \Id + \Id \otimes \wpD_{\sigma_3} -
      \rho \dd_{\rho} \right) F_{\times }^{\ast} (\zeta, \rho) [\phi_2^o,
      \phi_3^o]
      \end{aligned}
      \\ &
      =
      \begin{aligned}[t]
      & + \left( \wpD_{\zeta_2} \otimes \Id + \Id \otimes \wpD_{\zeta_3} \right)
      \left( \wpD_{\zeta_2} \otimes \Id \right)  F_{\times }^{\ast} (\zeta,
      \rho) [\tilde{\phi}_2^o, \phi_3^o]\\ &
      + \left( \wpD_{\zeta_2} \otimes \Id + \Id \otimes \wpD_{\zeta_3} \right)
      \left( \Id \otimes \wpD_{\zeta_3} \right) F^{\ast}_{\times} (\zeta,
      \rho) [\phi_2^o, \tilde{\phi}_3^o]\\ &
      - \left( \wpD_{\zeta_2} \otimes \wpD_{\zeta_3} \right) F_{\times
      }^{\ast} (\zeta, \rho) [\tilde{\phi}_2^o, \phi_3^o] - \left(
      \wpD_{\zeta_2} \otimes \wpD_{\zeta_3} \right)  F_{\times }^{\ast}
      (\zeta, \rho) [\phi_2^o, \tilde{\phi}_3^o]\\ &
      - \left( \wpD_{\sigma_2} \otimes \Id + \Id \otimes \wpD_{\sigma_3} -
      \rho \dd_{\rho} \right) F_{\times }^{\ast} (\zeta, \rho) [\phi_2^o,
      \phi_3^o]
      \end{aligned}
      \\ &
      =
      \begin{aligned}[t]
      &+ \beta \rho d_{\zeta} \left( \wpD_{\zeta_2} \otimes \Id \right) 
      F_{\times }^{\ast} (\zeta, \rho) [\tilde{\phi}_2^o, \phi_3^o]\\ &
      + \beta \rho d_{\zeta} \left( \Id \otimes \wpD_{\zeta_3} \right)
      F_{\times }^{\ast} (\zeta, \rho) [\phi_2^o, \tilde{\phi}_3^o]\\ &
      + \left( \wpD_{\zeta_2} \otimes \Id + \Id \otimes \wpD_{\zeta_3} - \beta
      \rho d_{\zeta} \right) \left( \wpD_{\zeta_2} \otimes \Id \right) 
      F_{\times }^{\ast} (\zeta, \rho) [\tilde{\phi}_2^o, \phi_3^o]\\ &
      + \left( \wpD_{\zeta_2} \otimes \Id + \Id \otimes \wpD_{\zeta_3} - \beta
      \rho d_{\zeta} \right) \left( \Id \otimes \wpD_{\zeta_3} \right)
      F^{\ast}_{\times} (\zeta, \rho) [\phi_2^o, \tilde{\phi}_3^o]\\ &
      - \left( \wpD_{\zeta_2} \otimes \wpD_{\zeta_3} \right) F_{\times
      }^{\ast} (\zeta, \rho) [\tilde{\phi}_2^o, \phi_3^o] - \left(
      \wpD_{\zeta_2} \otimes \wpD_{\zeta_3} \right)  F_{\times }^{\ast}
      (\zeta, \rho) [\phi_2^o, \tilde{\phi}_3^o]\\ &
      - \left( \wpD_{\sigma_2} \otimes \Id + \Id \otimes \wpD_{\sigma_3} -
      \rho \dd_{\rho} \right) F_{\times }^{\ast} (\zeta, \rho) [\phi_2^o,
      \phi_3^o] .
      \end{aligned}
    \end{aligned} \label{eq:scale-derivative-gain-of-cancellation:prod}
  \end{equation}
  We now concentrate on the inner \(\R^2_+\) integral in
  \(\LHS{\eqref{eq:single-tree:l-o-o:decomp}}\). Applying the fundamental
  theorem Differentiating and integrating back in scale the term \(F^{\ast}_3
  (\zeta, \sigma) [\phi_3^o]\) in the first summand gives and using identity
  \eqref{eq:scale-derivative-gain-of-cancellation:prod} we get
  \[ \begin{aligned}[t]&
       \int_{\R^2_+} \wpD_{\zeta} F_1^{\ast} (\zeta, \sigma)
       [\tilde{\phi}_1^l] F_2^{\ast} (\zeta, \sigma) [\phi_2^o] F_3^{\ast}
       (\zeta, \sigma) [\phi_3^o] \dd \zeta \frac{\dd \sigma}{\sigma}\\ &
       \qquad= \int_{\R^2_+} \wpD_{\zeta} F_1^{\ast} (\zeta, \sigma)
       [\tilde{\phi}_1^l] \int_0^{\sigma} \rho \dd_{\rho} F_{\times
       }^{\ast} (\zeta, \rho) [\phi_2^o, \phi_3^o] \frac{\dd \rho}{\rho}
       \dd \zeta \frac{\dd \sigma}{\sigma}\\ &
       \qquad= \begin{aligned}[t]&
         \int_{\R^2_+} \wpD_{\zeta} F_1^{\ast} (\zeta, \sigma)
         [\tilde{\phi}_1^l] \int_0^{\sigma} \beta \rho d_{\zeta} \left(
         \wpD_{\zeta_2} \otimes \Id \right)  F_{\times }^{\ast} (\zeta, \rho)
         [\tilde{\phi}_2^o, \phi_3^o] \frac{\dd \rho}{\rho} \dd \zeta
         \frac{\dd \sigma}{\sigma}\\ &
         + \int_{\R^2_+} \wpD_{\zeta} F_1^{\ast} (\zeta, \sigma)
         [\tilde{\phi}_1^l] \int_0^{\sigma} \beta \rho d_{\zeta}  \left( \Id
         \otimes \wpD_{\zeta_3} \right) F_{\times }^{\ast} (\zeta, \rho)
         [\phi_2^o, \tilde{\phi}_3^o] \frac{\dd \rho}{\rho} \dd \zeta
         \frac{\dd \sigma}{\sigma}\\ &
         - \textrm{B1} [\tilde{\phi}_1^l, \phi_2^o, \phi_3^o] - \textrm{M1}
         [\tilde{\phi}_1^l, \tilde{\phi}_2^o, \phi_3^o] - \textrm{M1}
         [\tilde{\phi}_1^l, \phi_2^o, \tilde{\phi}_3^o]\\ &
         + \textrm{B2}_{\left( \wpD_{\zeta } \otimes \Id \right)}
         [\tilde{\phi}_1^l, \tilde{\phi}_2^o, \phi_3^o] + \textrm{B2}_{\left(
         \Id \otimes \wpD_{\zeta } \right)} [\tilde{\phi}_1^l, \phi_2^o,
         \tilde{\phi}_3^o] .
       \end{aligned}
     \end{aligned} \]
  Integrating the first two summands on the \(\RHS{}\) by parts in \(\zeta\) gives
  \[ \begin{aligned}[t]&
       \int_{\R^2_+} \wpD_{\zeta} F_1^{\ast} (\zeta, \sigma)
       [\tilde{\phi}_1^l] \int_0^{\sigma} \beta \rho d_{\zeta} \left(
       \wpD_{\zeta_2} \otimes \Id \right)  F_{\times }^{\ast} (\zeta, \rho)
       [\tilde{\phi}_2^o, \phi_3^o] \frac{\dd \rho}{\rho} \dd \zeta
       \frac{\dd \sigma}{\sigma}\\ &
       = - \beta \int_{\R^2_+} \sigma d_{\zeta} \wpD_{\zeta} F_1^{\ast}
       (\zeta, \sigma) [\tilde{\phi}_1^l] \int_0^{\sigma} \frac{\rho}{\sigma}
       \left( \wpD_{\zeta_2} \otimes \Id \right)  F_{\times }^{\ast} (\zeta,
       \rho) [\tilde{\phi}_2^o, \phi_3^o] \frac{\dd \rho}{\rho} \dd
       \zeta \frac{\dd \sigma}{\sigma}\\ &
       =
       \begin{aligned}[t]
       &- \beta \int_{\R^2_+} \wpD_{\zeta}^2 F_1^{\ast} (\zeta, \sigma)
       [\tilde{\phi}_1^l] \int_0^{\sigma} \frac{\rho}{\sigma} \left(
       \wpD_{\zeta_2} \otimes \Id \right)  F_{\times }^{\ast} (\zeta, \rho)
       [\tilde{\phi}_2^o, \phi_3^o] \frac{\dd \rho}{\rho} \dd \zeta
       \frac{\dd \sigma}{\sigma}\\ &
       + \beta \int_{\R^2_+} \left( \wpD_{\zeta} - \sigma d_{\zeta} \right)
       \wpD_{\zeta} F_1^{\ast} (\zeta, \sigma) [\tilde{\phi}_1^l]
       \int_0^{\sigma} \frac{\rho}{\sigma} \left( \wpD_{\zeta_2} \otimes \Id
       \right)  F_{\times }^{\ast} (\zeta, \rho) [\tilde{\phi}_2^o, \phi_3^o]
       \frac{\dd \rho}{\rho} \dd \zeta \frac{\dd \sigma}{\sigma}
       \end{aligned}
       \\ &
       = - \textrm{M2}_{\left( \wpD_{\zeta_2} \otimes \Id \right)}
       [\tilde{\phi}_1^l, \tilde{\phi}_2^o, \phi_3^o] + \textrm{B3}_{\left(
       \wpD_{\zeta_2} \otimes \Id \right) } [\tilde{\phi}_1^l,
       \tilde{\phi}_2^o, \phi_3^o]
     \end{aligned} \]
  and
  \[ \begin{aligned}[t]&
       \int_{\R^2_+} \wpD_{\zeta} F_1^{\ast} (\zeta, \sigma)
       [\tilde{\phi}_1^l] \int_0^{\sigma} \beta \rho d_{\zeta} \left(
       \wpD_{\zeta_2} \otimes \Id \right)  F_{\times }^{\ast} (\zeta, \rho)
       [\tilde{\phi}_2^o, \phi_3^o] \frac{\dd \rho}{\rho} \dd \zeta
       \frac{\dd \sigma}{\sigma}\\ &
       \qquad= - \textrm{M2}_{\left( \Id \otimes \wpD_{\zeta } \right)}
       [\tilde{\phi}_1^l, \phi_2^o, \tilde{\phi}_3^o] + \textrm{B3}_{\left(
       \Id \otimes \wpD_{\zeta } \right) } [\tilde{\phi}_1^l, \phi_2^o,
       \tilde{\phi}_3^o] .
     \end{aligned} \]
  This concludes the proof of the decomposition
  \eqref{eq:single-tree:l-o-o:decomp}.
\end{proof}

\section{Forest estimates}\label{sec:forest}

In this section we prove bounds \eqref{eq:RN-domination-BHT},
\eqref{eq:iterated-Holder-type-bounds:linear}, and
\eqref{eq:iterated-Holder-type-bounds:bilinear}. These bounds are proven for
general functions \(F_j \in L^{\infty} (\R^{3}_{+}) \otimes
\Phi^{\infty}_{\gr}\) in place of \(\Emb [f_j] (\eta, y, t)\) specifically. The
results are generally obtained using interpolation using geometric
considerations on the support of the functions \(F_j\) and on properties of the
outer Lebesgue norms in play. The single tree estimates of
\Cref{sec:single-tree} serve as the \(L^{\infty}\) endpoint when proving bounds
\eqref{eq:iterated-Holder-type-bounds:linear}, and
\eqref{eq:iterated-Holder-type-bounds:bilinear}. Bound
\eqref{eq:RN-domination-BHT}, on the other hand relies on the correctness of
\Cref{def:integral-sizes} of the size \(\SI _{\Theta}\).

We begin this section by showing in \Cref{prop:TF-RN} that
\eqref{eq:RN-domination-BHT} holds. The role of this estimate allows us to pass
from trying to bound the trilinear form on {\LHS{\eqref{eq:BHT-dual-bounds}}}
directly to bounding appropriate outer Lebesgue norms of the integrand
\[ \Emb [f_1] (\eta, y, t) [\phi_0] \prod_{j = 2}^3 \Emb [f_j] \circ \Gamma_j
   (\eta, y, t) [\phi_0] . \]
Subsequently, our goal is to prove bounds
\eqref{eq:iterated-Holder-type-bounds:linear} and
\eqref{eq:iterated-Holder-type-bounds:bilinear} that we do in
\Cref{prop:holder-bound-trilinear}. This relies first on proving a
non-iterated version of these bounds in \Cref{lem:forest-estimate}. This is
just a straight-forward application of the outer Hölder inequality. Then we
deducing an atomic estimate for iterated outer Lebesgue spaces in
\Cref{lem:atomic}. This is the more involved procedure and involves. Then
\Cref{prop:holder-bound-trilinear} follows by interpolation as a corollary of
\Cref{lem:atomic}. \

\begin{proposition}
  \label{prop:TF-RN}Given any \(\phi_0 \in \Phi^N_{\gr}\), the bound
  \begin{equation}
    \left| \int_{\R^3_+} H (\eta, y, t) [\phi_0, \phi_0, \phi_0] \dd \eta \dd
    y \dd t \right| \lesssim \| \phi_0 \|_{\Phi^N_{\gr}}^3 \| H (\eta, y, t)
    \|_{L^1_{\nu_1} \fL^1_{\mu^1_{\Theta}, \nu_1} \SI _{\Theta} \Phi_{\gr}^N}
    \label{eq:TF-RN}
  \end{equation}
  holds for any compactly supported \(H \in L^{\infty} (\R^{3}_{+})
  \otimes^3 {\Phi^{\infty}_{\gr}}'\). As a consequence, the bound
  \eqref{eq:RN-domination-BHT} holds for all functions \(f_j \in \Sch (\R), j \in \{ 1, 2, 3 \}\).
\end{proposition}

\begin{proof}
  Bound \eqref{eq:RN-domination-BHT} follows from \eqref{eq:TF-RN} by applying
  it to \(H = \1_{\mathbb{K}_n} \prod_{j = 1}^3 \Emb [f_j] \circ \Gamma_j\)
  where \(\mathbb{K}_n \subset \R^3_+\) are compact subsets as specified in
  Remark \Cref{rmk:cpt-sets}. Convergence follows by
  \Cref{prop:wave-packet-represetion}: the integral on the
  {\LHS{\eqref{eq:RN-domination-BHT}}} is absolutely convergent.
  
  The bound \eqref{eq:TF-RN} follows by a two-fold application of
  \Cref{prop:outer-RN}. At the first level we obtain that
  \begin{equation}
    \left| \int  H (\eta, y, t) [\phi_0, \phi_0, \phi_0] \dd \eta \dd y \dd t
    \right| \lesssim \| H (\eta, y, t) \|_{L^1_{\mu^1_{\Theta}} \SI _{\Theta}
    \Phi_{\gr}^N} \label{eq:RN-domination-BHT:non-iter}
  \end{equation}
  as long as for any fixed tree \(T \in \TT_{\Theta}\) one has
  \[ \left| \int_T H (\eta, y, t) [\phi_0, \phi_0, \phi_0] \dd \eta \dd y \dd
     t \right| \lesssim_N \mu_{\Theta}^1 (T) \| H (\eta, y, t) \|_{\SI
     _{\Theta} \Phi_{\gr}^N} . \]
  To show the latter bound it is sufficient to decompose each \(\phi_j =
  \phi_0\) into \(\phi_j = \phi_{j, \theta}^o + \phi_{j, \theta}^l\) as described
  in \eqref{eq:lac-ov-wp-decomp} to obtain that \
  \[ H (\eta, y, t) [\phi_0, \phi_0, \phi_0] = \sum_{(\star_1, \star_2,
     \star_3) \in \{ l, o \}^3} H (\eta, y, t) [\phi_{1, \theta}^{\star_1},
     \phi_{2, \theta}^{\star_2}, \phi_{3, \theta}^{\star_3}] . \]
  The terms with \((\star_1, \star_2, \star_3) = (o, o, l)\) and \((\star_1,
  \star_2, \star_3) = (o, l, o)\) vanish because if \(\FT{\phi_0} 
  (\theta_{\Gamma_1}) \neq 0\) then \(\phi_{j, \theta}^o (z) = 0\) for \(j \in \{
  2, 3 \}\). This is because if \(\theta_{\Gamma_1} = \theta \in B_{\gr}\) then
  \(\theta_{\Gamma_2} \in \alpha_2 \beta_2 \left( B_{\gr} + \gamma_2 \right) =
  B_{\gr} - 1\) and \(\theta_{\Gamma_3} \in \alpha_3 \beta_3 \left( B_{\gr} +
  \gamma_3 \right) \in B_{(1 + \beta) \gr} + 1\); in particular
  \(\theta_{\Gamma_2} \notin B_{\gr}\) and \(\theta_{\Gamma_3} \notin B_{\gr}\).
  The claim then follows by directly integrating each term and recalling
  \Cref{def:integral-sizes} of \(\SI _{\Theta} \Phi_{\gr}^N\).
  
  From \eqref{eq:RN-domination-BHT:non-iter} it follows that that for any \(V^+
  \in \DD^{\cup}\) it holds that
  \[ \begin{aligned}[t]&
       \left| \int  \1_{V^+} H (\eta, y, t) [\phi_0, \phi_0, \phi_0] \dd \eta
       \dd y \dd t \right| \lesssim \left\| \1_{V^+} H (\eta, y, t)
       \right\|_{L^1_{\mu^1_{\Theta}, \nu_1} \SI _{\Theta} \Phi_{\gr}^N}\\ &
       \lesssim \nu_1 (V^+) \left\| \1_{V^+} H (\eta, y, t)
       \right\|_{\fL^1_{\mu^1_{\Theta}, \nu_1} \SI _{\Theta} \Phi_{\gr}^N} .
     \end{aligned} \]
  This allows us to apply \Cref{prop:outer-RN} again and obtain bound
  \eqref{eq:TF-RN}.
\end{proof}

\begin{lemma}
  \label{lem:forest-estimate} The bound
  \[ \begin{aligned}[t]&
       \| F_1 F_2 F_3 \|_{L^q_{\mu^1_{\Theta}} \left( \SI^{(l, l, l)}_{\Theta}
       \Phi_{\gr}^N + \SI^{(o, l, l)}_{\Theta} \Phi_{\gr}^N + \SI^{(l, l,
       o)}_{\Theta} \Phi_{\gr}^N + \SI^{(l, o, l)}_{\Theta} \Phi_{\gr}^N
       \right)} \qquad\\ &
       \qquad \lesssim \| F_1 \|_{L^{q_1}_{\mu^1_{\Theta}} \SF_{\Theta}^{u_1}
       \Phi_{4 \gr}^{N - 3}} \| F_2 \|_{L^{q_2}_{\mu^1_{\Theta}}
       \widetilde{\SF}_{\Gamma_2}^{u_2} \Phi_{4 \gr}^{N - 3}} \| F_3
       \|_{L^{q_3}_{\mu^1_{\Theta}} \widetilde{\SF}_{\Gamma_3}^{u_3} \Phi_{4
       \gr}^{N - 3}}
     \end{aligned} \]
  holds for any three functions \(F_1, F_2, F_3 \in L^{\infty}_{\tmop{loc}}
  (\R^{3}_{+}) {\otimes \Phi_{4 \gr}^{\infty}}'\) as long as \((q_1,
  q_2, q_3) \in [1, \infty]^3\) with \(q^{- 1} = \sum_{j = 1}^3 q_j^{- 1}\) and
  as long as \((u_1, u_2, u_3) \in [1, \infty]^3\) are such that \(\sum_{j = 1}^3
  u_j^{- 1} \leq 1\). Similarly, it holds that
  \[ \| F_1 F_2 F_3 \|_{L^q_{\mu^1_{\Theta}} \SI^{(l, o, o)}_{\Theta}
     \Phi_{\gr}^N} \lesssim \| F_1 \|_{L^{q_1}_{\mu^1_{\Theta}}
     \SF_{\Theta}^{u_1} \Phi_{4 \gr}^{N - 3}} \| F_2 F_3
     \|_{\widetilde{\SF}_{\Gamma_{\times}}^{u_{\times}} (\Phi \otimes \Phi)_{4
     \gr}^{N - 3}} \]
  as long as \((q_1, q_{\times}) \in [1, \infty]^2\) with \(q^{- 1} = q_1^{- 1}
  + q_{\times}^{- 1}\) and as long as \((u_1, u_{\times}) \in [1, \infty]^2\) are
  such that \(u_1^{- 1} + u_{\times}^{- 1} \leq 1\). The implicit constants are
  independent of \(F_1, F_2, F_3\).
\end{lemma}

\begin{proof}
  The claim follows from \Cref{thm:STE} by the outer Hölder inequality
  (\Cref{cor:outer-holder-classical}). To apply the outer Hölder notice that
  the bound \eqref{eq:abstract-size-holder} corresponds to the bounds
  \eqref{eq:STE:l-l-o} and \eqref{eq:STE:l-o-o} of \Cref{thm:STE}. Bound
  \eqref{eq:abstract-measure-holder}, on the other hand, follows directly from
  the fact that \(F_1 F_2 F_3\) is a point-wise product: if \(W_j \in
  \mathbb{T}^{\cup}_{\Theta}\) are such that
  \[ \1_{\R^3_+ \setminus W_j} F_j = 0 \]
  then
  \[ \1_{\R^3_+ \setminus W_{\cap}} F_1 F_2 F_3 = 0 \]
  where \(W_{\cap} = W_{j_0}\) for \(j_0\) such that \(\mu^1_{\Theta} (W_{j_0}) =
  \min_{j \in \{ 1, 2, 3 \}} (\mu^1_{\Theta} (W_{j }))\).
\end{proof}

\begin{lemma}
  \label{lem:atomic} Let \((q_1, q_2, q_3) \in [1, \infty)^3\) and let
  \[ q_{\times} \eqd (q_2^{- 1} + q_3^{- 1})^{- 1}, \qquad q \eqd (q_1^{- 1} +
     q_{\times}^{- 1})^{- 1} = \Big( \sum_{j = 1}^3 q_j^{- 1} \Big)^{- 1} .
  \]
  Then the bound
  \begin{equation}
    \begin{aligned}[t]&
      \| F_1 F_2 F_3 \|_{L^p_{\mu_{\Theta}^1} \left( \SI^{(l, l, l)}_{\Theta}
      \Phi_{\gr}^N + \SI^{(o, l, l)}_{\Theta} \Phi_{\gr}^N + \SI^{(l, l,
      o)}_{\Theta} \Phi_{\gr}^N + \SI^{(l, o, l)}_{\Theta} \Phi_{\gr}^N
      \right)}\\ &
      \qquad \lesssim \begin{aligned}[t]&
        \nu_1  ( \| F_1 \|_{\SF_{\Theta}^{u_1} \Phi_{4 \gr}^{N - 3}} > 0 )^{\frac{1}{p_1}}
        \min_{j \in \{ 2, 3 \}} \Big( \nu_{\beta} 
        ( \| {F_j} \|_{\widetilde{\SF}_{\Gamma_j}^{u_j}\Phi_{4 \gr}^{N - 3}} > 0 \Big) \Big)^{\frac{1}{p_{\times}}}
        \\ & \qquad
        \times \| F_1 \|_{\fL^{q_1, +}_{\mu_{\Theta}^{\infty}, \nu_1} \SF_{\Theta}^{u_1} \Phi_{4 \gr}^{N - 3}}
        \Big( \prod_{j = 2}^3 \| F_j
        \|_{X_{\mu_{\Theta}^1, \nu_{\beta}}^{q_j, r_j, +}
        \widetilde{\SF}_{\Gamma_j}^{u_j} \Phi_{4 \gr}^{N - 3}} \Big)
      \end{aligned}
    \end{aligned} \label{eq:atomic-trilinear}
  \end{equation}
  holds for any \(F_j \in L^{\infty}_{\tmop{loc}} (\R^{3}_{+}) \otimes
  \Phi_{4 \gr}^{\infty}\) as long as \((p_1, p_{\times}, p) \in {(q, +
  \infty]^3} \) with \(p^{- 1} = p_1^{- 1} + p_{\times}^{- 1}\), as long as \(r_j
  \in (0, q_j]\), \(j \in \{ 2, 3 \}\), are such that \(r_2^{- 1} + r_3^{- 1} \geq
  p_{\times}^{- 1}\), and as long as \((u_1, u_2, u_3) \in [1, \infty]^3\)
  satisfy \(\sum_{j = 1}^3 u_j^{- 1} \leq 1\).
  
  Similarly, let \((q_1, q_{\times}) \in [1, \infty)^2\) and let \(q \eqd (q_1^{-
  1} + q_{\times}^{- 1})^{- 1}\). The bound
  \begin{equation}
    \begin{aligned}[t]&
      \| F_1 F_2 F_3 \|_{L^p_{\mu_{\Theta}^1} \SI^{(l, o, o)}_{\Theta}
      \Phi^N_{\gr}} \lesssim \begin{aligned}[t]&
    \nu_1  ( \| F_1 \|_{\SF_{\Gamma_1}^{u_1} \Phi_{4 \gr}^{(N - 3)}} > 0 )^{\frac{1}{p_1}}
    \nu_{\beta} ( \| F_2 F_3 \|_{\widetilde{\SF}^{u_{\times}}_{\Gamma_2 \times \Gamma_3} \Phi_{4 \gr}^{N - 3}} > 0)^{\frac{1}{p_{\times}}}\\ &
        \qquad \times \| F_1 \|_{\fL^{q_1}_{\mu_{\Theta}^{\infty}, \nu_1}
          \SF_{\Theta}^{u_1} \Phi_{4 \gr}^{N - 3}}
        \| F_2 F_3 \|_{X_{\mu_{\Theta}^1, \nu_{\beta }}^{q_{\times}, r_{\times}}
        \widetilde{\SF}^{u_{\times}}_{\Gamma_2 \times \Gamma_3} \Phi_{4  \gr}^{N - 3}}
      \end{aligned}
    \end{aligned} \label{eq:atomic-bilinear}
\end{equation}
  holds for any \(F_j \in L^{\infty}_{\tmop{loc}} (\R^{3}_{+}) \otimes
  \Phi_{4 \gr}^{\infty}\) as long as \((p_1, p_{\times}, p) \in {(q, +
  \infty]^3} \) with \(p^{- 1} = p_1^{- 1} + p_{\times}^{- 1}\), as long as
  \(r_{\times} \in (0, q_{\times}]\), is such that \(r_{\times} \leq p_{\times}\),
  and as long as \((u_1, u_{\times}) \in [1, \infty]^3\) satisfy \(u_1^{- 1} +
  u_{\times}^{- 1} \leq 1\). 
\end{lemma}

\begin{proof}{Proof of \Cref{lem:atomic}}
  We first prove bound \eqref{eq:atomic-trilinear}. The condition \(r_j \leq
  q_j\) for \(j \in \{ 2, 3 \}\) is required for the sizes \(X_{\mu_{\Theta}^1,
  \nu_{\beta}}^{q_j, r_j, +} \widetilde{\SF}_{\Gamma_j}^{u_j} \Phi_{4 \gr}^{N
  - 3}\) to even be defined. We may assume that all terms on
  {\RHS{\eqref{eq:atomic-trilinear}}} are finite; otherwise there is nothing
  to prove. Let \(V_1 \in \mathbb{D}^{\cup}_1\) and \(V_{\times} \in
  \DD^{\cup}_{\beta }\) be such that \(\1_{\R^3_+ \setminus V_1} F_1 = 0\),
  \(\1_{\R^3_+ \setminus V_{\times}} F_2 F_3 = 0\), and
  \[ \begin{aligned}[t]&
  \nu_1 (V_1) \leq 2 \nu_1  ( \| F_1\|_{\SF_{\Theta}^{u_1} \Phi_{4 \gr}^{N - 3}} > 0 ),
  \\ &
  \nu_{\beta } (V_{\times}) \leq 2 \min_{j \in \{ 2, 3 \}} \nu_{\beta } 
       ( \| F_j \|_{\widetilde{\SF}_{\Gamma_j}^{u_j} \Phi_{4 \gr}^{N -3}} > 0 ) .
     \end{aligned} \]
  Let us first reason under the assumption that \(q < p_{\times} \leq
  q_{\times}\). Since \(r_2^{- 1} + r_3^{- 1} \geq p_{\times}^{- 1} \geq
  q_{\times}^{- 1}\) let us choose \(p_2 \in [r_2, q_2]\) and \(p_3 \in [r_3,
  q_2]\) such that \(p_{\times}^{- 1} = p_2^{- 1} + p_3^{- 1}\).
  
  By homogeneity, we may assume that \(\| F_1 \|_{\fL^{q_1,
  +}_{\mu_{\Theta}^{\infty}} \SF_{\Theta}^{u_1} \Phi_{4 \gr}^{N - 3}} = 1\) and
  \\\(\| F_j \|_{X_{\mu_{\Theta}^1, \nu_{\beta}}^{q_j, r_j, +}
  \widetilde{\SF}_{\Gamma_j}^{u_j} \Phi_{4 \gr}^{N - 3}} = 1\) for \(j \in \{ 2,
  3 \}\). Let us enact the atomic decomposition
  (\Cref{prop:atomic-decomposition}) of \(F_1\) to obtain a sequence of sets
  \(W_k \in \TT_{\Theta}^{\cup}\) so that
  \[ \begin{aligned}[t]&
       \left\| \1_{\Delta W_k} F_1 \right\|_{\SF_{\Theta}^{u_1} \Phi_{4
       \gr}^{N - 3}} \lesssim 2^{\frac{k}{q_1}}, \qquad \mu_{\Theta}^{\infty}
       (\Delta W_k) \lesssim 2^{- k}, \qquad \Delta W_k = W_{k - 1} \setminus
       W_k .
     \end{aligned} \]
  Since \(\| F_1 \|_{\SF_{\Theta}^{u_1} \Phi_{4 \gr}^{N - 3}} = \| F_1
  \|_{L^{\infty}_{\mu^{\infty}_{\Theta}} \SF_{\Theta}^{u_1} \Phi_{4 \gr}^{N -
  3}} \lesssim \| F_1 \|_{\fL^{q_1, +}_{\mu_{\Theta}^{\infty}, \nu_1}
  \SF_{\Theta}^{u_1} \Phi_{4 \gr}^{N - 3}}\) it holds that \(W_k = \emptyset\)
  for \(k \geq 1\). As a consequence we will only be concerned with \(k \leq 1\).
  Turning our attention to \(F_j\), according to the definition of the size \(\|
  \cdot \|_{X_{\mu_{\Theta}^1, \nu_{\beta}}^{q_j, r_j, +}
  \widetilde{\SF}_{\Gamma_j}^{u_j} \Phi_{4 \gr}^{N - 3}}\) (given in
  \eqref{eq:X-local-size} and \eqref{eq:X-size}) we have that
  \[ \left\| \1_{V_{\times}} \1_{\Delta W_k} F_j
     \right\|_{L^{\bar{r}_j}_{\mu^1_{\Theta}} \widetilde{\SF}_{\Gamma_j}^{u_j}
     \Phi_{4 \gr}^{N - 3}} \lesssim \nu_{\beta } (V_{\times})^{\frac{1}{p_j}}
     \mu^{\infty}_{\Theta_j} (\Delta W_k)^{\frac{1}{p_j} - \frac{1}{q_j}} \]
  since \(p_j \in [r_j, q_j]\). Using \Cref{lem:forest-estimate} we get that
  \[ \begin{aligned}[t]&
       \Big\| \1_{V_1} \1_{V_{\times}} \1_{\Delta W_k} \prod_{j = 1}^3 F_j
       \Big\| {_{L^{\bar{r}}_{\mu_{\Theta}^1} \left( \SI^{(l, l,
       l)}_{\Theta} \Phi_{\gr}^N + \SI^{(o, l, l)}_{\Theta} \Phi_{\gr}^N +
       \SI^{(l, l, o)}_{\Theta} \Phi_{\gr}^N + \SI^{(l, o, l)}_{\Theta}
       \Phi_{\gr}^N \right)}} \\ &
       \lesssim \left\| \1_{V_1} \1_{\Delta W_k} F_1
       \right\|_{L^{p_1}_{\mu^1_{\Theta}} \SF_{\Theta}^{u_1} \Phi_{4 \gr}^{N -
       3}} \prod_{j = 1}^3 \left\| \1_{V_{\times}} \1_{\Delta W_k} F_j
       \right\|_{L^{p_j}_{\mu^1_{\Theta}} \widetilde{\SF}_{\Gamma_j}^{u_j}
       \Phi_{4 \gr}^{N - 3}}\\ &
       \lesssim \nu_1 (V_1)^{\frac{1}{p_1}} \mu_{\Theta}^{\infty} (\Delta
       W_k)^{\frac{1}{p_1}} \left\| \1_{\Delta W_k} F_1
       \right\|_{\SF_{\Theta}^{u_1} \Phi_{4 \gr}^{N - 3}} \prod_{j = 2}^3
       \nu_1 (V_{\times})^{\frac{1}{\bar{r}_j}} \mu^{\infty}_{\Theta_j}
       (\Delta W_{k + 1})^{\frac{1}{p_j} - \frac{1}{q_j}}\\ &
       \lesssim \nu_1 (V_1)^{\frac{1}{p_1}} \nu_{\beta}
       (V_{\times})^{\frac{1}{p_2} + \frac{1}{p_3}} 2^{k \sum_{j = 1}^3
       \frac{1}{q_j} - \frac{1}{p_j}} = \nu_1 (V_1)^{\frac{1}{p_1}}
       \nu_{\beta} (V_{\times})^{\frac{1}{p_{\times}}} 2^{k \left( \frac{1}{q}
       - \frac{1}{p} \right)} .
     \end{aligned} \]
  Using the triangle inequality \eqref{eq:multi-triangle} we get that
  \[ \begin{aligned}[t]&
       {\Big\| \prod_{j = 1}^3 F_j \Big\|_{L^p_{\mu_{\Theta}^1} \left(
       \SI^{(l, l, l)}_{\Theta} \Phi_{\gr}^N + \SI^{(o, l, l)}_{\Theta}
       \Phi_{\gr}^N + \SI^{(l, l, o)}_{\Theta} \Phi_{\gr}^N + \SI^{(l, o,
       l)}_{\Theta} \Phi_{\gr}^N \right)}} \\ &
       \qquad \lesssim \nu_1 (V_1)^{\frac{1}{p_1}} \nu_{\beta}
       (V_{\times})^{\frac{1}{p_{\times}}} \sum_{k = - \infty}^1 k^{C_{\|
       \cdot \|}} 2^{k \left( \frac{1}{q} - \frac{1}{p} \right)} \lesssim
       \nu_1 (V_1)^{\frac{1}{p_1}} \nu_{\beta}
       (V_{\times})^{\frac{1}{p_{\times}}},
     \end{aligned} \]
  as required. The sum is convergent because \(\frac{1}{q} > \frac{1}{p}\).
  
  Now let us reason under the assumption that \(q_{\times} \leq p_{\times}\).
  From the above proof we have obtained that
  \[ \begin{aligned}[t]&
       \| F_1 F_2 F_3 \|_{L^p_{\mu_{\Theta}^1} \left( \SI^{(l, l, l)}_{\Theta}
       \Phi_{\gr}^N + \SI^{(o, l, l)}_{\Theta} \Phi_{\gr}^N + \SI^{(l, l,
       o)}_{\Theta} \Phi_{\gr}^N + \SI^{(l, o, l)}_{\Theta} \Phi_{\gr}^N
       \right)}\\ &
       \qquad = \left\| \1_{V_1 \cap V_{\times}} F_1 F_2 F_3
       \right\|_{L^p_{\mu_{\Theta}^1} \left( \SI^{(l, l, l)}_{\Theta}
       \Phi_{\gr}^N + \SI^{(o, l, l)}_{\Theta} \Phi_{\gr}^N + \SI^{(l, l,
       o)}_{\Theta} \Phi_{\gr}^N + \SI^{(l, o, l)}_{\Theta} \Phi_{\gr}^N
       \right)}\\ &
       \qquad \lesssim \nu_1 (V_1)^{\frac{1}{p_1}} \nu_{\beta}
       (V_{\times})^{\frac{1}{p_{\times}}} \| F_1 \|_{\fL^{q_1,
       +}_{\mu_{\Theta}^{\infty}, \nu_1} \SF_{\Theta}^{u_1} \Phi_{4 \gr}^{N -
       3}} \left( \prod_{j = 2}^3 \| F_j \|_{X_{\mu_{\Theta}^1,
       \nu_{\beta}}^{q_j, q_j, +} \widetilde{\SF}_{\Gamma_j}^{u_j} \Phi_{4
       \gr}^{N - 3}} \right)
     \end{aligned} \]
  for any \(p_1 \in (0, + \infty]\) such that \(p^{- 1} = p_1^{- 1} +
  q_{\times}^{- 1} < q^{- 1}\). We may assume that \(V_{\times} \subseteq V_1\)
  by carrying out the proof with \(\tilde{V}_{\times} = V_{\times} \cap V_1 \in
  \DD_{\beta}^{\cup}\) in place of \(V_{\times}\). It holds that \(\nu_{\beta}
  (V_1) \leq \nu_1 (V_1)\). Thus we obtain that for any \(p_{\times} >
  q_{\times}\) it holds that
  \[ \begin{aligned}[t]&
       \| F_1 F_2 F_3 \|_{L^p_{\mu_{\Theta}^1} \left( \SI^{(l, l, l)}_{\Theta}
       \Phi_{\gr}^N + \SI^{(o, l, l)}_{\Theta} \Phi_{\gr}^N + \SI^{(l, l,
       o)}_{\Theta} \Phi_{\gr}^N + \SI^{(l, o, l)}_{\Theta} \Phi_{\gr}^N
       \right)}\\ &
       \qquad \lesssim \nu_1 (V_1)^{\frac{1}{p}} \left( \frac{\nu_{\beta}
       (V_{\times} \cap V_1)}{\nu_{\beta} (V_1)}
       \right)^{\frac{1}{q_{\times}}} \| F_1 \|_{\fL^{q_1,
       +}_{\mu_{\Theta}^{\infty}, \nu_1} \SF_{\Theta}^{u_1} \Phi_{4 \gr}^{N -
       3}} \left( \prod_{j = 2}^3 \| F_j \|_{X_{\mu_{\Theta}^1,
       \nu_{\beta}}^{q_j, q_j, +} \widetilde{\SF}_{\Gamma_j}^{u_j} \Phi_{4
       \gr}^{N - 3}} \right)\\ &
       \hspace{2.0em} \lesssim \nu_1 (V_1)^{\frac{1}{p}} \left(
       \frac{\nu_{\beta} (V_{\times} \cap V_1)}{\nu_{\beta} (V_1)}
       \right)^{\frac{1}{p_{\times}}} \| F_1 \|_{\fL^{q_1,
       +}_{\mu_{\Theta}^{\infty}, \nu_1} \SF_{\Theta}^{u_1} \Phi_{4 \gr}^{N -
       3}} \left( \prod_{j = 2}^3 \| F_j \|_{X_{\mu_{\Theta}^1,
       \nu_{\beta}}^{q_j, q_j, +} \widetilde{\SF}_{\Gamma_j}^{u_j} \Phi_{4
       \gr}^{N - 3}} \right)\\ &
       \hspace{2.0em} \lesssim \nu_1 (V_1)^{\frac{1}{p} -
       \frac{1}{p_{\times}}} \nu_{\beta} (V_{\times} \cap
       V_1)^{\frac{1}{p_{\times}}} \| F_1 \|_{\fL^{q_1,
       +}_{\mu_{\Theta}^{\infty}, \nu_1} \SF_{\Theta}^{u_1} \Phi_{4 \gr}^{N -
       3}} \left( \prod_{j = 2}^3 \| F_j \|_{X_{\mu_{\Theta}^1,
       \nu_{\beta}}^{q_j, q_j, +} \widetilde{\SF}_{\Gamma_j}^{u_j} \Phi_{4
       \gr}^{N - 3}} \right) .
     \end{aligned} \]
  Thus we have shown the claim for any \(p_{\times} > q_{\times}\) and any \(p_1
  \in (q, + \infty]\) with the cited restrictions.
  
  Bound \eqref{eq:atomic-bilinear} is proved analogously with the following
  natural modifications. We require \(r_{\times} \leq q_{\times}\) for the size
  \(\widetilde{\SF}^{u_{\times}}_{\Gamma_2 \times \Gamma_3} \Phi_{4 \mf{r}}^{N
  - 3}\) to be defined. By homogeneity we assume that \(\| F_2 F_3
  \|_{X_{\mu_{\Theta}^1, \nu_{\beta}}^{q_{\times}, r_{\times}, +}
  \widetilde{\SF}^{u_{\times}}_{\Gamma_2 \times \Gamma_3} \Phi_{4 \mf{r}}^{N -
  3}} = 1\). Assuming \(p_{\times} \leq q_{\times}\) and taking the sets \(\Delta
  W_k\) to be the same as before, we obtain that
  \[ \left\| \1_{V_{\times}} \1_{\Delta W_k} F_2 F_3
     \right\|_{L^{p_{\times}}_{\mu^1_{\Theta}} \widetilde{\SF}_j^{u_j}}
     \lesssim \nu_{\beta } (V_{\times})^{\frac{1}{p_{\times}}}
     \mu^{\infty}_{\Theta_j} (\Delta W_k)^{\frac{1}{p_{\times}} -
     \frac{1}{q_{\times}}} \]
  and we conclude that
  \[ {\left\| \1_{V_1} \1_{V_{\times}} \1_{\Delta W_k} F_1 F_2 {F_3}_j
     \right\|_{L^p_{\mu_{\Theta}^1} \SI^{(l, o, o)}_{\Theta} \Phi^N_{\gr}}} 
     \lesssim \nu_1 (V_1)^{\frac{1}{p_1}} \nu_{\beta}
     (V_{\times})^{\frac{1}{p_{\times}}} 2^{k \left( \frac{1}{q} - \frac{1}{p}
     \right)} . \]
  Summing up in \(k \in \Z, k \leq 1\) yields the claim. To address the case
  \(p_{\times} > q_{\times}\) the procedure is the same as above and relies on
  noticing that one may assume \(\nu_{\beta} (V_{\times}) \leq \nu_{\beta}
  (V_1) \leq \nu_1 (V_1)\).
\end{proof}

\begin{proposition}
  \label{prop:holder-bound-trilinear}Let \((q_1, q_2, q_3) \in [1, \infty)^3\)
  and let
  \[ q_{\times} \eqd (q_2^{- 1} + q_3^{- 1})^{- 1}, \qquad q \eqd (q_1^{- 1} +
     q_{\times}^{- 1})^{- 1} = \Big( \sum_{j = 1}^3 q_j^{- 1} \Big)^{- 1},
  \]
  with \(q < 1\). Then the bounds
  \begin{equation}
    \begin{aligned}[t]&
      \| F_1 F_2 F_3 \|_{L^1_{\mu_{\Theta}^1} \left( \SI^{(l, l, l)}_{\Theta}
      \Phi_{\gr}^N + \SI^{(o, l, l)}_{\Theta} \Phi_{\gr}^N + \SI^{(l, l,
      o)}_{\Theta} \Phi_{\gr}^N + \SI^{(l, o, l)}_{\Theta} \Phi_{\gr}^N
      \right)} \qquad\\ &
      \qquad \lesssim \| F_1 \|_{L^{p_1}_{\nu_1} \fL^{q_1,
      +}_{\mu_{\Theta}^{\infty}, \nu_1} \SF_{\Gamma_1}^{u_1} \Phi_{4 \gr}^{N -
      3}} \Big( \prod_{j = 2}^3 \| F_j \|_{L^{p_j}_{\nu_{\beta }}
      X_{\mu_{\Theta}^1, \nu_{\beta }}^{q_j, r_j, +}
      \widetilde{\SF}_{\Gamma_j}^{u_j} \Phi_{4 \gr}^{N - 3}} \Big)
    \end{aligned} \label{eq:holder-bound-trilinear}
  \end{equation}
  holds for any \(F_j \in L^{\infty}_{\tmop{loc}} (\R^{3}_{+}) \otimes
  \Phi_{4 \gr}^{\infty}\) as long as \((p_1, p_2, p_3) \in [1, \infty]^3\) with
  \(\sum_{j = 1}^3 p_j^{- 1} = 1\), as long as \((u_1, u_2, u_3) \in [1,
  \infty]^3\) with \(\sum_{j = 1}^3 u_j^{- 1} \leq 1\), as long as \(r_j \in (0,
  q_j]\), \(j \in \{ 2, 3 \}\), are such that \(r_2^{- 1} + r_3^{- 1} \geq p_2^{-
  1} + p_3^{- 1}\).
  
  In particular, bound \eqref{eq:iterated-Holder-type-bounds:linear} holds
  under the above assumptions on the exponents.
  
  Similarly, the bound
  \begin{equation}
    \begin{aligned}[t]&
      \| F_1 F_2 F_3 \|_{L^1_{\mu_{\Theta}^1} \SI^{(l, o, o)}_{\Theta}
      \Phi^N_{\gr}} \lesssim \| F_1 \|_{L^{p_1}_{\nu_1} \fL^{q_1,
      +}_{\mu_{\Theta}^{\infty}, \nu_1} \SF_{\Theta}^{u_1} \Phi_{4 \mf{r}}^{N
      - 3}} \| F_2 F_3 \|_{L^{p_{\times}}_{\nu_{\beta }} X_{\mu_{\Theta}^1,
      \nu_{\beta }}^{q_{\times}, r_{\times}, +}
      \widetilde{\SF}^{u_{\times}}_{\Gamma_2 \times \Gamma_3} \Phi_{4
      \mf{r}}^{N - 3}}
    \end{aligned} \label{eq:holder-bound-bilinear}
  \end{equation}
  holds for any \(F_j \in L^{\infty}_{\tmop{loc}} (\R^{3}_{+}) \otimes
  \Phi_{4 \gr}^{\infty}\) as long as \((p_1, p_{\times}) \in [1, \infty]^2\) with
  \(p^{- 1} \eqd p_1^{- 1} + p_{\times}^{- 1}\), as long as \(r_{\times} \in [1,
  \infty]\) and \((q_1, q_{\times}) \in [1, \infty]^2\) are such that
  \[ r_{\times} < p_{\times} < q_{\times}, \]
  and
  \[ q_1^{- 1} + q_{\times}^{- 1} > 1, \]
  and as long as \((u_1, u_{\times}) \in [1, \infty]^2\) with \(u_1^{- 1} +
  u_{\times}^{- 1} \leq 1\).
  
  In particular, bound \eqref{eq:iterated-Holder-type-bounds:bilinear} holds
  under the above assumptions on the exponents. 
\end{proposition}

\begin{proof}
  \Cref{prop:outer-restricted-interpolation} can be applied to the bounds
  \eqref{eq:atomic-trilinear} to yield bound
  \eqref{eq:holder-bound-trilinear}. As a matter of fact the conditions of
  \Cref{lem:atomic} are open, the condition \(q < 1\) and \(p_1^{- 1} + p_2^{- 1}
  + p_3^{- 1} = 1\) imply that \(p_{\times}^{- 1} \eqd p_2^{- 1} + p_3^{- 1}\)
  satisfies \(p_{\times}^{- 1} < q^{- 1}\).
  
  Similarly \Cref{prop:outer-restricted-interpolation} can be applied to the
  bounds \eqref{eq:atomic-bilinear} to obtain bounds
  \eqref{eq:holder-bound-bilinear}. 
\end{proof}

\section{Non-uniform embedding bounds}\label{sec:non-uniform}

In this section we prove \Cref{thm:non-uniform-embedding-bounds}. We will
start by recording the following embedding bounds that have been shown in
{\cite{amentaBilinearHilbertTransform2020}}, the proof of which we will not
repeat here.

\begin{theorem}[{\cite[Theorem 1.3]{amentaBilinearHilbertTransform2020}}]
  \label{thm:non-uniform-embedding-bounds:worse} The bounds
  \[ \left\| \Emb [f] \right\|_{L^{p }_{\nu_1} \fL^{q, +}_{\mu_{\Theta}^1,
     \nu_1} \left( \SL^{\infty, \infty}_{\Theta} \Phi_{\gr}^N + \SL^{u,
     2}_{\Theta} \wpD_{\gr}^{N } + {\SL_{\Theta}^{(u, 1)}}  \dfct_{\gr}^N
     \right)} \lesssim \| f \|_{L^p (\R)} \]
  hold for any \(f \in \Sch (\R)\) as long as \(p \in (1, \infty]\),
  \(q \in (\max (p', 2), \infty)\), and \(u \in (1, \infty)\).
\end{theorem}

The statement {\cite[Theorem 1.3]{amentaBilinearHilbertTransform2020}} covers
more general embeddings for functions valued in Banach spaces. For our
purposes, it is sufficient to apply it for scalar-valued functions. That
corresponds to setting \(r = 2\) the statement of {\cite[Theorem
1.3]{amentaBilinearHilbertTransform2020}}. This corresponds to the case of
functions valued in Hilbert spaces.\quad

Bounds \eqref{eq:embedding-bounds:non-uniform-iterated} of
\Cref{thm:non-uniform-embedding-bounds} are an improvement over
\Cref{thm:non-uniform-embedding-bounds:worse} in two ways. First of all, the
size \(\SF^u_{\Theta} \Phi_{\gr}^N\) appearing in
\Cref{thm:non-uniform-embedding-bounds} contains the term \(\SJ_{\Theta}^u
\Phi_{\gr}^{N }\), absent in the bounds of
\Cref{thm:non-uniform-embedding-bounds:worse}. Next, we show that the
embedding holds with localized outer Lebesgue norm \(\fL^{q,
+}_{\mu_{\Theta}^{\infty}, \nu_1}\) in place of \(\fL^{q, +}_{\mu_{\Theta}^1,
\nu_1}\). According to \Cref{lem:measure-comparison-1-infty}, for \(\beta = 1\)
this is a strengthening of the result since we have that
\[ \nu (V)^{- \frac{1}{q}} \left\| \1_V F \right\|_{L^q_{\mu^1_{\Theta}} \SO}
   \lesssim \left\| \1_V F \right\|_{L^q_{\mu^{\infty}_{\Theta}} \SO} \]
for any \(F \in \Rad (\R^{3}_{+}) \otimes \Phi^{\infty}_{\gr}\) and, in
particular,
\[ \| F \|_{\fL^{q, +}_{\mu_{\Theta}^1, \nu_1} \SF^u_{\Theta} \Phi^N_{\gr}}
   \lesssim \left\| \1_V F \right\|_{L^{q, +}_{\mu^{\infty}_{\Theta}, \nu_1}
   \SF^u_{\Theta} \Phi^N_{\gr}} . \]
This strengthening costs us the \(p = \infty\) endpoint for the embedding.
\Cref{sec:non-uniform:SIO} is dedicated to showing the bound
\begin{equation}
  \left\| \Emb [f] \right\|_{L^p_{\nu_1} \fL^{q, +}_{\mu_{\Theta}^1, \nu_1}
  \SF_{\Theta}^u \Phi_{\mf{r}}^N} \lesssim \| f \|_{L^p (\R)}
  \label{eq:embedding-bounds:non-uniform-iterated:mu1}
\end{equation}
by appropriately controlling the \(\SJ_{\Theta}^u \Phi_{\gr}^{N }\) size using
the sizes appearing on the {\LHS{}} of the bound in
\Cref{thm:non-uniform-embedding-bounds:worse}.
\Cref{sec:non-unif:counting-function-improvement} is dedicated to deducing
bounds \eqref{eq:embedding-bounds:non-uniform-iterated} of
\Cref{thm:non-uniform-embedding-bounds} from
\eqref{eq:embedding-bounds:non-uniform-iterated:mu1} by showing that
super-level sets of \(\Emb [f]\) can be selected to satisfy appropriate bounds
in the \(\mu_{\Theta}^{\infty}\) outer measure.

\subsection{SIO truncation size control}\label{sec:non-uniform:SIO}

This section concentrates on proving the bound below..

\begin{proposition}
  \label{prop:SIO-size-domination}Given any sets \(V^+, V^- \in \DD_{1
  }^{\cup}\), and \(W^+, W^- \in \TT^{\cup}_{\Theta}\), the bounds
  \begin{equation}
    \begin{aligned}[t]&
      \left\| \1_{(V^+ \cap W^+) \setminus \left( V^- {\cup W^-}  \right)}
      \Emb [f] \right\|_{\SJ^u_{\Theta} \Phi_{\gr}^N}\\ &
      \qquad \lesssim \begin{aligned}[t]&
        \left\| \1_{(V^+ \cap W^+) \setminus \left( V^- {\cup W^-}  \right)}
        \Emb [f] \right\|_{\SL^{(2, 2)}_{\Theta} \wpD_{8 \gr}^{N - 8 }}\\ &
        \qquad + \left\| \1_{(V^+ \cap W^+) \setminus \left( V^- {\cup W^-} 
        \right)} \Emb [f] \right\|_{\SL^{(\infty, \infty)}_{\Theta} \Phi_{8
        \gr}^{N - 8}}
      \end{aligned}
    \end{aligned} \label{eq:SIO-bound:explicit}
  \end{equation}
  hold for all \(f \in \Sch (\R)\) as long as \(u \in [1, \infty)\).
  The implicit constant is independent of \(f\) and of the sets \(V^{\pm},
  W^{\pm}\).
\end{proposition}

This is sufficient to show
\eqref{eq:embedding-bounds:non-uniform-iterated:mu1}.

\begin{corollary}
  \label{cor:SIO-bound:mu1} The bounds
  \eqref{eq:embedding-bounds:non-uniform-iterated:mu1} hold for all \(f \in
  \Sch (\R)\) as long as \(p \in (1, \infty]\), \(q \in (\max (p',
  2), \infty)\), and \(u \in [1, \infty]\).
\end{corollary}

\begin{proof}
  According to Remark \Cref{rmk:iterated-outer-bounds-explicit},
  \Cref{thm:non-uniform-embedding-bounds:worse} provides us with a choice of
  \(V^- (\lambda) \in \DD^{\cup}_1\) and \(W_q (\lambda, \tau, V^+) \in
  \TT_{\Theta}^{\cup}\) satisfying \eqref{eq:iterated-outer-bounds-explicit} and
  \eqref{eq:iterated-outer-bounds-explicit-full} with
  \[ \begin{aligned}[t]&
       F = \Emb [f], \qquad A = \| f \|_{L^p}, \qquad \beta = 1,\\ &
       \SO = \left( \SL^{(\infty, \infty)}_{\Theta} \Phi_{8 \gr}^{N - 8} +
       \SL^{(u, 2)}_{\Theta} \wpD_{8 \gr}^{N - 8} + {\SL_{\Theta}^{(u, 1)}} 
       \dfct_{8 \gr}^{N - 8} \right) .
     \end{aligned} \]
  Bound \eqref{eq:SIO-bound:explicit} shows that this same choice also
  satisfies conditions \eqref{eq:iterated-outer-bounds-explicit} and
  \eqref{eq:iterated-outer-bounds-explicit-full} with \(\SO = \SF_{\Theta}^u
  \Phi_{\mf{r}}^N\), instead. Again, by Remark
  \Cref{rmk:iterated-outer-bounds-explicit}, this is exactly what is claimed
  by bound \eqref{eq:embedding-bounds:non-uniform-iterated:mu1}.
\end{proof}

\begin{proof}{Proof of \Cref{prop:SIO-size-domination}}
  After some preliminary geometric reductions, this proof is divided into two
  parts. First we reduce the claim of the proposition to the case \(u = 2\). We
  do so by assuming that bound \eqref{eq:SIO-bound:explicit} holds with \(u =
  2\) for any function \(f \in \Sch (\R)\) and any sets  \(V^+, V^-
  \in \DD_{1 }^{\cup}\), and \(W^+, W^- \in \TT^{\cup}_{\Theta}\), and then we
  deduce that bound \eqref{eq:SIO-bound:explicit} also holds for any \(u \in
  [1, + \infty)\).\quad We then conclude the proof by showing bound
  \eqref{eq:SIO-bound:explicit} holds with \(u = 2\).
  
  We start by introducing some notation. Let us fix \(T \in \TT_{\Theta}\) and
  use the shorthand
  \[ F^{\ast}_T (\theta, \zeta, \sigma) \eqd \left( \1_{(V^+ \cap W^+)
     \setminus \left( V^- {\cup W^-}  \right)} \Emb [f] \right) \circ \pi_T
     (\theta, \zeta, \sigma) . \]
  Let \(E^{\ast +} \eqd \pi_T^{- 1} \left( \mT_{\Theta} \cap V^+ \cap W^+
  \right)\) and \(E^{\ast -} \eqd \pi_T^{- 1} (V^- \cup W^-)\) and let \(E^{\ast}
  \eqd E^{\ast +} \setminus E^{\ast -}\). Let \(\mf{b}_{\pm} : \R^2 \rightarrow
  \R_+\) be two functions such that their graphs are the boundary of \(E^{\pm}\)
  and let \(\mf{b^{\ast}_{\pm}} : \Theta \times B_1 \rightarrow \R_+\) be two
  functions such that their graphs are the boundaries of \(E^{\ast \pm} \cap
  \Theta \times B_1 \times \R_+\). These functions are obtained from
  \Cref{lem:geometry-of-boundary}. Let
  \[ \begin{aligned}[t]&
       \widetilde{\mf{b}}^{\ast}_- (\zeta) = \max_{\theta \in \Theta} 
       \mf{b}^{\ast}_- (\zeta, \theta) \qquad \widetilde{\mf{b}}^{\ast}_+
       (\zeta) = \min_{\theta \in \Theta} \mf{b}^{\ast}_+ (\zeta, \theta) .
     \end{aligned} \]
  Let us begin by showing that for any \(T \in \TT_{\Theta}\) it holds that
  \begin{equation}
    \begin{aligned}[t]&
      \left\| \1_{(V^+ \cap W^+) \setminus \left( V^- {\cup W^-}  \right)}
      \Emb [f] \right\|_{\SJ^u_{\Theta} \Phi_{\gr}^N (T)}\\ &
      \qquad \lesssim \begin{aligned}[t]&
        \sup_{\phi} \Big\| \sup_{\sigma > \widetilde{\mf{b}}^{\ast}_-
        (\zeta)} \Big| \int_{\sigma}^{\widetilde{\mf{b}}^{\ast}_+ (\zeta)}
        \wpD_{\zeta} (\theta) F^{\ast} (\theta, \zeta, \rho) [\phi] \frac{\dd
        \rho}{\rho} \Big| \Big\|_{L^u_{\frac{\dd \theta \dd \zeta}{|
        \Theta |}} (\Theta \times B_1)}\\ &
        + \left\| \1_{(V^+ \cap W^+) \setminus \left( V^- {\cup W^-} 
        \right)} \Emb [f] \right\|_{\SL{}^{(\infty, \infty)}_{\Theta}
        \Phi_{\gr}^N (T)} .
      \end{aligned}
    \end{aligned} \label{eq:SIO:up-scale}
  \end{equation}
  As a matter of fact it holds that
  \[ \begin{aligned}[t]&
       \left\| \1_{(V^+ \cap W^+) \setminus \left( V^- {\cup W^-}  \right)}
       \Emb [f] \right\|_{\SJ^u_{\Theta} \Phi_{\gr}^N (T)}\\ &
       \qquad \lesssim \begin{aligned}[t]&
         \sup_{\phi} \Big\| \sup_{\sigma > \widetilde{\mf{b}}^{\ast}_-
         (\zeta)} \Big| \int_{\sigma}^{\widetilde{\mf{b}}^{\ast}_+ (\zeta)}
         \wpD_{\zeta} (\theta) F^{\ast} (\theta, \zeta, \rho) [\phi] \frac{\dd
         \rho}{\rho} \Big| \Big\|_{L^u_{\frac{\dd \theta \dd \zeta}{|
         \Theta |}} (\Theta \times B_1)} \quad\\ &
         \qquad + \begin{aligned}[t]&
           \Big\| \Big| \log \Big( \frac{\mf{b}^{\ast}_- (\zeta,
           \theta)}{\widetilde{\mf{b}}^{\ast}_- (\zeta)} \Big) \Big| +
           \Big| \log \Big( \frac{\mf{b}^{\ast}_+ (\zeta,
           \theta)}{\widetilde{\mf{b}}^{\ast}_+ (\zeta)} \Big)  \Big|
           \Big\|_{L^u_{\frac{\dd \theta \dd \zeta}{| \Theta |}} (\Theta
           \times B_1)}\\ &
           \times \sup_{\phi} \Big\| \wpD_{\zeta} (\theta) F^{\ast} (\theta,
           \zeta, \rho) [\phi] \Big\|_{L^{\infty}_{\frac{\dd \theta \dd
           \zeta}{| \Theta |} \frac{\dd \rho}{\rho}}}
         \end{aligned} 
       \end{aligned}
     \end{aligned} \]
  where the upper bound is taken over \(\phi \in \Phi_{\gr}^{\infty}\) with \(\|
  \phi \|_{\Phi_{\gr}^N} \leq 1\). We assume that the integrand vanishes when
  \(\widetilde{\mf{b}}^{\ast}_+ (\zeta) \leq \widetilde{\mf{b}}^{\ast}_-
  (\zeta)\). Since \(\| (- d_z + 2 \pi i \theta) \phi (z) \|_{\Phi_{\gr}^N}
  \leq \| \phi \|_{\Phi_{\gr}^N}\), we have that
  \[ \tmop{ssup}_{\phi} \left\| \wpD_{\zeta} (\theta) F^{\ast} (\theta,
     \zeta, \rho) [\phi] \right\|_{L^{\infty}_{\frac{\dd \theta \dd \zeta}{|
     \Theta |} \frac{\dd \rho}{\rho}}} \lesssim \left\| \1_{(V^+ \cap W^+)
     \setminus \left( V^- {\cup W^-}  \right)} \Emb [f]
     \right\|_{\SL^{(\infty, \infty)}_{\Theta} \Phi_{\gr}^N} . \]
  On the other hand, let \(\bar{\theta}\) be the point where \(\max_{\theta \in
  \Theta} \mf{b}^{\ast}_- (\zeta, \theta)\) is attained. From
  \eqref{eq:boundary-regularity-local} we get for \(\theta > \bar{\theta}\) that
  \[ 1 \geq \frac{\mf{b}^{\ast} (\theta, \zeta)}{\widetilde{\mf{b}}^{\ast}_-
     (\zeta)} \geq \frac{\theta_+ - \theta}{\theta_+ - \bar{\theta}} \]
  and for \(\theta < \bar{\theta}\) that
  \[ 1 \geq \frac{\mf{b}^{\ast} (\theta, \zeta)}{\widetilde{\mf{b}}^{\ast}_-
     (\zeta)} \geq \frac{\theta - \theta_-}{\bar{\theta} - \theta_-} \]
  where we denote \(\Theta = (\theta_-, \theta_+)\). This shows that
  \[ \Big\| \log \Big( \frac{\mf{b}^{\ast}_- (\zeta,
     \theta)}{\widetilde{\mf{b}}^{\ast}_- (\zeta)} \Big)  
     \Big\|_{L^u_{\frac{\dd \theta \dd \zeta}{| \Theta |}} (\Theta \times
     B_1)} + \Big\|  \log \Big( \frac{\mf{b}^{\ast}_+ (\zeta,
     \theta)}{\widetilde{\mf{b}}^{\ast}_+ (\zeta)} \Big) 
     \Big\|_{L^u_{\frac{\dd \theta \dd \zeta}{| \Theta |}} (\Theta \times
     B_1)} \lesssim 1. \]
  \tmtextbf{Reduction to the case \(u = 2\):}
  
  Recall that we assume bound \eqref{eq:SIO-bound:explicit} holds with \(u =
  2\). For \(u < 2\), bound \eqref{eq:SIO-bound:explicit} follows by the
  classical Hölder inequality that shows that for any \(T \in \TT_{\Theta}\)
  one has
  \[ \left\| \1_{(V^+ \cap W^+) \setminus \left( V^- {\cup W^-}  \right)}
     \Emb [f] \right\|_{\SJ^u_{\Theta} \Phi_{\gr}^N (T)} \lesssim \left\|
     \1_{(V^+ \cap W^+) \setminus \left( V^- {\cup W^-}  \right)} \Emb [f]
     \right\|_{\SJ^2_{\Theta} \Phi_{\gr}^N (T)} . \]
  We henceforth fix \(u \in (2, + \infty)\). The subsequent proof relies mainly
  on the fact that the size over trees is a \(\tmop{BMO}\)-style quantity and
  thus any two integrability exponents \(u_1, u_2\) can be compared. We adopt a
  procedure inspired by the proof of the John-Nirenberg inequality: we prove
  the reduction by induction on the scales of the trees \(T \in \TT_{\Theta}\)
  in play. According to \eqref{eq:SIO:up-scale} we need to show that for any
  \(T \in \TT_{\Theta}\) it holds that
  \[ \begin{aligned}[t]&
       \sup_{\phi} \left\| \sup_{\sigma > \widetilde{\mf{b}}^{\ast}_-
       (\zeta)} \left| \int_{\sigma}^{\widetilde{\mf{b}}^{\ast}_+ (\zeta)}
       \wpD_{\zeta} (\theta) F^{\ast} (\theta, \zeta, \rho) [\phi] \frac{\dd
       \rho}{\rho} \right| \right\|_{L^u_{\frac{\dd \theta \dd \zeta}{|
       \Theta |}} (\Theta \times B_1)} \lesssim \RHS{}
       \eqref{prop:SIO-size-domination}
     \end{aligned} \]
  Let us begin by showing that for any \(T \in \TT_{\Theta}\) and for any \(\phi
  \in \Phi_{\gr}^{\infty}\) with \(\| \phi \|_{\Phi_{\gr}^N} \leq 1\) it holds
  that
  \begin{equation}
    \begin{aligned}[t]&
      \left\| \sup_{\sigma > \widetilde{\mf{b}}^{\ast}_- (\zeta)} \left|
      \int_{\sigma}^{\widetilde{\mf{b}}^{\ast}_+ (\zeta)} \wpD_{\zeta}
      (\theta) F^{\ast} (\theta, \zeta, \rho) [\phi] \frac{\dd \rho}{\rho}
      \right| \right\|_{L^u_{\frac{\dd \theta \dd \zeta}{| \Theta |}} (\Theta
      \times B_1)}\\ &
      \qquad \lesssim \sup_{\phi'} \left\| \left( \int_{\Theta} \sup_{\sigma
      > \widetilde{\mf{b}}^{\ast}_- (\zeta)} \left|
      \int_{\sigma}^{\widetilde{\mf{b}}^{\ast}_+ (\zeta)} \wpD_{\zeta}
      (\theta) F^{\ast} (\theta, \zeta, \rho) [\phi] \frac{\dd \rho}{\rho}
      \right|^2 \frac{\dd \theta}{| \Theta |} \right)^{\frac{1}{2}}
      \right\|_{L^u (B_1)}
    \end{aligned} \label{eq:SIO:theta-integrability}
  \end{equation}
  where the upper bound gets taken over \(\phi' \in \Phi_{4 \gr}^{\infty}\) with
  \(\| \phi' \|_{\Phi_{4 \gr}^{N - 4}} \leq 1\). This amounts to reducing the
  exponent of the integral in the frequency variable \(\theta\). According to
  \eqref{eq:embedding} (that defines \(\Emb\)), to \eqref{eq:space-boost}
  (that defines \(\wpD_{\zeta} (\theta)\)), and to the symmetry property
  \eqref{eq:piT-symmetry}, we have that
  \[ \wpD_{\zeta} (\theta) F^{\ast} (\theta, \zeta, \rho) [\phi] = e^{2 \pi i
     \xi_T s_T \zeta} \rho \dd_{\zeta} \left( f_T \ast \Dil_{\rho}
     \Mod_{\theta} \phi^{\vee} (\zeta) \right) \]
  with \(f_T = s_T \Dil_{s_T^{- 1}} \Mod_{- \xi_T} \Tr_{- x_T} f\). This
  allows us to write that
  \[ \wpD_{\zeta} (\theta) F^{\ast} (\theta, \zeta, \rho) [\phi] =
     \wpD_{\zeta} (\theta') F^{\ast} (\theta', \zeta, \rho) \left[
     \Mod_{(\theta - \theta')} \phi \right] \]
  for any \(\theta' \in B_{\gr} (\theta)\) with \(\widetilde{\mf{b}}^{\ast}_-
  (\zeta) < \theta' < \widetilde{\mf{b}}^{\ast}_+ (\zeta)\) since then
  \((\theta', \zeta, \sigma) \in E^{\ast}\). Note also that \(\Mod_{(\theta -
  \theta')} \phi \in \Phi_{2 \gr}^{\infty}\) with \(\left\| \Mod_{(\theta -
  \theta')} \phi \right\|_{\Phi_{\gr}^N} \leq \| \phi \|_{\Phi_{2 \gr}^N}\). We
  thus have that
  \[ \begin{aligned}[t]&
       \sup_{\sigma \in \R_+} \left|
       \int_{\sigma}^{\widetilde{\mf{b}}^{\ast}_+ (\zeta)} \wpD_{\zeta}
       (\theta) F^{\ast} (\theta, \zeta, \rho) [\phi] \frac{\dd \rho}{\rho}
       \right| \leq\\ &
       \qquad \frac{1}{\gr} \int_{B_{\gr} (\theta) \cap \Theta} \sup_{\sigma
       \in \R_+} \left| \int_{\sigma}^{\widetilde{\mf{b}}^{\ast}_+ (\zeta)}
       \wpD_{\zeta} (\theta') F^{\ast} (\theta', \zeta, \rho) \left[
       \Mod_{(\theta - \theta')} \phi \right] \frac{\dd \rho}{\rho} \right|
       \dd \theta' .
     \end{aligned} \]
  Note that \(\gr < \left| B_{\gr} (\theta) \cap \Theta \right| < 2 \gr\). Using
  \Cref{lem:wave-packet-decomposition} we decompose \(\Mod_{(\theta - \theta')}
  \phi \in \Phi_{2 \gr}^N\) into an absolutely convergent sum:
  \[ \Mod_{(\theta - \theta')} \phi = \sum_{k \in \Z} a_k (\theta, \theta',
     \phi) \phi_k' . \]
  with \(\| \phi_k' \|_{\Phi_{4 \gr}^{N - 2}} \leq 1\) with uniformly summable
  coefficients. This shows that
  \[ \begin{aligned}[t]&
       \left\| \sup_{\sigma \in \R_+} \left|
       \int_{\sigma}^{\widetilde{\mf{b}}^{\ast}_+ (\zeta)} \wpD_{\zeta}
       (\theta) F^{\ast} (\theta, \zeta, \rho) [\phi] \frac{\dd \rho}{\rho}
       \right| \right\|_{L^u_{\frac{\dd \theta \dd \zeta}{| \Theta |}}
       (\Theta \times B_1)}\\ &
       \qquad \lesssim \sup_{\phi'} \int_{\Theta \times B_1} \left(
       \frac{1}{\gr} \int_{B_{\gr} (\theta) \cap \Theta} \sup_{\sigma \in
       \R_+} \left| \int_{\sigma}^{\widetilde{\mf{b}}^{\ast}_+ (\zeta)}
       \wpD_{\zeta} (\theta') F^{\ast} (\theta', \zeta, \rho) [\phi']
       \frac{\dd \rho}{\rho} \right|^2 \dd \theta' \right)^{\frac{u}{2}} \dd
       \zeta \frac{\dd \theta}{| \Theta |}\\ &
       \qquad \lesssim \sup_{\phi'} \left\| \left( \int_{\Theta} \sup_{\sigma
       \in \R_+} \left| \int_{\sigma}^{\mf{\tilde{b}}^{\ast}_+ (\zeta)}
       \wpD_{\zeta} (\theta) F^{\ast} (\theta', \zeta, \rho) [\phi'] \frac{\dd
       \rho}{\rho} \right|^2 \frac{\dd \theta'}{| \Theta |}
       \right)^{\frac{1}{2}} \right\|_{L^u (B_1)}
     \end{aligned} \]
  with the upper bounds being taken over \(\phi' \in \Phi^{\infty}_{\gr}\) with
  \(\| \phi' \|_{\Phi_{4 \gr}^{N - 2}} \leq 1\). This shows that bound
  \eqref{eq:SIO:theta-integrability} holds.
  
  With bound \eqref{eq:SIO:theta-integrability} in hand‚ to deduce bound
  \eqref{eq:SIO-bound:explicit} we now need to show that for every \(T \in
  \TT_{\Theta} \) it holds that
  \begin{equation}
    \begin{aligned}[t]&
      \sup_{\phi} \left\| \left( \int_{\Theta} \sup_{\sigma \in \R_+} \left|
      \int_{\sigma}^{\widetilde{\mf{b}}^{\ast}_+ (\zeta)} \wpD_{\zeta}
      (\theta) F^{\ast} (\theta, \zeta, \rho) [\phi] \frac{\dd \rho}{\rho}
      \right|^2 \frac{\dd \theta}{| \Theta |} \right)^{\frac{1}{2}}
      \right\|_{L^u (B_1)}\\ &
      \qquad \lesssim \sup_{\tmscript{\begin{aligned}[t]&
        T' \in \TT_{\Theta}\\ &
        s_{T'} \leq s_T
      \end{aligned}}} \begin{aligned}[t]&
        \left( \left\| \1_{(V^+ \cap W^+) \setminus \left( V^- {\cup W^-} 
        \right)} \Emb [f] \right\|_{\SL^{(2, 2)}_{\Theta} \wpD_{8 \gr}^{N - 8
        } (T')} \right\nobracket \qquad\\ &
        \qquad + \left. \left\| \1_{(V^+ \cap W^+) \setminus \left( V^- {\cup
        W^-}  \right)} \Emb [f] \right\|_{\SL^{(\infty, \infty)}_{\Theta}
        \Phi_{8 \gr}^{N - 8 } (T')} \right)
      \end{aligned}
    \end{aligned} \label{eq:tree-SIO-size}
  \end{equation}
  where the upper bound is taken over \(\phi \in \Phi_{4 \gr}^{\infty}\) with
  \(\| \phi \|_{\Phi_{4 \gr}^{N - 4}} \leq 1\). We do an induction on scales
  argument. Suppose we have shown that \(\eqref{eq:tree-SIO-size}\) holds with
  implicit constant \(C_{\ast}\) for fixed \(V^+ {, V^-} \), for any \(\tilde{W}^+,
  \tilde{W}^- \in \mathbb{T}_{\Theta}^{\cup}\), and for all \(T \in
  \TT_{\Theta}\) with \(s_T < S\) for some \(S > 0\). Let us show that it holds for
  those same \(V^+, V^-, \) and any \(W^+, W^- \in \mathbb{T}_{\Theta}^{\cup}\),
  and all \(T \in \TT_{\Theta}\) with \(s_T < C_{\tmop{ind}} S\) for some
  \(C_{\tmop{ind}} > 1\). This is sufficient to conclude that
  \(\eqref{eq:tree-SIO-size}\) holds for all \(T \in \TT_{\Theta}\). As a matter
  of fact, first check that it holds for all \(V^- \supset \{ (\eta, y, t) : t
  < \varepsilon \}\) for a fixed \(\varepsilon > 0\) because
  \eqref{eq:tree-SIO-size} holds trivially for \(s_T < S < \varepsilon\) and
  thus, by induction, \(\eqref{eq:tree-SIO-size}\) holds for all \(T \in
  \TT_{\Theta}\). Since the implicit constant in \eqref{eq:tree-SIO-size} is
  independent of \(\varepsilon > 0\) and {\RHS{\eqref{eq:tree-SIO-size}}} then
  we can pass to the limit and conclude that \eqref{eq:tree-SIO-size} holds
  for any \(V^+, V^-\), as required.
  
  We will fix the constants \(C_{\ast} > 1\) and \(C_{\tmop{ind}} > 1\) at the end
  of this discussion. By symmetry (identity \eqref{eq:size-symmetry} and
  subsequent discussion), we can suppose, without loss of generality, that \(T
  = T_{\Theta} (0, 0, 1)\) and \(S > C_{\tmop{ind}}^{- 1}\). Consider the
  function
  \[ \mf{F} (\zeta) \eqd  \left( \int_{\Theta} \sup_{\sigma \in \R_+} \left|
     \int_{\sigma}^{\widetilde{\mf{b}}^{\ast}_+ (\zeta)} F^{\ast} (\theta,
     \zeta, \rho) [\phi_{\theta}] \frac{\dd \rho}{\rho} \right|^2 \frac{\dd
     \theta}{| \Theta |} \right)^{\frac{1}{2}} \]
  with some fixed \(\phi \in \Phi_{4 \gr}^{\infty}\) with \(\| \phi \|_{\Phi_{4
  \gr}^{N - 4}} \leq 1\) and with \(\phi_{\theta} (z) \eqd (- d_z + 2 \pi i
  \theta) \phi (z)\). We need to show that
  \[ \int_{B_1} \mf{F} (\zeta)^u \dd \zeta \leq C_{\ast}
     \RHS{\eqref{eq:tree-SIO-size}}^u . \]
  We separate \(B_1\) into a part where \(\mf{F} (\zeta)\) is bounded and an
  exceptional set
  \[ E \eqd \left\{ {} \zeta \suchthat M \mf{F} (\zeta) \geq C_2 \left\|
     \mf{\mf{F}} (\zeta) \right\|_{L^2 (B_1)} \right\} \]
  for some \(C_2 \gg 1\) to be chosen later. \(E\) is open and we can represent \(E
  = \bigcup_{n \in \N} B_{s_n} (x_n)\) with pairwise disjoint \(B_{s_n} (x_n)\).
  Let us set \(\mf{\tilde{b}}_{+, n}^{\ast} (\zeta) \eqd 10 s_n - | \zeta - x_n
  |\) so that \(9 s_n < \mf{\tilde{b}}_{+, n} (\zeta) {\leq 10 s_n} \) on
  \(B_{s_n} (x_n)\). By the boundedness of the Hardy Littlewood maximal function
  it holds that \(\sum_{n \in \N} 2 s_n \leq C_M C_2^{- 1} .\) We can now write
  that
  \begin{equation}
    \mf{F} (\zeta)^u \leq \1_{B_1 \setminus E} (\zeta) \mf{F} (\zeta)^u + 2^u
    \sum_{n \in \N} \left( \mf{F}^{\uparrow}_n (\zeta)^u +
    \mf{F}^{\downarrow}_n (\zeta)^u \right)
    \label{eq:SIO:stopping-time-decomp}
  \end{equation}
  where
  \[ \begin{aligned}[t]&
       \mf{F}^{\uparrow}_n (\zeta) \eqd \1_{B_{s_n} (x_n)} (\zeta) \left(
       \int_{\Theta} \sup_{\sigma > s_n} \left|
       \int_{\sigma}^{\widetilde{\mf{b}}^{\ast}_+ (\zeta)} F^{\ast} (\theta,
       \zeta, \rho) [\phi_{\theta}] \frac{\dd \rho}{\rho} \right|^2 \frac{\dd
       \theta}{| \Theta |} \right)^{\frac{1}{2}},\\ &
       \mf{F}^{\downarrow}_n (\zeta) \eqd \1_{B_{s_n} (x_n)} (\zeta) \left(
       \int_{\Theta} \sup_{\sigma < \mf{\tilde{b}}_{+, n}^{\ast} (\zeta)}
       \left| \int_{\sigma}^{\mf{\tilde{b}}_{+, n}^{\ast} (\zeta)} F^{\ast}
       (\theta, \zeta, \rho) [\phi_{\theta}] \frac{\dd \rho}{\rho} \right|^2
       \frac{\dd \theta}{| \Theta |} \right)^{\frac{1}{2}} .
     \end{aligned} \]
  We will prove that
  \begin{equation}
    \int_{B_1} \mf{F} (\zeta)^u \dd \zeta \leq (1 + 2^u) (C_2 C_{u = 2} + C_M
    C_5 + C_M C_2^{- 1} C_{\ast})  \RHS{\eqref{eq:tree-SIO-size}}^u
    \label{eq:SIO:induction-step}
  \end{equation}
  that will allow us to conclude that
  \[ \int_{B_1} \mf{F} (\zeta)^u \dd \zeta \leq C_{\ast}
     \RHS{\eqref{eq:tree-SIO-size}}^u, \]
  as required, as long as \(C_2\) is chosen to be satisfy \(2 (1 + 2^u) < C_2\)
  and \(C_{\ast}\) satisfies \(2 (1 + 2^u) (C_2 C_{u = 2} + C_M C_5) \leq
  C_{\ast}\). Let us now show that \eqref{eq:SIO:induction-step} by bounding
  integrals of terms in \eqref{eq:SIO:stopping-time-decomp}.
  
  The term \(\1_{B_1 \setminus E} (\zeta) \mf{F} (\zeta)^u\) on
  {\RHS{\eqref{eq:SIO:stopping-time-decomp}}} can be bound since \(M \mf{F}
  (\zeta) \geq \mf{F} (\zeta)\); thus
  \[ \int_{B_1 \setminus E} \mf{F} (\zeta)^u \dd \zeta \leq C_2 \left\|
     \mf{\mf{F}} (\zeta) \right\|_{L^2 (B_1)} \leq C_2 C_{u = 2} 
     \RHS{\eqref{eq:tree-SIO-size}}^u \]
  for some \(C_{u = 2} > 0\) that comes from assuming that bound
  \eqref{eq:SIO-bound:explicit} holds with \(u = 2\).
  
  Let us now concentrate on the term \(\sum_{n \in \N} \mf{F}^{\uparrow}_n
  (\zeta)^u\) on {\RHS{\eqref{eq:SIO:stopping-time-decomp}}}. We will show that
  \[ \int_{B_1} \sum_{n \in \N} \mf{F}^{\uparrow}_n (\zeta)^u \dd \zeta =
     \sum_{n \in \N} \int_{B_{s_n} (x_n)} \mf{F}^{\uparrow}_n (\zeta)^u \dd
     \zeta \leq C_M C_5 C_{u = 2}  \RHS{\eqref{eq:tree-SIO-size}}^u \]
  for some fixed \(C_5 > 1\) by showing that for each \(n \in \N\) we have
  \[ \int_{B_{s_n} (x_n)} \mf{F}^{\uparrow}_n (\zeta)^u \dd \zeta \leq s_n C_2
     C_5 C_{u = 2}  \RHS{\eqref{eq:tree-SIO-size}}^u \]
  and using that \(\sum_{n \in \N} 2 s_n \leq C_M C_2^{- 1}\), i.e. the
  boundedness of the Hardy Littlewood maximal function. To do so we will argue
  that \(\mf{F}^{\uparrow}_n (\zeta)\) is essentially constant on \(B_{s_n}
  (x_n)\). Since \(\left\| \widetilde{\mf{b}}^{\ast}_+ (\zeta)
  \right\|_{\tmop{Lip}} \leq 1\) this shows that \(\min_{\zeta \in B_{s_n}
  (x_n)} \widetilde{\mf{b}}^{\ast}_+ (\zeta) \geq \min_{\zeta \in B_{s_n}
  (x_n)} \widetilde{\mf{b}}^{\ast}_+ (\zeta) - 2 s_n\). If
  \(\mf{F}^{\uparrow}_n\) is not constantly vanishing, then
  \(\widetilde{\mf{b}}^{\ast}_+ (\bar{\zeta}) > 9 s_n\) for some \(\bar{\zeta}
  \in B_{s_n} (x_n)\) and thus for any \(\zeta, \zeta' \in B_{s_n} (x_n)\) it
  holds that \(5 / 7 \leq \widetilde{\mf{b}}^{\ast}_+ (\zeta') /
  \widetilde{\mf{b}}^{\ast}_+ (\zeta) \leq 9 / 7\). Similarly, since \(9 s_n <
  \mf{\tilde{b}}_{+, n} {\leq 10 s_n} \) on \(B_{s_n} (x_n)\) it holds that \(9 /
  10 \leq \widetilde{\mf{b}}^{\ast}_{+, n} (\zeta') /
  \widetilde{\mf{b}}^{\ast}_+ (\zeta) \leq 10 / 9\) for any \(\zeta, \zeta' \in
  B_{s_n} (x_n)\). Finally either \(\mf{\tilde{b}}_- < 7 s_n\) on \(B_{s_n} (x_n)\)
  or \(5 / 7 \leq \widetilde{\mf{b}}^{\ast}_- (\zeta') /
  \widetilde{\mf{b}}^{\ast}_- (\zeta) \leq 9 / 7\) for any \(\zeta, \zeta' \in
  B_{s_n} (x_n)\). Let us use this to compare \(\mf{F}^{\uparrow}_n (\zeta)\)
  with \(\mf{F}^{\uparrow}_n (\zeta')\) for any \(\zeta, \zeta' \in B_{s_n}
  (x_n)\). We have
  \[ \left| \mf{F}^{\uparrow}_n (\zeta) - \mf{F}^{\uparrow}_n (\zeta')
     \right| \leq \begin{aligned}[t]&
       \left( \int_{\Theta} \left| \int_{\widetilde{\mf{b}}^{\ast}_+ (\zeta) /
       2}^{2 \widetilde{\mf{b}}^{\ast}_+ (\zeta)} | F^{\ast} (\theta, \zeta,
       \rho) [\phi_{\theta}] | + | F^{\ast} (\theta, \zeta', \rho)
       [\phi_{\theta}] | \frac{\dd \rho}{\rho} \right|^2 \frac{\dd \theta}{|
       \Theta |} \right)^{\frac{1}{2}}\\ &
       + \left( \int_{\Theta} \left| \int_{\mf{\tilde{b}}_{+, n}^{\ast}
       (\zeta) / 2}^{2 \mf{\tilde{b}}_{+, n}^{\ast} (\zeta)} | F^{\ast}
       (\theta, \zeta, \rho) [\phi_{\theta}] | + | F^{\ast} (\theta, \zeta',
       \rho) [\phi_{\theta}] | \frac{\dd \rho}{\rho} \right|^2 \frac{\dd
       \theta}{| \Theta |} \right)^{\frac{1}{2}}\\ &
       + \begin{aligned}[t]&
         \left( \int_{\Theta} \left| \int_0^1 \1_{\widetilde{\mf{b}}^{\ast}_-
         (\zeta) / 2 < \rho < 2 \widetilde{\mf{b}}^{\ast}_- (\zeta)} \right.
         \right. \qquad\\ &
         \qquad \times \left. \left. | F^{\ast} (\theta, \zeta, \rho)
         [\phi_{\theta}] | + | F^{\ast} (\theta, \zeta', \rho) [\phi_{\theta}]
         | \frac{\dd \rho}{\rho} \right|^2 \frac{\dd \theta}{| \Theta |}
         \right)^{\frac{1}{2}} .
       \end{aligned}\\ &
       + \begin{aligned}[t]&
         \left( \int_{\Theta} \left| \int_{2 \mf{\tilde{b}}_{+, n}^{\ast}
         (\zeta')}^{\widetilde{\mf{b}}^{\ast}_+ (\zeta') / 2}
         \1_{\widetilde{\mf{b}}^{\ast}_- (\zeta') \vee
         \widetilde{\mf{b}}^{\ast}_- (\zeta) < \rho <
         \widetilde{\mf{b}}^{\ast}_+ (\zeta') \wedge
         \widetilde{\mf{b}}^{\ast}_+ (\zeta)} \right\nobracket \right.
         \qquad\\ &
         \qquad \times \left. \left. | F^{\ast} (\theta, \zeta, \rho)
         [\phi_{\theta}] - F^{\ast} (\theta, \zeta', \rho) [\phi_{\theta}] |
         \frac{\dd \rho}{\rho} \right|^2 \frac{\dd \theta}{| \Theta |}
         \right)^{\frac{1}{2}} .
       \end{aligned}
     \end{aligned} \]

     Using \Cref{lem:wave-packet-decomposition} to decompose \(\phi_{\theta}\) we
  can bound all but the last summand on {\RHS{}} above by \(\left\| \1_{(V^+
  \cap W^+) \setminus \left( V^- {\cup W^-}  \right)} \Emb [f]
  \right\|_{\SL{}^{(\infty, \infty)}_{\Theta} \Phi_{8 \gr}^{N - 6} (T)}\). To
  bound the last summand notice that as long as \(\widetilde{\mf{b}}^{\ast}_-
  (\zeta') \vee \widetilde{\mf{b}}^{\ast}_- (\zeta) < \rho <
  \widetilde{\mf{b}}^{\ast}_+ (\zeta') \wedge \widetilde{\mf{b}}^{\ast}_+
  (\zeta)\) we have that
  \[ \begin{aligned}[t]&
       F^{\ast} (\theta, \zeta, \rho) [\phi_{\theta}] = f \ast \Dil_{\rho}
       \Mod_{\theta} \phi^{\vee} (\zeta)\\ &
       \qquad = f \ast \Tr_{(\zeta' - \zeta)} \Dil_{\rho} \Mod_{\theta}
       \phi^{\vee} (\zeta') = F^{\ast} (\theta, \zeta', \rho) \left[ e^{2 \pi
       i \theta \frac{(\zeta - \zeta')}{\rho}} \Tr_{\frac{(\zeta' -
       \zeta)}{\rho}} \phi_{\theta} \right]
     \end{aligned} \]
  Since \(\left\| e^{2 \pi i \frac{(\zeta - \zeta') \theta}{\rho}} \Tr_{\left(
  \frac{\zeta - \zeta'}{\rho} \right)} \phi_{\theta} - \phi_{\theta}
  \right\|_{\Phi_{4 \gr}^{N - 4}} \lesssim \frac{| \zeta - \zeta' |}{\rho} \|
  \phi_{\theta} \|_{\Phi_{4 \gr}^{N - 5}} \lesssim \frac{| \zeta - \zeta'
  |}{\rho} \| \phi \|_{\Phi_{4 \gr}^{N - 5}}\) for \(\frac{| \zeta - \zeta'
  |}{\rho} \lesssim 1\) we obtain that
  \[ \begin{aligned}[t]&
       \begin{aligned}[t]&
         \left( \int_{\Theta} \left| \int_{2 \mf{\tilde{b}}_{+, n}^{\ast}
         (\zeta')}^{\widetilde{\mf{b}}^{\ast}_+ (\zeta') / 2}
         \1_{\widetilde{\mf{b}}^{\ast}_- (\zeta') \vee
         \widetilde{\mf{b}}^{\ast}_- (\zeta) < \rho <
         \widetilde{\mf{b}}^{\ast}_+ (\zeta') \wedge
         \widetilde{\mf{b}}^{\ast}_+ (\zeta)} \right\nobracket \right.
         \qquad\\ &
         \qquad \times \left. \left. | F^{\ast} (\theta, \zeta, \rho)
         [\phi_{\theta}] - F^{\ast} (\theta, \zeta', \rho) [\phi_{\theta}] |
         \frac{\dd \rho}{\rho} \right|^2 \frac{\dd \theta}{| \Theta |}
         \right)^{\frac{1}{2}}
       \end{aligned}\\ &
       = \begin{aligned}[t]&
         \left( \int_{\Theta} \left| \int_{2 \mf{\tilde{b}}_{+, n}^{\ast}
         (\zeta')}^{\widetilde{\mf{b}}^{\ast}_+ (\zeta') / 2}
         \1_{\widetilde{\mf{b}}^{\ast}_- (\zeta') \vee
         \widetilde{\mf{b}}^{\ast}_- (\zeta) < \rho <
         \widetilde{\mf{b}}^{\ast}_+ (\zeta') \wedge
         \widetilde{\mf{b}}^{\ast}_+ (\zeta)} \right\nobracket \right.
         \qquad\\ &
         \qquad \times \left. \left. \left| F^{\ast} (\theta, \zeta', \rho)
         \left[ e^{2 \pi i \frac{(\zeta - \zeta') \theta}{\rho}} \Tr_{\left(
         \frac{\zeta - \zeta'}{\rho} \right)} \phi_{\theta} - \phi_{\theta}
         \right] \right| \frac{\dd \rho}{\rho} \right|^2 \frac{\dd \theta}{|
         \Theta |} \right)^{\frac{1}{2}}
       \end{aligned}\\ &
       \lesssim \begin{aligned}[t]&
         \left( \int_{\Theta} \left( \int_{2 \mf{\tilde{b}}_{+, n}^{\ast}
         (\zeta')}^{\widetilde{\mf{b}}^{\ast}_+ (\zeta') / 2}
         \1_{\widetilde{\mf{b}}^{\ast}_- (\zeta') \vee
         \widetilde{\mf{b}}^{\ast}_- (\zeta) < \rho <
         \widetilde{\mf{b}}^{\ast}_+ (\zeta') \wedge
         \widetilde{\mf{b}}^{\ast}_+ (\zeta)} \frac{| \zeta - \zeta' |}{\rho}
         \frac{\dd \rho}{\rho} \right)^2 \frac{\dd \theta}{| \Theta |}
         \right)^{\frac{1}{2}}\\ &
         \times \left\| \1_{(V^+ \cap W^+) \setminus \left( V^- {\cup W^-} 
         \right)} \Emb [f] \right\|_{\SL{}^{(\infty, \infty)}_{\Theta}
         \Phi_{4 \gr}^{N - 5} (T)}
       \end{aligned}\\ &
       \lesssim \left\| \1_{(V^+ \cap W^+) \setminus \left( V^- {\cup W^-} 
       \right)} \Emb [f] \right\|_{\SL{}^{(\infty, \infty)}_{\Theta} \Phi_{4
       \gr}^{N - 5} (T)}
     \end{aligned} \]
  where the last bound holds because integration happens over \(\rho \geq s_n\)
  and \(| \zeta - \zeta' | \lesssim s_n\). We can now average in \(\zeta \in
  B_{s_n} (x_n)\) to obtain that
  \[ \int_{B_{s_n} (x_n)} \mf{F}^{\uparrow}_n (\zeta)^u \dd \zeta
     \begin{aligned}[t]&
       \lesssim \int_{B_{s_n} (x_n)} \left( \frac{1}{2 s_n} \int_{B_{s_n}
       (x_n)} \left( \left| \mf{F}^{\uparrow}_n (\zeta') \right| + \left|
       \mf{F}^{\uparrow}_n (\zeta) - \mf{F}^{\uparrow}_n (\zeta') \right|
       \right) \dd \zeta' \right)^u \dd \zeta\\ &
       \lesssim \begin{aligned}[t]&
         s_n \left\| \1_{(V^+ \cap W^+) \setminus \left( V^- {\cup W^-} 
         \right)} \Emb [f] \right\|_{\SL{}^{(\infty, \infty)}_{\Theta}
         \Phi_{8 \gr}^{N - 6} (T)}^u\\ &
         + s_n \left( \frac{1}{2 s_n} \int_{B_{s_n} (x_n)} \left|
         \mf{\mf{\mf{F} }}^{\uparrow}_n (\zeta') \right| \dd \zeta' \right)^u
       \end{aligned}
     \end{aligned} \]
  By maximality of \(B_{s_n} (x_n) \subset E\) it holds that
  \[ \frac{1}{2 s_n} \int_{B_{s_n} (x_n)} \left| \mf{\mf{\mf{F}
     }}^{\uparrow}_n (\zeta') \right| \dd \zeta' \leq \frac{1}{2 s_n}
     \int_{B_{s_n} (x_n)} \left| \mf{\mf{\mf{F} }} (\zeta') \right| \dd \zeta'
     \lesssim C_2 \left\| \mf{\mf{\mf{F} }} \right\|_{L^2 (B_1)} . \]
  As a matter of fact
  \[ \frac{1}{2 s_n} \int_{B_{s_n} (x_n)} \left| \mf{\mf{\mf{F} }} (\zeta')
     \right| \dd \zeta' \lesssim M \mf{\mf{\mf{F} }} (x_n + s_n) \leq C_2
     \left\| \mf{\mf{\mf{F} }} \right\|_{L^2 (B_1)} \]
  because \(x_n + s_n \notin E\). Thus
  \[ \begin{aligned}[t]&
       \sum_n \int_{B_{s_n} (x_n)} \mf{F}^{\uparrow}_n (\zeta)^u \dd \zeta
       \leq C_M C_5 \left( \RHS{\eqref{eq:tree-SIO-size}}^u + C_2 \left\|
       \mf{\mf{\mf{F} }} \right\|_{L^2 (B_1)}^u \right)\\ &
       \leq C_M C_5 (1 + C_2 C_{u = 2}) \RHS{\eqref{eq:tree-SIO-size}}^u,
     \end{aligned} \]
  as required.
  
  Finally, let us concentrate on the term \(\sum_{n \in \N}
  \mf{F}^{\downarrow}_n (\zeta)^u\) on
  {\RHS{\eqref{eq:SIO:stopping-time-decomp}}}. Setting \(T_n = T_{\Theta} (0,
  x_n, 10 s_n)\) it holds that
  \[ \begin{aligned}[t]&
       \frac{1}{s_n} {\int_{B_{s_n} (x_n)}}  \mf{F}^{\downarrow}_n (\zeta)^u
       \dd \zeta\\ &
       \qquad = \frac{1}{s_n} {\int_{B_{s_n} (x_n)}}  \left( \int_{\Theta}
       \sup_{\sigma < \mf{\tilde{b}}_{+, n}^{\ast} (\zeta)} \left|
       \int_{\sigma}^{\mf{\tilde{b}}_{+, n}^{\ast} (\zeta)} F^{\ast} (\theta,
       \zeta, \rho) [\phi_{\theta}] \frac{\dd \rho}{\rho} \right|^2 \frac{\dd
       \theta}{| \Theta |} \right)^{\frac{u}{2}} \dd \zeta\\ &
       \qquad \lesssim \begin{aligned}[t]&
         \sup_{\phi'} \left\| \left( \int_{\Theta} \sup_{\sigma > 0} \left|
         \int_{\sigma}^{1 - | \zeta |} \wpD_{\zeta} (\theta) \left( \1_{(V^+
         \cap W^+) \setminus \left( V^- {\cup W^-}  \right)} \right. \right.
         \right. \right. \qquad\\ &
         \times \left. \left. \left. \left. \Emb [f] \right) \circ \pi_{T_n}
         (\theta, \zeta, \rho) [\phi'] \frac{\dd \rho}{\rho} \right|^2
         \frac{\dd \theta}{| \Theta |} \right)^{\frac{1}{2}} \right\|_{L^u
         (B_1)}^u
       \end{aligned}
     \end{aligned} \]
  where the upper bound is taken over \(\phi' \in \Phi_{4 \gr}^{\infty}\) with
  \(\| \phi' \|_{\Phi_{4 \gr}^{N - 4}} \leq 1\). Since \(s_n \leq C_M C_2^{- 1}\)
  by the boundedness of the Hardy Littlewood maximal function, we can choose
  \(C_2\) sufficiently large so that \(10 s_n \leq C_{\tmop{ind}}^{- 1}\) and we
  can apply the induction hypothesis to the bound above to get
  \[ \frac{1}{s_n} {\int_{B_{s_n} (x_n)}}  \mf{F}^{\downarrow}_n (\zeta)^u \dd
     \zeta \lesssim C_{\ast} \RHS{\eqref{eq:tree-SIO-size}}^u . \]
  Summing up, this gives
  \[ \sum_{n \in \N} {\int_{B_{s_n} (x_n)}}  \mf{F}^{\downarrow}_n (\zeta)^u
     \dd \zeta \leq C_M C_2^{- 2} C_{\ast} \RHS{\eqref{eq:tree-SIO-size}}^u .
  \]
  This concludes showing \eqref{eq:SIO:induction-step}.
  
  \tmtextbf{The case \(u = 2\):}
  
  The final step is showing that \eqref{eq:tree-SIO-size} holds for \(u = 2\).
  By symmetry considerations (identity \eqref{lem:geometry-of-boundary} and
  subsequent discussion) we can suppose, without loss of generality, that \(T =
  T_{\Theta} (0, 0, 1)\). We need to show that given any smooth \(\sigma \of
  \Theta \times B_1 \rightarrow [0, 1]\) it holds that
  \[ \left\| \int_{\sigma (\theta, \zeta)}^{\mf{\tilde{b}^{\ast}_+ (\zeta)}}
     \wpD_{\zeta} (\theta) F^{\ast} (\theta, \zeta, \rho) [\phi] \frac{\dd
     \rho}{\rho}  \right\|_{L^2_{\frac{\dd \theta \dd \zeta}{| \Theta |}}
     (\Theta \times B_1)} \lesssim \RHS{\eqref{eq:tree-SIO-size}} \]
  with \(\phi \in \Phi_{4 \gr}^{\infty}\) being any wave packet with \(\| \phi
  \|_{\Phi_{4 \gr}^{N - 4}} \leq 1\). Let \(\phi_{\theta} (z) \eqd (- d_z + 2
  \pi i \theta) \phi (z)\) so that \(\wpD_{\zeta} (\theta) F^{\ast} (\theta,
  \zeta, \rho) [\phi] = F^{\ast} (\theta, \zeta, \rho) [\phi_{\theta}]\). Next,
  we will show that we can assume that \(\sigma (\theta, \zeta) =
  \widetilde{\mf{b}}_- (\zeta)\) and, in particular, we may assume that \(\|
  \zeta \mapsto \sigma (\theta, \zeta) \|_{\tmop{Lip}} \leq 1\). For any \(C
  \geq 1\) it holds that
  \[ \begin{aligned}[t]&
       \left\| \int_{\sigma (\theta, \zeta)}^{\min \left( C \sigma (\theta,
       \zeta), \mf{\tilde{b}}^{\ast}_+ (\zeta) \right)} F^{\ast} (\theta,
       \zeta, \rho) [\phi_{\theta}] \frac{\dd \rho}{\rho}
       \right\|_{L^2_{\frac{\dd \theta \dd \zeta}{| \Theta |}} (\Theta \times
       B_1)} \qquad\\ &
       \lesssim \left\| \1_{(V^+ \cap W^+) \setminus \left( V^- {\cup W^-} 
       \right)} \Emb [f] \right\|_{\SL^{(\infty, \infty)}_{\Theta} \Phi_{4
       \gr}^{N - 4 }}
     \end{aligned} \]
  so we can assume without loss of generality that that \(\sigma (\theta,
  \zeta) \in \left( C \mf{\tilde{b}}^{\ast}_- (\zeta), C^{- 1}
  \mf{\tilde{b}}^{\ast}_+ (\zeta) \right)\) for some large \(C \gg 1\). Let us
  freeze \(\theta \in \Theta\) and, for the sake of the subsequent argument,
  assume that \(\zeta \geq 0\); the case \(\zeta < 0\) is handled symmetrically.
  Since \(\left\| \zeta \mapsto \widetilde{\mf{b}} (\zeta)
  \right\|_{\tmop{Lip}} \leq 1\) it holds that \(\rho \in \left(
  \mf{\tilde{b}}^{\ast}_- (\zeta'), \mf{\tilde{b}}^{\ast}_+ (\zeta') \right)\)
  if \(\rho \in \left( \sigma (\theta, \zeta), C^{- 1} \mf{\tilde{b}}^{\ast}_+
  (\zeta) \right)\) and \(| \zeta' - \zeta | \leq \sigma (\theta, \zeta)\). We
  decompose
  \[ \int_{\sigma (\theta, \zeta)}^{\mf{\tilde{b}}^{\ast}_+ (\zeta)} F^{\ast}
     (\theta, \zeta, \rho) [\phi_{\theta}] \frac{\dd \rho}{\rho} = \mathrm{I}
     (\theta, \zeta) + \mathrm{II} (\theta, \zeta) + \mathrm{III} (\theta,
     \zeta) + \mathrm{IV} (\theta, \zeta) \]
  with
  \[ \begin{aligned}[t]&
       \mathrm{I} (\theta, \zeta) \assign \int_{\sigma (\zeta,
       \theta)}^{\widetilde{\mf{b}}^{\ast}_+ (\zeta) / C} \left( F^{\ast}
       (\theta, \zeta, \rho) [\phi_{\theta}] \frac{\dd \rho}{\rho} -
       \frac{1}{\sigma (\theta, \zeta)} \int_{\zeta - \sigma (\theta,
       \zeta)}^{\zeta} F^{\ast} (\theta, \zeta', \rho) [\phi_{\theta}] \dd
       \zeta' \right) \frac{\dd \rho}{\rho},\\ &
       \mathrm{II} (\theta, \zeta) \assign - \frac{1}{\sigma (\theta, \zeta)}
       \int_{\zeta - \sigma (\theta, \zeta)}^{\zeta}
       \int_{\mf{\tilde{b}}^{\ast}_+ (\zeta) / C}^{\mf{\tilde{b}}^{\ast}_+
       (\zeta')} F^{\ast} (\theta, \zeta', \rho) [\phi_{\theta}] \frac{\dd
       \rho}{\rho} \dd \zeta',\\ &
       \mathrm{III} (\theta, \zeta) \assign - \frac{1}{\sigma (\theta, \zeta)}
       \int_{\zeta - \sigma (\theta, \zeta)}^{\zeta}
       \int_{\widetilde{\mf{b}}^{\ast}_- (\zeta')}^{\sigma (\theta, \zeta)}
       F^{\ast} (\theta, \zeta', \rho) [\phi_{\theta}] \frac{\dd
       \rho}{\rho} \dd \zeta',\\ &
       \mathrm{IV} (\theta, \zeta) \assign \frac{1}{\sigma (\theta, \zeta)}
       \int_{\zeta - \sigma (\theta, \zeta)}^{\zeta}
       \int_{\widetilde{\mf{b}}^{\ast}_-
       (\zeta')}^{\mf{\widetilde{\mf{b}}}^{\ast}_+ (\zeta')} F^{\ast} (\theta,
       \zeta', \rho) [\phi_{\theta}] \frac{\dd \rho}{\rho} \dd \zeta' .
     \end{aligned} \]
  Terms \(\mathrm{I} (\theta, \zeta)\) and \(\mathrm{II} (\theta, \zeta)\) are
  ``correction'' terms since we will show that they can be bounded by the size
  \(\left\| \1_{(V^+ \cap W^+) \setminus \left( V^- {\cup W^-}  \right)} \Emb
  [f] \right\|_{\SL{}^{(\infty, \infty)}_{\Theta} \Phi_{8 \gr}^{N - 8} (T)}\).
  We will then proceed to bound the latter \(2\) terms.
  
  For any value of \((\zeta, \rho)\) in the integral expression defining
  \(\mathrm{I} (\theta, \zeta)\), we have that
  \[ F^{\ast} (\theta, \zeta, \rho) [\phi_{\theta}] - \frac{1}{\sigma (\theta,
     \zeta)} \int_{\zeta - \sigma (\theta, \zeta)}^{\zeta} F^{\ast} (\theta,
     \zeta', \rho) [\phi_{\theta}] \dd \zeta' = F^{\ast} (\theta, \zeta,
     \rho) [\phi_{(\theta, \zeta)}] \]
  where
  \[ \phi_{(\theta, \zeta)} = \frac{1}{\sigma (\theta, \zeta)} \int_{\zeta -
     \sigma (\theta, \zeta)}^{\zeta} \left( \phi_{\theta} - e^{2 \pi i \theta
     \frac{(\zeta' - \zeta)}{\rho}} \Tr_{\frac{(\zeta - \zeta')}{\rho}}
     \phi_{\theta} \right) \dd \zeta' \]
  satisfies
  \[ \| \phi_{(\theta, \zeta)} (z) \|_{\Phi_{4 \gr}^{N - 4}} \lesssim \left|
     \frac{\sigma (\theta, \zeta)}{\rho} \right| \| \phi_{\theta} (z)
     \|_{\Phi_{4 \gr}^{N - 5}} . \]
  This holds since \
  \[ F^{\ast} (\theta, \zeta', \rho) [\phi_{\theta}] = F^{\ast} (\theta,
     \zeta, \rho) \left[ e^{2 \pi i \theta \frac{(\zeta' - \zeta)}{\rho}}
     \Tr_{\frac{(\zeta - \zeta')}{\rho}} \phi_{\theta} \right] \]
  for any \(\rho \in \left( \widetilde{\mf{b}}^{\ast}_- (\zeta') \vee
  \widetilde{\mf{b}}^{\ast}_- (\zeta), \widetilde{\mf{b}}^{\ast}_+ (\zeta')
  \wedge \widetilde{\mf{b}}^{\ast}_+ (\zeta) \right)\). Thus
  \[ \begin{aligned}[t]&
       | \mathrm{I} (\theta, \zeta) | \lesssim \int_{\sigma (\theta,
       \zeta)}^{\mf{\tilde{b}}^{\ast}_+ (\zeta) / C} \left| \frac{\sigma
       (\theta, \zeta)}{\rho} \right| \frac{\dd \rho}{\rho} 
       (\sup_{\tilde{\phi}} \sup_{\rho \in (0, 1)} | F^{\ast} (\theta, \zeta,
       \rho) [\tilde{\phi}] |)
     \end{aligned} \]
  with the upper bound taken over \(\tilde{\phi} \in \Phi^{\infty}_{4 \gr}\)
  with \(\| \tilde{\phi} \|_{\Phi_{4 \gr}^{N - 5}} \leq 1\). Integrating and
  using \Cref{lem:wave-packet-decomposition} this gives the bound
  \[ \| \mathrm{I} (\theta, \zeta) \|_{L^2_{\frac{\dd \theta \dd \zeta}{|
     \Theta |}} (\Theta \times B_1)} \lesssim \left\| \1_{(V^+ \cap W^+)
     \setminus \left( V^- {\cup W^-}  \right)} \Emb [f]
     \right\|_{\SL{}^{(\infty, \infty)}_{\Theta} \Phi_{8 \gr}^{N - 7} (T)} .
  \]
  The term \(\mathrm{II} (\zeta, \theta)\), again using
  \Cref{lem:wave-packet-decomposition}, satisfies
  \[ \begin{aligned}[t]&
       | \tmop{II} (\theta, \zeta) | \lesssim \frac{1}{\sigma (\theta, \zeta)}
       \int_{\zeta - \sigma (\theta, \zeta)}^{\zeta} \int_{\mf{b}^{\ast}_+
       (\zeta) / C}^{\mf{b}^{\ast}_+ (\zeta')}  (\sup_{\tilde{\phi}}
       \sup_{\rho \in (0, 1)} | F^{\ast} (\theta, \zeta, \rho) [\tilde{\phi}]
       |) \frac{\dd \rho}{\rho} \dd \zeta'
     \end{aligned} \]
  with the upper bound taken over \(\tilde{\phi} \in \Phi^{\infty}_{4 \gr}\)
  with \(\| \tilde{\phi} \|_{\Phi_{4 \gr}^{N - 5}} \leq 1\). Integrating and
  using \Cref{lem:wave-packet-decomposition} gives the bound
  \[ \| \mathrm{II} (\theta, \zeta) \|_{L^2_{\frac{\dd \theta \dd \zeta}{|
     \Theta |}} (\Theta \times B_1)} \lesssim \left\| \1_{(V^+ \cap W^+)
     \setminus \left( V^- {\cup W^-}  \right)} \Emb [f]
     \right\|_{\SL{}^{(\infty, \infty)}_{\Theta} \Phi_{8 \gr}^{N - 7} (T)} .
  \]
  Let us now deal with the term \(\mathrm{III} (\theta, \zeta)\). Recall (see
  \eqref{eq:embedded-no-defect}) that
  \[ F^{\ast} (\theta, \zeta', \rho) [\phi_{\theta}] = \wpD (\theta) F^{\ast}
     (\theta, \zeta', \rho) [\phi] = \rho \dd_{\zeta'} F^{\ast} (\theta,
     \zeta', \rho) [\phi] \]
  as long as \(\rho \in \left( \mf{b_-} (\theta, \zeta'), \mf{b_+} (\theta,
  \zeta') \right)\), which is the case for \(\zeta' \in (\zeta - \sigma (\theta,
  \zeta), \zeta)\). Integration by parts gives that
  \[ \mathrm{III} (\theta, \zeta) = \begin{aligned}[t]&
       - \int_{\widetilde{\mf{b}}^{\ast}_- (\zeta)}^{\sigma (\theta, \zeta)}
       \frac{\rho}{\sigma (\theta, \zeta)} F^{\ast} (\theta, \zeta, \rho)
       [\phi] \frac{\dd \rho}{\rho}\\ &
       + \int_{\widetilde{\mf{b}}^{\ast}_- (\zeta - \sigma (\theta,
       \zeta))}^{\sigma (\theta, \zeta)} \frac{\rho}{\sigma (\theta, \zeta)}
       F^{\ast} (\theta, \zeta - \sigma (\theta, \zeta), \rho) [\phi]
       \frac{\dd \rho}{\rho}\\ &
       - \frac{1}{\sigma (\theta, \zeta)} \int_{\zeta - \sigma (\theta,
       \zeta)}^{\zeta} \dd_{\zeta'} \widetilde{\mf{b}}^{\ast}_- (\zeta')
       F^{\ast} \left( \theta, \zeta', \widetilde{\mf{b}}^{\ast}_- (\zeta')
       \right) [\phi] \dd \zeta'
     \end{aligned} \]
  Since \(\left| \dd_{\zeta'} \widetilde{\mf{b}}^{\ast}_- (\zeta') \right| \leq
  1\) it holds that
  \[ \begin{aligned}[t]&
       \left| \frac{1}{\sigma (\theta, \zeta)} \int_{\zeta - \sigma (\theta,
       \zeta)}^{\zeta} \dd_{\zeta'} \widetilde{\mf{b}}^{\ast}_- (\zeta')
       F^{\ast} \left( \theta, \zeta', \widetilde{\mf{b}}^{\ast}_- (\zeta')
       \right) [\phi] \dd \zeta' \right|\\ &
       \qquad \leq \frac{1}{\sigma (\theta, \zeta)} \int_{\zeta - \sigma
       (\theta, \zeta)}^{\zeta} \left| F^{\ast} \left( \theta, \zeta',
       \widetilde{\mf{b}}^{\ast}_- (\zeta') \right) [\phi] \right| \dd
       \zeta' \lesssim \sup_{\zeta' \in B_1} \left| F^{\ast} \left( \theta,
       \zeta', \widetilde{\mf{b}}^{\ast}_- (\zeta') \right) [\phi] \right| .
     \end{aligned} \]
  The other two summands satisfy
  \[ \begin{aligned}[t]&
       \left| \int_{\widetilde{\mf{b}}^{\ast}_- (\zeta)}^{\sigma (\theta,
       \zeta)} \frac{\rho}{\sigma (\theta, \zeta)} F^{\ast} (\theta, \zeta,
       \rho) [\phi] \frac{\dd \rho}{\rho} \right| \lesssim \sup_{\rho \in
       \left( \widetilde{\mf{b}}^{\ast}_- (\zeta), \sigma (\theta, \zeta)
       \right)} | F^{\ast} (\theta, \zeta, \rho) [\phi] |
     \end{aligned} \]
  and
  \[ \begin{aligned}[t]&
       \left| \int_{\widetilde{\mf{b}}^{\ast}_- (\zeta - \sigma (\theta,
       \zeta))}^{\sigma (\theta, \zeta)} \frac{\rho}{\sigma (\theta, \zeta)}
       F^{\ast} (\theta, \zeta - \sigma (\theta, \zeta), \rho) [\phi]
       \frac{\dd \rho}{\rho} \right|\\ &
       \qquad \lesssim \sup_{\zeta' \in B_1} \sup_{\rho \in \left(
       \widetilde{\mf{b}}^{\ast}_- (\zeta'), \widetilde{\mf{b}}^{\ast}_+
       (\zeta') \right)} | F^{\ast} (\theta, \zeta', \rho) [\phi] | .
     \end{aligned} \]
  Thus \(\| \mathrm{III} (\zeta, \theta) \|_{L^2_{\frac{\dd \zeta \dd \theta}{|
  \Theta |}} (B_1 \times \Theta)} \lesssim \left\| \1_{V^+ \setminus (V^-
  \cup W )} \Emb [f] \right\|_{\SL^{(2, \infty)}_{\Theta}
  {\Phi_{\gr}^{\tilde{N} - 3}}'}\).
  
  It remains to deal with the term \(\mathrm{IV} (\theta, \zeta)\). Up to an
  application of the Hardy-Littlewood maximal function, \(\mathrm{IV} (\theta,
  \zeta)\) can be compared with the original quantity to bound but with \(\sigma
  (\theta, \zeta) = \mf{\tilde{b}_- (\zeta)}\). More explicitly \(| \mathrm{IV}
  (\theta, \zeta) | \leq \HLM (\widetilde{\mathrm{IV}} (\theta, \cdot))
  (\zeta)\) where
  \[ \widetilde{\mathrm{IV}} (\theta, \zeta) = \int_{\mf{\tilde{b}}^{\ast}_-
     (\zeta)}^{\mf{\tilde{b}}^{\ast}_+ (\zeta)} F^{\ast} (\theta, \zeta, \rho)
     [\phi_{\theta}] \frac{\dd \rho}{\rho} \]
  so, using the boundedness of the Hardy-Littlewood maximal function \(M\) on
  \(L^2\), we need to show that
  \[ \| \widetilde{\mathrm{IV}} (\theta, \zeta) \|_{L^2_{\frac{\dd \theta \dd
     \zeta}{| \Theta |}} (\Theta \times B_1)} \lesssim
     \RHS{\eqref{eq:tree-SIO-size}} . \]
  We have
  \[ \begin{aligned}[t]&
       \left| \int_{\mf{\tilde{b}}^{\ast}_- (\zeta)}^{\mf{\tilde{b}}^{\ast}_+
       (\zeta)} F^{\ast} (\theta, \zeta, \rho) [\phi_{\theta}] \frac{\dd
       \rho}{\rho} \right|^2\\ &
       \qquad = 2 \Re \left( \int_{\mf{\tilde{b}}^{\ast}_-
       (\zeta)}^{\mf{\tilde{b}}^{\ast}_+ (\zeta)} F^{\ast} (\theta, \zeta,
       \rho) [\phi_{\theta}] \int_{\rho }^{\mf{\tilde{b}}^{\ast}_+ (\zeta)}
       \overline{F^{\ast} (\theta, \zeta, \rho') [\phi_{\theta}]} \frac{\dd
       \rho'}{\rho} \frac{\dd \rho}{\rho} \right) .
     \end{aligned} \]
  As before, recall that
  \[ F^{\ast} (\theta, \zeta, \rho) [\phi_{\theta}] = \rho' \dd_{\zeta'}
     F^{\ast} (\theta, \zeta, \rho') [\phi] \]
  Integrating by parts we obtain that
  \[ \begin{aligned}[t]&
       \int_{B_1} \left| \int_{\mf{\tilde{b}}^{\ast}_-
       (\zeta)}^{\mf{\tilde{b}}^{\ast}_+ (\zeta)} F^{\ast} (\theta, \zeta,
       \rho) [\phi_{\theta}] \frac{\dd \rho}{\rho} \right|^2 \dd \zeta\\ &
       \qquad \begin{aligned}[t]&
         = 2 \Re \int_{B_1} \left( \int_{\mf{\tilde{b}}^{\ast}_-
         (\zeta)}^{\mf{\tilde{b}}^{\ast}_+ (\zeta)} F^{\ast} (\theta, \zeta,
         \rho) [\phi_{\theta}] \int_{\rho }^{\mf{\tilde{b}}^{\ast}_+ (\zeta)}
         \rho' \dd_{\zeta} \overline{F^{\ast} (\theta, \zeta, \rho') [\phi]}
         \frac{\dd \rho'}{\rho'} \frac{\dd \rho}{\rho} \right) \dd \zeta\\ &
         = \begin{aligned}[t]&
           - 2 \Re \int_{B_1} \left( \int_{\mf{\tilde{b}}^{\ast}_-
           (\zeta)}^{\mf{\tilde{b}}^{\ast}_+ (\zeta)} \rho \dd_{\zeta}
           F^{\ast} (\theta, \zeta, \rho) [\phi_{\theta}] \int_{\rho
           }^{\mf{\tilde{b}}^{\ast}_+ (\zeta)} \frac{\rho'}{\rho}
           \overline{F^{\ast} (\theta, \zeta, \rho') [\phi]} \frac{\dd
           \rho'}{\rho'} \frac{\dd \rho}{\rho} \right) \dd \zeta\\ &
           + 2 \Re \int_{B_1} \left( \dd_{\zeta} \mf{\tilde{b}}^{\ast}_-
           (\zeta) F^{\ast} (\theta, \zeta, \rho) [\phi_{\theta}]
           \int_{\mf{\tilde{b}}^{\ast}_- (\zeta)}^{\mf{\tilde{b}}^{\ast}_+
           (\zeta)} \frac{\rho'}{\mf{\tilde{b}}^{\ast}_- (\zeta)}
           \overline{F^{\ast} (\theta, \zeta, \rho') [\phi]} \frac{\dd
           \rho'}{\rho'} \right) \dd \zeta\\ &
           - 2 \Re \int_{B_1} \left( \int_{\mf{\tilde{b}}^{\ast}_-
           (\zeta)}^{\mf{\tilde{b}}^{\ast}_+ (\zeta)} F^{\ast} (\theta, \zeta,
           \rho) [\phi_{\theta}] \frac{\dd \rho}{\rho} \dd_{\zeta}
           \mf{\tilde{b}}^{\ast}_+ (\zeta) \overline{F^{\ast} \left( \theta,
           \zeta, \mf{\tilde{b}}^{\ast}_+ (\zeta) \right) [\phi]} \right) \dd
           \zeta .
         \end{aligned}
       \end{aligned}
     \end{aligned} \]
  Using the bounds \(\left| \dd_{\zeta} \mf{\tilde{b}}^{\ast}_- (\zeta) \right|
  \leq 1\), \(\left| \dd_{\zeta} \mf{\tilde{b}}^{\ast}_+ (\zeta) \right| \leq
  1\), the identity
  \[ \rho \dd_{\zeta} F^{\ast} (\theta, \zeta, \rho) [\phi_{\theta}] = \wpD
     (\theta) F^{\ast} (\theta, \zeta, \rho) [\phi_{\theta}] = F^{\ast}
     (\theta, \zeta, \rho) [\phi_{\theta}'] \]
  with \(\phi'_{\theta} (z) = (- d_z + 2 \pi i \theta) \phi_{\theta} (z)\), the
  fact that \(\| \phi_{\theta}' \|_{\Phi_{4 \gr}^{N - 4}} \lesssim \|
  \phi_{\theta} \|_{\Phi_{4 \gr}^{N - 4}} \lesssim \| \phi \|_{\Phi_{4 \gr}^{N
  - 4}} \leq 1\), and Young's convolution inequality in the multiplicative
  measure \(\frac{\dd \rho'}{\rho'}\) we can integrate in \((\theta, \zeta) \in
  \Theta \times B_1\) to obtain that
  \[ \begin{aligned}[t]&
       \int_{\Theta \times B_1} \left| \int_{\mf{\tilde{b}}^{\ast}_-
       (\zeta)}^{\mf{\tilde{b}}^{\ast}_+ (\zeta)} \rho \dd_{\zeta} F^{\ast}
       (\theta, \zeta, \rho) [\phi_{\theta}] \int_{\rho
       }^{\mf{\tilde{b}}^{\ast}_+ (\zeta)} \frac{\rho'}{\rho}
       \overline{F^{\ast} (\theta, \zeta, \rho') [\phi]} \frac{\dd
       \rho'}{\rho'} \frac{\dd \rho}{\rho} \right| \frac{\dd \theta \dd
       \zeta}{| \Theta |}\\ &
       \qquad \lesssim \left\| \1_{(V^+ \cap W^+) \setminus \left( V^- {\cup
       W^-}  \right)} \Emb [f] \right\|_{\SL^{(2, 2)}_{\Theta} \Phi_{4
       \gr}^{N - 4 }}^2,
     \end{aligned} \]
  \[ \begin{aligned}[t]&
       \int_{\Theta \times B_1} \left| \dd_{\zeta} \mf{\tilde{b}}^{\ast}_-
       (\zeta) F^{\ast} (\theta, \zeta, \rho) [\phi_{\theta}]
       \int_{\mf{\tilde{b}}^{\ast}_- (\zeta)}^{\mf{\tilde{b}}^{\ast}_+
       (\zeta)} \frac{\rho'}{\mf{\tilde{b}}^{\ast}_- (\zeta)}
       \overline{F^{\ast} (\theta, \zeta, \rho') [\phi]} \frac{\dd
       \rho'}{\rho'} \right| \frac{\dd \theta \dd \zeta}{| \Theta |}\\ &
       \qquad \lesssim \begin{aligned}[t]&
         \left\| \1_{(V^+ \cap W^+) \setminus \left( V^- {\cup W^-}  \right)}
         \Emb [f] \right\|_{\SL^{(2, 2)}_{\Theta} \Phi_{4 \gr}^{N - 4 }}\\ &
         \qquad \times \left\| \1_{(V^+ \cap W^+) \setminus \left( V^- {\cup
         W^-}  \right)} \Emb [f] \right\|_{\SL^{(\infty, \infty)}_{\Theta}
         \Phi_{4 \gr}^{N - 4 }},
       \end{aligned} 
     \end{aligned} \]
  and
  \[ \begin{aligned}[t]&
       \int_{\Theta \times B_1} \left| \int_{\mf{\tilde{b}}^{\ast}_-
       (\zeta)}^{\mf{\tilde{b}}^{\ast}_+ (\zeta)} F^{\ast} (\theta, \zeta,
       \rho) [\phi_{\theta}] \frac{\dd \rho}{\rho} \dd_{\zeta}
       \mf{\tilde{b}}^{\ast}_+ (\zeta) \overline{F^{\ast} \left( \theta,
       \zeta, \mf{\tilde{b}}^{\ast}_+ (\zeta) \right) [\phi]} \right|
       \frac{\dd \theta \dd \zeta}{| \Theta |}\\ &
       \qquad \lesssim \begin{aligned}[t]&
         \left\| \1_{(V^+ \cap W^+) \setminus \left( V^- {\cup W^-}  \right)}
         \Emb [f] \right\|_{\SL^{(\infty, \infty)}_{\Theta} \Phi_{4 \gr}^{N -
         4 }}\\ &
         \qquad \times \| \widetilde{\mathrm{IV}} (\zeta, \theta)
         \|_{L^2_{\frac{\dd \zeta \dd \theta}{| \Theta |}} (\Theta \times
         B_1)} .
       \end{aligned}
     \end{aligned} \]
  With the above we have shown that
  \[ \| \widetilde{\mathrm{IV}} (\zeta, \theta) \|_{L^2_{\frac{\dd \zeta \dd
     \theta}{| \Theta |}} (\Theta \times B_1)} \lesssim
     \RHS{\eqref{eq:tree-SIO-size}} \Big( \RHS{\eqref{eq:tree-SIO-size}} + \|
     \widetilde{\mathrm{IV}} (\zeta, \theta) \|_{L^2_{\frac{\dd \zeta \dd
     \theta}{| \Theta |}} (\Theta \times B_1)} \Big) . \]
  This implies that
  \[ \| \widetilde{\mathrm{IV}} (\zeta, \theta) \|_{L^2_{\frac{\dd \zeta \dd
     \theta}{| \Theta |}} (\Theta \times B_1)} \lesssim
     \RHS{\eqref{eq:tree-SIO-size}}, \]
  as required 
\end{proof}

\subsection{Counting function
improvement}\label{sec:non-unif:counting-function-improvement}

To conclude that \Cref{thm:non-uniform-embedding-bounds} holds we now need to
improve the bound \eqref{eq:embedding-bounds:non-uniform-iterated:mu1}. In
particular, we need to bound \(\Emb [f]\) in the localized outer Lebesgue
quasi-norm \(L^p_{\nu_1} \fL^{q, +}_{\mu_{\Theta}^{\infty}, \nu_1}
\SF_{\Theta}^u \Phi_{\mf{r}}^N\) in place of \(L^p_{\nu_1} \fL^{q,
+}_{\mu_{\Theta}^1, \nu_1} \SF_{\Theta}^u \Phi_{\mf{r}}^N\).

\begin{proof}{Proof of \Cref{thm:non-uniform-embedding-bounds}:
\(\mu^1_{\Theta}\) to \(\mu^{\infty}_{\Theta}\)}
  Fix \(q_0 \in [1, + \infty)\). We aim to show that for any \(q \in (q_0, +
  \infty]\) it holds that
  \[ \| F \|_{L^p_{\nu_1} \fL^{q, +}_{\mu^{\infty}_{\Theta}, \nu_1}
     \SF^u_{\Theta} \Phi_{\gr}^N (V )} \lesssim \| F \|_{L^p_{\nu_1} \fL^{q_0,
     +}_{\mu^1_{\Theta}, \nu_1} \SF^u_{\Theta} \Phi_{\gr}^N} \]
  for any \(p \in (1, + \infty)\). To do so let us use the restricted weak
  interpolation. To apply \Cref{prop:outer-restricted-interpolation} we can
  suppose that \(F\) is supported on \(V \in \DD^{\cup}_1\) and we need to check
  that
  \[ \| F \|_{L^p_{\nu_1} \fL^{q, +}_{\mu^{\infty}_{\Theta}, \nu_1}
     \SF^u_{\Theta} \Phi_{\gr}^N} \lesssim \nu_1 (V)^{\frac{1}{p}} \| F
     \|_{\fL^{q_0, +}_{\mu^1_{\Theta}, \nu_1} \SF^u_{\Theta} \Phi_{\gr}^N} \]
  for any \(p > 1\). If \(q = + \infty\) the claim is immediate because for any
  \(V^+\)
  \[ \| F \|_{\fL^{\infty}_{\mu^{\infty}_{\Theta}, \nu_1} \SF^u_{\Theta}
     \Phi_{\gr}^N (V^+)} = \left\| \1_{V^+} F \right\|_{\SF^u_{\Theta}
     \Phi_{\gr}^N} = \| F \|_{\fL^{\infty}_{\mu^1_{\Theta}, \nu_1}
     \SF^u_{\Theta} \Phi_{\gr}^N (V^+)} \leq \| F \|_{\fL^{q_0,
     +}_{\mu^1_{\Theta}, \nu_{\beta}} \SF^u_{\Theta} \Phi_{\gr}^N} . \]
  Thus it is sufficient to show that
  \begin{equation}
    \| F \|_{L^p_{\nu_1} \fL^q_{\mu^{\infty}_{\Theta}, \nu_1} \SF^u_{\Theta}
    \Phi_{\gr}^N} \lesssim \nu_1 (V)^{\frac{1}{p}} \| F \|_{\fL^{q_0,
    +}_{\mu^1_{\Theta}, \nu_1} \SF^u_{\Theta} \Phi_{\gr}^N} .
    \label{eq:mu1-muinfty:p-end}
  \end{equation}
  We will show that for any \(V^+ \in \mathbb{D}^{\cup}_1\) there exists
  \(V^{\tmop{ecc}} (V^+) \in \mathbb{D}^{\cup}_1\) with \(\nu_1 (V^{\tmop{ecc}}
  (V^+)) \leq \frac{1}{2} \nu_1 (V^+)\) such that
  \begin{equation}
    \left\| \1_{{V^+}  \setminus V^{\tmop{ecc}} (V^+)} F \right\|_{L^{\bar{q},
    \infty}_{\mu^{\infty}_{\Theta}} \SF^u_{\Theta} \Phi_{\gr}^N} \lesssim
    \left\| \1_{V^+} F \right\|_{\fL^{q_0, +}_{\mu^1_{\Theta}, \nu_1}
    \SF^u_{\Theta} \Phi_{\gr}^N} \label{eq:mu1-muinfty:eccept}
  \end{equation}
  for any \(q_0 < \bar{q} < q\). Let us show that this is sufficient. Start by
  setting \(V_0 = V\) and \(V_{k + 1} = V^{\tmop{ecc}} (V_k)\) to obtain that
  \[ F = \sum_{k = 0}^{\infty} \1_{V_k \setminus V_{k + 1}} F_k \]
  so by quasi-subadditivity for any \(\varepsilon > 0\) it holds that
  \[ \| F \|_{L^p_{\nu_1} \fL^q_{\mu^{\infty}_{\Theta}, \nu_1} \SF^u_{\Theta}
     \Phi_{\gr}^N} \lesssim \sum_{k = 0}^{\infty} 2^{\varepsilon k} \left\|
     \1_{V_k \setminus V_{k + 1}} F \right\|_{L^p_{\nu_1}
     \fL^q_{\mu^{\infty}_{\Theta}, \nu_1} \SF^u_{\Theta} \Phi_{\gr}^N} . \]
  By \Cref{def:localized-outer-lebesgue}, by the construction of \(V_k\), and by
  log-convexity of outer Lebesgue norms (\eqref{eq:outer-log-convexity} in
  \Cref{prop:outer-properties}) we have that
  \[ \begin{aligned}[t]&
       \left\| \1_{V_k \setminus V_{k + 1}} F \right\|_{L^{\infty}_{\nu_1}
       \fL^q_{\mu^{\infty}_{\Theta}, \nu_1} \SF^u_{\Theta} \Phi_{\gr}^N}
       \lesssim \left\| \1_{V_k \setminus V_{k + 1}} F \right\|_{L^{\bar{q},
       \infty}_{\mu^{\infty}_{\Theta}} \SF^u_{\Theta} \Phi_{\gr}^N}^{\bar{q} /
       q} \left\| \1_{V_k \setminus V_{k + 1}} F
       \right\|_{L^{\infty}_{\mu^{\infty}_{\Theta}} \SF^u_{\Theta}
       \Phi_{\gr}^N}^{1 - \bar{q} / q}\\ &
       \leq \left\| \1_{V_k} F \right\|_{\fL^{q_0}_{\mu^1_{\Theta}, \nu_1}
       \SF^u_{\Theta} \Phi_{\gr}^N}^{\bar{q} / q} \left\| \1_{V_k \setminus
       V_{k + 1}} F \right\|_{\fL^{\infty}_{\mu^1_{\Theta}, \nu_1}
       \SF^u_{\Theta} \Phi_{\gr}^N}^{1 - \bar{q} / q} \leq \| F \|_{\fL^{q_0,
       +}_{\mu^1_{\Theta}, \nu 1} \SF^u_{\Theta} \Phi_{\gr}^N} .
     \end{aligned} \]
  Since \(\1_{V_k \setminus V_{k + 1}} F\) is supported on \(V_k\) that satisfies
  \(\nu_1 (V_k) \leq 2^{- k} \nu (V)\) it holds that
  \[
  \begin{aligned}[t]
  & \| F \|_{L^p_{\nu_1} \fL^q_{\mu^{\infty}_{\Theta}, \nu_1} \SF^u_{\Theta}
     \Phi_{\gr}^N} \\ & \qquad\lesssim \sum_{k = 0}^{\infty} 2^{\varepsilon k} 2^{- k /
     p} \nu (V)^{1 / p} \left\| \1_{V_k} F \right\|_{^{} \fL^{q_0,
     +}_{\mu^1_{\Theta}, \nu_1} \SF^u_{\Theta} \Phi_{\gr}^N} \lesssim \nu
     (V)^{1 / p} \| F \|_{\fL^{q_0, +}_{\mu^1_{\Theta}, \nu_1} \SF^u_{\Theta}
     \Phi_{\gr}^N} 
  \end{aligned}
  \]
  as long as \(\varepsilon < \frac{1}{p}\). This is what was required by
  \eqref{eq:mu1-muinfty:p-end}.
  
  Let us now show that \(\eqref{eq:mu1-muinfty:eccept}\) holds. Assume by
  homogeneity that \(\| F \|_{\fL^{q_0, +}_{\mu^1_{\Theta}, \nu_1}
  \SF^u_{\Theta} \Phi_{\gr}^N} = 1\). Since \(\left\| \1_{V^+} F
  \right\|_{L^{q_0, \infty}_{\mu^1_{\Theta}} \SF^u_{\Theta} \Phi_{\gr}^N} \leq
  \left\| \1_{V^+} F \right\|_{L^{q_0}_{\mu^1_{\Theta}} \SF^u_{\Theta}
  \Phi_{\gr}^N}\), for any \(k \in \N\) there exists \(W_k \in
  \TT^{\cup}_{\Theta}\) such that
  \[ \left\| \1_{V^+ \setminus W_k} F \right\|_{\SF^u_{\Theta} \Phi_{\gr}^N}
     \leq 2^{- \frac{k}{q_0}}, \qquad \mu^1_{\Theta} (W_k) \leq 2^k \nu_1
     (V^+) . \]
  Let \(\mc{T}_k^1 \subset \TT_{\Theta}\) be such that \(W_k \subset \bigcup_{T
  \in \mc{T}_k^1} T\) and \(\left\| N_{\mc{T}_k^1} \right\|_{L^1} \lesssim 2^k
  \nu_1 (V^+)\). Set
  \[ \mathbb{V}_k = \left\{ x \in \R : \HLM N_{\mc{T}_k^1} > C
     2^{\frac{\bar{q}}{q_0} k} \right\} \]
  and using the fact that \(\mathbb{V}_k\) is open let us represent it as a
  countable union of its connected components (open balls):
  \[ \mathbb{V}_k = \bigcup_{n \in \N} B_{s_{k, n}} (x_{k, n}) . \]
  Setting \(V_k = \bigcup_{n \in \N} D (x_{k, n}, 100 s_{k, n})\) and
  \(V^{\tmop{ecc}} (V^+) = \bigcup_{k \in \N} V_k\), by the boundedness of the
  Hardy Littlewood maximal function we have that
  \[ \nu_1 (V^{\tmop{ecc}} (V^+)) \leq \sum_{k \in \N} \nu_1 (V_k) \lesssim
     \sum_{k \in \N} C^{- 1} 2^{k \left( 1 - \frac{\bar{q}}{q_0} \right)}
     \nu_1 (V^+) \lesssim C^{- 1} \nu_1 (V^+) \]
  as required, as long as \(C > 1\) is chosen large enough. Let us now show that
  \[ \left\| \1_{V^+ \setminus (V^{\tmop{ecc}} (V^+) \cup W_k^{\infty})} F
     \right\|_{\SF^u_{\Theta} \Phi_{\gr}^N} \leq 2^{- \frac{k}{q_0}} \text{
     and } \mu^{\infty}_{\Theta} (W_k^{\infty}) \lesssim
     2^{\frac{\bar{q}}{q_0} k} \]
  where
  \[ \begin{aligned}[t]&
       W^{\infty}_k \eqd \bigcup_{T \in \mc{T}^{\infty}_k} T\\ &
       \mc{T}^{\infty}_k \eqd \begin{aligned}[t]&
         \left( \bigcup_{n \in \N} \left\{ T \in \mc{T}^1_k \suchthat B_{s_T}
         (x_T) \cap B_{s_{k, n}} (x_{k, n}) \neq \emptyset, s_T > 10 s_{k, n}
         \right\} \right)\\ &
         \qquad \cup \left\{ T \in \mc{T}^1_k \suchthat B_{s_T} (x_T) \cap V_k
         = \emptyset \right\} .
       \end{aligned}
     \end{aligned} \]
  The two bounds imply that \(\left\| \1_{V^+ \setminus V^{\tmop{ecc}} (V^+)} F
  \right\|_{L^{\bar{q}, \infty}_{\mu^{\infty}_{\Theta}}}^{\bar{q}} \lesssim
  1\), concluding the proof of \(\eqref{eq:mu1-muinfty:eccept}\). The former bound
  holds because \(W_k \subset V_k \cup W^{\infty}_k\). As a matter of fact, \(W_k
  \subset \bigcup_{T \in \mc{T}_k^1} T\) and if \(T \in \mc{T}_k^1 \setminus
  \mc{T}^{\infty}_k\) then there exists \(n \in \N\) such that \(B_{s_T} (x_T)
  \cap B_{s_{k, n}} (x_{k, n})\) and \(s_T < 10 s_{k, n}\) but then \(B_{s_T}
  (x_T) \subset B_{100 s_{k, n}} (x_{k, n})\) and thus \(T \subset V_k\). \
  
  To show that \(\mu^{\infty}_{\Theta} (W_k^{\infty}) \lesssim
  2^{\frac{\bar{q}}{q_0} k}\) it is sufficient to show that
  \(N_{\mc{T}_k^{\infty}} (x) \lesssim 2^{\frac{\bar{q}}{q_0} k}\) for all \(x
  \in \R\). If \(x \notin \mathbb{V}_k\) then the claim follows by the definition
  of \(\mathbb{V}_k\) and from the fact that \(N_{\mc{T}_k^{\infty}} (x) \leq
  N_{\mc{T}_k^1} (x) \leq \HLM N_{\mc{T}_k^1} (x)\). If \(x \in
  \mathbb{V}_k\)then \(x \in B_{s_{k, n}} (x_{k, n})\) for some \(n \in \N\). Let
  \[ \mc{T}_k^{\infty} (x) \eqd \left\{ T \in \mc{T}_k^{\infty} \suchthat
     B_{s_T} (x_T) \ni x \right\} \]
  so that \(N_{\mc{T}_k^{\infty}} (x)\) coincides with the cardinality of
  \(\mc{T}_k^{\infty} (x)\). Given the way \(\mc{T}_k^{\infty}\) was defined for
  any \(T \in \mc{T}_k^{\infty} (x)\) it holds that \(s_T > 10 s_{k, n}\) and
  thus either \(B_{s_T} (x_T) \supset (x, x + 4 s_T)\) or  \(B_{s_T} (x_T)
  \supset (x - 4 s_T, x)\) showing that
  \[ \int_{B_{3 s_{k, n}} (x_{k, n})} N_{\mc{T}^{\infty}_k} (y) \dd y > 3
     s_{k, n} N_{\mc{T}_k^{\infty}} (x) . \]
  Thus for any \(y \in B_{3 s_{k, n}} (x_{k, n})\) it holds that \(\HLM
  N_{\mc{T}_k^1} (y) \gtrsim N_{\mc{T}_k^{\infty}} (x)\) while we know from the
  maximality of \(B_{s_{k, n}} (x_{k, n})\) that there exists \(y \in B_{3 s_{k,
  n}} (x_{k, n}) \setminus B_{s_{k, n}} (x_{k, n})\) with \(\HLM N_{\mc{T}_k^1}
  (y) < C 2^{\frac{\bar{q}}{q_0} k}\), showing the claimed bound
  \(N_{\mc{T}_k^{\infty}} (x) \lesssim 2^{\frac{\bar{q}}{q_0} k}\).
\end{proof}

\section{Uniform embedding bounds - linear case
}\label{sec:uniform-embeddings:linear}

In this section we prove \Cref{thm:uniform-embedding-bounds:linear} about the
linear embedding bounds \eqref{eq:embedding-bounds:uniform-iterated}. of.
Recall from the setup of \Cref{thm:uniform-embedding-bounds:linear} that
\(\Gamma = \Gamma_{(\alpha, \beta, \gamma)}\) is as described \eqref{eq:gamma},
\(\Theta\) is an bounded open interval with \(B_4 \subsetneq \Theta \subsetneq
\R\), and let \(\Theta^{\tmop{in}}\) is a bounded open interval with \(B_{2^{- 5}}
(- \gamma) \subsetneq \Theta^{\tmop{in}} \subsetneq B_{2^{- 3}} (- \gamma)\)
and let \(\Theta^{\tmop{ex}} \eqd \Theta \setminus \Theta^{\tmop{in}}\). We will
consistently use the notation \(\theta_{\Gamma} \eqd \alpha \beta (\theta +
\gamma)\) as introduced in \Cref{sec:derived-sizes}. Furthermore
\(\Theta_{\Gamma} \eqd \alpha \beta (\Theta + \gamma)\),
\(\Theta^{\tmop{in}}_{\Gamma} \eqd \{ \theta_{\Gamma} : \theta \in
\Theta^{\tmop{in}} \}\), and \(\Theta^{\tmop{ex}}_{\Gamma} \eqd \{
\theta_{\Gamma} : \theta \in \Theta^{\tmop{ex}} \}\). As in the rest of the
paper we assume that the parameter \(N\) is large enough and \(\gr > 0\) is small
enough. These parameters, that govern the rate of decay and the Fourier
support of wave packets \(\phi\) appear in the size \(\widetilde{\SF}^u_{\Gamma}
\Phi_{\mf{r}}^N\) and are ubiquitous in the subsequent section. The values \(N >
100\) and \(\gr < 2^{- 10}\) are, in particular, always admissible and the choice
of \(N\) and \(\gr\) can be made independently of \(\beta \in (0, 1]\). All implicit
constants appearing in \Cref{thm:uniform-embedding-bounds:linear} and in this
entire section are uniform in \(\beta \in (0, 1]\). The constants are also
uniform in \(\Theta\) and \(\gr\) as long as \(| \Theta |\) is uniformly bounded and
\(\gr > 0\) is bounded away from \(0\).

While \Cref{thm:uniform-embedding-bounds:linear} holds for all \(\beta \in (0,
1]\), we are mainly interested in the case \(0 < \beta \ll 1\).

In this section, for simplicity, the reader may assume that \(\alpha = \beta^{-
1}\) and \(\gamma = 0\). This maintains the important features of the subsequent
discussion and of the proofs but avoids unnecessary complications. Under these
assumption one has that \(\theta_{\Gamma} = \theta\), \(\Theta_{\Gamma} =
\Theta\). There does not, however, seem to be a direct way to deduce the
statements below and \Cref{thm:uniform-embedding-bounds:linear} itself in
their full generality from these restricted assumptions on the parameters
\((\alpha, \beta, \gamma)\).

The boundedness properties the map \(F \mapsto F \circ \Gamma\) are encoded in
the two propositions below. They allow us to control outer Lebesgue
quasi-norms with two auxiliary sizes that do not appear, per se, in the size
\(\widetilde{\SF}^u_{\Gamma} \Phi_{\mf{r}}^N\). However, for functions of the
form \(F = \Emb [f]\) these auxiliary sizes can be used to control outer
Lebesgue quasi-norms with size \(\widetilde{\SF}^u_{\Gamma} \Phi_{\mf{r}}^N\).

\begin{proposition}[Uniform bounds for exterior size]
  \label{prop:unif-gamma-bounds:Sex}The bounds
  \begin{equation}
    \begin{aligned}[t]&
      \| F \circ \Gamma \|_{L^r_{\mu_{\Theta}^1} \Gamma^{\ast} \SL^{(u,
      2)}_{(\Theta, \Theta^{\tmop{ex}})} \Phi_{\gr}^N} \lesssim \| F
      \|_{L^r_{\mu_{\Theta_{\Gamma}}^1} \left( \SL^{(\infty,
      \infty)}_{\Theta_{\Gamma}} + \SL^{(\bar{u}, 2)}_{(\Theta_{\Gamma},
      \Theta^{\tmop{ex}}_{\Gamma})} \right) \Phi_{2 \gr}^{N - 4}},
    \end{aligned} \label{eq:unif-gamma-bounds:Sex}
  \end{equation}
  hold for all \(F \in L^{\infty}_{\tmop{loc}} (\R^{3}_{+}) {\otimes
  \Phi_{2 \gr}^{\infty}}'\) as long as \(2 \leq u < r \leq + \infty\), and as
  long as \(\bar{u} \in (u, \infty)\) is sufficiently large, depending on \((u,
  r)\). 
\end{proposition}

\begin{proposition}[Uniform bounds for singular size]
  \label{prop:unif-gamma-bounds:singular-size}Let \(V  \in
  \mathbb{D}^{\cup}_{\beta}\), \(W  \in \TT_{\Theta}^{\cup}\) and \(E = V \cap W\)
  or \(E = V \cup W\). The bounds
  \begin{equation}
    \begin{aligned}[t]&
      \left\|  F \circ \Gamma (\eta, y, t) t \dd_t \1_E (\eta, y, t)
      \right\|_{L^r_{\mu_{\Theta}^1} \Gamma^{^{\ast}} \SL^{(u, 1)}_{(\Theta,
      \Theta^{\tmop{in}} )} \Phi_{\gr}^N} \qquad\\ &
      \lesssim \| F (\eta, y, t)  \|_{L^r_{\mu_{\Theta_{\Gamma}}^1}
      \SL^{(\infty, \infty)}_{\Theta_{\Gamma}} \Phi_{2 \gr}^{N - 4}}
    \end{aligned} \label{eq:unif-gamma-bounds:singular-size}
  \end{equation}
  hold for all \(F \in L^{\infty}_{\tmop{loc}} (\R^{3}_{+}) {\otimes
  \Phi_{2 \gr}^{\infty}}'\) as long as \(1 < u < r \leq \infty\). The implicit
  constant is independent of \(E\). The derivative in \(t \dd_t \1_E (\eta, y,
  t)\) is intended in the distributional sense and is a Radon measure on
  \(\R^3_+\).
\end{proposition}

These two statements are proved in \Cref{sec:bound-exterior-size} and in
\Cref{sec:unif-bounds:singular} respectively. The two sections follow a
similar outline: first we prove a geometric covering lemma on \(\R^3_+\)
(Lemmata \ref{lem:ext-tree-covering} and \ref{lem:singular-tree-covering}),
from this we deduce a weak-endpoint result for scalar valued functions \(F \in
\Rad (\R^{3}_{+})\). Finally, using restricted weak interpolation and
\Cref{lem:wave-packet-decomposition} we deduce the desired claim.

The next lemma shows that sizes \(\Gamma^{\ast} \SL^{(u, 2)}_{(\Theta,
\Theta^{\tmop{ex}})} \Phi_{\gr}^N\) and \(\Gamma^{^{\ast}} \SL^{(u,
1)}_{(\Theta, \Theta^{\tmop{in}} )} \Phi_{\gr}^N\) can be used to control the
size \(\widetilde{\SF}^u_{\Gamma} \Phi_{\mf{r}}^N\), that appears in
\Cref{thm:uniform-embedding-bounds:linear}.

\begin{lemma}[Uniform bounds for the \(\widetilde{\SF}_{\Gamma}^u \Phi_{\gr}^N\)
size]
  \label{lem:unif-derived-size-bound}Let \(V^+, V^- \in
  \mathbb{D}^{\cup}_{\beta}\) and \(W^+ \in \TT_{\Theta}^{\cup}\). The bounds
  \[ \begin{aligned}[t]&
       \Bigl\| \1_{(V^+ \cap W^+) \setminus V^-} \Emb [f] \circ \Gamma
       \Bigr\|_{L^r_{\mu^1_{\Theta}} \widetilde{\SF}^u_{\Gamma}
       {\Phi_{\mf{r}}^N} }\\ &
       \lesssim \begin{aligned}[t]&
         \Bigl\| \1_{(V^+ \cap W^+) \setminus V^-} \Emb [f] \circ \Gamma
         \Bigr\|_{L^r_{\mu^1_{\Theta}} \Gamma^{\ast} \SL^{(u, 2)}_{(\Theta,
         \Theta^{\tmop{ex}})} \Phi_{16 \gr}^{N - 10}}\\ &
         + \Bigl\| \1_{\overline{(V^+ \cap W^+)} \setminus V^-} \Emb [f]
         \circ \Gamma (\eta, y, t) t \dd_t \1_{(V^+ \cap W^+)} (\eta, y, t) 
         \Bigr\|_{L^r_{\mu^1_{\Theta}} {\Gamma^{\ast}}  \SL^{(u,
         1)}_{(\Theta, \Theta^{\tmop{in}})} \Phi_{16 \gr}^{N - 10}}\\ &
         + \Bigl\| \1_{\overline{(V^+ \cap W^+)} \setminus V^-} \Emb [f]
         \circ \Gamma (\eta, y, t) t \dd_t \1_{\left( \R^3_+ \setminus V^-
         \right)} (\eta, y, t)  \Bigr\|_{L^r_{\mu^1_{\Theta}} {\Gamma^{\ast}}
         \SL^{(u, 1)}_{(\Theta, \Theta^{\tmop{in}})} \Phi_{16 \gr}^{N - 10}}
       \end{aligned}
     \end{aligned} \]
  hold for any \(f \in \Sch (\R)\) and for all \(r \in (u, \infty]\).
  The implicit constant is independent of \(V^+\), \(V^-\), and \(W^+\).
\end{lemma}

We prove this lemma in \Cref{sec:unif-derived-size-bound}. Now we conclude by
showing how \Cref{prop:unif-gamma-bounds:Sex},
\Cref{prop:unif-gamma-bounds:singular-size}, and
\Cref{lem:unif-derived-size-bound} can be used to deduce the embedding bounds
\eqref{eq:embedding-bounds:uniform-iterated} of
\Cref{thm:uniform-embedding-bounds:linear}.

\begin{proof}{Proof of \Cref{thm:uniform-embedding-bounds:linear}}
  Using homogeneity let us assume that \(\| f_j \|_{L^{p_j} (\R)}
  = 1\). According to the discussion in Remark
  \Cref{rmk:iterated-outer-bounds-explicit} we need to exhibit a choice of
  sets \(V^- (\lambda) \in \DD^{\cup}_{\beta}\) such that
  \[ \begin{aligned}[t]&
       \int_0^{\infty} \nu_{\beta} \left( {V^-}  (\lambda^{1 / p}) \right)
       \dd \lambda \lesssim 1,\\ &
       \left\| \1_{\R^3_+ \setminus V^- (\lambda)}  \left( \Emb [f] \circ
       \Gamma \right) \right\|_{X^{q, r, +}_{\mu^1_{\Theta}, \nu_{\beta}}
       \widetilde{\SF}^u_{\Gamma} \Phi_{\mf{r}}^N} \lesssim \lambda .
     \end{aligned} \]
  Using \Cref{thm:non-uniform-embedding-bounds}. with \(N + 14\), \(32 \gr\) and
  some large enough \(\bar{u} \in (2, \infty)\) in place of \(N\), \(\gr\), and \(u\),
  we can find \(\widetilde{V^-} (\lambda) \in \DD^{\cup}_1\) such that
  \[ \int_0^{\infty} \nu_{\beta} (\widetilde{V^-} (\lambda^{1 / p})) \dd
     \lambda \lesssim 1, \qquad \left\| \1_{\R^3_+ \setminus \widetilde{V^-}
     (\lambda)}  \Emb [f] \right\|_{\fL^{q, +}_{\mu^{\textit{}
     \infty}_{\Theta_{\Gamma}}, \nu_1} \SF^{\bar{u}}_{\Theta} \Phi_{32 \gr}^{N
     + 14}} \lesssim \lambda . \]
  Recall that \(\SF^{\bar{u}}_{\Theta} \Phi_{32 \gr}^{N + 14} \gtrsim
  \SL^{(\infty, \infty)}_{\Theta_{\Gamma}} \Phi_{32 \gr}^{N + 14} +
  \SL^{(\bar{u}, 2)}_{(\Theta_{\Gamma}, \Theta^{\tmop{ex}}_{\Gamma})} \Phi_{32
  \gr}^{N + 14}\) because \(\SL^{(\bar{u}, 2)}_{\Theta} \wpD_{32 \gr}^{N + 14}
  \gtrsim \SL^{(\bar{u}, 2)}_{(\Theta_{\Gamma}, \Theta^{\tmop{ex}}_{\Gamma})}
  \Phi_{32 \gr}^{N + 14}\) since \(\Theta^{\tmop{ex}}_{\Gamma} \cap B_{64 \gr} =
  \emptyset\). We set \(V^- (\lambda) \eqd \Gamma^{- 1} (\widetilde{V^-}
  (\lambda) ) \in \DD^{\cup}_{\beta}\) and we claim that
  \[ \begin{aligned}[t]&
       \nu_{\beta} (V^- (\lambda)) = \nu_1 (\widetilde{V^-} (\lambda) )\\ &
       \left\| \1_{\R^3_+ \setminus V^- (\lambda)} \left( \Emb [f] \circ
       \Gamma \right) \right\|_{X^{q, r, +}_{\mu^1_{\Theta}, \nu_{\beta}}
       \widetilde{\SF}^u_{\Gamma} \Phi_{\mf{r}}^N} \lesssim \lambda .
     \end{aligned} \]
  The first bound follows since \(\Gamma (D_{\beta} (x, s)) = D_1 (x, s)\) and
  \\ \(\nu_{\beta} (D_{\beta} (x, s)) = \nu_1 (D_1 (x, s)) = s\). To show the
  latter bound, according to Remark \Cref{rmk:iterated-outer-bounds-explicit},
  we need to prove that for any \(V^+ \in \DD^{\cup}_{\beta}\) and \(W^+ \in
  \TT_{\Theta}^{\cup}\), and for \(\bar{r} \in \{ r, q \}\) the bounds
  \begin{equation}
    \left\| \1_{(V^+ \cap W^+) \setminus V^- (\lambda)} \left( \Emb [f] \circ
    \Gamma \right) \right\|_{L^{\bar{r}}_{\mu^1_{\Theta}}
    \widetilde{\SF}^u_{\Gamma} \Phi_{\mf{r}}^N} \lesssim \lambda \nu_{\beta}
    (V^+)^{\frac{1}{\bar{r}}} \mu^{\infty}_{\Theta} (W^+)^{\frac{1}{\bar{r}} -
    \frac{1}{q}} \label{eq:ueb:local-claim-r}
  \end{equation}
  hold. \Cref{lem:unif-derived-size-bound} allows us, instead of working with
  the size \(\widetilde{\SF}^u_{\Gamma} \Phi_{\mf{r}}^N\) directly, to work
  with two alternative quantities: specifically those appearing in
  {\LHS{\eqref{prop:unif-gamma-bounds:Sex}}} and
  {\LHS{\eqref{prop:unif-gamma-bounds:singular-size}}}. More precisely,
  according to \Cref{lem:unif-derived-size-bound} it is sufficient to prove
  that
  \begin{equation}
    \Bigl\| \1_{(V^+ \cap W^+) \setminus V^- (\lambda)} \Emb [f] \circ \Gamma
    \Bigr\|_{L^{\bar{r}}_{\mu^1_{\Theta}} \Gamma^{\ast} \SL^{(u,
    2)}_{(\Theta, \Theta^{\tmop{ex}})} \Phi_{16 \gr}^{N - 10}} \lesssim
    \lambda \nu_{\beta} (V^+)^{\frac{1}{\bar{r}}} \mu^{\infty}_{\Theta}
    (W^+)^{\frac{1}{\bar{r}} - \frac{1}{q}}, \label{eq:ueb:local-claim-Sex}
  \end{equation}
  and that
  \begin{equation}
    \begin{aligned}[t]&
      \begin{aligned}[t]&
        \Bigl\| \1_{\overline{(V^+ \cap W^+)} \setminus V^- (\lambda)} \Emb
        [f] \circ \Gamma (\eta, y, t) t \dd_t \1_{(V^+ \cap W^+)} (\eta, y, t)
        \Bigr\|_{L^{\bar{r}}_{\mu^1_{\Theta}} {\Gamma^{\ast}}  \SL^{(u,
        1)}_{(\Theta, \Theta^{\tmop{in}})} \Phi_{16 \gr}^{N - 10}}\\ &
        \qquad + \Bigl\| \1_{\overline{(V^+ \cap W^+)} \setminus V^-
        (\lambda)} \Emb [f] \circ \Gamma (\eta, y, t) t \dd_t \1_{\left(
        \R^3_+ \setminus V^- \right)} (\eta, y, t) 
        \Bigr\|_{L^{\bar{r}}_{\mu^1_{\Theta}} {\Gamma^{\ast}}  \SL^{(u,
        1)}_{(\Theta, \Theta^{\tmop{in}})} \Phi_{16 \gr}^{N - 10}}
      \end{aligned}\\ &
      \qquad \lesssim \lambda \nu_{\beta} (V^+)^{\frac{1}{\bar{r}}}
      \mu^{\infty}_{\Theta} (W^+)^{\frac{1}{\bar{r}} - \frac{1}{q}} .
    \end{aligned} \label{eq:ueb:local-claim-sing}
\end{equation}

  Since \Cref{prop:unif-gamma-bounds:Sex} and
  \Cref{prop:unif-gamma-bounds:singular-size} show that
  \[ \LHS{} \eqref{eq:ueb:local-claim-Sex} + \LHS{}
     \eqref{eq:ueb:local-claim-sing} \leq \left\| \1_{\Gamma ((V^+ \cap W^+)
     \setminus V^- (\lambda))} \Emb [f]
     \right\|_{L^{\bar{r}}_{\mu_{\Theta_{\Gamma}}^1} \left( \SL^{(\infty,
     \infty)}_{\Theta_{\Gamma}} + \SL^{(\bar{u}, 2)}_{(\Theta_{\Gamma},
     \Theta^{\tmop{ex}}_{\Gamma})} \right) \Phi_{32 \gr}^{N - 14}} \]
  as long as \(\bar{u}\) is large enough, all we need to do is show that
  \[
  \begin{aligned}[t] & 
  \left\| \1_{\Gamma ((V^+ \cap W^+) \setminus V^- (\lambda))} \Emb [f]
     \right\|_{L^{\bar{r}}_{\mu_{\Theta_{\Gamma}}^1} \left( \SL^{(\infty,
     \infty)}_{\Theta_{\Gamma}} + \SL^{(\bar{u}, 2)}_{(\Theta_{\Gamma},
     \Theta^{\tmop{ex}}_{\Gamma})} \right) \Phi_{32 \gr}^{N - 14}} \\ & \qquad \lesssim
     \lambda \nu_{\beta} (V^+)^{\frac{1}{\bar{r}}} \mu^{\infty}_{\Theta}
     (W^+)^{\frac{1}{\bar{r}} - \frac{1}{q}}, 
  \end{aligned}
  \]
  which will be the goal for the rest of the proof. This follows essentially
  by \eqref{eq:outer-Lp-finite-support}, up to controlling
  \(\mu^1_{\Theta_{\Gamma}} (W^+)\). More precisely, since \(\bar{r} \leq q\), by
  \eqref{eq:outer-Lp-finite-support} it holds that
  \[ \begin{aligned}[t]&
       \left\| \1_{\Gamma ((V^+ \cap W^+) \setminus V^- (\lambda))} \Emb [f]
       \right\|_{L^{\bar{r}}_{\mu_{\Theta_{\Gamma}}^1} \left( \SL^{(\infty,
       \infty)}_{\Theta_{\Gamma}} + \SL^{(\bar{u}, 2)}_{(\Theta_{\Gamma},
       \Theta^{\tmop{ex}}_{\Gamma})} \right) \Phi_{32 \gr}^{N - 14}}\\ &
       \lesssim \mu^1_{\Theta_{\Gamma}} (\Gamma (V^+) \cap \Gamma
       (W^+))^{\frac{1}{\bar{r}} - \frac{1}{q}} \left\| \1_{\Gamma  ((V^+
       \cap W^+) \setminus V^- (\lambda))} F
       \right\|_{L^q_{\mu_{\Theta_{\Gamma}}^1} \left( \SL^{(\infty,
       \infty)}_{\Theta_{\Gamma}} + \SL^{(\bar{u}, 2)}_{(\Theta_{\Gamma},
       \Theta^{\tmop{ex}}_{\Gamma})} \right) \Phi_{32 \gr}^{N - 14}},
     \end{aligned} \]
  while \Cref{lem:measure-comparison-1-infty} applied with \(\beta = 1\) gives
  that
  \[ \mu^1_{\Theta_{\Gamma}} (\Gamma (V^+) \cap W) \lesssim \nu_1 (\Gamma
     (V^+)) \mu^{\infty}_{\Theta_{\Gamma}} (\Gamma (V^+) \cap W) \quad \forall
     W \in \TT^{\cup}_{\Theta_{\Gamma}} . \]
  Using the above bound and \eqref{eq:outer-measure-Lp-comparison} we obtain
  that
  \[ \begin{aligned}[t]&
       \left\| \1_{\Gamma  ((V^+ \cap W^+) \setminus V^-)} F
       \right\|_{L^q_{\mu_{\Theta_{\Gamma}}^1} \left( \SL^{(\infty,
       \infty)}_{\Theta_{\Gamma}} + \SL^{(\bar{u}, 2)}_{(\Theta_{\Gamma},
       \Theta^{\tmop{ex}}_{\Gamma})} \right) \Phi_{32 \gr}^{N - 14}}\\ &
       \qquad \lesssim \nu_1 (\Gamma (V^+))^{\frac{1}{q}} \left\| \1_{\Gamma 
       ((V^+ \cap W^+) \setminus V^-)} F
       \right\|_{L^q_{\mu_{\Theta_{\Gamma}}^{\infty}} \left( \SL^{(\infty,
       \infty)}_{\Theta_{\Gamma}} + \SL^{(\bar{u}, 2)}_{(\Theta_{\Gamma},
       \Theta^{\tmop{ex}}_{\Gamma})} \right) \Phi_{32 \gr}^{N - 14}} .
     \end{aligned} \]
  We also have that
  \[ \mu^1_{\Theta_{\Gamma}} (\Gamma (V^+) \cap \Gamma
     (W^+))^{\frac{1}{\bar{r}} - \frac{1}{q}} \leq \nu_{\beta}
     (V^+)^{\frac{1}{\bar{r}} - \frac{1}{q}} \mu^{\infty}_{\Theta}
     (W^+)^{\frac{1}{\bar{r}} - \frac{1}{q}} \]
  because \(\Gamma (V^+) \in \DD^{\cup}_1\), \(\Gamma (W^+) \in
  \TT_{\Theta_{\Gamma}}^{\cup}\), \(\nu_1 (\Gamma (V^+)) = \nu_{\beta} (V^+)\),
  and \\\(\mu^{\infty}_{\Theta} (W^+) = \mu^{\infty}_{\Theta_{\Gamma}} (\Gamma
  (W^+))\). The estimates above, combined, give that
  \[ \begin{aligned}[t]&
       \left\| \1_{\Gamma ((V^+ \cap W^+) \setminus V^-)} \Emb [f]
       \right\|_{L^{\bar{r}}_{\mu_{\Theta_{\Gamma}}^1} \left( \SL^{(\infty,
       \infty)}_{\Theta_{\Gamma}} + \SL^{(\bar{u}, 2)}_{(\Theta_{\Gamma},
       \Theta^{\tmop{ex}}_{\Gamma})} \right) \Phi_{32 \gr}^{N - 14}}\\ &
        \lesssim \nu_{\beta} (V^+)^{\frac{1}{\bar{r}}}
       \mu^{\infty}_{\Theta} (W^+)^{\frac{1}{\bar{r}} - \frac{1}{q}} \left\|
       \1_{\Gamma  ((V^+ \cap W^+) \setminus V^-)} \Emb [f]
       \right\|_{L^q_{\mu_{\Theta_{\Gamma}}^{\infty}} \left( \SL^{(\infty,
       \infty)}_{\Theta_{\Gamma}} + \SL^{(\bar{u}, 2)}_{(\Theta_{\Gamma},
       \Theta^{\tmop{ex}}_{\Gamma})} \right) \Phi_{32 \gr}^{N - 14}} .
     \end{aligned} \]
  This is the required bound \eqref{eq:ueb:local-claim-Sex} since \(\Gamma (V^-
  (\lambda)) = \widetilde{V^-} (\lambda)\) so
  \[ \begin{aligned}[t]&
       \left\| \1_{\Gamma  ((V^+ \cap W^+) \setminus V^- (\lambda))} \Emb [f]
       \right\|_{L^q_{\mu_{\Theta_{\Gamma}}^{\infty}} \left( \SL^{(\infty,
       \infty)}_{\Theta_{\Gamma}} + \SL^{(\bar{u}, 2)}_{(\Theta_{\Gamma},
       \Theta^{\tmop{ex}}_{\Gamma})} \right) \Phi_{32 \gr}^{N - 14}}\\ &
       \qquad \lesssim \left\| \1_{\R^3_+ \setminus \widetilde{V^-} (\lambda)}  \Emb
       [f] \right\|_{\fL^{q, +}_{\mu^{\textit{} \infty}_{\Theta_{\Gamma}},
       \nu_1} \SF^{\bar{u}}_{\Theta} \Phi_{32 \gr}^{N - 14}} \lesssim \lambda
     \end{aligned} \]
  by construction.
\end{proof}

\subsection{Uniform bounds for the exterior
size}\label{sec:bound-exterior-size}

We now prove \Cref{prop:unif-gamma-bounds:Sex}. We begin with a uniform
covering argument.

\begin{lemma}[Exterior tree covering]
  \label{lem:ext-tree-covering}Suppose \(F \in \Rad (\R^3_+)\) is a compactly
  supported Radon measure. For all \(\lambda > 0\) there exists a set
  \(W_{\lambda} \subset \TT^{\cup}_{\Theta}\) and finite collection of trees
  \(\mathcal{T} \subset \mathbb{T}_{\Theta}\) such that
  \[ \Bigl\| \1_{\R^3_+ \setminus W_{\lambda}} F \Bigr\|_{\SL^{(1,
     1)}_{(\Theta, \Theta^{\tmop{ex}})}} \leq \lambda, \qquad W_{\lambda}
     \supset \bigcup_{T \in \mathcal{T}} T, \qquad \mu^1_{\Theta}
     (W_{\lambda}) \lesssim  \sum_{T \in \mathcal{T}} \mu^1_{\Theta} (T) . \]
  Each \(T \in \mc{T}\) is endowed with a distinguished subset \(X_T \subset
  \pi_T \left( \mT_{\Theta^{\tmop{ex}}} \right)\) such that
  \[ \begin{aligned}[t]&
       \frac{1}{\mu^1_{\Theta} (T)} \int_{X_T} | F (\eta, y, t) | \dd \eta
       \dd y \dd t \geq \lambda .
     \end{aligned} \]
  The subsets \(X_T\) satisfy
  \begin{equation}
    \left\| \sum_{T \in \mathcal{T}} \1_{\Gamma (X_T)} (\eta, y, t)
    \right\|_{\SL^{(\infty, \infty)}_{\Theta_{\Gamma}}} \lesssim 1
    \label{eq:ext-tree-covering:ext-summability}
  \end{equation}
  and
  \begin{equation}
    \left\| \sum_{T \in \mathcal{T}} \1_{\Gamma (X_T)} (\eta, y, t)
    \right\|_{\SL^{(1, 1)}_{(\Theta_{\Gamma}, \Theta^{\tmop{in}}_{\Gamma})}}
    \lesssim 1. \label{eq:ext-tree-covering:inner-summability}
  \end{equation}
  The implicit constants do not depend on \(\Gamma\).
\end{lemma}

\begin{proof}
  Let us define local sizes \({\SL^{(1, 1)}_{(\Theta, \Theta^{\tmop{ex} +})}} \)
  and \({\SL^{(1, 1)}_{(\Theta, \Theta^{\tmop{ex} -})}} \) by setting
  \[ \begin{aligned}[t]&
       \Theta^{\tmop{ex} \pm} \eqd \Theta \cap \{ \theta \suchthat \pm \theta
       \geq \gamma \} .
     \end{aligned} \]
  It suffices to prove this lemma with \({\SL^{(1, 1)}_{(\Theta,
  \Theta^{\tmop{ex} +})}} \) in place of \({\SL^{(1, 1)}_{(\Theta,
  \Theta^{\tmop{ex}})}} \). A symmetric proof handles the size \({\SL^{(1,
  1)}_{(\Theta, \Theta^{\tmop{ex} -})}} \) and the two results can then be
  added together.
  
  Without loss of generality we assume that \(F\) is supported on \(\mathbb{K}
  \subset [- S ; S]^2 \times [S^{- 1} ; S]\) for some \(S \gg 1\), and by
  homogeneity we assume that \(\|F\|_{ {{\SL^{(1, 1)}_{(\Theta,
  \Theta^{\tmop{ex} +})}} } } = 1\). This guarantees a bound on the parameters
  \((\xi_T, x_T, s_T)\) of any tree \(T \in \mathbb{T}_{\Theta}\) for which
  \(\|F\|_{{{\SL^{(1, 1)}_{(\Theta, \Theta^{\tmop{ex} +})}} }  (T)} \geq
  \lambda\). As a matter of fact, it must hold that \(s_T > S^{- 1}\), because
  otherwise \(T_{\Theta} (\xi_T, x_T, s_T)  \cap \mathbb{K}= \emptyset\). It
  also holds that \(s_T \leq 2 S \lambda^{- 1}\) because
  \[ \begin{aligned}[t]&
       s_T \|F\|_{{{\SL^{(1, 1)}_{(\Theta, \Theta^{\tmop{ex} +})}} } 
       (T_{\Theta} (\xi_T, x_T, s_T) )} \leq 2 S \|F\|_{{{\SL^{(1,
       1)}_{(\Theta, \Theta^{\tmop{ex} +})}} }  (T_{\Theta} (\xi_T, 0, 2 S) )}
       \leq 2 S
     \end{aligned} \]
  since \(T_{\Theta} (\xi_T, x_T, s_T) \cap \mathbb{K} \subset T_{\Theta}
  (\xi_T, 0, 2 S)  \cap \mathbb{K}\) and since we assumed that
  \(\|F\|_{{{\SL^{(1, 1)}_{(\Theta, \Theta^{\tmop{ex} +})}} } } = 1\). It also
  holds that \(| x_T | < S + 2 S \lambda^{- 1}\) because otherwise, again
  \(T_{\Theta} (\xi_T, x_T, s_T) \cap \mathbb{K}= \emptyset\). For the same
  reason we must have that \(| \xi_T | < 4 S^2 (S^2 + | \Theta |) \lambda^{-
  1}\).
  
  We want to construct \(\mathcal{T}\) by an iterative algorithm. The following
  preparatory step guarantees that the algorithm will terminate. We discretize
  the set of all trees by setting
  \[ \begin{aligned}[t]&
       \R^{3, \tmop{latt}}_+ \assign \left\{ (\eta, y, t) : \eta t, y t^{- 1}
       \in \Z, t \in 2^{\Z} \right\},\\ &
       \mathbb{T}_{\Theta}^{\tmop{latt}} \assign \left\{ T_{\Theta}  (\eta, y,
       t) : (\eta, y, t) \in \R^{3, \tmop{latt}}_+ \right\} .
     \end{aligned} \]
  Given any \(T \in \mathbb{T}_{\Theta} \) we associate to it a collection of
  trees \(\mathcal{W} (T) \subset \mathbb{T}_{\Theta}^{\tmop{latt}}\) given by
  \[ \mathcal{W} (T) \assign \left\{ T_{\Theta} (\xi, x, s)  \in
     \mathbb{T}_{\Theta}^{\tmop{latt}} \suchthat 1  \leq \frac{s}{s_T} < 2^4,
     |x - x_T | \leq 2 s_T, s_T | \xi - \xi_T | < 2^8 \right\} . \]
  It holds that \(\mu_{\Theta}^1 \bigl( \bigcup_{T' \in \mathcal{W} (T)} T'
  \bigr) \lesssim \mu_{\Theta}^1 (T)\) and \\\(\bigcup_{\heartsuit \in \{- 1, 0,
  1\}} T_{(\xi_T + \heartsuit s_T^{- 1}, x_T, s_T)} \subset \bigcup_{T' \in
  \mathcal{W} (T)} T'\). Take \(\mathcal{K}^{\tmop{latt}}_{\lambda} \subset
  \mathbb{T}_{\Theta}^{\tmop{latt}}\) to be
  \[ \mathcal{K}^{\tmop{latt}}_{\lambda} \assign \bigcup_{(\xi_T, x_T, s_T)
     \in A \subset \R^3_+} \mathcal{W} (T_{\Theta}  (\xi_T, x_T, s_T)) \subset
     \mathbb{T}_{\Theta}^{\tmop{latt}} \]
  where
  \[ A \eqd \{ (\xi_T, x_T, s_T) \suchthat S^{- 1} \leq s_T \leq 2 S
     \lambda^{- 1}, |x_T | \leq S + 2 S \lambda^{- 1}, | \xi_T | \leq 4 S^2
     (S^2 + | \Theta |) \lambda^{- 1} \}, \]
  so that the trees in \(\mathcal{K}^{\tmop{latt}}_{\lambda}\) cover
  \(\mathbb{K}\) and so that for any tree \(T \in \TT_{\Theta}\) such that
  \(\|F\|_{{{\SL^{(1, 1)}_{(\Theta, \Theta^{\tmop{ex} +})}} }  (T)} \geq
  \lambda\) it holds that \(\mathcal{W} (T) \subset
  \mathcal{K}^{\tmop{latt}}_{\lambda}\). The collection
  \(\mathcal{K}^{\tmop{latt}}_{\lambda}\) is finite, since the tops of trees in
  \(\mathcal{K}^{\tmop{latt}}_{\lambda}\) are a bounded subset of the discrete
  set \(\R^{3, \tmop{latt}}_+\).
  
  Let us now proceed to the actual selection algorithm. Given a collection of
  trees \(\mathcal{X} \subset \mathbb{T}_{\Theta} \) we say that \(T \in
  \mathcal{X}\) is {\tmem{quasi-maximal}} if
  \[ T' \in \mathcal{X} \Longrightarrow \xi_{T'} \leq \xi_T +
     \varepsilon_{\max} \]
  with a constant \(\varepsilon_{\max} = \varepsilon_{\max} (\lambda, S)\) to be
  determined later. We start with the collection \(\mathcal{T}_0 = \emptyset\)
  and let \(\mathbb{K}_0 \assign \mathbb{K}\). Suppose that at step \(n\) we have
  a collection \(\mathcal{T}_n = \{ T_1, \ldots, T_n \} \subset
  \mathbb{T}_{\Theta} \) and a subset \(\mathbb{K}_n \subset \mathbb{K}\). At
  step \(n + 1\) let
  \[ \mathcal{X}_{n + 1} \assign \Bigl\{ T \in \mathbb{T}_{\Theta} : \|
     \1_{{\mathbb{K}_n} } F\|_{{{\SL^{(1, 1)}_{(\Theta, \Theta^{\tmop{ex}
     +})}} }  (T)} \geq \lambda \Bigr\} ; \]
  if \(\mathcal{X}_{n + 1}\) is empty, we terminate the iteration, otherwise we
  choose a quasi-maximal element \(T_{n + 1}\) of \(\mathcal{X}_{n + 1}\); this
  can be done since \(\Bigl\{ (\xi, x, s) \in \R^3_+ \suchthat T_{\Theta} (\xi,
  x, s)  \in \mathcal{X}_{n + 1} \Bigr\}\) is precompact and non-empty, and
  thus has at least one quasi-maximal element. Let \(\mathcal{T}_{n + 1}
  =\mathcal{T}_n \cup \{ T_{n + 1} \}\) and let \(X_{T'} \assign \mathbb{K}_n
  \cap \pi_{T_{n + 1}} \left( \mT_{\Theta^{\tmop{ex} +}} \right) \). Now set \
  \(\mathbb{K}_{n + 1} =\mathbb{K}_n \setminus \bigcup_{T' \in \mathcal{W}
  (T_{n + 1})} T'\). This process terminates after finitely many steps: if
  \(T_{n + 1}\) gets selected then
  \[ \mathcal{W} (T_{n + 1}) \cap \left( \mathcal{K}^{\tmop{latt}}_{\lambda}
     \setminus \bigcup_{n' = 0}^n \mathcal{W} (T_{n'}) \right) \neq \emptyset,
  \]
  given that \(T_{n'} \subset \bigcup_{T' \in \mathcal{W} (T_{n'})} T'\) for
  any \(n'\) and that
  \[ \|F\|_{{{\SL^{(1, 1)}_{(\Theta, \Theta^{\tmop{ex} +})}} }  (T_{n + 1})}
     \geq \| \1_{{\mathbb{K}_n} } F\|_{{{\SL^{(1, 1)}_{(\Theta,
     \Theta^{\tmop{ex} +})}} }  (T_{n + 1})} > \lambda \]
  and thus \(\mathcal{W} (T) \subset \mathcal{K}^{\tmop{latt}}_{\lambda}\). By
  construction, \(\mathcal{K}^{\tmop{latt}}_{\lambda} \setminus \bigcup_{n' =
  0}^n \mathcal{W} (T_{n'})\) is thus a sequence of finite collections of trees
  that is strictly decreasing in \(n\) as long as the algorithm continues. The
  collection eventually stabilizes so the algorithm terminates.
  
  The procedure described above yields a collection of trees \(\mathcal{T}=
  \bigcup_{n \in \N} \mc{T}_n\) and distinguished subsets \(X_{T'} \subset
  \pi_{T'} \left( \mT_{\Theta^{\tmop{ex} +}} \right)\) for any \(T' \in
  \mathcal{T}\). According to the selection criterion, it holds that
  \[ \frac{1}{\mu^1_{\Theta} (T)} \int_{X_T} | F (\eta, y, t) | \dd \eta
     \dd y \dd t \geq \lambda . \]
  We set \(W_{\lambda} = \bigcup_{T \in \mathcal{T}} W (T) \supset \bigcup_{T
  \in \mathcal{T}} T\) and thus we have \(\mu^1_{\Theta} (W_{\lambda}) \lesssim
  \sum_{T \in \mathcal{T}} \mu^1_{\Theta} (T)\). Since the procedure terminated
  it holds that \(\left\| \1_{\mathbb{K} \setminus W_{\lambda}} F
  \right\|_{\SL^{(1, 1)}_{(\Theta, \Theta^{\tmop{ex} +})}} \leq \lambda\). It
  remains to show bounds \eqref{eq:ext-tree-covering:ext-summability} and
  \eqref{eq:ext-tree-covering:inner-summability}. The former follows since the
  sets \(X_T\) are constructed to be pairwise disjoint. To show the latter bound
  let us show that
  \begin{equation}
    \int_{\R_+} \sum_{T \in \mathcal{T}} \1_{X_T} (\xi + \theta t^{- 1}, x, t)
    \frac{\dd t}{t} \lesssim 1 + \log \left( 1 + \frac{| \Theta^{\tmop{ex}
    +} |}{\dist (\theta, \Theta^{\tmop{ex} +})} \right)
    \label{eq:treesel-contra-claim-integral}
  \end{equation}
  for any \(\xi, x \in \R\) and \(\theta \in \Theta^{\tmop{in}}\) with some
  absolute constant. This is sufficient to deduce
  \eqref{eq:ext-tree-covering:inner-summability}: using a change of variables
  we would then rewrite condition
  \eqref{eq:ext-tree-covering:inner-summability} and get
  \[ \begin{aligned}[t]&
       \left\| \sum_{T \in \mathcal{T}} \1_{\Gamma (X_T)} (\eta, y, t)
       \right\|_{\SL^{(1, 1)}_{(\Theta_{\Gamma},
       \Theta^{\tmop{in}}_{\Gamma})} (T_{\Theta_{\Gamma}} (\alpha \xi, x,
       s))}\\ &
       \leq \int_{B_1} \int_{\Theta^{\tmop{in}}_{\Gamma}} \int_{\R_+} \sum_{T
       \in \mathcal{T}} \1_{\Gamma (X_T)} (\alpha \xi + \theta (s \sigma)^{-
       1}, x  + s  \zeta, s  \sigma) \frac{\dd \sigma}{\sigma} \frac{\dd
       \zeta \dd \theta}{| \Theta_{\Gamma} |}\\ &
       \leq \int_{B_1} \int_{\Theta^{\tmop{in}} } \int_{\R_+} \sum_{T \in
       \mathcal{T}} \1_{\Gamma (X_T)} (\alpha \xi + \alpha \beta (\theta' +
       \gamma) (s  \beta \sigma')^{- 1}, x  + s  \zeta, s  \beta \sigma')
       \frac{\dd \sigma'}{\sigma'} \frac{\dd \zeta \dd \theta'}{| \Theta
       |}\\ &
       \leq \int_{B_1} \int_{\Theta^{\tmop{in}}} \int_{\R_+} \sum_{T \in
       \mathcal{T}} \1_{X_T} (\xi + \theta' (s \sigma')^{- 1}, x_T + s  \zeta,
       s  \sigma') \frac{\dd \sigma'}{\sigma'} \frac{\dd \zeta \dd
       \theta'}{| \Theta |}\\ &
       \leq \int_{B_1} \int_{\Theta^{\tmop{in}}} \left( 1 + \log \left( 1 +
       \frac{| \Theta^{\tmop{ex}} |}{\dist (\theta', \Theta^{\tmop{ex}})}
       \right) \right) \frac{\dd \zeta \dd \theta'}{| \Theta |} \lesssim 1 + |
       \Theta | .
     \end{aligned} \]
  Fix \(\xi, x \in \R\) and \(\theta \in \Theta^{\tmop{in}}\) and let us prove
  \eqref{eq:treesel-contra-claim-integral} by showing that
  \begin{equation}
    \frac{\tau_+}{\tau_-} \leq \frac{2 | \Theta^{\tmop{ex} +} |}{\dist
    (\theta, \Theta^{\tmop{ex} +})} + 1, \label{eq:treesel-contra-claim}
  \end{equation}
  where
  \[ \begin{aligned}[t]&
       \tau_+ \assign \sup \left\{ t \in [0, + \infty) \suchthat (\xi + \theta
       t^{- 1}, x, t) \in \bigcup_{T \in \mathcal{T}} X_T \right\},\\ &
       \tau_- \assign \inf \left\{ t \in [0, + \infty) \suchthat (\xi + \theta
       t^{- 1}, x, t) \in \bigcup_{T \in \mathcal{T}} X_T \right\} .
     \end{aligned} \]
  The quantity \(\tau_+\) is positive and finite unless \(\sum_{T \in
  \mathcal{T}} \1_{X_T} (\xi + \theta t^{- 1}, x, t) = 0\) for all \(t\), but
  then {\LHS{\eqref{eq:treesel-contra-claim-integral}}} vanishes. According to
  the definition of \(\tau_+\) we may find \(t_0 \in \left( \frac{\tau_+}{1 +
  \varepsilon}, \tau_+ \right)\) and a tree \(T_0\) such that \((\xi + \theta t^{-
  1}, x, t) \in X_{T_0}\).
  
  Let \(\mathcal{T}_{>}\) be the subset of trees in \(\mc{T}\) selected
  \tmtextit{after} \(T_0\) by the algorithm of this proof. It holds that
  \[ \inf \left\{ t \in [0, + \infty) \suchthat (\xi + \theta t^{- 1}, x, t)
     \in \bigcup_{T \in \mathcal{T}_{>}} X_T \right\} \geq t_0 =
     \frac{\tau_+}{1 + \varepsilon_{\max}} . \]
  As a matter of fact \((\xi + \theta \tilde{t}^{- 1}, x, \tilde{t}) \in T_0\)
  for any \(\tilde{t} < t_0\) so in particular \((\xi + \theta \tilde{t}^{- 1},
  x, \tilde{t}) \notin X_{T_{>}} \subset T_{>} \setminus T_0\) for any \(T_{>}
  \in \mc{T}\) selected after \(T_0\).
  
  Let us then focus a tree \(T_{<} \in \mc{T}_{<}\) selected before \(T_0\). By
  quasi-maximality we have that \(- \varepsilon_{\max} + \xi_{T_0} \leq
  \xi_{T_{<}}\) and thus
  \[ \xi - \xi_{T_{<}} \leq \xi - \xi_{T_0} + \varepsilon_{\max} . \]
  If \(\xi_{T_{<}} \geq \xi\) then we would have that \(\pi_{T_{<}} \left(
  \mT_{\Theta^{\tmop{ex} +}} \right) \cap \left\{ (\xi + \theta t^{- 1}, x, t)
  : t \in \R_+ \right\} = \emptyset\) so \(T_{<}\) would not contribute to
  defining \(\tau^-\). Suppose \(\xi_{T_{<}} < \xi\); then it holds that
  \[ \begin{aligned}[t]
       \left\{ t : (\xi + \theta t^{- 1}, x, t) \in \pi_{T_{<}} \left(
       \mT_{\Theta^{\tmop{ex} +}} \right) \right\} &\subset \left( \frac{\dist
       (\theta, \Theta^{\tmop{ex} +})}{\xi - \xi_{T_{<}}}, + \infty \right)
       \\ & \subset \left( \frac{\dist (\theta, \Theta^{\tmop{ex} +})}{\xi -
       \xi_{T_0} {+ \varepsilon_{\max}} }, + \infty \right)
     \end{aligned} \]
  To bound \(\frac{\dist (\theta, \Theta^{\tmop{ex} +})}{\xi - \xi_{T_0} {+
  \varepsilon_{\max}} }\) from below we use the fact that \((\xi + \theta t_0^{-
  1}, x, t_0) \in X_{T_0}\) and, as a consequence,
  \[ \frac{\dist (\theta, \Theta^{\tmop{ex} +})}{t_0} < \xi - \xi_{T_0} <
     \frac{| \Theta^{\tmop{ex} +} |}{t_0} \]
  because \(X_{T_0} \subset \left\{ (\xi_{T_0} + \bar{\theta} \bar{t}^{- 1},
  \bar{x}, \bar{t}) : \bar{\theta} \in \Theta^{\tmop{ex} +}, \bar{x} \in
  B_{s_{T_0}} (x_{T_0}), \bar{t} < s_{T_0} \right\}\). It follows that
  \[ \begin{aligned}[t]&
       \left\{ t : (\xi + \theta t^{- 1}, x, t) \in \pi_{T_{<}} \left(
       \mT_{\Theta^{\tmop{ex} +}} \right) \right\} \subset \left( t_0 
       \frac{\dist (\theta, \Theta^{\tmop{ex} +})}{| \Theta^{\tmop{ex} +} | {+
       \varepsilon_{\max}}  t_0}, + \infty \right) .
     \end{aligned} \]
  Recall that \(t_0 < 2 S \lambda^{- 1}\) since all trees \(T\) in \(\mc{T}\)
  satisfy \(s_T < 2 S \lambda^{- 1}\). It is sufficient to take
  \(\varepsilon_{\max} = \frac{\lambda | \Theta^{\tmop{ex} +} |}{2 S}\) to yield
  \[ \begin{aligned}[t]&
       \left\{ t : (\xi + \theta t^{- 1}, x, t) \in \pi_{T_{<}} \left(
       \mT_{\Theta^{\tmop{ex} +}} \right) \right\} \subset \left(
       \frac{\tau_+}{(1 + \varepsilon)}  \frac{\dist (\theta,
       \Theta^{\tmop{ex} +})}{2 | \Theta^{\tmop{ex} +} |}, + \infty \right)
     \end{aligned} \]
  and prove that \
  \[ \tau_- > \frac{\tau_+}{(1 + \varepsilon)} \min \left( \frac{\dist
     (\theta, \Theta^{\tmop{ex} +})}{2 | \Theta^{\tmop{ex} +} | (1 +
     \varepsilon)}, 1 \right) \]
  Since \(\varepsilon > 0\) can be chosen arbitrarily, this proves our claim.
\end{proof}

The covering lemma above allows us to obtain the following uniform weak \(L^1\)
bound.

\begin{lemma}
  \label{lem:unif-Sex-bound}The bound
  \begin{equation}
    \| F \circ \Gamma \|_{L^{1, \infty}_{\mu_{\Theta}^1} \Gamma^{\ast}
    \SL^{(1, 1)}_{(\Theta, \Theta^{\tmop{ex}})}} \lesssim \| F
    \|_{L^1_{\mu_{\Theta_{\Gamma}}^1} \left( \SL^{(\infty,
    \infty)}_{\Theta_{\Gamma}} + \SL^{(1, 1)}_{(\Theta_{\Gamma},
    \Theta^{\tmop{ex}}_{\Gamma})} \right)} \label{eq:unif-Sex-bound}
  \end{equation}
  holds for any \(F \in \Rad (\R^{3}_{+})\).
\end{lemma}

\begin{proof}
  We need to show that for every \(\lambda > 0\) there exists a set \(W_{\lambda}
  = \bigcup_{T \in \mathcal{T} \subset \mathbb{T}_{\Theta}} T\) such that
  \[ \begin{aligned}[t]&
       \left\| \1_{\R^3_+ \setminus W_{\lambda} } F \circ \Gamma
       \right\|_{\SL^{(1, 1)}_{(\Theta, \Theta^{\tmop{ex}})}} \leq \lambda,\\ &
       \mu_{\Theta}^1 (W_{\lambda}) \lesssim \lambda^{- 1}
       \|F\|_{L^1_{\mu_{\Theta_{\Gamma} }^1} \left( \SL^{(\infty,
       \infty)}_{{\Theta_{\Gamma} } } + \SL^{(1, 1)}_{(\Theta_{\Gamma} ,
       \Theta_{\Gamma}^{\tmop{ex}})} \right)} .
     \end{aligned} \]
  We may assume that \(F\) is supported on a compact set by the discussion at
  the end of \Cref{sec:sizes}. Fix \(\lambda > 0\) and let \(W_{\lambda} \assign
  \bigcup_{T \in \mathcal{T }_{\lambda} \subset \mathbb{T}_{\Theta}} T\) where
  \(\mathcal{T}_{\lambda} \subset \mathbb{T}_{\Theta}\) is obtained by applying
  Lemma \ref{lem:ext-tree-covering} to \(F \circ \Gamma \) in place of \(F\). It
  remains to prove the required bounds on \(\mu_{\Theta}^1 (W_{\lambda})\). For
  each tree \(T \in \mathcal{T}_{\lambda}\), Lemma \ref{lem:ext-tree-covering}
  gives the bound
  \[ \lambda \mu^1_{\Theta} (T) \leq \int_{X_T} | F \circ \Gamma (\eta, y, t)
     | \dd \eta \dd y \dd t. \]
  Summing the bound over all trees \(T \in \mathcal{T}_{\lambda}\) and applying
  a change of variables gives
  \[ \begin{aligned}[t]&
       \lambda \mu^1_{\Theta} (W_{\lambda}) \begin{aligned}[t]&
         \lesssim \lambda \sum_{T \in \mathcal{T}_{\lambda}} \mu^1_{\Theta}
         (T) \lesssim \int_{\R^3_+} \left( \sum_{T \in \mathcal{T}_{\lambda}}
         \1_{X_T} (\eta, y, t) \right) | F \circ \Gamma  (\eta, y, t) | \dd
         \eta \dd y \dd t\\ &
         = \int_{\R^3_+} \left( \sum_{T \in \mathcal{T}_{\lambda}} \1_{X_T}
         (\eta, y, t) \right) | F (\alpha (\eta + \gamma t^{- 1}), y, \beta t)
         | \dd \eta \dd y \dd t\\ &
         = \int_{\R^3_+} \left( \sum_{T \in \mathcal{T}_{\lambda}} \1_{\Gamma
         (X_T)} (\alpha (\eta + \gamma t^{- 1}), y, \beta t) \right) | F
         (\alpha (\eta + \gamma t^{- 1}), y, \beta t) | \dd \eta \dd y
         \dd t\\ &
         = \int_{\R^3_+} \frac{1}{\alpha \beta} \left( \sum_{T \in
         \mathcal{T}_{\lambda}} \1_{\Gamma (X_T)} (\eta, y, t) \right) | F
         (\eta, y, t) | \dd \eta \dd y \dd t.
       \end{aligned}
     \end{aligned} \]
  It remains to bound the \(\RHS{}\) of this inequality from above. By
  \Cref{prop:outer-RN} we have that
  \[ \begin{aligned}[t]&
       \int_{\R^3_+} \left( \sum_{T \in \mathcal{T}_{\lambda}} \1_{\Gamma
       (X_T)} (\eta, y, t) \right) | F (\eta, y, t) | \dd \eta \dd y
       \dd t\\ &
       \qquad \lesssim \|F\|_{L^1_{\mu_{\Theta_2}^1} \left( \SL^{(\infty,
       \infty)}_{\Theta_{\Gamma}} + \SL^{(1, 1)}_{(\Theta_{\Gamma},
       \Theta_{\Gamma}^{\tmop{ex}})} \right)} \Bigl\| \sum_{{T \in
       \mathcal{T}_{\lambda}} } \1_{\Gamma (X_T)} (\eta, y, t)
       \Bigr\|_{L^{\infty}_{\mu_{\Theta}^1} \left( \SL^{(1,
       1)}_{(\Theta_{\Gamma}, \Theta_{\Gamma}^{\tmop{in}})} + \SL^{(\infty,
       \infty)}_{\Theta_{\Gamma}} \right)} .
     \end{aligned} \]
  Since
  \[ \begin{aligned}[t]&
       \Bigl\| \sum_{{T \in \mathcal{T}_{\lambda}} } \1_{\Gamma (X_T)} (\eta,
       y, t) \Bigr\|_{L^{\infty}_{\mu_{\Theta}^1} \left( \SL^{(1,
       1)}_{(\Theta_{\Gamma}, \Theta_{\Gamma}^{\tmop{in}})} + \SL^{(\infty,
       \infty)}_{(\Theta_{\Gamma}, \Theta_{\Gamma}^{\tmop{ex}})} \right)}
       \lesssim 1
     \end{aligned} \]
  this concludes the proof.
\end{proof}

We are now ready to deduce \Cref{prop:unif-gamma-bounds:Sex} from
\Cref{lem:unif-Sex-bound}. We use \Cref{lem:wave-packet-decomposition} to
reduce to the scalar case i.e. to proving bounds for \(F \in \Rad (\R^{3}_{+})\) instead of for \(F \in \Rad (\R^{3}_{+}) \otimes
\Phi^N_{\gr}\). Then we conclude using a restricted weak type interpolation
argument.

\begin{proof}{Proof of \Cref{prop:unif-gamma-bounds:Sex}}
  We begin by reducing to the scalar bound. In particular, a straightforward
  adaptation of \Cref{cor:wave-packet-decomposition-and-sizes} to the case
  of the sizes \(\Gamma^{\ast} \SL^{(u, 2)}_{(\Theta, \Theta^{\tmop{ex}})}\)
  shows that
  \[ \| F \circ \Gamma \|_{L^r_{\mu_{\Theta}^1} \Gamma^{\ast} \SL^{(u,
     2)}_{(\Theta, \Theta^{\tmop{ex}})} \Phi^N_{\gr}} \lesssim \sup_{\phi} \|
     F \circ \Gamma (\eta, y, t) [\phi] \|_{L^r_{\mu_{\Theta}^1} \Gamma^{\ast}
     \SL^{(u, 2)}_{(\Theta, \Theta^{\tmop{ex}})}} \]
  where the upper bound is taken over all \(\phi \in \Phi^{N'}_{2 \gr}\) with
  \(\| \phi \|_{\Phi^{N'}_{2 \gr}} \lesssim 1\), as long as \(N \geq N' + 4\).
  Bound \eqref{eq:unif-gamma-bounds:Sex} can thus be deduced from
  \begin{equation}
    \| F \circ \Gamma \|_{L^r_{\mu_{\Theta}^1} \Gamma^{\ast} \SL^{(u,
    2)}_{(\Theta, \Theta^{\tmop{ex}})}} \lesssim \| F
    \|_{L^r_{\mu_{\Theta_{\Gamma}}^1} \left( \SL^{(\infty,
    \infty)}_{\Theta_{\Gamma}} + \SL^{(\bar{u}, 2)}_{(\Theta_{\Gamma},
    \Theta^{\tmop{ex}}_{\Gamma})} \right)}
    \label{eq:unif-gamma-bound-non-iter:Sex:scalar}
  \end{equation}
  for all \(F \in \Rad (\R^{3}_{+})\). The remaining step is to deduce
  bound \eqref{eq:unif-gamma-bound-non-iter:Sex:scalar} from bound
  \eqref{eq:unif-Sex-bound} of \Cref{lem:unif-Sex-bound}. In particular let us
  show that
  \begin{equation}
    \| F \circ \Gamma \|_{L^r_{\mu_{\Theta}^1} \Gamma^{\ast} \SL^{(u,
    1)}_{(\Theta, \Theta^{\tmop{ex}})}} \lesssim \| F
    \|_{L^r_{\mu_{\Theta_{\Gamma}}^1} \left( \SL^{(\infty,
    \infty)}_{\Theta_{\Gamma}} + \SL^{(\bar{u}, 1)}_{(\Theta_{\Gamma},
    \Theta^{\tmop{ex}}_{\Gamma})} \right)}
    \label{eq:unif-gamma-bound-non-iter:Sex:scalar:weak:non-diag}
  \end{equation}
  for \(1 \leq u < \bar{u} \leq \infty\) and \(r \in \left( u \frac{\bar{u} -
  1}{\bar{u} - u}, \infty \right]\) from \eqref{eq:unif-Sex-bound}. Bound
  \eqref{eq:unif-gamma-bound-non-iter:Sex:scalar} follows from
  \eqref{eq:unif-gamma-bound-non-iter:Sex:scalar:weak:non-diag} by applying it
  to \(F^2\) in lieu of \(F\) and choosing \(\bar{u}\) appropriately large. For any
  \(\bar{u} > u\) the \(r = \infty\) endpoint for
  \eqref{eq:unif-gamma-bound-non-iter:Sex:scalar:weak:non-diag}, given by
  \[ \| F \circ \Gamma \|_{L^{\infty}_{\mu_{\Theta}^1} \Gamma^{\ast} \SL^{(u,
     1)}_{(\Theta, \Theta^{\tmop{ex}})}} \leq \| F
     \|_{L^{\infty}_{\mu_{\Theta}^1} \SL^{(\bar{u}, 1)}_{(\Theta_{\Gamma},
     \Theta_{\Gamma}^{\tmop{ex}})}}, \]
  follows from the fact that \(\Gamma (T_{\Theta} (\xi_T, x_T, s_T)) \subset
  T_{\Theta_{\Gamma}} (\alpha \xi_T, x_T, s_T)\). We can use
  \Cref{prop:outer-restricted-interpolation} to conclude that bound
  \eqref{eq:unif-gamma-bound-non-iter:Sex:scalar:weak:non-diag} holds for all
  other \(r > u^{\ast} \eqd u \frac{\bar{u} - 1}{\bar{u} - u}\) if we show the
  \(u^{\ast}\) restricted weak endpoint
  \begin{equation}
    \begin{aligned}[t]&
      \| F \circ \Gamma \|_{L^{u^{\ast}, \infty}_{\mu_{\Theta}^1}
      \Gamma^{\ast} \SL^{(u, 1)}_{(\Theta, \Theta^{\tmop{ex}})}} \lesssim
      \begin{aligned}[t]&
        \mu_{\Theta}^1 \left( \|F\|_{\left( \SL^{(\infty,
        \infty)}_{\Theta_{\Gamma}} + \SL^{(1, 1)}_{(\Theta_{\Gamma},
        \Theta_{\Gamma}^{\tmop{ext}})} \right)} > 0
        \right)^{\frac{1}{u^{\ast}}}\\ &
        \qquad \times \|F\|_{\left( \SL^{(\infty, \infty)}_{\Theta_{\Gamma}} +
        \SL^{(\bar{u}, 1)}_{(\Theta_{\Gamma}, \Theta_{\Gamma}^{\tmop{ex}})}
        \right)} .
      \end{aligned}
    \end{aligned} \label{eq:unif-gamma-bound-non-iter:Sex:scalar:Lrestricted}
  \end{equation}
  Fix any \(\tau > 0\) and let us apply \Cref{lem:unif-Sex-bound}: there exists
  \(W_{\tau} \in \TT_{\Theta}^{\cup}\) such that
  \[ \begin{aligned}[t]&
       \mu^1_{\Theta} (W_{\tau}) \lesssim \begin{aligned}[t]&
         \tau^{- 1} \mu_{\Theta}^1 \left( \|F\|_{\left( \SL^{(\infty,
         \infty)}_{\Theta} + \SL^{(\bar{u}, 1)}_{(\Theta_{\Gamma},
         \Theta_{\Gamma}^{\tmop{ext}})} \right)} > 0 \right) \qquad\\ &
         \qquad \times \|F\|_{\left( \SL^{(\infty, \infty)}_{\Theta} +
         \SL^{(\bar{u}, 1)}_{(\Theta_{\Gamma}, \Theta_{\Gamma}^{\tmop{ex}})}
         \right)},
       \end{aligned}\\ &
       \| \1_{\R^3_+ \setminus W_{\tau}} F \circ \Gamma \|_{\Gamma^{\ast}
       \SL^{(1, 1)}_{(\Theta, \Theta^{\tmop{ex}})}} \leq \tau .
     \end{aligned} \]
  By \(\log\)-convexity of the \(\SL^{(u, 1)}_{(\Theta, \Theta^{\tmop{ex}})}\)
  sizes in the parameter \(u \in [1, \infty]\), we have that
  \[ \begin{aligned}[t]&
       \left\| \1_{\R^3_+ \setminus W_{\tau}} F \circ \Gamma
       \right\|_{\Gamma^{\ast} \SL^{(u, 1)}_{(\Theta, \Theta^{\tmop{ext}})}}
       \leq \| \1_{\R^3_+ \setminus W_{\tau}} F \circ \Gamma \|_{\Gamma^{\ast}
       \SL^{(1, 1)}_{(\Theta, \Theta^{\tmop{ex}})}}^{\frac{1}{u^{\ast}}} \|
       \1_{\R^3_+ \setminus W_{\tau}} F \circ \Gamma \|_{\Gamma^{\ast}
       \SL^{(\bar{u}, 1)}_{(\Theta, \Theta^{\tmop{ex}})}}^{\frac{u^{\ast} -
       1}{u^{\ast}}}\\ &
       \leq \tau^{\frac{1}{u^{\ast}}} \|F\|_{\Gamma^{\ast} \SL^{(\bar{u},
       1)}_{(\Theta, \Theta^{\tmop{ex}})}}^{\frac{u^{\ast} - 1}{u^{\ast}}}
     \end{aligned} . \]
  Combining the estimate above with the measure bound on \(W_{\tau}\) we get
  that
  \[ \mu_{\Theta}^1 (W_{\tau})  \left\| \1_{\R^3_+ \setminus W_{\tau}} F \circ
     \Gamma \right\|_{\SL^{(u, 1)}_{(\Theta, \Theta^{\tmop{ext}})}}^{u_{\ast}}
     \lesssim \begin{aligned}[t]&
       \mu_{\Theta}^1 \left( \|F\|_{\left( \SL^{(\infty,
       \infty)}_{\Theta_{\Gamma}} + \SL^{(\bar{u}, 1)}_{(\Theta_{\Gamma},
       \Theta_{\Gamma}^{\tmop{ex}})} \right)} > 0 \right) \qquad\\ &
       \qquad \times \|F\|_{\left( \SL^{(\infty, \infty)}_{\Theta_{\Gamma}} +
       \SL^{(\bar{u}, 1)}_{(\Theta_{\Gamma}, \Theta_{\Gamma}^{\tmop{ex}})}
       \right)}^{u^{\ast}} .
     \end{aligned} \]
  Since \(\tau > 0\) is arbitrary we have shown that
  \eqref{eq:unif-gamma-bound-non-iter:Sex:scalar:Lrestricted} holds.
\end{proof}

\subsection{Uniform bounds for the singular
size}\label{sec:unif-bounds:singular}

We now prove \Cref{prop:unif-gamma-bounds:singular-size}, proceeding roughly
in the same manner as in \Cref{sec:bound-exterior-size} while proving
\Cref{prop:unif-gamma-bounds:Sex}. We begin with a uniform covering argument
for the interior part of trees.

\begin{lemma}[Singular tree covering]
  \label{lem:singular-tree-covering}Let \(V  \in \mathbb{D}^{\cup}_{\beta}\), \(W
  \in \TT_{\Theta}^{\cup}\) and let \(E = V \cap W\) or \(E = V \cup W\). Suppose
  \(F \in L^{\infty}_{\tmop{loc}} (\R^3_+)\) is a compactly supported function.
  For all \(\lambda > 0\) there exists a set \(W_{\lambda} \subset
  \TT^{\cup}_{\Theta}\) and finite collection of trees \(\mathcal{T} \subset
  \mathbb{T}_{\Theta}\) such that
  \[ \begin{aligned}[t]&
       \Bigl\| \1_{\R^3_+ \setminus W_{\lambda}} F (\eta, y, t) t \dd_t \1_E
       (\eta, y, t) \Bigr\|_{\SL^{(1, 1)}_{(\Theta, \Theta^{\tmop{in}})}}
       \leq \lambda,\\ &
       W_{\lambda} \supset \bigcup_{T \in \mathcal{T}} T, \qquad
       \mu^1_{\Theta} (W_{\lambda}) \lesssim  \sum_{T \in \mathcal{T}}
       \mu^1_{\Theta} (T) .
     \end{aligned} \]
  Each \(T \in \mc{T}\) is endowed with a distinguished bounded subset \(X_T
  \subset \pi_T \left( \mT_{\Theta^{\tmop{in}}} \right)\) such that
  \[ \frac{1}{\mu^1_{\Theta} (T)} \int_{X_T} | F (\eta, y, t) | t \dd_t \1_E
     (\eta, y, t) \dd \eta \dd y \dd t \geq \lambda . \]
  and
  \begin{equation}
    \left\| \left( \sum_{T \in \mathcal{T}} \1_{\Gamma (X_T)} (\eta, y, t)
    \right) t \dd_t \1_{\Gamma (E)} (\eta, y, t) \right\|_{\SL^{(1,
    1)}_{\Theta_{\Gamma}}} \lesssim 1 \label{eq:int-tree-covering:summability}
  \end{equation}
  for any \(\Gamma\) as in \eqref{eq:gamma}. The implicit constants do not
  depend on the set \(E\) nor on \(\Gamma\).
\end{lemma}

\begin{proof}{Proof of \Cref{lem:singular-tree-covering}}
  The proof follows along the same lines of the proof of
  \Cref{lem:ext-tree-covering}. We begin by making a geometric assumption on
  \(W\). If \(W = \bigcup_{n \in \N} T_n\) then let us set \(W^{\pm} \eqd
  \bigcup_{n \in \N} T_n^{\pm}\) where
  \[ T^{\pm} : \{ (\eta \pm h, y, t), (\eta, y, t) \in T, h \geq 0 \} ; \]
  notice that \(T^+ \in \TT^{\cup}_{\Theta}\) if \(T \in \TT^{\cup}_{\Theta}\) and
  it holds that \(T = T^+ \cap T^-\). Thus, if \(E = V \cap W\) it holds that \(E =
  E^+ \cap E^-\) where \(E^{\pm} \eqd V \cap W^{\pm}\) and we have that
  \[ \left| F (\eta, y, t) t \dd_t \1_E (\eta, y, t) \right| \leq F (\eta, y,
     t)  \left| t \dd_t \1_{E^-} (\eta, y, t) \right| + F (\eta, y, t)  \left|
     t \dd_t \1_{E^+} (\eta, y, t) \right| \]
  Similarly, if \(E = V \cup W\) then \(E = E^+ \cup E^-\) and
  \[ \left| F (\eta, y, t) t \dd_t \1_E (\eta, y, t) \right| \leq F (\eta, y,
     t)  \left| t \dd_t \1_{E^-} (\eta, y, t) \right| + F (\eta, y, t)  \left|
     t \dd_t \1_{E^+} (\eta, y, t) \right| . \]
  By the quasi triangle inequality, it is thus sufficient to show the claim
  for \(E^+\) in place of \(E\). We reason in the case \(E^+ = V \cap W^+\) and we
  write \(F\) in place of \(F (\eta, y, t)  \1_{E^-} (\eta, y, t)\).
  
  Without loss of generality we assume that \(F\) is supported on \(\mathbb{K}
  \subset [- S ; S]^2 \times [S^{- 1} ; S]\) for some \(S \gg 1\). It follows
  from \Cref{lem:geometry-of-boundary} that \\ \(\left\| F \circ \Gamma (\eta, y,
  t) t \dd_t \1_E (\eta, y, t) \right\|_{\SL^{(1, 1)}_{(\Theta,
  \Theta^{\tmop{in}})}} < \infty\) so by homogeneity we assume that
  \[ \left\| F \circ \Gamma (\eta, y, t) t \dd_t \1_E (\eta, y, t)
     \right\|_{\SL^{(1, 1)}_{(\Theta, \Theta^{\tmop{in}})}} = 1. \]
  The parameters \((\xi_T, x_T, s_T)\) of any tree \(T \in \mathbb{T}_{\Theta}\)
  for which
  \[ \left\| F \circ \Gamma (\eta, y, t) t \dd_t \1_E (\eta, y, t)
     \right\|_{\SL^{(1, 1)}_{(\Theta, \Theta^{\tmop{in}})} (T)} \geq \lambda
  \]
  are bounded and \(s_T\) is bounded away from \(0\). As a matter of fact
  \[ \begin{aligned}[t]&
       s_T \|F (\eta, y, t) t \dd_t \1_{E^+} (\eta, y, t) \|_{{{\SL^{(1,
       1)}_{(\Theta, \Theta^{\tmop{in}})}} }  (T_{\Theta} (\xi_T, x_T, s_T)
       )}\\ &
       \qquad \leq 2 S \|F (\eta, y, t) t \dd_t \1_{E^+} (\eta, y, t)
       \|_{{{\SL^{(1, 1)}_{(\Theta, \Theta^{\tmop{in}})}} }  (T_{\Theta}
       (\xi_T, 0, 2 S) )} .
     \end{aligned} \]

  We want to construct \(\mathcal{T}\) by an iterative procedure. We enact the
  same preparatory steps that we have done in the proof of
  \Cref{lem:ext-tree-covering}: we introduce
  \(\mathbb{T}_{\Theta}^{\tmop{latt}}\), the discrete collection of trees; to
  any \(T \in \mathbb{T}_{\Theta} \) we associate a collection of trees
  \(\mathcal{W} (T) \subset \mathbb{T}_{\Theta}^{\tmop{latt}}\) such that \
  \(\mu_{\Theta}^1 \bigl( \bigcup_{T' \in \mathcal{W} (T)} T' \bigr) \lesssim
  \mu_{\Theta}^1 (T)\) and \(\bigcup_{\heartsuit \in \{- 1, 0, 1\}} T_{(\xi_T +
  \heartsuit s_T^{- 1}, x_T, s_T)} \subset \bigcup_{T' \in \mathcal{W} (T)}
  T'\). We define a finite collection \(\mathcal{K}^{\tmop{latt}}_{\lambda}
  \subset \mathbb{T}_{\Theta}^{\tmop{latt}}\) so that the trees in
  \(\mathcal{K}^{\tmop{latt}}_{\lambda}\) cover \(\mathbb{K}\) and so that for any
  tree \(T \in \TT_{\Theta}\) such that\\ \(\left\| F \circ \Gamma (\eta, y, t) t
  \dd_t \1_E (\eta, y, t) \right\|_{\SL^{(1, 1)}_{(\Theta,
  \Theta^{\tmop{in}})} (T)} \geq \lambda\) it holds that \(\mathcal{W} (T)
  \subset \mathcal{K}^{\tmop{latt}}_{\lambda}\).
  
  Let us now proceed to the selection algorithm. Given a collection of trees
  \(\mathcal{X} \subset \mathbb{T}_{\Theta} \) we say that \(T \in \mathcal{X}\)
  is {\tmem{quasi-maximal}} if
  \[ T' \in \mathcal{X} \Longrightarrow \xi_{T'} \leq \xi_T +
     \varepsilon_{\max} \]
  with \(\varepsilon_{\max} \leq 2^{- 100} S.\) Let us start with the collection
  \(\mathcal{T}_0 = \emptyset\) and let \(\mathbb{K}_0 \assign \mathbb{K}\).
  Suppose that at step \(n\) we have a collection \(\mathcal{T}_n = \{ T_1,
  \ldots, T_n \} \subset \mathbb{T}_{\Theta} \) and a subset \(\mathbb{K}_n
  \subset \mathbb{K}\). At step \(n + 1\) let
  \[ \mathcal{X}_{n + 1} \assign \Bigl\{ T \in \mathbb{T}_{\Theta} : \|
     \1_{{\mathbb{K}_n} } F \circ \Gamma (\eta, y, t) t \dd_t \1_E (\eta, y,
     t) \|_{{{\SL^{(1, 1)}_{(\Theta, \Theta^{\tmop{in}})}} }  (T)} \geq
     \lambda \Bigr\} ; \]
  if \(\mathcal{X}_{n + 1}\) is empty, we terminate the iteration. Otherwise
  choose a quasi-maximal element \(T_{n + 1}\) of \(\mathcal{X}_{n + 1}\) and let
  \(\mathcal{T}_{n + 1} =\mathcal{T}_n \cup \{ T_{n + 1} \}\), \(X_{T'} \assign
  \mathbb{K}_n \cap \pi_{T_{n + 1}} \left( \mT_{\Theta^{\tmop{in}}} \right) \)
  and \(\mathbb{K}_{n + 1} =\mathbb{K}_n \setminus \bigcup_{T' \in \mathcal{W}
  (T_{n + 1})} T'\). This process terminates after finitely many steps. This
  procedure yields a collection of trees \(\mathcal{T}= \bigcup_{n \in \N}
  \mc{T}_n\) and distinguished subsets \(X_{T'} \subset \pi_{T'} \left(
  \mT_{\Theta^{\tmop{in}}} \right)\) for any \(T' \in \mathcal{T}\). According to
  the selection criterion (and Lemma \Cref{lem:geometry-of-boundary}) we have
  that
  \[ \frac{1}{\mu^1_{\Theta} (T)} \int_{X_T} | F (\eta, y, t) | t \dd_t \1_E
     (\eta, y, t) \dd \eta \dd y \dd t \geq \lambda . \]
  We set \(W_{\lambda} = \bigcup_{T \in \mathcal{T}} W (T) \supset \bigcup_{T
  \in \mathcal{T}} T\) and we have \(\mu^1_{\Theta} (W_{\lambda}) \lesssim
  \sum_{T \in \mathcal{T}} \mu^1_{\Theta} (T)\). Since the procedure terminated
  it holds that
  \[ \| \1_{\R^3_+ \setminus W_{\lambda}} F \circ \Gamma (\eta, y, t) t \dd_t
     \1_E (\eta, y, t) \|_{{{\SL^{(1, 1)}_{(\Theta, \Theta^{\tmop{in}})}} } 
     (T)} \leq \lambda . \]
  It remains to show that \eqref{eq:int-tree-covering:summability} holds. It
  holds that \(\Gamma (E^+) = \Gamma (V) \cap \Gamma (W^+)\) with \(\Gamma (V)
  \in \DD_1^{\cup}\) and \(\Gamma (W^+) \in \TT_{\Theta_{\Gamma}}^{\cup}\) so let
  \(\mf{b} (\eta, y)\) be the function whose graph is the boundary of \(\Gamma
  (E^+)\), as described in \Cref{lem:geometry-of-boundary}. Denoting
  \(\Theta_{\Gamma} = (\theta_{\Gamma -}, \theta_{\Gamma +})\), according to
  \Cref{lem:geometry-of-boundary} it holds that
  \begin{equation}
    0 \leq - \dd_{\eta} \left( \frac{1}{\mf{b} (\eta, y)} \right) \leq |
    \theta_{\Gamma -} |^{- 1} ; \label{eq:singular-size-integrand-b-ODE}
  \end{equation}
  the condition \(\dd_{\eta} \mf{b} (\eta, y) \geq 0\) comes from the fact that
  \(\eta \mapsto \mf{b} (\eta, y)\) is non-decreasing: if \((\eta, y, t) \in
  \Gamma (E^+)\) then \((\eta', y, t) \in \Gamma (E^+)\) for any \(\eta' \geq
  \eta\). It follows that for any \(\eta_0 \in \R\) and any \(\eta_0 < \eta <
  \frac{| \theta_{\Gamma -} |}{\mf{b} (\eta_0, y)}\) one has
  \begin{equation}
    \begin{aligned}[t]&
      \mf{b} (\eta_0, y) \leq \mf{b} (\eta, y), \qquad \forall \eta \geq
      \eta_0,\\ &
      \mf{b} (\eta, y) \leq \frac{| \theta_{\Gamma -} | \mf{b} (\eta_0, y)}{|
      \theta_{\Gamma -} | - \mf{b} (\eta_0, y) (\eta - \eta_0)}, \qquad
      \forall \eta \in \left( \eta_0, \frac{| \theta_{\Gamma -} |}{\mf{b}
      (\eta_0, y)} \right) .
    \end{aligned} \label{eq:singular-size-integrand-b-ODE-sol}
  \end{equation}
  We must bound the expression
  \begin{equation}
    \begin{aligned}[t]&
      \Bigl\| \sum_{T \in \mathcal{T}} \1_{\Gamma (X_T)} (\eta, y, t) t \dd_t
      \1_{\Gamma (E)} (\eta, y, t) \Bigr\|_{{{\SL^{(1, 1)}_{\Theta_{\Gamma}}}
      }  (T)}\\ &
      = \frac{1}{s_T} \int_{B_{s_T} (x_T)} \left( \int_{\R } \left( \sum_{T
      \in \mathcal{T}} \1_{\Gamma (X_T)} \left( \eta, y, \mf{b} (\eta, y)
      \right) \right) \1_T \left( \eta, y, \mf{b} (\eta, y) \right) \mf{b}
      (\eta, y) \frac{\dd \eta}{| \Theta |} \right) \dd y.
    \end{aligned} \label{eq:singular-size-integrand}
  \end{equation}
  Without loss of generality, let us suppose that \(T = T_{\Theta_{\Gamma}} (0,
  0, 1)\). Fix \(y \in B_1 (0)\) and fix \(\bar{\theta} > \theta_{\Gamma -}\) close
  to \(\theta_{\Gamma -}\); in particular let \(\bar{\theta} < - | \theta_{\Gamma
  -}^{\tmop{in}} |\).
  
  Let \(\bar{\eta}\) be the lower bound of \(\eta \in \R\) such that \(\eta \mf{b}
  (\eta, y) > \bar{\theta}\) (or \(+ \infty\) if no such \(\eta\) exists). The
  quantity \(\eta \mapsto \eta \mf{b} (\eta, y)\) is non-decreasing, according
  to \eqref{eq:singular-size-integrand-b-ODE}. Let us start by showing that \
  \[ \int_{\eta > \bar{\eta}} \left( \sum_{T \in \mathcal{T}} \1_{\Gamma
     (X_T)} \left( \eta, y, \mf{b} (\eta, y) \right) \right) \1_T \left( \eta,
     y, \mf{b} (\eta, y) \right) \mf{b} (\eta, y) \frac{\dd \eta}{| \Theta |}
     \lesssim 1. \]
  This is relevant only if \(\bar{\eta} < + \infty\). From
  \eqref{eq:singular-size-integrand-b-ODE-sol} we can deduce that for
  \(\bar{\eta} < \eta < \frac{\bar{\theta} - \theta_{\Gamma -}}{2 \mf{b}
  (\bar{\eta}, y)}\) it holds that
  \[ \mf{b} (\eta, y) < \frac{2 | \theta_{\Gamma -} | \mf{b} (\bar{\eta},
     y)}{\bar{\theta} - \theta_{\Gamma}} . \]
  When the graph \(\eta \mapsto \mf{b} (\eta, y)\) intersects \(T\) we have that
  \(\mf{b} (\eta, y) \leq \frac{\theta_{\Gamma_+}}{\eta}\). Thus for \(\eta >
  \frac{\bar{\theta} - \theta_{\Gamma}}{2 \mf{b} (\bar{\eta}, y)}\) we have
  \[ \mf{b} (\eta, y) < \frac{\theta_{\Gamma_+} 2 \mf{b} (\bar{\eta},
     y)}{\bar{\theta} - \theta_{\Gamma_-}} . \]
  We have shown that for all \(\eta > \bar{\eta}\) when the graph \(\eta \mapsto
  \mf{b} (\eta, y)\) intersects \(T\) it holds that
  \[ \mf{b} (\eta, y) < \frac{(2 | \theta_{\Gamma -} | + \theta_{\Gamma_+}) 2
     \mf{b} (\bar{\eta}, y)}{| \Theta |} . \]
  The graph \(\eta \mapsto \mf{b} (\eta, y)\) intersects \(T\) only if \(\eta <
  \frac{\theta_{\Gamma +}}{\mf{b} (\bar{\eta}, y)}\) so
  \[ \int_{\eta > \bar{\eta}} \left( \sum_{T \in \mathcal{T}} \1_{\Gamma
     (X_T)} \left( \eta, y, \mf{b} (\eta, y) \right) \right) \1_T \left( \eta,
     y, \mf{b} (\eta, y) \right) \mf{b} (\eta, y) \frac{\dd \eta}{| \Theta |}
     \leq 2 (| \theta_{\Gamma -} | + \theta_{\Gamma_+}) . \]

  It thus remains to bound the inner integrand of
  \eqref{eq:singular-size-integrand} restricted to when \(\eta < \bar{\eta}\)
  and in particular when \(\theta_{\Gamma -} < \eta b (\eta, y) <
  \bar{\theta}\). Let \(A = \{ (\eta, y, t) : \theta_{\Gamma -} < \eta b (\eta,
  y) < \bar{\theta} \} .\) Let \(T_0 \in \mc{T}\) be the tree selected earliest
  by the algorithm among those for which the graph \(\eta \mapsto \mf{b} (\eta,
  y)\) intersects \(\Gamma (X_{T_0}) \cap T \cap A\). We need to show the bound
  \begin{equation}
    \int_{\eta = - \infty }^{\bar{\eta}} \left( \sum_{T \in \mathcal{T}_{> 0}}
    \1_{\Gamma (X_T)} \left( \eta, y, \mf{b} (\eta, y) \right) \right) \1_{T
    \cap A} \left( \eta, y, \mf{b} (\eta, y) \right) \mf{b} (\eta, y)
    \frac{\dd \eta}{| \Theta |} \lesssim 1
    \label{eq:singular-size-integrand-y-fixed}
  \end{equation}
  where \(\mathcal{T}_{> 0}\) is the sub collection of those trees in \(\mc{T}\)
  selected after \(T_{n_0}\). Let \(\eta_0\) be the lower bound of \(\eta\) for
  which the graph \(\eta \mapsto \mf{b} (\eta, y)\) intersects \(\Gamma \left(
  X_{T_{n_0}} \right) \cap T\). It holds that
  \[ \alpha \xi_{T_{n_0}} < \eta_0 + | \theta_{\Gamma -}^{\tmop{in}} | \mf{b}
     (\eta_0, y)^{- 1} < 0 \]
  because \(\eta_0 \mf{b} (\eta_0, y) < \bar{\theta} < - | \theta_{\Gamma
  -}^{\tmop{in}} |\). It follows that \(\left\{ (\eta, y, t) \in T : \eta \leq
  \eta_0, t < \mf{b} (\eta_0, t) \right\} \subset \Gamma (T_{n_0})\) and thus
  the graph \(\eta \mapsto \mf{b} (\eta, y)\) over \(\eta \in (- \infty, \eta_0]\)
  cannot intersect \(\Gamma (X_{T_{>}}) \cap T \cap A\) for any \(T_{>} \in
  \mc{T}\) selected after \(T_{n_0}\). We are down to bounding
  \[ \int_{\eta  = \eta_0}^{\bar{\eta}} \left( \sum_{T \in \mathcal{T}_{> 0}}
     \1_{\Gamma (X_T)} \left( \eta, y, \mf{b} (\eta, y) \right) \right) \1_{T
     \cap A} \left( \eta, y, \mf{b} (\eta, y) \right) \mf{b} (\eta, y)
     \frac{\dd \eta}{| \Theta |} \]
  Any tree \(T_{>} \in \mathcal{T}_{> 0}\) satisfies
  \[ \alpha \xi_{T'} < \alpha \xi_{T_{n_0}} + \alpha \varepsilon_{\max} <
     \eta_0 + | \theta_{\Gamma -}^{\tmop{in}} | \mf{b} (\eta_0, y)^{- 1} +
     \alpha \varepsilon_{\max} < \eta_0 + 2 | \theta_{\Gamma -}^{\tmop{in}} |
     \mf{b} (\eta_0, y)^{- 1} \]
  by the selection condition, given that \(\epsilon_{\max} < \frac{|
  \theta_{\Gamma -}^{\tmop{in}} |}{S}\). The graph \(\eta \mapsto \mf{b} (\eta,
  y)\) intersects \(\Gamma (X_{T_{>}}) \cap T \subset \Gamma \left( \pi_{T_{>}}
  \left( \mT_{\Theta^{\tmop{in}}} \right) \right) \cap T\) only if
  \[ \mf{b} (\eta, y) < \frac{| \theta^{\tmop{in}}_{\Gamma +} |}{\eta -
     \alpha \xi_{T_{>}}} . \]
  This shows that if \(\eta > \alpha \xi_{T_{>}} + \varepsilon' \mf{b} (\eta_0,
  y)^{- 1}\) for some \(\varepsilon' > 0\) then
  \[ \mf{b} (\eta, y) < \frac{| \theta^{\tmop{in}}_{\Gamma +} | \mf{b}
     (\eta_0, y)}{\varepsilon'} . \]
  We have already established that
  \[ \mf{b} (\eta, y) \leq \frac{| \theta_{\Gamma -} | \mf{b} (\eta_0, y)}{|
     \theta_{\Gamma -} | - \mf{b} (\eta_0, y) (\eta - \eta_0)} \]
  so if \(\eta < \alpha \xi_{T'} + \varepsilon' \mf{b} (\eta_0, y)^{- 1}\) then
  \(0 < \eta - \eta_0 < (2 | \theta_{\Gamma -}^{\tmop{in}} | + \varepsilon')
  \mf{b} (\eta_0, y)^{- 1}\) so
  \[ \mf{b} (\eta, y) \leq \frac{| \theta_{\Gamma -} | \mf{b} (\eta_0, y)}{|
     \theta_{\Gamma -} | - (2 | \theta_{\Gamma -}^{\tmop{in}} | +
     \varepsilon')} . \]
  Choosing \(\varepsilon' = 1 / 4\) we get that
  \[ \mf{b} (\eta, y) \leq (| \theta_{\Gamma -} | + 4 |
     \theta^{\tmop{in}}_{\Gamma +} |) \mf{b} (\eta_0, y) \]
  when the graph \(\eta \mapsto \mf{b} (\eta, y)\) intersects \(\Gamma
  (X_{T_{>}}) \cap T\). Since \(\eta \mapsto \mf{b} (\eta, y)\) is non-decreasing
  it holds that the graph \(\eta \mapsto \mf{b} (\eta, y)\) intersects \(T\) only
  if \(\eta \mf{b} (\eta_0, y) < \theta_{\Gamma +}\) so we have shown that
  \[ \begin{aligned}[t]&
       \int_{\eta  = \eta_0}^{\bar{\eta}} \left( \sum_{T \in \mathcal{T}_{>
       0}} \1_{\Gamma (X_T)} \left( \eta, y, \mf{b} (\eta, y) \right) \right)
       \1_{T \cap A} \left( \eta, y, \mf{b} (\eta, y) \right) \mf{b} (\eta, y)
       \frac{\dd \eta}{| \Theta |}\\ &       \qquad 
       \leq \int_{\eta  = \theta_{\Gamma -} / \mf{b} (\eta_0,
       y)}^{\theta_{\Gamma +} / \mf{b} (\eta_0, y)} (| \theta_{\Gamma -} | + 4
       | \theta^{\tmop{in}}_{\Gamma +} |) \mf{b} (\eta_0, y) \frac{\dd \eta}{|
       \Theta |} \lesssim | \theta_{\Gamma -} | + 4 |
       \theta^{\tmop{in}}_{\Gamma +} | .
     \end{aligned} \]
  The claim follows. 
\end{proof}

The covering lemma above allows us to obtain the following uniform weak \(L^1\)
bound.

\begin{lemma}
  \label{lem:unif-sing-bound}Let \(V  \in \mathbb{D}^{\cup}_{\beta}\), \(W  \in
  \TT_{\Theta}^{\cup}\) and \(E = V \cap W\) or \(E = V \cup W\). The bound
  \begin{equation}
    \begin{aligned}[t]&
      \left\| F \circ \Gamma (\eta, y, t) t \dd_t \1_E (\eta, y, t)
      \right\|_{L^{1, \infty}_{\mu_{\Theta}^1} \Gamma^{\ast} \SL^{(1,
      1)}_{(\Theta, \Theta^{\tmop{in}})}} \lesssim  \| F (\eta, y, t) 
      \|_{L^1_{\mu_{\Theta}^1} \SL^{(\infty, \infty)}_{\Theta_{\Gamma}}},
    \end{aligned} \label{eq:unif-sing-bound}
  \end{equation}
  holds for any function \(F \in L^{\infty}_{\tmop{loc}} (\R^3_+)\). The
  implicit constant is independent of the set \(E\).
\end{lemma}

\begin{proof}
  We need to show that for every \(\lambda > 0\) there exists a set \(W_{\lambda}
  = \bigcup_{T \in \mathcal{T} \subset \mathbb{T}_{\Theta}} T\) such that
  \[ \begin{aligned}[t]&
       \left\| \1_{\R^3_+ \setminus W_{\lambda} } F \circ \Gamma (\eta, y, t)
       t \dd_t \1_E (\eta, y, t) \right\|_{\Gamma^{\ast} \SL^{(1,
       1)}_{(\Theta, \Theta^{\tmop{in}})}} \leq \lambda,\\ &
       \mu_{\Theta}^1 (W_{\lambda}) \lesssim \lambda^{- 1}
       \|F\|_{L^1_{{\mu_{\Theta }^1}_{\Gamma}} \SL^{(\infty,
       \infty)}_{\Theta_{\Gamma}}} .
     \end{aligned} \]
  We may assume that \(F\) is supported on a compact set by the discussion at
  the end of \Cref{sec:sizes}. Fix \(\lambda > 0\) and let \(W_{\lambda} \assign
  \bigcup_{T' \in \mathcal{T }_{\lambda} \subset \mathbb{T}_{\Theta}} T'\)
  where \(\mathcal{T}_{\lambda} \subset \mathbb{T}_{\Theta}\) is obtained by
  applying \Cref{lem:singular-tree-covering} to \(F \circ \Gamma \) in place of
  \(F\). It remains to prove the required bounds on \(\mu_{\Theta}^1
  (W_{\lambda})\). Let \((\eta, y) \mapsto \mf{b} (\eta, y)\) be the function
  whose graph is the boundary of \(E \), as described in
  \Cref{lem:geometry-of-boundary}. For each tree \(T' \in
  \mathcal{T}_{\lambda}\), Lemma \ref{lem:singular-tree-covering} gives the
  bound
  \[ \int_{\R^2} \1_{X_{T'}} \left( \eta, y, \mf{b} (\eta, y) \right) \left| F
     \circ \Gamma (\eta, y, \mf{b} (\eta, y)) \right| \mf{b} (\eta, y) \dd
     \eta \dd y \geq \lambda \mu^1_{\Theta} (T) . \]
  and thus
  \[ \begin{aligned}[t]
       \lambda \mu^1_{\Theta} (W_{\lambda}) &\lesssim  \lambda \sum_{T' \in
       \mathcal{T}_{\lambda}} \mu^1_{\Theta} (T')\\ &
       \lesssim \int_{\R^2} \left( \sum_{T' \in \mathcal{T}_{\lambda}}
       \1_{X_T} \left( \eta, y, \mf{b} (\eta, y) \right) \right) \left| F
       \circ \Gamma (\eta, y, \mf{b} (\eta, y)) \right| \mf{b} (\eta, y)
       \dd \eta \dd y\\ &
       \lesssim \begin{aligned}[t]&
         \int_{\R^2} \left( \sum_{T' \in \mathcal{T}_{\lambda}} \1_{\Gamma
         (X_{T'})} \left( \alpha \left( \eta + \gamma \mf{b} (\eta, y)^{- 1}
         \right), y, \beta \mf{b} (\eta, y) \right) \right)\\ &
         \qquad \times \left| F (\alpha \left( \eta + \gamma \mf{b} (\eta,
         y)^{- 1} \right), y, \beta \mf{b} (\eta, y)) \right| \mf{b} (\eta, y)
         \dd \eta \dd y.
       \end{aligned}
     \end{aligned} \]
  Let us enact the Lipschitz change of variables \(\eta' (\eta, y) = \alpha
  \left( \eta + \gamma \mf{b} (\eta, y)^{- 1} \right)\). We have the following
  bound on the Jacobian:
  \[ \frac{\dd \eta' (\eta, y)}{\dd \eta} = \alpha \left( 1 - \gamma
     \frac{\partial_{\eta} \mf{b} (\eta, y)}{\mf{b} (\eta, y)^2} \right)
     \gtrsim \alpha ; \]
  this holds since \(| \gamma | < 1\) and \(| \theta_+ |, | \theta_- | \geq 2\)
  and thus, according to \eqref{eq:boundary-regularity}, it holds that \(\left|
  \gamma \frac{\partial_{\eta} \mf{b} (\eta, y)}{\mf{b} (\eta, y)^2} \right| <
  \frac{1}{2}\). Let \((\eta', y) \mapsto \mf{b}_{\Gamma} (\eta', y)\) be the
  function whose graph is the boundary of \(\Gamma (E)\).
  \Cref{lem:geometry-of-boundary} applies since \(\Gamma (E) = \Gamma (V) \cap
  \Gamma (W)\) (or, respectively, \(\Gamma (E) = \Gamma (V) \cup \Gamma (W)\))
  and it holds that \(\Gamma (V) \in \DD_1^{\cup}\) and \(\Gamma (W) \in
  \TT^{\cup}_{\Theta_{\Gamma}} .\) It follows that
  \[ \beta \mf{b} (\eta, y) = \mf{b}_{\Gamma} (\eta', y) \]
  since the boundary of \(E\) is the set
  \( \left\{ \left( \eta, y, \mf{b} (\eta, y) \right), (\eta, y) \in \R^2
     \right\} \)
  while the image under \(\Gamma\) of this set is
  \[ \begin{aligned}[t]&
       \Gamma \left( \left\{ \left( \eta, y, \mf{b} (\eta, y) \right), (\eta,
       y) \in \R^2 \right\} \right)\\ &
       \qquad \begin{aligned}[t]&
         = \left\{ \left( \left( \eta + \gamma \mf{b} (\eta, y)^{- 1} \right),
         y, \beta \mf{b} (\eta, y) \right) \suchthat (\eta, y) \in \R^2
         \right\}\\ &
         = \left\{ \left( \eta', y, \mf{b}_{\Gamma} (\eta', y) \right)
         \suchthat (\eta', y) \in \R^2 \right\} .
       \end{aligned}
     \end{aligned} \]
  We get that
  \[ \begin{aligned}[t]&
       \lambda \mu^1_{\Theta} (W_{\lambda})
       \lesssim \frac{1}{\alpha \beta} \int_{\R^2}
       \begin{aligned}[t]
       & \left( \sum_{T' \in
       \mathcal{T}_{\lambda}} \1_{\Gamma (X_{T'})} \left( \eta', y,
     \mf{b}_{\Gamma} (\eta', y) \right) \right)
     \\ & \qquad\times\left| F (\eta', y,
     \mf{b}_{\Gamma} (\eta', y)) \right| \mf{b}_{\Gamma} (\eta', y) \dd \eta' \dd y
       \end{aligned}
     \end{aligned} \]
  It remains to bound the \(\RHS{}\) of this inequality from above. By
  \Cref{prop:outer-RN} and by the Hölder inequality
  (\Cref{cor:outer-holder-classical}) we have that
  \[ \begin{aligned}[t]&
       \int_{\R^2} \left( \sum_{T' \in \mathcal{T}_{\lambda}} \1_{\Gamma
       (X_{T'})} \left( \eta', y, \mf{b}_{\Gamma} (\eta', y) \right) \right)
       \left| F (\eta', y, \mf{b}_{\Gamma} (\eta', y)) \right| \mf{b}_{\Gamma}
       (\eta', y) \dd \eta' \dd y\\ &
       = \int_{\R^2} \left( \sum_{T' \in \mathcal{T}_{\lambda}} \1_{\Gamma
       (X_{T'})} (\eta', y, t') \right) | F (\eta', y, t') | t' \delta \left(
       t' - \mf{b}_{\Gamma} (\eta', y) \right) \dd \eta' \dd y\\ &
       \leq \left\| \left( \sum_{T' \in \mathcal{T}_{\lambda}} \1_{\Gamma
       (X_{T'})} (\eta', y, t') \right) F (\eta', y, t') t' \delta \left( t' -
       \mf{b}_{\Gamma} (\eta', y) \right)
       \right\|_{L^1_{\mu_{\Theta_{\Gamma}}^1} \SL^{(1,
       1)}_{\Theta_{\Gamma}}}\\ &
       \leq \| F (\eta', y, t') \|_{L^1_{\mu_{\Theta_{\Gamma}}^1}
       \SL^{(\infty, \infty)}_{\Theta_{\Gamma}}} \left\| \left( \sum_{T' \in
       \mathcal{T}_{\lambda}} \1_{\Gamma (X_{T'})} (\eta', y, t') \right) t'
       \delta \left( t' - \mf{b}_{\Gamma} (\eta', y) \right)
       \right\|_{L^{\infty}_{\mu_{\Theta_{\Gamma}}^1} \SL^{(1,
       1)}_{\Theta_{\Gamma}}}
     \end{aligned} \]
  Using bound \eqref{eq:int-tree-covering:summability} provided by
  \Cref{lem:singular-tree-covering} we have
  \[ \begin{aligned}[t]&
       \left\| \left( \sum_{T' \in \mathcal{T}_{\lambda}} \1_{\Gamma
       (X_{T'})} (\eta', y, t') \right) t' \delta \left( t' - \mf{b}_{\Gamma}
       (\eta', y) \right) \right\|_{L^{\infty}_{\mu_{\Theta_{\Gamma}}^1}
       \SL^{(1, 1)}_{\Theta_{\Gamma}}}\\ &
       \qquad = \left\| \left( \sum_{T' \in \mathcal{T}_{\lambda}} \1_{\Gamma
       (X_{T'})} (\eta', y, t') \right) t' \dd_t \1_{\Gamma (E)} (\eta', y,
       t') \right\|_{L^{\infty}_{\mu_{\Theta_{\Gamma}}^1} \SL^{(1,
       1)}_{\Theta_{\Gamma}}} \lesssim 1
     \end{aligned} \]
  that concludes the proof.
\end{proof}

We are now ready to deduce \Cref{prop:unif-gamma-bounds:singular-size} from
\Cref{lem:unif-sing-bound}. The procedure is analogous to the deduction of
\Cref{prop:unif-gamma-bounds:Sex} from \Cref{lem:unif-Sex-bound}: first we
reduce to a scalar inequality and then we conclude using restricted weak type
interpolation with one endpoint provided by \Cref{lem:unif-sing-bound}. The
\(L^{\infty}\) endpoint of this argument is more involved than in the proof of
\Cref{prop:unif-gamma-bounds:Sex}. The reader is encouraged to refer to the
proof of \Cref{prop:unif-gamma-bounds:Sex} for the explanation of the common
steps with the proof below, that here are presented succinctly.

\begin{proof}{Proof of \Cref{prop:unif-gamma-bounds:singular-size}}
  Bound \eqref{eq:unif-gamma-bounds:singular-size} follows from the scalar
  bound
  \begin{equation}
    \left\|  F \circ \Gamma (\eta, y, t) t \dd_t \1_E (\eta, y, t)
    \right\|_{L^r_{\mu_{\Theta}^1} \Gamma^{^{\ast}} \SL^{(u, 1)}_{(\Theta,
    \Theta^{\tmop{in}} )}} \lesssim \| F (\eta, y, t) 
    \|_{L^r_{\mu_{\Theta_{\Gamma}}^1} \SL^{(\infty,
    \infty)}_{\Theta_{\Gamma}}}
    \label{eq:unif-gamma-bound-non-iter:singular:scalar}
  \end{equation}
  for all \(F \in L^{\infty}_{\tmop{loc}} (\R^{3}_{+})\), as has been
  discussed in \Cref{cor:wave-packet-decomposition-and-sizes} and at the
  beginning of the proof of \Cref{prop:unif-gamma-bounds:Sex}. Next, we deduce
  the bound \eqref{eq:unif-gamma-bound-non-iter:singular:scalar} from the
  bound \eqref{eq:unif-sing-bound}. It is sufficient to prove
  \eqref{eq:unif-gamma-bound-non-iter:singular:scalar} with \(u = 1\) since \(F
  \in L^{\infty}_{\tmop{loc}} (\R^{3}_{+})\) is arbitrary,
  \[ \| F (\eta, y, t)  \|_{L^r_{\mu_{\Theta_{\Gamma}}^1} \SL^{(\infty,
     \infty)}_{\Theta_{\Gamma}}} = \| | F (\eta, y, t) |^u  \|_{L^{r /
     u}_{\mu_{\Theta_{\Gamma}}^1} \SL^{(\infty, \infty)}_{\Theta_{\Gamma}}}^{1
     / u}, \]
  and by \eqref{eq:boundary-singular-size} it holds that
  \[ \begin{aligned}[t]&
       \left\|  F \circ \Gamma (\eta, y, t) t \dd_t \1_E (\eta, y, t)
       \right\|_{L^r_{\mu_{\Theta}^1} \Gamma^{^{\ast}} \SL^{(u, 1)}_{(\Theta,
       \Theta^{\tmop{in}} )}}\\ &
       \qquad \lesssim \left\|  | F |^u \circ \Gamma (\eta, y, t) t \dd_t
       \1_E (\eta, y, t) \right\|_{L^{r / u}_{\mu_{\Theta}^1}
       \Gamma^{^{\ast}} \SL^{(1, 1)}_{(\Theta, \Theta^{\tmop{in}} )}}^{1 / u}
       .
     \end{aligned} \]
  Now, supposing \(u = 1\), we prove the bound by interpolation
  (\Cref{prop:outer-restricted-interpolation}). The weak \(L^1\) endpoint is
  given by \Cref{lem:unif-sing-bound}. The \(r = \infty\) endpoint for
  \eqref{eq:unif-gamma-bound-non-iter:Sex:scalar:weak:non-diag}, given by
  \[ \left\|  F \circ \Gamma (\eta, y, t) t \dd_t \1_E (\eta, y, t)
     \right\|_{L^{\infty}_{\mu_{\Theta}^1} \Gamma^{^{\ast}} \SL^{(1,
     1)}_{(\Theta, \Theta^{\tmop{in}} )}} \leq \| F (\eta, y, t) 
     \|_{L^r_{\mu_{\Theta_{\Gamma}}^1} \SL^{(\infty,
     \infty)}_{\Theta_{\Gamma}}} . \]
  This bound holds because
  \[ \begin{aligned}[t]&
       \left\|  F \circ \Gamma (\eta, y, t) t \dd_t \1_E (\eta, y, t)
       \right\|_{L^{\infty}_{\mu_{\Theta}^1} \Gamma^{^{\ast}} \SL^{(1,
       1)}_{(\Theta, \Theta^{\tmop{in}} )}}\\ &
       \qquad \leq \|  F \circ \Gamma (\eta, y, t) 
       \|_{L^{\infty}_{\mu_{\Theta}^1} \Gamma^{^{\ast}} \SL^{(\infty,
       \infty)}_{\Theta}} \left\|  t \dd_t \1_E (\eta, y, t)
       \right\|_{L^{\infty}_{\mu_{\Theta}^1} \Gamma^{^{\ast}} \SL^{(1,
       1)}_{(\Theta, \Theta^{\tmop{in}} )}}\\ &
       \hspace{2.0em} \leq \|  F \|_{L^{\infty}_{\mu_{\Theta_{\Gamma}}^1}
       \SL^{(\infty, \infty)}_{\Theta_{\Gamma}}} \left\|  t \dd_t \1_E (\eta,
       y, t) \right\|_{L^{\infty}_{\mu_{\Theta}^1} \Gamma^{^{\ast}} \SL^{(1,
       1)}_{(\Theta, \Theta^{\tmop{in}} )}} \leq \|  F
       \|_{L^{\infty}_{\mu_{\Theta_{\Gamma}}^1} \SL^{(\infty,
       \infty)}_{\Theta_{\Gamma}}} .
     \end{aligned} \]
  The last bound requires us to show that
  \begin{equation}
    \left\|  t \dd_t \1_E (\eta, y, t) \right\|_{\Gamma^{^{\ast}} \SL^{(1,
    1)}_{(\Theta, \Theta^{\tmop{in}} )} (T)} \lesssim 1 \qquad \forall T \in
    \TT_{\Theta}, \label{eq:boundary-bound}
  \end{equation}
  that we will now show and, allowing us to conclude the proof. Let \((\eta, y)
  \mapsto \mf{b} (\eta, y)\) is the function whose graph is the boundary of
  \(E\), as described by \Cref{lem:geometry-of-boundary}. The same lemma gives
  that
  \[ \begin{aligned}[t]&
       \left\|  t \dd_t \1_E (\eta, y, t) \right\|_{\Gamma^{^{\ast}}
       \SL^{(1, 1)}_{(\Theta, \Theta^{\tmop{in}} )} (T)}\\ &
       \qquad = \int_{\R^2_+} \1_{\pi_T \left( \mT_{\Theta^{\tmop{in}}}
       \right)} \left( \eta, y, \mf{b} (\eta, y) \right) \mf{b} (\eta, y)
       \frac{\dd y \dd \eta}{s_T | \Theta |} .
     \end{aligned} \]
  By symmetry, assume \(T = T_{\Theta} (0, 0, 1)\). Let us denote \(\Theta =
  (\theta_-, \theta_+)\) and \(\Theta^{\tmop{in}} = (\theta_-^{\tmop{in}},
  \theta^{\tmop{in}}_+)\). Integrating bounds \eqref{eq:boundary-regularity},
  for \(\eta > 0\) it holds that
  \[ \frac{| \theta_+ | \mf{b} (0, y)}{| \theta_+ | + \eta \mf{b} (0, y) }
     \leq \mf{b} (\eta, y), \]
  and if also \(\eta < | \theta_- | \mf{b} (0, y)^{- 1}\) then
  \[ \mf{b} (\eta, y) \leq \frac{| \theta_- | \mf{b} (0, y)}{| \theta_- | -
     \eta \mf{b} (0, y) } . \]
  When \(\eta < 0\) it holds that
  \[ \frac{| \theta_- | \mf{b} (0, y)}{| \theta_- | - \eta \mf{b} (0, y) }
     \leq \mf{b} (\eta, y) \leq \frac{| \theta_+ | \mf{b} (0, y)}{| \theta_+ |
     + \eta \mf{b} (0, y) } \]
  and if also \(\eta > - | \theta_+ | \mf{b} (0, y)^{- 1}\) then
  \[ \mf{b} (\eta, y) \leq \frac{| \theta_+ | \mf{b} (0, y)}{| \theta_+ | +
     \eta \mf{b} (0, y) } . \]
  Since \(\left( \eta, y, \mf{b} (\eta, y) \right) \in \pi_T \left(
  \mT_{\Theta^{\tmop{in}}} \right)\) only if \(\left| \eta \mf{b} (\eta, y)
  \right| < \theta^{\tmop{in}}_{\ast} \eqd \max (| \theta^{\tmop{in}}_+ |, |
  \theta^{\tmop{in}}_- |)\).
  
  Let us show that if the graph of \(\mf{b}\) intersects \(\pi_T \left(
  \mT_{\Theta^{\tmop{in}}} \right)\) then \(\mf{b} (\eta, y) \lesssim \mf{b} (0,
  y)\). Supposing that \(\eta > 0\), if \(\mf{b}\) intersects \(\pi_T \left(
  \mT_{\Theta^{\tmop{in}}} \right)\) then
  \[ \mf{b} (\eta, y) \leq \frac{| \theta_- | \mf{b} (0, y)}{| \theta_- | -
     \eta \mf{b} (0, y) } \1_{\eta < \frac{| \theta_- |}{2 \mf{b} (0, y)} } +
     \frac{\theta^{\tmop{in}}_{\ast}}{\eta} \1_{\eta > \frac{| \theta_- |}{2
     \mf{b} (0, y)} } \leq 2 \left( 1 + \frac{\theta^{\tmop{in}}_+}{| \theta_-
     |} \right) \mf{b} (0, y) \]
  On the other hand if \(\eta < 0\) then
  \[ \mf{b} (\eta, y) \leq \frac{| \theta_+ | \mf{b} (0, y)}{| \theta_+ | +
     \eta \mf{b} (0, y) } \1_{\eta > - \frac{| \theta_+ |}{2 \mf{b} (0, y)}} +
     \frac{\theta^{\tmop{in}}_{\ast}}{| \eta |} \1_{\eta < - \frac{| \theta_+
     |}{2 \mf{b} (0, y)}} \leq 2 \left( 1 + \frac{| \theta^{\tmop{in}}_- |}{|
     \theta_+ |} \right) \mf{b} (0, y) . \]
  Let us show that if the graph of \(\mf{b}\) intersects \(\pi_T \left(
  \mT_{\Theta^{\tmop{in}}} \right)\) then \(| \eta | \lesssim \mf{b} (0, y)^{-
  1}\). Given the bounds on \(\mf{b}\), the graph of \(\mf{b}\) intersects \(\pi_T
  \left( \mT_{\Theta^{\tmop{in}}} \right)\) only if
  \[ \begin{aligned}[t]&
       \frac{| \theta_+ | \mf{b} (0, y)}{| \theta_+ | + \eta \mf{b} (0, y) } <
       \frac{\theta^{\tmop{in}}_{\ast}}{\eta} \text{ when } \eta > 0,\\ &
       \frac{| \theta_- | \mf{b} (0, y)}{| \theta_- | + | \eta | \mf{b} (0, y)
       } < \frac{\theta^{\tmop{in}}_{\ast}}{| \eta |} \text{ when } \eta < 0,
     \end{aligned} \]
  and thus
  \[ \begin{aligned}[t]&
       \mf{b} (0, y) (| \theta_+ | - \theta^{\tmop{in}}_{\ast}) | \eta | < |
       \theta_+ | \theta^{\tmop{in}}_{\ast} \text{ when } \eta > 0,\\ &
       \mf{b} (0, y) (| \theta_- | - \theta^{\tmop{in}}_{\ast}) | \eta | < |
       \theta_+ | \theta^{\tmop{in}}_{\ast} \text{ when } \eta < 0.
     \end{aligned} \]
  Since \(\theta^{\tmop{in}}_{\ast} < 1 < \max (| \theta_+ |, | \theta_- |)\) it
  follows that \(| \eta | \lesssim \mf{b} (0, y)^{- 1}\). Together with the
  bound \(\mf{b} (\eta, y) \lesssim \mf{b} (0, y)\) it shows that
  \[ \int_{\R^2_+} \1_{\pi_T \left( \mT_{\Theta^{\tmop{in}}} \right)} \left(
     \eta, y, \mf{b} (\eta, y) \right) \mf{b} (\eta, y) \frac{\dd y \dd
     \eta}{s_T | \Theta |} \lesssim 1 \]
  as required.
\end{proof}

\subsection{Uniform bounds for the \(\widetilde{\SF}_{\Gamma}^u \Phi_{\gr}^N\)
size}\label{sec:unif-derived-size-bound}

In this section we prove \Cref{lem:unif-derived-size-bound}. The size
\(\widetilde{\SF}_{\Gamma}^u \Phi_{\gr}^N\), as defined in
\eqref{eq:uniform-embedding-full-size:linear}, consists of three parts; the
three lemmata \ref{lem:uniform-defect-bound},
\ref{lem:unif-lac-size-domination}, and \ref{lem:unif-max-size-domination}
below address these parts: the sizes \({\Gamma^{\ast}}  \SL^{(u, 1)}_{(\Theta,
\Theta^{\tmop{in}})} \dfct_{\gr}^N\), \({\Gamma^{\ast}}  \SL^{(u, 2)}_{\Theta}
\wpD_{\gr}^N\), and \({\beta^{\frac{1}{u}}}  \Gamma^{\ast} \SL^{(u,
\infty)}_{\Theta} \Phi_{\gr}^N\) respectively. \Cref{lem:uniform-defect-bound},
up next, amounts to showing that the defect of \(\1_{(V^+ \cap W^+) \setminus
V^-} \Emb [f] \circ \Gamma\) is essentially concentrated on the boundary of the
region \((V^+ \cap W^+) \setminus V^-\) and thus can be controlled by the
quantity on {\LHS{\eqref{prop:unif-gamma-bounds:singular-size}}}, introduced
in \Cref{prop:unif-gamma-bounds:singular-size}.

\begin{lemma}[Defect size domination]
  \label{lem:uniform-defect-bound}Let \(V^+, V^- \in \mathbb{D}^{\cup}_{\beta}\)
  and \(W^+, W^- \in \TT_{\Theta}^{\cup}\). The bounds
  \begin{equation}
    \begin{aligned}[t]&
      \Bigl\| \1_{(V^+ \cap W^+) \setminus (V^- \cup W^-)} \left( \Emb [f]
      \circ \Gamma \right) \Bigr\|_{{\Gamma^{\ast}}  \SL^{(u, 1)}_{(\Theta,
      \Theta^{\tmop{in}})} \dfct_{\gr}^N}\\ &
      \lesssim \begin{aligned}[t]&
        \Bigl\| \1_{\overline{(V^+ \cap W^+)} \setminus (V^- \cup W^-)}
        \left( \Emb [f] \circ \Gamma \right) (\eta, y, t) t \dd_t \1_{(V^+
        \cap W^+)} (\eta, y, t)  \Bigr\|_{{\Gamma^{\ast}}  \SL^{(u,
        1)}_{(\Theta, \Theta^{\tmop{in}})} \Phi_{16 \gr}^{N - 8}}\\ &
        + \Bigl\| \1_{\overline{(V^+ \cap W^+)} \setminus (V^- \cup W^-)}
        \left( \Emb [f] \circ \Gamma \right) (\eta, y, t) t \dd_t \1_{\left(
        \R^3_+ \setminus V^- \right)} (\eta, y, t)  \Bigr\|_{{\Gamma^{\ast}} 
        \SL^{(u, 1)}_{(\Theta, \Theta^{\tmop{in}})} \Phi_{16 \gr}^{N - 8}}\\ &
        + \Bigl\| \1_{(V^+ \cap W^+) \setminus (V^- \cup W^-)} \left( \Emb
        [f] \circ \Gamma \right) \Bigr\|_{\Gamma^{\ast} \SL^{(u,
        2)}_{(\Theta, \Theta^{\tmop{ex}})} \Phi_{16 \gr}^{N - 8}}
      \end{aligned}
    \end{aligned} \label{eq:geometry-of-defects}
\end{equation}

  hold for any \(f \in \Sch (\R)\) and for all \(u \in [1, +
  \infty]\). The implicit constant is independent of the sets \(V^+\), \(V^-\),
  \(W^+\), and \(W^-\).
\end{lemma}

\begin{proof}
  To show \eqref{eq:geometry-of-defects} it is sufficient to prove that
  \[ \begin{aligned}[t]&
       \begin{aligned}[t]&
         \Bigl\| \1_{(V^+ \cap W^+) \setminus (V^- \cup W^-)} \left( \Emb [f]
         \circ \Gamma \right) \Bigr\|_{{\Gamma^{\ast}}  \SL^{(u,
         1)}_{(\Theta, \Theta^{\tmop{in}})} \dfct_{\zeta, \gr}^N (T)}\\ &
         \qquad + \Bigl\| \1_{(V^+ \cap W^+) \setminus (V^- \cup W^-)} \left(
         \Emb [f] \circ \Gamma \right) \Bigr\|_{{\Gamma^{\ast}}  \SL^{(u,
         1)}_{(\Theta, \Theta^{\tmop{in}})} \dfct_{\sigma, \gr}^N (T)}
       \end{aligned}\\ &
       \lesssim \begin{aligned}[t]&
         \Bigl\| \1_{\overline{(V^+ \cap W^+)} \setminus (V^- \cup W^-)}
         \left( \Emb [f] \circ \Gamma \right) (\eta, y, t) t \dd_t \1_{(V^+
         \cap W^+)} (\eta, y, t)  \Bigr\|_{{\Gamma^{\ast}}  \SL^{(u,
         1)}_{(\Theta, \Theta^{\tmop{in}})} \Phi_{16 \gr}^{N - 8}}\\ &
         + \Bigl\| \1_{\overline{(V^+ \cap W^+)} \setminus (V^- \cup W^-)}
         \left( \Emb [f] \circ \Gamma \right) (\eta, y, t) t \dd_t \1_{\left(
         \R^3_+ \setminus V^- \right)} (\eta, y, t)  \Bigr\|_{{\Gamma^{\ast}}
         \SL^{(u, 1)}_{(\Theta, \Theta^{\tmop{in}})} \Phi_{16 \gr}^{N - 8}}\\ &
         + \Bigl\| \1_{\overline{(V^+ \cap W^+)} \setminus (V^- \cup W^-)}
         \left( \Emb [f] \circ \Gamma \right) \Bigr\|_{\Gamma^{\ast} \SL^{(u,
         2)}_{(\Theta, \Theta^{\tmop{ex}})} \Phi_{16 \gr}^{N - 8}}
       \end{aligned}
     \end{aligned} \]
  holds for any \(T \in \TT_{\Theta}\). Let us fix \(f \in \Sch (\R)\). By symmetry, according to \eqref{eq:size-symmetry} and the
  subsequent discussion, we can assume that \({T = T_{\Theta}}  (0, 0, 1)\) so
  that \(\Gamma \circ \pi_T (\theta, \zeta, \sigma) = (\alpha \beta (\theta +
  \gamma) (\beta \sigma)^{- 1}, \zeta, \beta \sigma) = (\theta_{\Gamma} (\beta
  \sigma)^{- 1}, \zeta, \beta \sigma)\). For conciseness we set \(F_{\Gamma} =
  \Emb [f] \circ \Gamma\). By \eqref{eq:embedded-no-defect} it holds that
  \[ \begin{aligned}[t]&
       \left( \Gamma^{\ast} \wpD_{\zeta} (t \eta) - \beta t \dd_y \right)
       F_{\Gamma} (\eta, y, t) = 0,\\ &
       \left( \Gamma^{\ast} \wpD_{\sigma} (t \eta) - t \dd_t + \eta \dd_{\eta}
       \right) F_{\Gamma} (\eta, y, t) = 0.
     \end{aligned} \]
  According to definitions \eqref{eq:gamma-defect-size} (see also the
  discussion below \eqref{eq:embedded-no-defect}) it holds that
  \begin{equation}
    \begin{aligned}[t]&
      \| F_{\Gamma} (\eta, y, t) \|_{\Gamma^{\ast} {\SL_{(\Theta,
      \Theta^{\tmop{in}}) }^{(u, 1)}}  \dfct_{\zeta, \gr}^{N } (T)}\\ &
      \leq \begin{aligned}[t]&
        {\left\| \left( \beta t \dd_y \1_{(W^+ \cap V^+)} (\eta, y, t)
        \right) \left( \1_{\R^3_+ \setminus (W^- \cup V^-)} F_{\Gamma} \right)
        (\eta, y, t) \right\|_{\SL_{(\Theta, \Theta^{\tmop{in}})}^{(u, 1)}
        \Phi_{\mf{r}}^N (T)}} \\ &
        + {\left\| \left( \beta t \dd_y \1_{\R^3_+ \setminus V^-} (\eta, y,
        t) \right) \left( \1_{\overline{(W^+ \cap V^+)} \setminus (W^- \cup
        V^-)} F_{\Gamma} \right) (\eta, y, t) \right\|_{\SL_{(\Theta,
        \Theta^{\tmop{in}})}^{(u, 1)} \Phi_{\mf{r}}^N (T)}} \\ &
        + {\left\| \left( \beta t \dd_y \1_{\R^3_+ \setminus W^-} (\eta, y,
        t) \right) \left( \1_{\overline{(W^+ \cap V^+)} \setminus (W^- \cup
        V^-)} F_{\Gamma} \right) (\eta, y, t) \right\|_{\SL_{(\Theta,
        \Theta^{\tmop{in}})}^{(u, 1)} \Phi_{\mf{r}}^N (T)}} ,
      \end{aligned}
    \end{aligned} \label{eq:geometry-of-defects:space}
  \end{equation}
  and
  \begin{equation}
    \begin{aligned}[t]&
      \| F_{\Gamma} (\eta, y, t) \|_{\Gamma^{\ast} {\SL_{(\Theta,
      \Theta^{\tmop{in}}) }^{(u, v)}}  \dfct_{\sigma, \gr}^{N } (T)}\\ &
      \leq \begin{aligned}[t]&
        {\left\| \left( \left( t \dd_t - \eta \dd_{\eta} \right) \1_{(W^+
        \cap V^+)} (\eta, y, t) \right) \left( \1_{\R^3_+ \setminus (W^- \cup
        V^-)} F_{\Gamma} \right) (\eta, y, t) \right\|_{\SL_{(\Theta,
        \Theta^{\tmop{in}})}^{(u, 1)} \Phi_{\mf{r}}^N (T)}} \\ &
        + {\left\| \left( \left( t \dd_t - \eta \dd_{\eta} \right) \1_{\R^3_+
        \setminus V^-} (\eta, y, t) \right) \left( \1_{\overline{(W^+ \cap
        V^+)} \setminus (W^- \cup V^-)} F_{\Gamma} \right) (\eta, y, t)
        \right\|_{\SL_{(\Theta, \Theta^{\tmop{in}})}^{(u, 1)} \Phi_{\mf{r}}^N
        (T)}} \\ &
        + {\left\| \left( \left( t \dd_t - \eta \dd_{\eta} \right) \1_{\R^3_+
        \setminus W^-} (\eta, y, t) \right) \left( \1_{\overline{(W^+ \cap
        V^+)} \setminus (W^- \cup V^-)} F_{\Gamma} \right) (\eta, y, t)
        \right\|_{\SL_{(\Theta, \Theta^{\tmop{in}})}^{(u, 1)} \Phi_{\mf{r}}^N
        (T) .}} 
      \end{aligned}
    \end{aligned} \label{eq:geometry-of-defects:scale}
\end{equation}

  Let \(\mf{b}_+\), \(\mf{b}_-\) and \(\mf{b}_{W^-}\) be the functions whose
  graphs are the boundaries of \(V^+ \cap W^+\), \(V^-\), and \(W^-\) respectively,
  given by \Cref{lem:geometry-of-boundary}. According to
  \eqref{eq:boundary-regularity} we have
  \[ \begin{aligned}[t]&
       \left| \beta t \dd_y \1_{V^+ \cap W^+} (\eta, y, t) \right| \lesssim t
       \delta \left( t - \mf{b}_+ (\eta, y) \right),\\ &
       \left| \beta t \dd_y \1_{\R^3_+ \setminus V^-} (\eta, y, t) \right|
       \lesssim t \delta \left( t - \mf{b}_- (\eta, y) \right),\\ &
       \left| \beta t \dd_y \1_{\R^3_+ \setminus W^-} (\eta, y, t) \right|
       \lesssim t \delta \left( t - \mf{b}_{W^-} (\eta, y) \right) .
     \end{aligned} \]
  Similarly, when \((\eta, y, t) \in \pi_T \left( \mT_{\Theta} \right)\), or in
  general when \(\eta t\) is bounded it holds that
  \[ \begin{aligned}[t]&
       \left| \left( t \dd_t - \eta \dd_{\eta} \right) \1_{V^+ \cap W^+}
       (\eta, y, t) \right| \lesssim t \delta \left( t - \mf{b}_+ (\eta, y)
       \right),\\ &
       \left| \left( t \dd_t - \eta \dd_{\eta} \right) \1_{\R^3_+ \setminus
       V^-} (\eta, y, t) \right| \lesssim t \delta \left( t - \mf{b}_- (\eta,
       y) \right),\\ &
       \left| \left( t \dd_t - \eta \dd_{\eta} \right) \1_{\R^3_+ \setminus
       W^-} (\eta, y, t) \right| \lesssim t \delta \left( t - \mf{b}_{W^-}
       (\eta, y) \right) .
     \end{aligned} \]
  As a matter of fact, if \(\mf{b}\) is any of the functions \(\mf{b}_+,
  \mf{b}_-, \mf{b}_{W^-}\) then
  \[ \eta \dd_{\eta} \1_{t > \mf{b} (\eta, y)} = - \delta \left( t - \mf{b}
     (\eta, y) \right) \eta \dd_{\eta} \mf{b} (\eta, y) = - t \delta \left( t
     - \mf{b} (\eta, y) \right) \eta t \frac{\dd_{\eta} \mf{b} (\eta,
     y)}{\mf{b} (\eta, y)^2} \]
  and, according to \eqref{eq:boundary-regularity}, \(\left| \eta t
  \frac{\dd_{\eta} \mf{b} (\eta, y)}{\mf{b} (\eta, y)^2} \right| \lesssim 1\).
  Thus the first two summands in {\RHS{\eqref{eq:geometry-of-defects:space}}}
  and in {\RHS{\eqref{eq:geometry-of-defects:scale}}} are bounded by the first
  two summands in {\RHS{}}\eqref{eq:geometry-of-defects}. It remains to show
  that the last summand of {\RHS{\eqref{eq:geometry-of-defects:space}}} and
  the last summand of {\RHS{\eqref{eq:geometry-of-defects:scale}}} are each
  bounded by {\RHS{}}\eqref{eq:geometry-of-defects}. According to the
  discussion above, this follows by showing the bound
  \[ {\left\| \left( \1_{\overline{(W^+ \cap V^+)} \setminus (W^- \cup V^-)}
     F_{\Gamma} \right) (\eta, y, t) t \delta \left( t - \mf{b}_{W^-} (\eta,
     y) \right) \right\|_{\SL_{(\Theta, \Theta^{\tmop{in}})}^{(u, 1)}
     \Phi_{\mf{r}}^N (T)}}  \lesssim \RHS{} \eqref{eq:geometry-of-defects} .
  \]
  Let \(\mf{b}_+^{\ast}\), \(\mf{b}_-^{\ast}\) and \(\mf{b}_{W^-}^{\ast}\) be the
  functions \(\mf{b}_+\), \(\mf{b}_-\) and \(\mf{b}_{W^-}\) expressed in local
  coordinates of \(\pi_T\) i.e. such that \eqref{eq:boundary-pullback-condition}
  holds for \(\mf{b} = \mf{b}_+\), \(\mf{b} = \mf{b}_-\) or \(\mf{b} =
  \mf{b}_{W^-}\), respectively. By \eqref{eq:boundary-singular-size-local} we
  need to show that
  \[ \int_{\Theta^{\tmop{in}} \times B_1} \left| \1_{E^{\ast}} F^{\ast}
     \left( \theta, \zeta, \mf{b}_{W^-}^{\ast} (\theta, \zeta) \right) [\phi]
     \right|^u \frac{\dd \theta \dd \zeta}{| \Theta |} \lesssim \RHS{}
     \eqref{eq:geometry-of-defects}^u \| \phi \|_{\Phi^N_{\gr}} \]
  where we use the shorthand \(F^{\ast} =
  \1_{\mT_{\Theta^{\tmop{in}}_{\Gamma}}} \left( \Emb [f] \circ \Gamma \circ
  \pi_T \right)\) and\\ \(E^{\ast} = \left\{ (\theta, \zeta, \sigma) :
  \mf{b}_-^{\ast} (\theta, \zeta) < \sigma < \mf{b}_+^{\ast} (\theta, \zeta)
  \right\}\). First we record that locally in \(\theta\) and \(\sigma\) the
  function \(F^{\ast}\) is almost constant. As a matter of fact, it follows form
  the definitions \eqref{eq:embedding} and \eqref{eq:gamma} of \(\Emb\) and
  \(\Gamma\) that
  \[ F^{\ast} (\theta, \zeta, \sigma) [\phi] =
     \1_{\mT_{\Theta^{\tmop{in}}_{\Gamma}}} (\theta, \zeta, \sigma) f \ast
     \Dil_{\sigma} \Mod_{\theta_{\Gamma}} \phi^{\vee} (\zeta) \]
  where \(\phi^{\vee} (z) = \phi (- z)\). Thus, if we define \(\phi_{(\theta',
  \sigma', \theta, \sigma)}\) by setting
  \begin{equation}
    \Dil_{\sigma'} \Mod_{\theta'_{\Gamma}} \phi_{(\theta', \sigma', \theta,
    \sigma)}^{\vee} \eqd \Dil_{\sigma} \Mod_{\theta_{\Gamma}} \phi^{\vee}
    \label{eq:wave-packet-move} .
  \end{equation}
  it holds that \(F^{\ast} (\theta, \zeta, \sigma) [\phi] = F^{\ast} (\theta',
  \zeta', \sigma') [\phi_{(\theta', \sigma', \theta, \sigma)}]\) as long as
  \((\theta', \zeta', \sigma') \in \mT_{\Theta^{\tmop{in}}_{\Gamma}}\).
  Furthermore, any \(\bar{r} > 0\) small enough independently of \(\Gamma\), if \(|
  \theta' - \theta | < \bar{r}\) and \(| \sigma' / \sigma - 1 | < \bar{r}\) one
  has that \(\| \phi_{(\theta', \sigma', \theta, \sigma)} \|_{\Phi^N_{2^{l + 1}
  \gr}} \lesssim \| \phi \|_{\Phi^N_{2^l \gr}}\) for any \(l \in \{ 0, \ldots,
  10 \}\), any \(N < N_0\) for some fixed arbitrary large \(N_0 \in \N\) and any
  \(\phi \in \Phi_{2^l \gr}^N\). Let us fix such a small \(\bar{r} > 0\) so that
  the facts above hold.
  
  Next we want to show that close to any point \(\left( \theta, \zeta,
  \mf{b}_{W^-}^{\ast} (\theta, \zeta) \right) \in
  \mT_{\Theta^{\tmop{in}}_{\Gamma}} \cap E^{\ast}\) there exists a small
  neighborhood in \(\pi_T^{- 1} (T \setminus W^-)\) such that the contribution
  of \(\left| \1_{E^{\ast}} F^{\ast} \left( \theta, \zeta, \mf{b}_{W^-}^{\ast}
  (\theta, \zeta) \right) [\phi] \right|^u\) to the integral we are bounding is
  controlled by \(F^{\ast}\) on that neighborhood. We will then, in turn, bound
  \(F^{\ast}\) on that neighborhood by \(\RHS{} \eqref{eq:geometry-of-defects}\).
  For small enough \(r_{\theta}, r_{\zeta}, r_{\sigma} > 0\) let
  \[ \begin{aligned}[t]&
       R (\theta, \zeta) \eqd \left\{ (\theta', \zeta') \in
       \Theta^{\tmop{in}}_{\Gamma} \times B_1 \of | \theta' - \theta | <
       r_{\theta}, | \zeta' - \zeta | < r_{\zeta} \mf{b}_{W^-}^{\ast} (\theta,
       \zeta) \right\}\\ &
       R^0 (\theta, \zeta) \eqd \left\{ (\theta', \zeta', \sigma') \of
       (\theta', \zeta') \in R (\theta, \zeta),
       \frac{\sigma'}{\mf{b}_{W^-}^{\ast} (\theta, \zeta)} \in [1 - 10
       r_{\sigma}, 1 + 10 r_{\sigma}] \right\},\\ &
       R^+ (\theta, \zeta) \eqd \left\{ (\theta', \zeta', \sigma') \of
       (\theta', \zeta') \in R (\theta, \zeta),
       \frac{\sigma'}{\mf{b}_{W^-}^{\ast} (\theta, \zeta)} \in (1 + 2
       r_{\sigma}, 1 + 8 r_{\sigma}) \right\},\\ &
       R^{+ +} (\theta, \zeta) \eqd \left\{ (\theta', \zeta', \sigma') \of
       (\theta', \zeta') \in R (\theta, \zeta),
       \frac{\sigma'}{\mf{b}_{W^-}^{\ast} (\theta, \zeta)} \in (1 + 4
       r_{\sigma}, 1 + 6 r_{\sigma}) \right\} .
     \end{aligned} \]
  Clearly \(R^{+ +} (\theta, \zeta) \subset R^+ (\theta, \zeta) \subset R^0
  (\theta, \zeta) \subset R (\theta, \zeta) \times \R_+\). From the bounds
  \eqref{eq:boundary-regularity-local} and in particular from the fact that
  \(W^- \in \TT_{\Theta}^{\cup}\) so \(\left| \partial_{\zeta}
  \mf{b}^{\ast}_{W^-} (\theta, \zeta) \right| \leq 1\) we can see that
  \(r_{\theta}, r_{\zeta}, r_{\sigma}\) can be chosen small enough so that
  \[ \begin{aligned}[t]&
       \left( \theta', \zeta', \mf{b}_{W^-}^{\ast} (\theta', \zeta') \right)
       \subset R^0 (\theta, \zeta) \qquad \forall (\theta', \zeta') \in R
       (\theta, \zeta),\\ &
       (\theta', \zeta', \sigma'), (\theta'', \zeta', \sigma'') \in R^0
       (\theta, \zeta) \Longrightarrow | \theta' - \theta | < \bar{r}, |
       \sigma' / \sigma - 1 | < \bar{r},\\ &
       R^{+ +} (\theta, \zeta) \subset R^+ (\theta, \zeta) \subset \pi_T^{- 1}
       \left( \R^3_+ \setminus W^- \right) .
     \end{aligned} \]
  As a matter of fact for fixed \(r_{\theta} > 0\) and \(r_{\zeta} > 0\) it holds
  that for any \((\theta', \zeta') \in R (\theta, \zeta)\) one has
  \[ \mf{b}_{W^-}^{\ast} (\theta, \zeta) (e^{- C r_{\theta}} - r_{\zeta}) <
     \mf{b}_{W^-}^{\ast} (\theta', \zeta') < \mf{b}_{W^-}^{\ast} (\theta,
     \zeta) (e^{C r_{\theta}} + r_{\zeta}) \]
  where \(C = \frac{1}{\dist \left( \Theta^{\tmop{in}}, \R \setminus \Theta
  \right) - r_{\theta}} \lesssim 1\). Thus, choosing \(10 r_t < | (e^{C
  r_{\theta}} + r_{\zeta}) - 1 | + | (e^{- C r_{\theta}} - r_{\zeta}) - 1 |\)
  provides us with the first claim. The other claims can be satisfied
  analogously. It holds that
  \begin{equation}
    \begin{aligned}[t]&
      \int_{R (\theta, \zeta)} \left| {\1_{E^{\ast}}}  F^{\ast} \left(
      \theta', \zeta', \mf{b}_{W^-}^{\ast} (\theta', \zeta') \right) [\phi]
      \right|^u \dd \zeta' \dd \theta'\\ &
      \lesssim \begin{aligned}[t]&
        \int_{R^+ (\theta, \zeta)} \left| {\1_{E^{\ast}}}  F^{\ast} (\theta',
        \zeta', \sigma') \left[ \phi_{\left( \theta'', \sigma', \theta',
        \mf{b}_{W^-}^{\ast} (\theta', \zeta') \right)} \right] \right|^u\\ &
        \qquad \times \left( \1_{L_{(\theta, \zeta)}} (\zeta') + \sigma'
        \delta \left( \sigma' - \mf{b}_-^{\ast} (\theta'', \zeta') \right) +
        \sigma' \delta \left( \sigma' - \mf{b}_+^{\ast} (\theta', \zeta')
        \right) \right) \dd \zeta' \dd \theta' \frac{\dd \sigma'}{\sigma'}
      \end{aligned}
    \end{aligned} \label{eq:defect-stopped-bound}
\end{equation}

  where
  \[ L_{(\theta, \zeta)} = \bigcap_{\theta'} \left\{ \zeta' \in B_1 : \text{}
     (\theta', \zeta') \times \left\{ \mf{b}_-^{\ast} (\theta', \zeta'),
     \mf{b}_+^{\ast} (\theta', \zeta') \right\} \cap R^{+ +} (\theta, \zeta) =
     \emptyset \right\} \]
  with the intersection taken over all \(\theta'\) with \((\theta', \zeta') \in R
  (\theta, \zeta)\). The bound follows because for every fixed \(\zeta'\) either
  \(\1_{L_{(\theta, \zeta)}} (\zeta') \neq 0\) and thus
  \[ \begin{aligned}[t]&
       \1_{R (\theta, \zeta)} (\theta', \zeta') \left| {\1_{E^{\ast}}} 
       F^{\ast} \left( \theta', \zeta', \mf{b}_{W^-}^{\ast} (\theta', \zeta')
       \right) [\phi] \right|^u \qquad\\ &
       \lesssim \left| {\1_{E^{\ast}}}  F^{\ast} (\theta', \zeta', \sigma')
       \left[ \phi_{\left( \theta', \sigma', \theta', \mf{b}_{W^-}^{\ast}
       (\theta', \zeta') \right)} \right] \right|^u
     \end{aligned} \]
  for every \(\theta'\) with \((\theta', \zeta') \in R (\theta, \zeta)\).
  Otherwise, if \(\1_{L_{(\theta, \zeta)}} (\zeta') = 0\) then there exist
  \(\bar{\theta}\) with \((\bar{\theta}, \zeta') \in R (\theta, \zeta)\) such that
  either \(\left( \bar{\theta}, \zeta', \mf{b}_-^{\ast} (\bar{\theta}, \zeta')
  \right) \in R^{+ +} (\theta, \zeta) \neq \emptyset\) or \(\left( \bar{\theta},
  \zeta', \mf{b}_+^{\ast} (\bar{\theta}, \zeta') \right) \in R^{+ +} (\theta,
  \zeta) \neq \emptyset\). In the former case, there exists an absolute
  \(\bar{r} > 0\) depending on \(r_{\theta}, r_{\zeta}, r_{\sigma} > 0\) such that
  \(\left( \theta'', \zeta', \mf{b}_-^{\ast} (\theta', \zeta') \right) \in R^+
  (\theta, \zeta)\) for all \(| \theta'' - \theta | < \bar{r}\). Similarly,
  \(\left( \theta'', \zeta', \mf{b}_+^{\ast} (\theta', \zeta') \right) \in R^+
  (\theta, \zeta)\) in the case \(\left( \bar{\theta}, \zeta', \mf{b}_+^{\ast}
  (\bar{\theta}, \zeta') \right) \in R^{+ +} (\theta, \zeta) \neq \emptyset\).
  The claim then follows by integrating over \(\zeta'\) and \(\theta''\).
  
  Our argument then proceeds as follows. We need to show that
  \begin{equation}
    \RHS{\eqref{eq:defect-stopped-bound}} \lesssim \mf{b}_{W^-}^{\ast}
    (\theta, \zeta) \RHS{\eqref{eq:geometry-of-defects}}^u \| \phi
    \|_{\Phi^N_{\gr}} \label{eq:defect-stopped-bound-bound}
  \end{equation}
  Then we will show the following fact:
  \begin{equation}
    \int_{\Theta^{\tmop{in}}_{\Gamma} \times B_1} \frac{1}{\mf{b}_{W^-}^{\ast}
    (\theta, \zeta)} \1_{R (\theta, \zeta)} (\theta', \zeta') \dd \theta \dd
    \zeta \gtrsim \1_{\Theta^{\tmop{in}} \times B_1} (\theta', \zeta') .
    \label{eq:defect-stopping-construction}
  \end{equation}
  From these two bounds we can conclude the proof: using
  \eqref{eq:defect-stopping-construction} we get that
  \[ \begin{aligned}[t]&
       \int_{\Theta^{\tmop{in}}_{\Gamma} \times B_1} \left| {\1_{E^{\ast}}} 
       F^{\ast} \left( \theta', \zeta', \mf{b}_{W^-}^{\ast} (\theta', \zeta')
       \right) [\phi] \right|^u \dd \theta' \dd \zeta'\\ &
       \lesssim \int_{\Theta^{\tmop{in}}_{\Gamma} \times B_1}
       \frac{1}{\mf{b}_{W^-}^{\ast} (\theta, \zeta)} \left( \int_{R (\theta,
       \zeta)} \left| {\1_{E^{\ast}}}  F^{\ast} \left( \theta', \zeta',
       \mf{b}_{W^-}^{\ast} (\theta', \zeta') \right) [\phi] \right|^u \dd
       \theta' \dd \zeta' \right) \dd \theta \dd \zeta ;
     \end{aligned} \]
  using \eqref{eq:defect-stopped-bound-bound} we obtain that
  \[ \begin{aligned}[t]&
       \int_{\Theta^{\tmop{in}}_{\Gamma} \times B_1}
       \frac{1}{\mf{b}_{W^-}^{\ast} (\theta, \zeta)} \left( \int_{R (\theta,
       \zeta)} \left| {\1_{E^{\ast}}}  F^{\ast} \left( \theta', \zeta',
       \mf{b}_{W^-}^{\ast} (\theta', \zeta') \right) [\phi] \right|^u \dd
       \theta' \dd \zeta' \right) \dd \theta \dd \zeta\\ &
       \lesssim \int_{\Theta^{\tmop{in}}_{\Gamma} \times B_1} \RHS{}
       \eqref{eq:geometry-of-defects}^u \| \phi \|_{\Phi^N_{\gr}} \dd \theta
       \dd \zeta,
     \end{aligned} \]
  and this concludes the claim. Let us show that
  \eqref{eq:defect-stopping-construction} holds. For any fixed \(\left(
  \theta', \zeta' \right) \in \Theta^{\tmop{in}}_{\Gamma} \times B_1\) and any
  \((\theta'', \zeta'') \in R (\theta', \zeta')\) it holds that
  \(\mf{b}_{W^-}^{\ast} (\theta'', \zeta'') > \mf{b}_{W^-}^{\ast} (\theta,
  \zeta) (1 - 10 r_t) .\) It follows that if \(| \theta'' - \theta' | <
  r_{\theta}\) and \(| \zeta'' - \zeta' | < r_{\zeta} (1 - 10 r_t)
  \mf{b}_{W^-}^{\ast} (\theta', \zeta')\) then \((\theta', \zeta') \in R
  (\theta'', \zeta'')\). It follows that
  \[ \int_{\Theta^{\tmop{in}}_{\Gamma} \times B_1} \1_{R (\theta'', \zeta'')}
     (\theta', \zeta') \dd \theta'' \dd \zeta'' \gtrsim \mf{b}_{W^-}^{\ast}
     (\theta', \zeta') \]
  with an absolute constant depending on \(r_{\theta}, r_{\zeta}, r_{\sigma}\)
  and \(r''\), as required.
  
  It remains to show that \eqref{eq:defect-stopped-bound-bound} holds. Note
  that \(\pi_T \left( R^+ (\theta, \zeta) \cap
  \mT_{\Theta^{\tmop{in}}_{\Gamma}} \right) \subset \pi_{T'} \left(
  \mT_{\Theta^{\tmop{in}}_{\Gamma}} \right)\) with \(T' = T_{\Theta} \left(
  \xi_T, x_T + \zeta, 2 \mf{b}_{W^-}^{\ast} (\theta, \zeta) \right)\). As a
  matter of fact, \((\theta', \zeta', \sigma') \in R^+ (\theta, \zeta) \cap
  \mT_{\Theta^{\tmop{in}}}\) only if \(\theta' \in \Theta^{\tmop{in}}_{\Gamma}\),
  \(| \zeta' - \zeta | < r_{\zeta} \mf{b}_{W^-}^{\ast} (\theta, \zeta) \), and
  \(\sigma' < (1 + 8 r_t) \mf{b}_{W^-}^{\ast} (\theta, \zeta)\). In particular,
  this means that if \((\eta', y', t') = \pi_T (\theta', \zeta', \sigma')\) then
  \[ \begin{aligned}[t]&
       t' < (1 + 8 r_t) \mf{b}_{W^-}^{\ast} (\theta, \zeta),\\ &
       | y' - (x_T + \zeta) | < r_{\zeta} \mf{b}_{W^-}^{\ast} (\theta,
       \zeta),\\ &
       t' (\eta' - \xi_T) \in \Theta^{\tmop{in}}_{\Gamma},
     \end{aligned} \]
  that shows the claimed inclusion. Applying
  \Cref{lem:wave-packet-decomposition} to \(\phi_{\left( \theta', \sigma',
  \theta', \mf{b}_{W^-}^{\ast} (\theta', \zeta') \right)}\) we get that
  \[ \phi_{\left( \theta', \sigma', \theta', \mf{b}_{W^-}^{\ast} (\theta',
     \zeta') \right)} = \sum_k a_k \left( \theta', \sigma', \theta',
     \mf{b}_{W^-}^{\ast} (\theta', \zeta') \right) \widetilde{\phi_k} \]
  with \(\| \widetilde{\phi_k} \|_{\Phi^{N - 2}_{4 \gr}} \lesssim 1\) and
  \(\left| a_k \left( \theta', \sigma', \theta', \mf{b}_{W^-}^{\ast} (\theta',
  \zeta') \right) \right| \lesssim \langle k \rangle^{- 2} \| \phi
  \|_{\Phi^N_{\gr}}\) so
  \[ \begin{aligned}[t]&
       \begin{aligned}[t]&
         \int_{R^+ (\theta, \zeta)} \left| {\1_{E^{\ast}}}  F^{\ast} (\theta',
         \zeta', \sigma') \left[ \phi_{\left( \theta', \sigma', \theta',
         \mf{b}_{W^-}^{\ast} (\theta', \zeta') \right)} \right] \right|^u\\ &
         \qquad \times \left( \1_{L_{(\theta, \zeta)}} (\zeta') + \sigma'
         \delta \left( \sigma' - \mf{b}_-^{\ast} (\theta', \zeta') \right) +
         \sigma' \delta \left( \sigma' - \mf{b}_+^{\ast} (\theta', \zeta')
         \right) \right) \dd \zeta' \dd \theta' \frac{\dd \sigma'}{\sigma'}
       \end{aligned}\\ &
       \lesssim \begin{aligned}[t]&\| \phi \|_{\Phi^N_{\gr}} \sup_{\tilde{\phi}} 
         \int_{R^+ (\theta, \zeta)} \left| {\1_{E^{\ast}}}  F^{\ast} (\theta',
         \zeta', \sigma') [\tilde{\phi}] \right|^u\\ &
         \qquad \times \left( \1_{L_{(\theta, \zeta)}} (\zeta') + \sigma'
         \delta \left( \sigma' - \mf{b}_-^{\ast} (\theta', \zeta') \right) +
         \sigma' \delta \left( \sigma' - \mf{b}_+^{\ast} (\theta', \zeta')
         \right) \right) \dd \zeta' \dd \theta' \frac{\dd \sigma'}{\sigma'} .
       \end{aligned}
     \end{aligned} \]
  where the upper bound is taken over \(\tilde{\phi} \in \Phi^{\infty}_{4 \gr}\)
  with \(\| \tilde{\phi} \|_{\Phi^{N - 2}_{4 \gr}} \lesssim 1\). By
  \eqref{eq:boundary-singular-size-local} we obtain that
  \[ \begin{aligned}[t]&
       \begin{aligned}[t]&
         \| \phi \|_{\Phi^N_{\gr}} \sup_{\tilde{\phi}} \int_{R^+ (\theta,
         \zeta)} \left| {\1_{E^{\ast}}}  F^{\ast} (\theta', \zeta', \sigma')
         [\tilde{\phi}] \right|^u\\ &
         \qquad \times \left( \sigma' \delta \left( \sigma' - \mf{b}_-^{\ast}
         (\theta', \zeta') \right) + \sigma' \delta \left( \sigma' -
         \mf{b}_+^{\ast} (\theta', \zeta') \right) \right) \dd \zeta' \dd
         \theta' \frac{\dd \sigma'}{\sigma'}
       \end{aligned}\\ &
       \lesssim \begin{aligned}[t]&
         \| \phi \|_{\Phi^N_{\gr}} \mu^1_{\Theta} (T') \left( \left\|
         \1_{\overline{(W^+ \cap V^+)} \setminus (W^- \cup V^-)} F_{\Gamma}
         (\eta, y, t) t \delta \left( t - \mf{\mf{b}_-} (\eta, y) \right)
         \right\|_{\SL^{(u, 1)}_{(\Theta, \Theta^{\tmop{in}})} (T')} \qquad
         \right.\\ &
         + \left. \left\| \1_{\overline{(W^+ \cap V^+)} \setminus (W^- \cup
         V^-)} F_{\Gamma} (\eta, y, t) t \delta \left( t - \mf{\mf{b}_+}
         (\eta, y) \right) \right\|_{\SL^{(u, 1)}_{(\Theta,
         \Theta^{\tmop{in}})} (T')} \right)
       \end{aligned}\\ &
       \lesssim \| \phi \|_{\Phi^N_{\gr}} \RHS{}
       \eqref{eq:geometry-of-defects}^u \mu^1_{\Theta} (T') \lesssim \| \phi
       \|_{\Phi^N_{\gr}} \mf{b}_{W^-}^{\ast} (\theta, \zeta) \RHS{}
       \eqref{eq:geometry-of-defects}^u,
     \end{aligned} \]
  as required. It remains to show the bound \
  \[ \left( \frac{1}{\mf{b}_{W^-}^{\ast} (\theta, \zeta)} \int_{R^+ (\theta,
     \zeta)} \1_{L_{(\theta, \zeta)}} (\zeta') \left| {\1_{E^{\ast}}} 
     F^{\ast} (\theta', \zeta', \sigma') [\tilde{\phi}] \right|^u \dd \zeta'
     \dd \theta' \frac{\dd \sigma'}{\sigma'} \right)^{\frac{1}{u}} \lesssim
     \RHS{} \eqref{eq:geometry-of-defects}  \]
  for any fixed \(\tilde{\phi} \in \Phi^{\infty}_{4 \gr}\) with \(\| \tilde{\phi}
  \|_{\Phi^{N - 2}_{4 \gr}} \leq 1\). It holds that \(\pi_T \left( R^{+ +}
  (\theta, \zeta) \cap \mT_{\Theta^{\tmop{in}}_{\Gamma}} \right) \subset
  \pi_{T''} \left( \mT_{\Theta^{\tmop{in}}_{\Gamma}} \right)\) with \(T'' =
  T_{\Theta} \left( \xi_T - \frac{1}{2 \mf{b}_{W^-}^{\ast} (\theta, \zeta)},
  x_T + \zeta, 2 \mf{b}_{W^-}^{\ast} (\theta, \zeta) \right)\). As a matter of
  fact, \((\theta', \zeta', \sigma') \in R^{+ +} (\theta, \zeta) \cap
  \mT_{\Theta^{\tmop{in}}_{\Gamma}}\) only if \(\theta' \in
  \Theta^{\tmop{in}}_{\Gamma}\), \(| \zeta' - \zeta | < r_{\zeta}
  \mf{b}_{W^-}^{\ast} (\theta, \zeta) \), and \((1 + 4 r_t) \mf{b}_{W^-}^{\ast}
  (\theta, \zeta) < \sigma' < (1 + 6 r_t) \mf{b}_{W^-}^{\ast} (\theta,
  \zeta)\). In particular, this means that if \((\eta', y', t') = \pi_T
  (\theta', \zeta', \sigma')\) then
  \[ \begin{aligned}[t]&
       (1 + 4 r_t) \mf{b}_{W^-}^{\ast} (\theta, \zeta) < t' < (1 + 6 r_t)
       \mf{b}_{W^-}^{\ast} (\theta, \zeta),\\ &
       | y' - (x_T + \zeta) | < r_{\zeta} \mf{b}_{W^-}^{\ast} (\theta,
       \zeta),\\ &
       t' (\eta' - \xi_T) \in \Theta^{\tmop{in}}_{\Gamma} .
     \end{aligned} \]
  But then
  \[ t' \left( \eta' - \left( \xi_T - \frac{1}{2 \mf{b}_{W^-}^{\ast} (\theta,
     \zeta)} \right) \right) \in \Theta^{\tmop{in}}_{\Gamma} + \left( \frac{1
     + 4 r_t}{2}, \frac{1 + 6 r_t}{2} \right) \subset
     \Theta^{\tmop{ex}}_{\Gamma}, \]
  as required. If \(u \leq 2\) this is sufficient to conclude: it holds that
  \[ \begin{aligned}[t]&
       \left( \frac{1}{\mf{b}_{W^-}^{\ast} (\theta, \zeta)} \int_{R^+ (\theta,
       \zeta)} \1_{L_{(\theta, \zeta)}} (\zeta') \left| {\1_{E^{\ast}}} 
       F^{\ast} (\theta', \zeta', \sigma') [\tilde{\phi}] \right|^u \dd \zeta'
       \dd \theta' \frac{\dd \sigma'}{\sigma'} \right)^{\frac{1}{u}}\\ &
       \qquad \lesssim \left\| \1_{\overline{(W^+ \cap V^+)} \setminus (W^-
       \cup V^-)} \1_{\pi_T (R^+ (\theta, \zeta))} F_{\Gamma}
       \right\|_{{\Gamma^{\ast}}  \SL^{(u, u)}_{(\Theta, \Theta^{\tmop{ex}})}
       \Phi_{4 \gr}^{N - 2} (T'')} \\ &
       \qquad \lesssim
       \begin{aligned}[t]
       &\left\| \1_{\overline{(W^+ \cap V^+)} \setminus (W^-
       \cup V^-)} \1_{\pi_T (R^+ (\theta, \zeta))} F_{\Gamma}
       \right\|_{{\Gamma^{\ast}}  \SL^{(u, 2)}_{(\Theta, \Theta^{\tmop{ex}})}
       \Phi_{4 \gr}^{N - 2} (T'')}\\ &
       \qquad \times \left\| \1_{\overline{(W^+ \cap V^+)} \setminus (W^- \cup V^-)}
       \1_{\pi_T (R^+ (\theta, \zeta))} \right\|_{{\Gamma^{\ast}}  \SL^{\left(
       \infty, \frac{2 u}{2 - u} \right)}_{(\Theta, \Theta^{\tmop{ex}})}
       (T'')}
       \end{aligned}
     \end{aligned} \]
  and
  \[ \begin{aligned}[t]&
       \left\| \1_{\overline{(W^+ \cap V^+)} \setminus (W^- \cup V^-)}
       \1_{\pi_T (R^+ (\theta, \zeta))} \right\|_{{\Gamma^{\ast}}  \SL^{\left(
       \infty, \frac{2 u}{2 - u} \right)}_{(\Theta, \Theta^{\tmop{ex}})}
       (T'')}^{\frac{2 u}{2 - u}}\\ &
       \qquad \lesssim \sup_{\theta \in \Theta} \sup_{\zeta \in B_1} \int_0^1
       \left| \1_{R^+ (\theta, \zeta)} (\pi_T^{- 1} \pi_{T''} (\theta',
       \zeta', \sigma'))  \right|^{\frac{2 u}{2 - u}} \frac{\dd
       \sigma'}{\sigma'}\\ &
       \qquad\lesssim \sup_{\theta \in \Theta} \sup_{\zeta \in B_1} \int_0^1 \1_{(1
       + 2 r_{\sigma}) / 2 < \sigma' < (1 + 8 r_{\sigma}) / 2} \frac{\dd
       \sigma'}{\sigma'}
     \end{aligned} \]
  since \((\theta'', \zeta'', \sigma'') = \pi_T^{- 1} \pi_{T''} (\theta',
  \zeta', \sigma')\) only if \(\sigma'' = 2 \mf{b}_{W^-}^{\ast} (\theta, \zeta)
  \sigma'\). If \(u > 2\) we use that the integrand \({\1_{E^{\ast}}}  F^{\ast}
  (\theta', \zeta', \sigma') [\tilde{\phi}]\) is almost constant over \(R^+
  (\theta, \zeta)\). As a matter of fact it holds that
  \[ \1_{L_{(\theta, \zeta)}} (\zeta') \left| {\1_{E^{\ast}}}  F^{\ast}
     (\theta', \zeta', \sigma') [\phi] \right| = \1_{L_{(\theta, \zeta)}}
     (\zeta') \left| {\1_{E^{\ast}}}  F^{\ast} (\theta'', \zeta', \sigma'')
     \left[ {\phi }_{(\theta'', \sigma'', \theta', \sigma')} \right] \right|
  \]
  as per \eqref{eq:wave-packet-move} for any \((\theta'', \zeta', \sigma'')
  \in R^{+ +} (\theta, \zeta)\). Applying \Cref{lem:wave-packet-decomposition}
  to \({\phi }_{(\theta'', \sigma'', \theta', \sigma')}\)
  \[ \begin{aligned}[t]&
       \begin{aligned}[t]&
         \1_{L_{(\theta, \zeta)}} (\zeta') \left| {\1_{E^{\ast}}}  F^{\ast}
         (\theta', \zeta', \sigma') [\tilde{\phi}] \right|\\ &
         \qquad = \sup_{\widetilde{\widetilde{\phi }}} \left( \frac{\int_{R^{+
         +} (\theta, \zeta)} \1_{L_{(\theta, \zeta)}} (\zeta') \left|
         {\1_{E^{\ast}}}  F^{\ast} (\theta'', \zeta', \sigma'') \left[
         \widetilde{\widetilde{\phi }} \right] \right|^2 \frac{\dd
         \sigma''}{\sigma''} \dd \theta'}{\int_{R^{+ +} (\theta, \zeta)}
         \1_{L_{(\theta, \zeta)}} (\zeta') \frac{\dd \sigma''}{\sigma''} \dd
         \theta'} \right)^{\frac{1}{2}}
       \end{aligned} 
     \end{aligned} \]
  where the upper bound is taken over \(\widetilde{\widetilde{\phi }} \in
  \Phi^{\infty}_{16 \gr}\) with \(\| \tilde{\phi} \|_{\Phi^{N - 4}_{16 \gr}}
  \lesssim 1\). Since
  \[ \int_{R^{+ +} (\theta, \zeta)} \1_{L_{(\theta, \zeta)}} (\zeta')
     \frac{\dd \sigma''}{\sigma''} \dd \theta' \gtrsim 1 \]
  we have that
  \[ \begin{aligned}[t]&
       \frac{1}{\mf{b}_{W^-}^{\ast} (\theta, \zeta)} \int_{R^+ (\theta,
       \zeta)} \1_{L_{(\theta, \zeta)}} (\zeta') \left| {\1_{E^{\ast}}} 
       F^{\ast} (\theta', \zeta', \sigma') [\tilde{\phi}] \right|^u \dd \zeta'
       \dd \theta' \frac{\dd \sigma'}{\sigma'}\\ &
       \lesssim \sup_{\widetilde{\widetilde{\phi }}} \int_{B_1} \left(
       \int_{\Theta^{\tmop{in}} \times [0, 1]} \1_{L_{(\theta, \zeta)}}
       (\zeta') \left| \1_{R^{+ +} (\theta, \zeta)} {\1_{E^{\ast}}}  F^{\ast}
       (\theta'', \zeta', \sigma'') \left[ \widetilde{\widetilde{\phi }}
       \right] \right|^2 \frac{\dd \sigma''}{\sigma''} \dd \theta' \right)^{u
       / 2} \dd \zeta' .
     \end{aligned} \]
  Now, the same procedure as before allows us, using that \(u \geq 2\) to deduce
  that
  \[ \begin{aligned}[t]&
       \left( \frac{1}{\mf{b}_{W^-}^{\ast} (\theta, \zeta)} \int_{R^+ (\theta,
       \zeta)} \1_{L_{(\theta, \zeta)}} (\zeta') \left| {\1_{E^{\ast}}} 
       F^{\ast} (\theta', \zeta', \sigma') [\tilde{\phi}] \right|^u \dd \zeta'
       \dd \theta' \frac{\dd \sigma'}{\sigma'} \right)^{\frac{1}{u}}\\ &
       \qquad \lesssim
       \begin{aligned}[t] &
       \left\| \1_{\overline{(W^+ \cap V^+)} \setminus (W^-
       \cup V^-)} \1_{\pi_T (R^+ (\theta, \zeta))} F_{\Gamma}
       \right\|_{{\Gamma^{\ast}}  \SL^{(u, 2)}_{(\Theta, \Theta^{\tmop{ex}})}
       \Phi_{4 \gr}^{N - 2} (T'')}\\ &
       \times \left\| \1_{\overline{(W^+ \cap V^+)} \setminus (W^- \cup V^-)}
       \1_{\pi_T (R^+ (\theta, \zeta))} \right\|_{{\Gamma^{\ast}} 
       \SL^{(\infty, \infty)}_{(\Theta, \Theta^{\tmop{ex}})} (T'')}
       \end{aligned}
       \\ &
       \qquad \lesssim \left\| \1_{\overline{(W^+ \cap V^+)} \setminus (W^-
       \cup V^-)} \1_{\pi_T (R^+ (\theta, \zeta))} F_{\Gamma}
       \right\|_{{\Gamma^{\ast}}  \SL^{(u, 2)}_{(\Theta, \Theta^{\tmop{ex}})}
       \Phi_{4 \gr}^{N - 2} (T'')} .
     \end{aligned} \]
\end{proof}

\begin{lemma}[Lacunary size domination]
  \label{lem:unif-lac-size-domination}Given any sets \(V^+, V^- \in \DD_{\beta
  }^{\cup}\), and let \(W^+, W^- \in \TT_{\Theta}^{\cup}\), the bounds
  \begin{equation}
    \begin{aligned}[t]&
      \bigl\| \1_{(V^+ \cap W^+) \setminus (V^- \cup W^-)} \left( \Emb [f]
      \circ \Gamma \right) \bigr\|_{{\Gamma^{\ast}}  \SL^{(u, 2)}_{\Theta}
      \wpD_{\gr}^N} \qquad\\ &
      \qquad \lesssim \begin{aligned}[t]&
        \bigl\| \1_{(V^+ \cap W^+) \setminus (V^- \cup W^-)}  \left( \Emb [f]
        \circ \Gamma \right) \bigr\|_{{\Gamma^{\ast}}  \SL^{(u, 2)}_{(\Theta,
        \Theta^{\tmop{ex}})} \Phi_{16 \gr}^{N - 10}}\\ &
        + \bigl\| \1_{(V^+ \cap W^+) \setminus (V^- \cup W^-)}  \left( \Emb
        [f] \circ \Gamma \right) \bigr\|_{{\Gamma^{\ast}}  \SL^{(u,
        1)}_{(\Theta, \Theta^{\tmop{in}})} \dfct^{N - 10}_{16 \gr}}
      \end{aligned}
    \end{aligned} \label{eq:unif-lac-size-domination}
  \end{equation}
  hold for all \(f \in \Sch (\R)\) as long as \(u \in [1, \infty]\).
  The implicit constant is independent of \(f\) and of the sets \(V^{\pm},
  W^{\pm}\).
\end{lemma}

\begin{proof}
  Let us fix \(T \in \TT_{\Theta}\) and show that
  \[ \begin{aligned}[t]&
       \bigl\| \1_{(V^+ \cap W^+) \setminus (V^- \cup W^-)} \left( \Emb [f]
       \circ \Gamma \right) \bigr\|_{{\Gamma^{\ast}}  \SL^{(u, 2)}_{\Theta}
       \wpD_{\gr}^N (T)} \lesssim \RHS{\eqref{eq:unif-lac-size-domination}} .
     \end{aligned} \]
  By symmetry, (see \eqref{eq:size-symmetry} and the subsequent discussion),
  we can assume that \({T = T_{\Theta}}  (0, 0, 1)\) so that
  \[ \Gamma \circ \pi_T (\theta, \zeta, \sigma) = (\alpha \beta (\theta +
     \gamma) (\beta \sigma)^{- 1}, \zeta, \beta \sigma) = (\theta_{\Gamma}
     (\beta \sigma)^{- 1}, \zeta, \beta \sigma) = \pi_T (\theta_{\Gamma},
     \zeta, \beta \sigma) . \]
  For the sake of this proof, let us suppose that \(\alpha \beta > 0\) and
  \(\gamma < 0\). Any other combination of signs can be dealt with, mutatis
  mutandis, following the same argument. We use the shorthand \(E_+^{\ast}
  \assign \pi_T^{- 1} (V^+ \cap W^+ \cap T)\), \(E_-^{\ast} \assign \pi_T^{- 1}
  (V^- \cup W^-)\), and \(F^{\ast} \eqd \1_{E_+^{\ast} \setminus E_-^{\ast}}
  \Emb{} [f] \circ \Gamma \circ \pi_T\). From \Cref{def:lacunary-size} we have
  that
  \[ \begin{aligned}[t]&
       \bigl\| \1_{(V^+ \cap W^+) \setminus (V^- \cup W^-)} \left( \Emb [f]
       \circ \Gamma \right) \bigr\|_{{\Gamma^{\ast}}  \SL^{(u, 2)}_{\Theta}
       \wpD_{\gr}^N (T)} = \sup_{\phi} \| F^{\ast} (\theta, \zeta, \sigma)
       [\phi_{\theta}] \|_{L^u_{\frac{\dd \theta \dd \zeta}{| \Theta |}
       } L^2_{\frac{\dd \sigma}{\sigma}}}\\ &
       \phi_{\theta} (z) \eqd (- d_z + 2 \pi i \theta_{\Gamma}) \phi (z)
     \end{aligned} \]
  with the upper bound taken over \(\phi \in \Phi^{\infty}_{\gr}\) with \(\| \phi
  \|_{\Phi^N_{\gr}} \leq 1\). First we address the case \(\theta \in
  \Theta^{\tmop{ex}}_{\Gamma}\) i.e. \(\theta_{\Gamma} \notin
  \Theta^{\tmop{in}}\). Since
  \[ \begin{aligned}[t]&
       \wpD_{\zeta} (\theta_{\Gamma}) F^{\ast} (\theta, \zeta, \sigma) [\phi]
       = F^{\ast} (\theta, \zeta, \sigma) [\phi_{\theta}]
     \end{aligned} \]
  and \(\| - d_z \phi \|_{\Phi_{\gr}^N} + \| 2 \pi i \theta_{\Gamma} \phi
  \|_{\Phi_{\gr}^N} \lesssim \| \phi \|_{\Phi_{\gr}^N}\) it follows immediately
  that
  \[ \begin{aligned}[t]&
       \sup_{\phi}  \int_{\R^2} \left( \int_{\R_+}
       \1_{\mT_{\Theta_{\Gamma}^{\tmop{ex}}}} \1_{E_+^{\ast} \setminus
       E_-^{\ast}} \left| \wpD_{\zeta} (\theta_{\Gamma}) F^{\ast} (\theta,
       \zeta, \sigma) [\phi] \right|^2 \hspace{0.17em} \frac{\dd
       \sigma}{\sigma} \right)^{\frac{u}{2}} \hspace{0.17em} \frac{\dd \zeta
       \dd \theta}{| \Theta |}\\ &
       \qquad \begin{aligned}[t]&
         \lesssim \sup_{\phi}  \int_{\R^2} \left( \int_{\R_+}
         \1_{\mT_{\Theta_{\Gamma}^{\tmop{ex}}}} \1_{E_+^{\ast} \setminus
         E_-^{\ast}} | F^{\ast} (\theta, \zeta, \sigma) [\phi] |^2
         \hspace{0.17em} \frac{\dd \sigma}{\sigma} \right)^{\frac{u}{2}}
         \hspace{0.17em} \frac{\dd \zeta \dd \theta}{| \Theta |}\\ &
         \lesssim \| \1_{(V^+ \cap W^+) \setminus (V^- \cup W^-)} \left( \Emb
         [f] \circ \Gamma \right) \|^u_{{\Gamma^{\ast}}  \SL^{(u,
         2)}_{(\Theta, \Theta^{\tmop{ex}})} \Phi_{\gr}^N} .
       \end{aligned}
     \end{aligned} \]
  The rest of the proof is dedicated to dealing with the case \(\theta \in
  \Theta^{\tmop{in}}_{\Gamma}\). Note that \(\FT{\phi_{\theta}}\) vanishes at \(-
  \theta_{\Gamma}\) while for any \(\theta \notin \Theta^{\tmop{in}}_{\Gamma}\),
  \(\FT{\phi}\) vanishes on a large neighborhood of \(- \theta_{\Gamma}\) simply
  because \(| \theta_{\Gamma} | > \gr\) and \(\spt \FT{\phi} \subset B_{\gr}\).
  This brings us to compare \(F^{\ast} (\theta, \zeta, \sigma) [\phi_{\theta}]\)
  or \(\theta \in \Theta^{\tmop{in}}_{\Gamma}\) with \(F^{\ast} (\theta', \zeta,
  \sigma) [\phi'_{\theta'}]\) for \(\theta' \in \Theta^{\tmop{ex}}_{\Gamma}\) and
  for an appropriate \(\phi'_{\theta} \in \Phi^{\infty}_{16 \gr}\), and to
  expect to be able to bound \(\LHS{\eqref{eq:unif-lac-size-domination}}\) by
  the first addend on {\RHS{\eqref{eq:unif-lac-size-domination}}}. In
  practice, we also need to control some remainder terms by the boundary
  integral
  \[ \bigl\| \1_{(V^+ \cap W^+) \setminus (V^- \cup W^-)}  \left( \Emb [f]
     \circ \Gamma \right) \bigr\|_{{\Gamma^{\ast}}  \SL^{(u, 1)}_{(\Theta,
     \Theta^{\tmop{in}})} \dfct^{N - 10}_{16 \gr}} . \]
  This proof goes along the lines of {\cite[Lemma
  5.3]{amentaBilinearHilbertTransform2020}}, however care is taken to
  \begin{itemize}
    \item make sure the estimates are uniform in \(\Gamma\) and in particular in
    the deformation parameter \(\beta \in (0, 1]\),
    
    \item use the defect size \({\Gamma^{\ast}}  \SL^{(u, 1)}_{(\Theta,
    \Theta^{\tmop{in}})} \dfct^{N - 10}_{16 \gr}\) instead of the size
    \(\Gamma^{\ast} \SL^{(\infty, \infty)}_{\Theta} \Phi_{16 \gr}^{N - 10}\)
    that is not available in this setting. 
  \end{itemize}
  Let \(\mf{b}_+^{\ast}\), \(\mf{b}_-^{\ast}\) be the two functions whose graphs
  are the boundaries of \(E_+^{\ast}\) and \(E_-^{\ast}\) respectively so that
  \(E_+^{\ast} \setminus E_-^{\ast} = \left\{ (\theta, \zeta, \sigma) :
  \mf{b}_-^{\ast} (\theta, \zeta) < \sigma < \mf{b}_+^{\ast} (\theta, \zeta)
  \right\}\). The existence of such functions is guaranteed by
  \Cref{lem:geometry-of-boundary}. Our goal is to show that
  \begin{equation}
    \sup_{\phi}  \int_{\R^2} \left( \int_{\R_+}
    \1_{\mT_{\Theta_{\Gamma}^{\tmop{in}}}} \1_{E_+^{\ast} \setminus
    E_-^{\ast}} \left| \wpD_{\zeta} (\theta_{\Gamma}) F^{\ast} (\theta, \zeta,
    \sigma) [\phi] \right|^2 \hspace{0.17em} \frac{\dd \sigma}{\sigma}
    \right)^{\frac{u}{2}} \hspace{0.17em} \frac{\dd \zeta \dd \theta}{| \Theta
    |} \lesssim \RHS{\eqref{eq:unif-lac-size-domination}}
    \label{eq:ld:local-coordinates:in}
  \end{equation}
  with \(\phi\) as before.
  
  Let us denote \(\Theta^{\tmop{in}}_{\Gamma} = (\theta^{\tmop{in}}_{\Gamma -},
  \theta^{\tmop{in}}_{\Gamma +})\), \(\Theta^{\tmop{in}} =
  (\theta^{\tmop{in}}_-, \theta^{\tmop{in}}_+)\) with \(\theta^{\tmop{in}}_{\pm}
  = \alpha \beta \theta^{\tmop{in}}_{\Gamma \pm} + \gamma\), let
  \(\bar{\Theta}_{\Gamma} \eqd (\bar{\theta}_{\Gamma -}, \bar{\theta}_{\Gamma
  +}) = (\theta^{\tmop{in}}_{\Gamma -} - (\alpha \beta)^{- 1},
  \theta^{\tmop{in}}_{\Gamma +} + (\alpha \beta)^{- 1})\) be an enlargement of
  \(\Theta^{\tmop{in}}_{\Gamma}\) that still satisfies \(\bar{\Theta}_{\Gamma}
  \subset \Theta\), and let \(\bar{\Theta}_{\Gamma}^{\tmop{ex}} =
  \bar{\Theta}_{\Gamma}  \setminus \Theta^{\tmop{in}}_{\Gamma}\).
  
  Our main claim is that for any \((\theta, \zeta, \sigma) \in
  \mT_{\Theta^{\tmop{in}}_{\Gamma}} \cap (E_+^{\ast} \setminus E_-^{\ast})\)
  and for any \(\phi \in \Phi_{\gr}^{\infty}\) with \(\| \phi \|_{\Phi_{\gr}^N}
  \leq 1\) we have that
  \begin{equation}
    \begin{aligned}[t]&
      \left| \wpD_{\zeta} (\theta_{\Gamma}) F^{\ast} (\theta, \zeta, \sigma)
      [\phi] \right|\\ &
      \qquad \lesssim \begin{aligned}[t]&
        \sup_{\tilde{\phi}} \left( \int_{\bar{\Theta}_{\Gamma}^{\tmop{ex}}}
        \int_{\mf{b}_-^{\ast} (\theta', \zeta)}^{\mf{b}_+^{\ast} (\theta',
        \zeta)} e^{- \left| \log \left( \frac{\sigma}{\sigma'} \right)
        \right|} | F^{\ast} (\theta', \zeta, \sigma') [\tilde{\phi}] |
        \frac{\dd \sigma'}{\sigma'} \dd \theta' \right.\\ &
        \qquad + \begin{aligned}[t]&
          \int_{\Theta^{\tmop{in}}_{\Gamma}} {\1_{\mf{b}_-^{\ast} (\theta',
          \zeta) < \mf{b}_+^{\ast} (\theta', \zeta)}}  e^{- \left| \log \left(
          \frac{\sigma}{\mf{b}_-^{\ast} (\theta', \zeta)} \right) \right|}
          \left| F^{\ast} (\theta', \zeta, \mf{b}_-^{\ast} (\theta', \zeta))
          [\tilde{\phi}] \right|\\ &
          \qquad \times (1 + \delta (\theta - \theta')) \dd \theta'
        \end{aligned}\\ &
        \qquad \left\nobracket + \int_{\Theta^{\tmop{in}}_{\Gamma}}
        {\1_{\mf{b}_-^{\ast} (\theta', \zeta) < \mf{b}_+^{\ast} (\theta',
        \zeta)}}  e^{- \left| \log \left( \frac{\sigma}{\mf{b}_+^{\ast}
        (\theta', \zeta)} \right) \right|} \left| F^{\ast} (\theta', \zeta,
        \mf{b}_+^{\ast} (\theta', \zeta)) [\tilde{\phi}] \right| \dd \theta'
        \right)
      \end{aligned}
    \end{aligned} \label{eq:ld:full-decomposition-bound}
\end{equation}

  with the upper bound taken over functions \(\tilde{\phi} \in \Phi_{8 \gr}^{N
  - 8}\) with \(\| \tilde{\phi} \|_{\Phi_{8 \gr}^{N - 8}} \leq 1\). Furthermore,
  if \(\sigma < C_{\tmop{sep}} \mf{b}_-^{\ast} (\theta, \zeta)\) for some fixed
  large \(C_{\tmop{sep}} \gg 1\) the above statement also holds for any \(\phi
  \in \Phi_{4 \gr}^{\infty}\) with \(\| \phi \|_{\Phi_{4 \gr}^N} \leq 1\) on
  {\LHS{\eqref{eq:ld:full-decomposition-bound}}} while the upper bound is
  still taken over \(\tilde{\phi} \in \Phi_{8 \gr}^{N - 8}\) on
  {\RHS{\eqref{eq:ld:full-decomposition-bound}}}.
  
  We now show how this would allow us to conclude that
  \(\eqref{eq:ld:local-coordinates:in} \lesssim
  \RHS{\eqref{eq:unif-lac-size-domination}}^u\), as required. Applying
  \Cref{lem:wave-packet-decomposition} we represent any \(\tilde{\phi} \in
  \Phi_{8 \gr}^{N - 8}\) with \(\| \tilde{\phi} \|_{\Phi_{8 \gr}^{N - 8}} \leq
  1\) as
  \[ \tilde{\phi} = \sum_{k \in \Z} a_k \widetilde{ \tilde{\phi}}_k \]
  with \(\widetilde{ \tilde{\phi}}_k \in \Phi_{16 \gr}^{N - 10}\), \(\|
  \widetilde{ \tilde{\phi}}_k \|_{\Phi_{16 \gr}^{N - 10}} \leq 1\), and \(| a_k
  | \lesssim \langle k \rangle^{- 2}\). Using
  \eqref{eq:ld:full-decomposition-bound} and the triangle inequality we get
  that \(\eqref{eq:ld:local-coordinates:in} \lesssim \mathrm{I} + \mathrm{I}^- +
  \mathrm{I}^+\) where
  \[ \begin{aligned}[t]&
       \mathrm{I} \eqd \sup_{\tilde{\phi}} \begin{aligned}[t]&
         \int_{\Theta_{\Gamma}^{\tmop{in}} \times B_1} \left( \int_{[0, 1]}
         {\1 }_{\mT_{\Theta_{\Gamma}^{\tmop{in}}} \cap (E_+^{\ast} \setminus
         E_-^{\ast}) } (\theta, \zeta, \sigma)  \right.\\ &
         \qquad \times \left. \left( \int_{\bar{\Theta}_{\Gamma}^{\tmop{ex}}}
         \int_{\mf{b}_-^{\ast} (\theta', \zeta)}^{\mf{b}_+^{\ast} (\theta',
         \zeta)} e^{- \left| \log \left( \frac{\sigma}{\sigma'} \right)
         \right|} | F^{\ast} (\theta', \zeta, \sigma') [\tilde{\phi}] |
         \frac{\dd \sigma'}{\sigma'} \dd \theta' \right)^2 \frac{\dd
         \sigma}{\sigma} \right)^{\frac{u}{2}} \frac{\dd \zeta \dd \theta}{|
         \Theta |},
       \end{aligned}\\ &
       \mathrm{I}^- \eqd \sup_{\tilde{\phi}} \begin{aligned}[t]&
         \int_{\Theta_{\Gamma}^{\tmop{in}} \times B_1} \left( \int_{[0, 1]}
         {\1 }_{\mT_{\Theta_{\Gamma}^{\tmop{in}}} \cap (E_+^{\ast} \setminus
         E_-^{\ast}) } (\theta, \zeta, \sigma)  \left(
         \int_{\Theta^{\tmop{in}}_{\Gamma}} e^{- \left| \log \left(
         \frac{\sigma}{\mf{b}_-^{\ast} (\theta', \zeta)} \right) \right|}
         \right. \right\nobracket \qquad\\ &
         \qquad \times \left\nobracket \left\nobracket \left| F^{\ast}
         (\theta', \zeta, \mf{b}_-^{\ast} (\theta', \zeta)) [\tilde{\phi}]
         \right| (1 + \delta (\theta - \theta')) \dd \theta' \right)^2
         \frac{\dd \sigma}{\sigma} \right)^{\frac{u}{2}} \frac{\dd \zeta \dd
         \theta}{| \Theta |},
       \end{aligned} \\ &
       \mathrm{I}^+ \eqd \sup_{\tilde{\phi}} \begin{aligned}[t]&
         \int_{\Theta_{\Gamma}^{\tmop{in}} \times B_1} \left( \int_{[0, 1]}
         {\1 }_{\mT_{\Theta_{\Gamma}^{\tmop{in}}} \cap (E_+^{\ast} \setminus
         E_-^{\ast}) } (\theta, \zeta, \sigma)  \right.\\ &
         \qquad \times \left. \left( \int_{\Theta^{\tmop{in}}_{\Gamma}} e^{-
         \left| \log \left( \frac{\sigma}{\mf{b}_+^{\ast} (\theta', \zeta)}
         \right) \right|} \left| F^{\ast} (\theta', \zeta, \mf{b}_+^{\ast}
         (\theta', \zeta)) [\tilde{\phi}] \right| \dd \theta' \right)^2
         \frac{\dd \sigma}{\sigma} \right)^{\frac{u}{2}} \frac{\dd \zeta \dd
         \theta}{| \Theta |} .
       \end{aligned}
     \end{aligned} \]
  where the upper bound, this time, is taken over functions \(\tilde{\phi}
  \in \Phi_{16 \gr}^{N - 10}\) with \\ \(\| \tilde{\phi} \|_{\Phi_{16 \gr}^{N -
  10}} \leq 1\). We bound the term \(\mathrm{I}\) using the Minkowski inequality
  and using the Young convolution inequality with respect to the
  multiplicative Haar measure \(\frac{\dd \sigma'}{\sigma'}\):
  \[ \mathrm{I} \textrm{} \begin{aligned}[t]&
       \lesssim \sup_{\tilde{\phi}} \int_{\bar{\Theta}_{\Gamma}^{\tmop{ex}}
       \times B_1} \left( \int_{\mf{b}_-^{\ast} (\theta',
       \zeta)}^{\mf{b}_+^{\ast} (\theta', \zeta)} | F^{\ast} (\theta', \zeta,
       \sigma') [\tilde{\phi}] | ^2 \frac{\dd \sigma'}{\sigma'}
       \right)^{\frac{u}{2}} \hspace{0.17em} \dd \zeta \dd \theta'\\ &
       \lesssim \Bigl\| \1_{V^+ \setminus V^-} \1_{\R^3_+ \setminus W^-}
       \left( \Emb [f] \circ \Gamma \right) \Bigr\|_{{\Gamma^{\ast}} 
       \SL^{(u, 2)}_{(\Theta, \Theta^{\tmop{ex}})} \Phi_{8 \gr}^{N - 8} (T)}^u
       .
     \end{aligned} \]
  Bounding the terms \(\mathrm{I}^-\) and \(\mathrm{I}^+\) makes use of the defect
  size.
  \[ \begin{aligned}[t]&
       \mathrm{I}^+ \lesssim \sup_{\tilde{\phi}} \begin{aligned}[t]&
         \int_{\Theta_{\Gamma}^{\tmop{in}} \times B_1} \left. \left(
         \int_{\Theta^{\tmop{in}}_{\Gamma}} {\1_{\mf{b}_-^{\ast} (\theta',
         \zeta) < \mf{b}_+^{\ast} (\theta', \zeta)}}  \left| F^{\ast}
         (\theta', \zeta, \mf{b}_+^{\ast} (\theta', \zeta)) [\tilde{\phi}]
         \right| \dd \theta' \right. \right)^{\frac{u}{2}} \frac{\dd \zeta \dd
         \theta}{| \Theta |}
       \end{aligned} \qquad\\ &
       \lesssim \sup_{\tilde{\phi}} \begin{aligned}[t]&
         \int_{\Theta^{\tmop{in}}_{\Gamma} \times B_1} {\1_{\mf{b}_-^{\ast}
         (\theta', \zeta) < \mf{b}_+^{\ast} (\theta', \zeta)}}  \left|
         F^{\ast} (\theta', \zeta, \mf{b}_+^{\ast} (\theta', \zeta))
         [\tilde{\phi}] \right|^{\frac{u}{2}} \frac{\dd \zeta \dd \theta'}{|
         \Theta |}
       \end{aligned} \lesssim \RHS{\eqref{eq:unif-lac-size-domination}}^u .
     \end{aligned} \]
  Analogously,
  \[ \begin{aligned}[t]&
       \mathrm{I}^- \lesssim \sup_{\tilde{\phi}} \begin{aligned}[t]&
         \int_{\Theta_{\Gamma}^{\tmop{in}} \times B_1} \left(
         \int_{\bar{\Theta}} {\1_{\mf{b}_-^{\ast} (\theta', \zeta) <
         \mf{b}_+^{\ast} (\theta', \zeta)}}  \left| F^{\ast} (\theta', \zeta,
         \mf{b}_-^{\ast} (\theta', \zeta)) [\tilde{\phi}] \right| \right.
         \qquad\\ &
         \times \left\nobracket (1 + \delta (\theta - \theta')) \dd \theta'
         \right)^{\frac{u}{2}} \frac{\dd \zeta \dd \theta}{| \Theta |}
       \end{aligned}\\ &
       \qquad \lesssim \sup_{\tilde{\phi}} \begin{aligned}[t]&
         \int_{\bar{\Theta}_{\Gamma} \times B_1} {\1_{\mf{b}_-^{\ast}
         (\theta', \zeta) < \mf{b}_+^{\ast} (\theta', \zeta)}}  \left|
         F^{\ast} (\theta', \zeta, \mf{b}_-^{\ast} (\theta', \zeta))
         [\tilde{\phi}] \right|^{\frac{u}{2}} \frac{\dd \zeta \dd \theta'}{|
         \Theta |}
       \end{aligned}
       \\ & \qquad\lesssim \RHS{\eqref{eq:unif-lac-size-domination}}^u .
     \end{aligned} \]
  The last inequality in each of the chains above follows from
  \Cref{lem:geometry-of-boundary} and in particular from
  \eqref{eq:boundary-singular-size-local}.
  
  The remaining part of this proof is dedicated to showing that
  \eqref{eq:ld:full-decomposition-bound} holds. In
  {\LHS{}}\eqref{eq:ld:full-decomposition-bound} fix \((\theta, \zeta, \sigma)
  \in \mT_{\Theta} \cap (E_+^{\ast} \setminus E_-^{\ast})\) and \(\phi \in
  \Phi^{\infty}_{\gr}\) with \(\| \phi \|_{\Phi^N_{\gr}} \leq 1\). We will
  discuss two cases. First we consider when \(\sigma \leq C_{\tmop{sep}}
  \mf{b}^{\ast}_- (\theta, \zeta)\) for some very large \(C_{\tmop{sep}} \gg 1\).
  This corresponds to the point \((\theta, \zeta, \sigma)\) being close to the
  boundary of \(E^{\ast}_-\). We then address the more general case when \(\sigma
  \in \left( C_{\tmop{sep}}  \mf{b}^{\ast}_- (\theta, \zeta), \mf{b}^{\ast}_+
  (\theta, \zeta) \right)\).
  
  We start with the former case. In this regime we also allow \(\phi\) to have
  larger frequency support: \(\phi \in \Phi_{4 \gr}^{\infty}\) with \(\| \phi
  \|_{\Phi_{4 \gr}^N} \leq 1\) instead of only \(\phi \in \Phi_{\gr}^{\infty}\);
  this will be necessary in the other regime. We further distinguish two
  cases: \tmtextbf{Case 1.0} and \tmtextbf{Case 1.5}. The former is
  when \(\sigma < \left( 1 + 2^{- 10} \gr \right) \mf{b}^{\ast}_- (\theta,
  \zeta)\) or if \(| \theta_{\Gamma} | < \gr\) while the latter is when \(\sigma
  \leq C_{\tmop{sep}} \mf{b}^{\ast}_- (\theta, \zeta)\) but the conditions of
  \tmtextbf{Case 1.0} do not hold.
  
  \tmtextbf{Case 1.0}: Suppose \(\sigma < \left( 1 + 2^{- 10} \gr \right)
  \mf{b}^{\ast}_- (\theta, \zeta)\) or \(| \theta_{\Gamma} | < \gr\) with \(\sigma
  \leq C_{\tmop{sep}} \mf{b}^{\ast}_- (\theta, \zeta)\). We will use the fact
  that
  \[ \begin{aligned}[t]&
       \wpD_{\zeta} (\theta) F^{\ast} (\theta, \zeta, \sigma) [\phi] =
       F^{\ast} (\theta, \zeta, \mf{b}^{\ast}_- (\theta, \zeta)) [\tilde{\phi}
       ],\\ &
       \widetilde{\phi }^{\vee} \eqd \Mod_{- \theta_{\Gamma}} {\Dil }_{\sigma
       / \mf{b}^{\ast}_- (\theta, \zeta)} \Mod_{\theta_{\Gamma}}
       \phi_{\theta}^{\vee},\\ &
       \phi_{\theta} (z) \eqd (- d_z + 2 \pi i \theta_{\Gamma}) \phi (z)
     \end{aligned} \]
  This identity follows from \eqref{eq:embedding} that defines \(\Emb\). Since
  \[ \spt \left( \FT{{\Dil }_{\sigma} \Mod_{\theta_{\Gamma}}
     \phi_{\theta}^{\vee}} \right) \subset \left( \sigma^{- 1} \left(
     \theta_{\Gamma} - 4 \gr \right), \sigma^{- 1} \left( \theta_{\Gamma} + 4
     \gr \right) \right) \]
  it holds that
  \[ \spt \left( \FT{\widetilde{\phi  }^{\vee}} \right) \subset \left(
     \frac{\mf{b}^{\ast}_- (\theta, \zeta)}{\sigma} \left( \theta_{\Gamma} - 4
     \gr \right) - \theta_{\Gamma}, \frac{\mf{b}^{\ast}_- (\theta,
     \zeta)}{\sigma} \left( \theta_{\Gamma} + 4 \gr \right) - \theta_{\Gamma}
     \right) . \]
  Since \(\theta_{\Gamma} \in \Theta\) and thus \(| \theta_{\Gamma} | < 10\) and
  either \(| \theta_{\Gamma} | < \gr\) or \(\frac{1}{1 + 2^{- 10} \gr} <
  \frac{\mf{b}^{\ast}_- (\theta, \zeta)}{\sigma} < 1\) we have that \(\left|
  \left( \frac{\mf{b}^{\ast}_- (\theta, \zeta)}{\sigma} - 1 \right)
  \theta_{\Gamma} \right| < 2 \gr\). This implies that \(\tilde{\phi}  \in
  \Phi_{8 \gr}^N\) with \(\| \tilde{\phi}  \|_{\Phi_{8 \gr}^N}
  \lesssim_{C_{\tmop{sep}}} \| \phi \|_{\Phi_{4 \gr}^N} \leq 1\). This shows
  bound \eqref{eq:ld:full-decomposition-bound} in this case because we have
  obtained that
  \[ \left| \wpD_{\zeta} (\theta_{\Gamma}) F^{\ast} (\theta, \zeta, \sigma)
     [\phi] \right| \lesssim \sup_{\tilde{\phi}} \begin{aligned}[t]&
       \int_{\Theta^{\tmop{in}}_{\Gamma}} {\1_{\mf{b}_-^{\ast} (\theta',
       \zeta) < \mf{b}_+^{\ast} (\theta', \zeta)}}  e^{- \left| \log \left(
       \frac{\sigma}{\mf{b}_-^{\ast} (\theta', \zeta)} \right) \right|}\\ &
       \qquad \times \left| F^{\ast} (\theta', \zeta, \mf{b}_-^{\ast}
       (\theta', \zeta)) [\tilde{\phi}] \right| (1 + \delta (\theta -
       \theta')) \dd \theta'
     \end{aligned} \]

  \tmtextbf{Case 1.5}: We now suppose that \(1 + 2^{- 10} \gr \leq
  \frac{\sigma}{\mf{b}^{\ast}_- (\theta, \zeta)} < C_{\tmop{sep}}\) and \(|
  \theta_{\Gamma} | > \gr\). We assume that \(\theta > - \gamma +
  \frac{\gr}{\alpha \beta}\), the case \(\theta < - \gamma - \frac{\gr}{\alpha
  \beta}\) is dealt with symmetrically. Let \(\left( \omega , \mf{b}^{\ast}_-
  (\omega, \zeta) \right)\) be the point in \(\Theta \times \R_+\) where the
  segment connecting the points \((\theta, \sigma) \in \Theta \times \R_+\) and
  \((- \gamma, 0) \in \Theta \times [0, + \infty)\) intersects the graph of
  \(\theta' \mapsto \mf{b}^{\ast}_- (\theta', \zeta)\). We claim that
  \begin{itemize}
    \item such a point exists,
    
    \item it is unique,
    
    \item it holds that \(\mf{b}^{\ast}_- (\omega, \zeta) < \mf{b}^{\ast}_+
    (\omega, \zeta)\),
    
    \item there exists a small \(\bar{r} > 0\) such that \((\omega, \omega +
    \bar{r}) \subset (- \gamma, \theta)\) and \(\mf{b}^{\ast}_- (\theta', \zeta)
    < \mf{b}^{\ast}_+ (\theta', \zeta)\) for all \(\theta' \in (\omega, \omega +
    \bar{r})\).
  \end{itemize}
  Assuming for now that this holds, we define \(\tilde{\phi}_{\theta'}\) by
  setting
  \[ {\Dil }_{\mf{b}^{\ast}_- (\theta', \zeta)} \Mod_{\theta'_{\Gamma}}
     \widetilde{\phi }_{\theta'}^{\vee} \eqd {\Dil }_{\sigma}
     \Mod_{\theta_{\Gamma}} \phi_{\theta}^{\vee} \]
  so that
  \[ \wpD_{\zeta} (\theta_{\Gamma}) F^{\ast} (\theta, \zeta, \sigma) [\phi] =
     \frac{1}{\bar{r}} \int_{\omega}^{\omega + \bar{r}} F^{\ast} (\theta',
     \zeta, \mf{b}^{\ast}_- (\theta', \zeta)) [{\widetilde{\phi }_{\theta'}} ]
     \dd \theta' . \]
  It holds that \(\spt \left( \FT{\widetilde{\phi  }_{\omega}} \right) \subset
  B_{4 \gr \frac{\mf{b}^{\ast}_- (\omega, \zeta)}{\sigma}} \subset B_{4 \gr}\);
  using bound \eqref{eq:boundary-regularity-local}, and in particular that
  \[ e^{- | \theta' - \omega |} < \frac{\mf{b}^{\ast}_- (\theta',
     \zeta)}{\mf{b}^{\ast}_- (\omega, \zeta)} < e^{| \theta' - \omega |}, \]
  we deduce, by continuity, that \(\spt \left( \FT{\widetilde{\phi 
  }_{\theta'}} \right) \subset B_{8 \gr}\) if \(\bar{r} > 0\) is chosen small
  enough. Furthermore, since \(\sigma < C_{\tmop{sep}} \mf{b}^{\ast}_- (\theta,
  \zeta)\) and \(\mf{b}^{\ast}_- (\omega, \zeta) > e^{- | \theta - \gamma |}
  \mf{b}^{\ast}_- (\theta, \zeta)\) it holds that \(\sigma \lesssim
  C_{\tmop{sep}} \mf{b}^{\ast}_- (\omega, \zeta)\). Thus \(\widetilde{\phi
  }_{\theta'} \in \Phi_{8 \gr}^N\) with \(\left\| {\widetilde{\phi }_{\theta'}} 
  \right\|_{\Phi_{8 \gr}^N} \lesssim_{C_{\tmop{sep}}} \| \phi \|_{\Phi_{4
  \gr}^N} \leq 1\). An application of \Cref{lem:wave-packet-decomposition}
  allows us to conclude that
  \[ \left| \wpD_{\zeta} (\theta_{\Gamma}) F^{\ast} (\theta, \zeta, \sigma)
     [\phi] \right| = \sup_{\tilde{\phi}} \frac{1}{\bar{r}}
     \int_{\omega}^{\omega + \bar{r}} \left| F^{\ast} (\theta', \zeta,
     \mf{b}^{\ast}_- (\theta', \zeta)) [\widetilde{\phi } ] \right| \dd
     \theta' \]
  with the upper bound taken over \(\widetilde{\phi } \in \Phi_{16
  \gr}^{\infty}\) with \(\| \widetilde{\phi }  \|_{\Phi_{16 \gr}^{N - 2}}
  \lesssim_{C_{\tmop{sep}}} 1.\) This would conclude showing bound
  \eqref{eq:ld:full-decomposition-bound} in this case. It remains to prove the
  properties about the point \(\left( \omega , \mf{b}^{\ast}_- (\omega, \zeta)
  \right)\) claimed above.
  
  At least one point of intersection between the segment connecting \
  \((\theta, \sigma) \in \Theta \times [0, + \infty)\) and \((- \gamma, 0) \in
  \Theta \times [0, + \infty)\) and the graph of \(\theta' \rightarrow
  \mf{b}_-^{\ast} (\theta', \zeta)\) exists because \(\mf{b}_-^{\ast} (- \gamma,
  \zeta) > 0\), \(\mf{b}_-^{\ast} (\theta, \zeta) < \sigma\), and \(\theta'
  \mapsto \mf{b}_-^{\ast} (\theta', \zeta)\) is continuous. Let us call \(\left(
  \omega, \mf{b}_-^{\ast} (\omega, \zeta) \right)\) one such point. Let us show
  that the intersection is unique. As a matter of fact, if \(\mf{b}\) is any
  function satisfying \eqref{eq:boundary-regularity-local} then its graph can
  only cross the segment connecting \((\theta, \sigma) \in \Theta \times [0,
  + \infty)\) and \((- \gamma, 0) \in \Theta \times [0, + \infty)\) from above to
  below. To see this suppose that \(\theta_0\) is the \(\Theta\) coordinate of a
  point of intersection: we then have
  \[ \mf{b} (\theta_0, \zeta) = \sigma \frac{\theta_0 + \gamma}{\theta +
     \gamma} . \]
  By \eqref{eq:boundary-regularity-local} we have that \(\dd_{\theta'}
  \mf{b}^{\ast}  (\theta', \zeta) \leq \frac{\mf{b}^{\ast}  (\theta',
  \zeta)}{\theta' - \theta_-}\) from which we get that for any \(\theta' >
  \theta_0\)
  \[ \mf{b}^{\ast} (\theta', \zeta) \leq \mf{b}^{\ast} (\theta_0, \zeta)
     \frac{\theta' - \theta_-}{\theta_0 - \theta_-} = \sigma \frac{\theta_0 +
     \gamma}{\theta + \gamma} \frac{\theta' - \theta_-}{\theta_0 - \theta_-}
     \leq \sigma \frac{\theta' + \gamma}{\theta + \gamma} ; \]
  we used that \(\theta_0 - \theta_- > \theta_0 + \gamma\). This is exactly the
  property we claimed. By the same reasoning, if we were to have
  \(\mf{b}^{\ast}_+ (\omega, \zeta) < \mf{b}^{\ast}_- (\omega, \zeta) = \sigma
  \frac{\omega + \gamma}{\theta + \gamma}\) then it would holds that
  \[ \mf{b}^{\ast}_+ (\theta, \zeta) \leq \sigma \]
  which is a contradiction to the fact that \((\theta, \zeta, \sigma) \in
  \mT_{\Theta} \cap (E_+^{\ast} \setminus E_-^{\ast})\).
  
  Let us now address the existence of \(\bar{r} > 0\). Since \(\sigma >
  \mf{b}^{\ast}_- (\theta, \zeta) \left( 1 + 2^{- 10} \gr \right)\),
  \(\mf{b}^{\ast}_- (\theta', \zeta) < \mf{b}^{\ast}_- (\theta, \zeta) e^{|
  \theta' - \theta |}\) and the slope of the segment connecting \((- \gamma, 0)\)
  to \((\theta, \sigma)\) is bounded from above by \(\alpha \beta
  \frac{\mf{b}^{\ast}_- (\theta, \zeta)}{\gr}\) it follows that \(\omega <
  \theta - 2 \bar{r}\) for some universal \(\bar{r} > 0\) independent of
  \(\Gamma\), \(\theta\) and \(\zeta\). Finally, for any \(\theta' \in (\omega,
  \omega + \bar{r})\) it holds that \(\mf{b^{\ast}_- (\theta', \zeta)} \leq
  \mf{b^{\ast}_+ (\theta', \zeta)}\) because otherwise, as before, we would be
  able to deduce that \(\mf{b}^{\ast}_+ (\theta, \zeta) \leq \sigma\) in
  contradiction with the fact that \((\theta, \zeta, \sigma) \in \mT_{\Theta}
  \cap (E_+^{\ast} \setminus E_-^{\ast})\).
  
  \tmtextbf{Case 2:} Now let us discuss the case \(\sigma > C_{\tmop{sep}}
  \mf{b}^{\ast}_- (\theta, \zeta)\). Let \\ \(\mf{b}^{\ast}_- (\zeta) =
  \sup_{\theta' \in \bar{\Theta}_{\Gamma}} \mf{b}^{\ast}_- (\theta', \zeta)\)
  and \(\mf{b}^{\ast}_- (\zeta) = \inf_{\theta' \in \bar{\Theta}_{\Gamma}}
  \mf{b}^{\ast}_+ (\theta', \zeta)\). By \eqref{eq:boundary-regularity-local}
  and in particular from the fact that for any \(\theta' \in
  \bar{\Theta}_{\Gamma}\) it holds that
  \[ e^{- | \bar{\Theta}_{\Gamma} |} \mf{b}^{\ast}_{\pm} (\theta, \zeta) \leq
     \mf{b}^{\ast}_{\pm} (\theta', \zeta) \leq e^{| \bar{\Theta}_{\Gamma} |}
     \mf{b}^{\ast}_{\pm} (\theta, \zeta) \]
  so we obtain that \(\mf{b}^{\ast}_- (\theta, \zeta) \mf{< b}^{\ast}_- (\zeta)
  < 2 \mf{b}^{\ast}_- (\theta, \zeta)\), \(\mf{b}^{\ast}_+ (\theta, \zeta) \mf{<
  b}^{\ast}_+ (\zeta) < 2 \mf{b}^{\ast}_+ (\theta, \zeta)\), and in particular
  \(\mf{b}^{\ast}_+ (\zeta) > \frac{C_{\tmop{sep}}}{4} \mf{b}^{\ast}_- (\zeta)\)
  and \(\sigma > \frac{C_{\tmop{sep}}}{2} \mf{b}^{\ast}_- (\zeta)\). We set
  \(C_{\tmop{sep}} = 2^{11}\). Fix \(\chi \in \Phi^{\infty}_2\) with \(\FT{\chi}
  \left( \FT{z} \right) \in [0, 1]\) and \(\FT{\chi} \left( \FT{z} \right) = 1\)
  for \(\FT{z} \in B_{7 / 4}\) and let \(\chi_{\gr} (z) = \gr \chi \left( \gr z
  \right)\). Let us define
  \[ M^+ \left( \FT{z} \right) = \int_{\bar{\Theta}_{\Gamma}} \FT{\chi_{\gr}}
     \left( \mf{b}^{\ast}_+ (\theta', \zeta) \FT{z} - \theta'_{\Gamma} \right)
     \dd \theta' . \]
  It holds that \(M^+ \left( \FT{z} \right) \geq 0\), \(M^+ \left( \FT{z} \right)
  \gtrsim 1\) for \(\left| \FT{z} \right| < \frac{1}{2 \mf{b}^{\ast}_+
  (\zeta)}\), and \(\left| \partial_{\FT{z}}^n M^+ \left( \FT{z} \right)
  \right| \lesssim \mf{b}^{\ast}_+ (\zeta)^{- n}\). Recall that for \((\theta,
  \zeta, \sigma) \in \mT_{\Theta} \cap (E_+^{\ast} \setminus E_-^{\ast})\) we
  have \(\wpD_{\zeta} (\theta_{\Gamma}) F^{\ast} (\theta, \zeta, \sigma) [\phi]
  = f \ast \Dil_{\sigma} \Upsilon\) with \(\Upsilon  \assign \dd_z
  \tmop{Mod}_{\theta_{\Gamma}} \phi^{\vee} (z)\); we keep the dependence on
  \(\theta_{\Gamma}\) implicit. It holds that
  \[ \spt \left( \FT{\Dil_{\sigma} \Upsilon } \right) \subset \left(
     \sigma^{- 1} \left( \theta_{\Gamma} - \gr \right), \sigma^{- 1} \left(
     \theta_{\Gamma} + \gr \right) \right) . \]
  Let us decompose
  \[ \FT{\Dil_{\sigma} \Upsilon } \left( \FT{z} \right) \begin{aligned}[t]&
       = \FT{\chi} \left( 4 \mf{b}^{\ast}_+ (\zeta) \FT{z} \right)
       \FT{\Dil_{\sigma} \Upsilon_{\theta_{\Gamma}}} \left( \FT{z} \right) +
       \left( 1 - \FT{\chi} \left( 4 \mf{b}^{\ast}_+ (\zeta) \FT{z} \right)
       \right) \FT{\Dil_{\sigma} \Upsilon } \left( \FT{z} \right)\\ &
       = \FT{\Upsilon_{\sigma}^{\uparrow}} + \FT{\Upsilon_{\sigma}^0} .
     \end{aligned} \]
  It holds that
  \[ \FT{\Upsilon_{\sigma}^{\uparrow}} = \int_{\bar{\Theta}_{\Gamma}}
     \frac{\FT{\Upsilon_{\sigma}^{\uparrow}} \left( \FT{z} \right)}{M^+ \left(
     \FT{z} \right)} \FT{\chi_{\gr}} \left( \mf{b}^{\ast}_+ (\theta', \zeta)
     \FT{z} - \theta'_{\Gamma} \right) \dd \theta' \]
  so let us define
  \[ \FT{\phi_{\theta'}} \left( \FT{z} \right) \eqd
     \frac{\FT{\Upsilon_{\sigma}^{\uparrow}} \left( \frac{\FT{z} +
     \theta'_{\Gamma}}{\mf{b}^{\ast}_+ (\theta', \zeta)} \right)}{M^+ \left(
     \frac{\FT{z} + \theta'_{\Gamma}}{\mf{b}^{\ast}_+ (\theta', \zeta)}
     \right)} \FT{\chi_{\gr}} \left( \FT{z} \right) \]
  so that \(\phi_{\theta'} \in \Phi^{\infty}_{2 \gr}\), \(\| \phi_{\theta'}
  \|_{\Phi^N_{2 \gr}} \lesssim 1\) and
  \[ \begin{aligned}[t]&
       \wpD_{\zeta} (\theta_{\Gamma}) F^{\ast} (\theta, \zeta, \sigma) [\phi]
       =
       \begin{aligned}[t]
       & \int_{\bar{\Theta}_{\Gamma}^{\tmop{in}}} F^{\ast} (\theta', \zeta,
       \mf{b}^{\ast}_+ (\theta', \zeta)) [\phi_{\theta'}] \dd \theta' \\ &\qquad 
       + \int_{\bar{\Theta}_{\Gamma}^{\tmop{ex}}} F^{\ast} (\theta', \zeta,
       \mf{b}^{\ast}_+ (\theta', \zeta)) [\phi_{\theta'}] \dd \theta' .
       \end{aligned}       
     \end{aligned} \]
  An application of \Cref{lem:wave-packet-decomposition} allows us to conclude
  that
  \[ \left| \int_{\bar{\Theta}_{\Gamma}^{\tmop{in}}} F^{\ast} (\theta',
     \zeta, \mf{b}^{\ast}_+ (\theta', \zeta)) [\phi_{\theta'}] \dd \theta'
     \right| \lesssim \sup_{\tilde{\phi}} \left|
     \int_{\bar{\Theta}_{\Gamma}^{\tmop{in}}} F^{\ast} (\theta', \zeta,
     \mf{b}^{\ast}_+ (\theta', \zeta)) [\tilde{\phi}] \dd \theta' \right| \]
  with the upper bound taken over \(\tilde{\phi} \in \Phi_{4 \gr}^{\infty}\)
  with \(\| \tilde{\phi} \|_{\Phi_{4 \gr}^{N - 4}} \leq 1.\) To bound the second
  addend, define \(\phi_{\theta', \sigma'}\) for all \(\sigma' \in \left(
  \frac{\mf{b}^{\ast}_+ (\theta', \zeta)}{1 + \bar{r}}, \mf{b}^{\ast}_+
  (\theta', \zeta) \right)\) for some sufficiently small \(\bar{r} > 0\) by
  setting
  \[ \Dil_{\sigma'} \Mod_{\theta'_{\Gamma}} \phi_{\theta', \sigma'} =
     \Dil_{\mf{b}^{\ast}_+ (\theta', \zeta)} \Mod_{\theta'_{\Gamma}}
     \phi_{\theta'} . \]
  Note that if \(\bar{r} > 0\) is small enough and since \(\mf{b}^{\ast}_+
  (\zeta) / \mf{b}^{\ast}_- (\zeta) > C_{\tmop{sep}} / 4\) we have that \
  \(\mf{b}^{\ast}_- (\theta', \zeta) < \sigma' < \mf{b}^{\ast}_+ (\theta',
  \zeta)\) and \(\phi_{\theta', \sigma'} \in \Phi_{4 \gr}^{\infty}\) with \(\|
  \phi_{\theta', \sigma'} \|_{\Phi_{4 \gr}^{N - 2}} \lesssim 1\). An
  application of \Cref{lem:wave-packet-decomposition} allows us to conclude
  that
  \[ \begin{aligned}[t]&
       \left| \int_{\bar{\Theta}_{\Gamma}^{\tmop{in}}} F^{\ast} (\theta',
       \zeta, \mf{b}^{\ast}_+ (\theta', \zeta)) [\phi_{\theta'}] \dd \theta'
       \right| \lesssim \left| \int_{\bar{\Theta}_{\Gamma}^{\tmop{in}}}
       \int_{\frac{\mf{b}^{\ast}_+ (\theta', \zeta)}{1 +
       \bar{r}}}^{\mf{b}^{\ast}_+ (\theta', \zeta)} F^{\ast} (\theta', \zeta,
       \mf{b}^{\ast}_+ (\theta', \zeta)) [\phi_{\theta', \sigma'}] \dd \theta'
       \frac{\dd \sigma'}{\sigma'} \right|\\ &
       \sup_{\tilde{\phi}} \int_{\bar{\Theta}_{\Gamma}^{\tmop{in}}}
       \int_{\frac{\mf{b}^{\ast}_+ (\theta', \zeta)}{1 +
       \bar{r}}}^{\mf{b}^{\ast}_+ (\theta', \zeta)} \left| F^{\ast} (\theta',
       \zeta, \mf{b}^{\ast}_+ (\theta', \zeta)) [\tilde{\phi}] \right| \dd
       \theta'
     \end{aligned} \]
  with the upper bound taken over \(\tilde{\phi} \in \Phi_{4 \gr}^{\infty}\)
  with \(\| \tilde{\phi} \|_{\Phi_{8 \gr}^{N - 4}} \leq 1.\) It remains to
  decompose \(\FT{\Upsilon_{\sigma}^0}\). Recall that \
  \[ \begin{aligned}[t]&
       \FT{\Upsilon_{\sigma}^0} \left( \FT{z} \right) = \left( 1 - \FT{\chi}
       \left( 4 \mf{b}^{\ast}_+ (\zeta) \FT{z} \right) \right)
       \FT{\Dil_{\sigma} \Upsilon} \left( \FT{z} \right),\\ &
       \spt \left( \FT{\Upsilon_{\sigma}^0} \right) \subset B_{\frac{1 + 4
       \gr}{4 \sigma}} \setminus B_{\frac{7}{16 \mf{b}_+^{\ast} (\zeta)}} .
     \end{aligned} \]
  For any \(\theta' \in \bar{\Theta}_{\Gamma}^{\tmop{ex}}\) it holds that
  \(\theta'_{\Gamma} > 2^{- 4}\) or \(3 < \theta'_{\Gamma} < - 2^{- 4}\)
  \[ m_{\theta'_{\Gamma}} \left( \FT{z} \right)  = \int_1^{\infty} \FT{\chi}
     \left( \sigma \FT{z} - \theta'_{\Gamma} \right) \frac{\dd \sigma}{\sigma}
     . \]
  Note that \(\spt \left( \FT{\chi} \left( \sigma \FT{z} - \theta'_{\Gamma}
  \right) \right) \subset \left( \sigma^{- 1} \left( \theta'_{\Gamma} - 2 \gr
  \right), \sigma^{- 1} \left( \theta'_{\Gamma} + 2 \gr \right) \right)\) and \
  \(\FT{\chi} \left( \sigma \FT{z} - \theta'_{\Gamma} \right) = 1\) on \(\left(
  \sigma^{- 1} \left( \theta'_{\Gamma} - 7 \gr / 4 \right), \sigma^{- 1}
  \left( \theta'_{\Gamma} + 7 \gr / 4 \right) \right)\). It follows that if
  \(\theta'_{\Gamma}\) then
  \[ \spt \left( m_{\theta'_{\Gamma}} \left( \FT{z} \right)  \right) \subset
     \left( 0, \theta'_{\Gamma} + 2 \gr \right) \]
  and
  \[ m_{\theta'_{\Gamma}} \left( \FT{z} \right) = \int_0^{\infty} \FT{\chi}
     \left( \FT{z} - \theta'_{\Gamma} \right) \frac{\dd \FT{z}}{\FT{z}} \]
  for \(\FT{z} \in \left( 0, \theta'_{\Gamma} - 2 \gr \right)\);
  \(\int_0^{\infty} \FT{\chi} \left( \FT{z} - \theta'_{\Gamma} \right)
  \frac{\dd \FT{z}}{\FT{z}}\) is bounded above and away from \(0\) uniformly in
  \(\theta'_{\Gamma}\). It follows that
  \[ m \left( \FT{z} \right) = \int_{\bar{\Theta}^{\tmop{ex}}_{\Gamma}}
     m_{\theta'_{\Gamma}} \left( \FT{z} \right) \dd \theta' \]
  satisfies
  \[ \spt \left( m \left( \FT{z} \right)  \right) \subset \left( - 2 - 2 \gr,
     2 + 2 \gr \right) \]
  and
  \[ m \left( \FT{z} \right) = \int_{\bar{\Theta}^{\tmop{ex}}_{\Gamma}}
     \int_0^{\infty} \FT{\chi}_{\gr} \left( \FT{z} - \theta'_{\Gamma} \right)
     \frac{\dd \FT{z}}{\FT{z}} \dd \theta' \]
  for \(\FT{z} \in B_{2^{- 6} - 2 \gr}\). Let us re-define \(m (0) =
  \lim_{\FT{z} \rightarrow 0} m \left( \FT{z} \right)\) to obtain a smooth
  function \(\FT{z} \rightarrow m \left( \FT{z} \right) .\) Now set
  \(M_{\sigma_-, \sigma_+} \left( \FT{z} \right) \eqd \frac{m \left( \sigma_-
  \FT{z} \right) - m \left( \sigma_+ \FT{z} \right)}{m (0)}\) for any \(\sigma_-
  < \sigma_+\). It holds that
  \[ \begin{aligned}[t]&
       M_{\sigma_-, \sigma_+} \left( \FT{z} \right) \geq 0,\\ &
       M_{\sigma_-, \sigma_+} \left( \FT{z} \right) = 0 \qquad \text{if }
       \left| \FT{z} \right| > 2 \sigma_-^{- 1} \left( 1 + \gr \right),\\ &
       M_{\sigma_-, \sigma_+} \left( \FT{z} \right) = 0 \qquad \text{if }
       \left| \FT{z} \right| < 2 \sigma_+^{- 1} \left( 2^{- 7} - \gr
       \right),\\ &
       M_{\sigma_-, \sigma_+} \left( \FT{z} \right) = 1 \qquad \text{if } 2
       \sigma_+^{- 1} \left( 1 + \gr \right) < \left| \FT{z} \right| < 2
       \sigma_-^{- 1} \left( 2^{- 7} - \gr \right) .
     \end{aligned} \]
  The first claim follows from the fact that
  \[ M_{\sigma_-, \sigma_+} \left( \FT{z} \right) = \frac{1}{m (0)}
     \int_{\bar{\Theta}^{\tmop{ex}}_{\Gamma}} \int_{\sigma_-}^{\sigma_+}
     \FT{\chi}_{\gr} \left( \sigma \FT{z} - \theta'_{\Gamma} \right) \frac{\dd
     \sigma}{\sigma} \dd \theta' \]
  while subsequent claims follow from support considerations on \(m\) above. Let
  us write
  \[ \FT{\Upsilon_{\sigma}^0} \left( \FT{z} \right) = M_{\mf{b }^{\ast}_-
     (\zeta), \mf{b }^{\ast}_+ (\zeta)} \left( \FT{z} \right)
     \FT{\Upsilon_{\sigma}^0} \left( \FT{z} \right) . \]
  This identity holds because of support considerations and since
  \(\frac{\mf{b}^{\ast}_- (\zeta) C_{\tmop{sep}}}{2} < \sigma < \mf{b}_-^{\ast}
  (\zeta)\). We obtain that
  \[ \begin{aligned}[t]&
       \FT{\Upsilon_{\sigma}^0} \left( \FT{z} \right) =
       \int_{\bar{\Theta}^{\tmop{ex}}_{\Gamma}} \int_0^{\infty}
       \FT{\Upsilon_{\sigma}^0} \left( \FT{z} \right) \FT{\chi}_{\gr} \left(
       \sigma' \FT{z} - \theta'_{\Gamma} \right) \frac{\dd \sigma'}{\sigma'}
       \dd \theta'\\ &
       \qquad = \int_{\bar{\Theta}^{\tmop{ex}}_{\Gamma}} \int_0^{\infty}
       \left( 1 - \FT{\chi} \left( 4 \mf{b}^{\ast}_+ (\zeta) \FT{z} \right)
       \right) \sigma \FT{z} \FT{\phi} \left( \sigma \FT{z} - \theta_{\Gamma}
       \right) \FT{\chi}_{\gr} \left( \sigma' \FT{z} - \theta'_{\Gamma}
       \right) \frac{\dd \sigma'}{\sigma'} \dd \theta'
     \end{aligned} \]
  By support considerations it holds that
  \[ \FT{\phi} \left( \sigma \FT{z} - \theta_{\Gamma} \right) \FT{\chi}_{\gr}
     \left( \sigma' \FT{z} - \theta'_{\Gamma} \right) = 0 \]
  unless \(\sigma' \gtrsim \sigma\) and \(\left( 1 - \FT{\chi} \left( 4
  \mf{b}^{\ast}_+ (\zeta) \FT{z} \right) \right) \FT{\chi}_{\gr} \left(
  \sigma' \FT{z} - \theta'_{\Gamma} \right) = 0\) unless \(\sigma' \gtrsim
  \mf{b}^{\ast}_+ (\zeta) \FT{z}\). It follows that setting
  \[ \phi_{\theta', \sigma'} \left( \sigma' \FT{z} - \theta'_{\Gamma} \right)
     \eqd \left( 1 - \FT{\chi} \left( 4 \mf{b}^{\ast}_+ (\zeta) \FT{z} \right)
     \right) \sigma' \FT{z} \FT{\phi} \left( \sigma \FT{z} - \theta_{\Gamma}
     \right) \FT{\chi}_{\gr} \left( \sigma' \FT{z} - \theta'_{\Gamma} \right)
  \]
  we obtain a function \(\phi_{\theta', \zeta', \sigma'} \in \Phi_{2 \gr}^N\)
  with \({\| \phi_{\theta', \zeta', \sigma'} \|_{\Phi_{2 \gr}^N}}  \lesssim 1\)
  so
  \[ \FT{\Upsilon_{\sigma}^0} \left( \FT{z} \right) =
     \int_{\bar{\Theta}^{\tmop{ex}}_{\Gamma}} \int_0^{\infty}
     \frac{\sigma}{\sigma'} \1_{\sigma' > a \sigma} \phi_{\theta', \zeta',
     \sigma'} \left( \sigma' \FT{z} - \theta'_{\Gamma} \right) \frac{\dd
     \sigma'}{\sigma'} \dd \theta' . \]
  or some \(a > 0\). Making sure to choose \(C_{\tmop{sep}} > \frac{4}{a}\) and
  applying \Cref{lem:wave-packet-decomposition} allows us to conclude that
  \[ | f \ast \Upsilon_{\sigma}^0 (\zeta) | \lesssim \sup_{\tilde{\phi}}
     \int_{\bar{\Theta}^{\tmop{ex}}_{\Gamma}} \int_0^{\infty}
     \frac{\sigma}{\sigma'} \1_{\sigma' > a \sigma} | F^{\ast} (\theta',
     \zeta', \sigma') [\tilde{\phi}] | \frac{\dd \sigma'}{\sigma'} \dd \theta'
     . \]
  with the upper bound taken over \(\tilde{\phi} \in \Phi_{4 \gr}^{\infty}\)
  with \(\| \tilde{\phi} \|_{\Phi_{8 \gr}^{N - 4}} \leq 1\), as required. 
\end{proof}

\begin{lemma}[Maximal size domination]
  \label{lem:unif-max-size-domination}Let \(V^+, V^- \in
  \mathbb{D}^{\cup}_{\beta}\), and let \(W^+, W^- \in \TT_{\Theta}^{\cup}\). The
  bounds
  \begin{equation}
    \begin{aligned}[t]&
      \| \1_{(V^+ \cap W^+) \setminus (V^- \cup W^-)} \left( \Emb [f] \circ
      \Gamma \right) \|_{{\Gamma^{\ast}}  \SL^{(u, \infty)}_{(\Theta,
      \Theta^{\tmop{in}})} \Phi_{\gr}^N} \qquad\\ &
      \qquad \lesssim \beta^{- \frac{1}{u}} \| \1_{(V^+ \cap W^+) \setminus
      (V^- \cup W^-)} \left( \Emb [f] \circ \Gamma \right) \|_{\left(
      {\Gamma^{\ast}}  \SL^{(u, 2)}_{(\Theta, \Theta^{\tmop{ex}})} \Phi_{4
      \gr}^{N - 2} {+ \Gamma^{\ast}}  \SL^{(u, 1)}_{\Theta} \dfct^{N - 2}_{4
      \gr} \right)},
    \end{aligned} \label{eq:unif-max-size-domination}
  \end{equation}
  hold for any \(f \in \Sch (\R)\) as long as \(u \in [2, \infty]\).
  The implicit constant is independent of \(f\) and of the sets \(V^{\pm},
  W^{\pm}\).
\end{lemma}

\begin{proof}
  Let us fix \(T \in \mathbb{T}_{\Theta}\) and show that
  \[ \| \1_{(V^+ \cap W^+) \setminus (V^- \cup W^-)} \left( \Emb [f] \circ
     \Gamma \right) \|_{{\Gamma^{\ast}}  \SL^{(u, \infty)}_{(\Theta,
     \Theta^{\tmop{in}})} \Phi_{\gr}^N (T)} \lesssim
     \RHS{\eqref{eq:unif-lac-size-domination} .} \]
  By symmetry, (see \eqref{eq:size-symmetry} and the subsequent discussion),
  we can assume that \({T = T_{\Theta}}  (0, 0, 1)\) so that
  \[ \Gamma \circ \pi_T (\theta, \zeta, \sigma) = (\alpha \beta (\theta +
     \gamma) (\beta \sigma)^{- 1}, \zeta, \beta \sigma) = (\theta_{\Gamma}
     (\beta \sigma)^{- 1}, \zeta, \beta \sigma) = \pi_T (\theta_{\Gamma},
     \zeta, \beta \sigma) . \]
  We use the shorthand \(E_+^{\ast} \assign \mT_{\Theta^{\tmop{in}}}  \cap
  \pi_T^{- 1} (V^+ \cap W^+)\), \(E_-^{\ast} \assign \pi_T^{- 1} (V^- \cup
  W^-)\), and \(F^{\ast} \eqd \1_{E_+^{\ast} \setminus E_-^{\ast}} \Emb{} [f]
  \circ \Gamma \circ \pi_T\). We have that
  \[ \begin{aligned}[t]&
       \| \1_{(V^+ \cap W^+) \setminus (V^- \cup W^-)} \left( \Emb [f] \circ
       \Gamma \right) \|_{{\Gamma^{\ast}}  \SL^{(u, \infty)}_{(\Theta,
       \Theta^{\tmop{in}})} \Phi_{\gr}^N (T)}\\ &
       = \sup_{\phi} \sup_{\sigma} \int_{\Theta^{\tmop{in}} \times B_1}
       F^{\ast} (\theta, \zeta, \sigma (\theta, \zeta)) [\phi]^u \frac{\dd
       \theta \dd \zeta}{| \Theta |}
     \end{aligned} \]
  where the upper bound is taken over all \(\phi \in \Phi^{\infty}_{\gr}\) with
  \(\| \phi \|_{\Phi^N_{\gr}} \leq 1\) and over all continuous functions \(\sigma
  : \Theta \times B_1 \rightarrow (0, + \infty)\). Let \(\mf{b}_+^{\ast}\),
  \(\mf{b}_-^{\ast}\) be the two functions whose graphs are the boundaries of
  \(E_+^{\ast}\) and \(E_-^{\ast}\) respectively so that \(E_+^{\ast} \setminus
  E_-^{\ast} = \left\{ (\theta, \zeta, \sigma) : \mf{b}_-^{\ast} (\theta,
  \zeta) < \sigma < \mf{b}_+^{\ast} (\theta, \zeta) \right\}\). The existence
  of such functions is guaranteed by \Cref{lem:geometry-of-boundary}. Let
  \(\bar{r} > 0\) be a small absolute constant to depending only on \(\gr\) be
  determined later. We distinguish the points \((\theta, \zeta)\) based on how
  far \(\sigma (\theta, \zeta)\) is from the boundary of \(E_+^{\ast} \setminus
  E_-^{\ast}\). Let
  \[ \begin{aligned}[t]&
       U_- \eqd \left\{ (\theta, \zeta) \in \Theta^{\tmop{in}} \times B_1
       \suchthat \mf{b}_-^{\ast} (\theta, \zeta) \leq \sigma (\theta, \zeta) <
       (1 + \bar{r}) \mf{b}_-^{\ast} (\theta, \zeta) \right\},\\ &
       U_+ \eqd \left\{ (\theta, \zeta) \in \Theta^{\tmop{in}} \times B_1
       \suchthat (1 - \bar{r}) \mf{b}_+^{\ast} (\theta, \zeta) < \sigma
       (\theta, \zeta) \leq \mf{b}_+^{\ast} (\theta, \zeta) \right\},\\ &
       U_0 \eqd \left\{ (\theta, \zeta) \in \Theta^{\tmop{in}} \times B_1
       \suchthat (1 + \bar{r}) \mf{b}_-^{\ast} (\theta, \zeta) \leq \sigma
       (\theta, \zeta) \leq (1 - \bar{r}) \mf{b}_+^{\ast} (\theta, \zeta)
       \right\} .
     \end{aligned} \]
  If \((\theta, \zeta) \in U_-\) we use that
  \[ \left( \Emb [f] \circ \Gamma \right) (\theta, \zeta, \sigma (\theta,
     \zeta)) [\phi] = \left( \Emb [f] \circ \Gamma \right) \left( \theta,
     \zeta, \mf{b}_-^{\ast} (\theta, \zeta) \right) [\phi_{(\theta, \zeta)}]
  \]
  with
  \[ \phi_{(\theta, \zeta)} \eqd \Mod_{\theta_{\Gamma}} \Dil_{\sigma (\theta,
     \zeta) / \mf{b}_-^{\ast} (\theta, \zeta)} \Mod_{- \theta_{\Gamma}} \phi .
  \]
  Since \(1 \leq \sigma (\theta, \zeta) / \mf{b}_-^{\ast} (\theta, \zeta) < 1 +
  \bar{r}\) we have that \(\phi_{(\theta, \zeta)} \in \Phi_{2 \gr}^{\infty}\)
  with \(\| \phi_{(\theta, \zeta)} \|_{\Phi_{2 \gr}^N} \lesssim 1\). Using
  \Cref{lem:wave-packet-decomposition} we get that
  \[ \begin{aligned}[t]&
       \int_{\Theta^{\tmop{in}} \times B_1} \1_{U_-} (\theta, \zeta) F^{\ast}
       (\theta, \zeta, \sigma (\theta, \zeta)) [\phi]^u \frac{\dd \theta \dd
       \zeta}{| \Theta |}\\ &
       \qquad \lesssim \sup_{\tilde{\phi}} \int_{\Theta^{\tmop{in}} \times
       B_1} \1_{U_-} (\theta, \zeta) F^{\ast} \left( \theta, \zeta,
       \mf{b}_-^{\ast} (\theta, \zeta) \right) [\phi]^u \frac{\dd \theta \dd
       \zeta}{| \Theta |}
     \end{aligned} \]
  with the upper bound taken over \(\tilde{\phi} \in \Phi_{2 \gr}^{\infty}\)
  with \(\| \tilde{\phi} \|_{\Phi_{4 \gr}^{N - 2}} \leq 1\). The {\RHS{}} of the
  entry above is bounded by \(\| \1_{(V^+ \cap W^+) \setminus (V^- \cup W^-)}
  \left( \Emb [f] \circ \Gamma \right) \|_{{\Gamma^{\ast}}  \SL^{(u,
  1)}_{\Theta} \dfct^{N - 2}_{4 \gr} (T)}\) and thus by
  \(\RHS{\eqref{eq:unif-lac-size-domination}^u}\), as required. The same
  procedure applies to the case when \((\theta, \zeta) \in U_+\): Set
  \(\phi_{(\theta, \zeta)} \eqd \Mod_{\theta_{\Gamma}} \Dil_{\sigma (\theta,
  \zeta) / \mf{b}_+^{\ast} (\theta, \zeta)} \Mod_{- \theta_{\Gamma}} \phi\) to
  obtain that
  \[ \begin{aligned}[t]&
       \int_{\Theta^{\tmop{in}} \times B_1} \1_{U_+} (\theta, \zeta) F^{\ast}
       (\theta, \zeta, \sigma (\theta, \zeta)) [\phi]^u \frac{\dd \theta \dd
       \zeta}{| \Theta |}\\ &
       \qquad \lesssim \sup_{\tilde{\phi}} \int_{\Theta^{\tmop{in}} \times
       B_1} \1_{U_-} (\theta, \zeta) F^{\ast} \left( \theta, \zeta,
       \mf{b}_+^{\ast} (\theta, \zeta) \right) [\phi]^u \frac{\dd \theta \dd
       \zeta}{| \Theta |} \lesssim \RHS{\eqref{eq:unif-lac-size-domination}^u}
       .
     \end{aligned} \]

  In the case \((\theta, \zeta) \in U_0\) we will just show the pointwise bound
  \(| F^{\ast} (\theta, \zeta, \sigma (\theta, \zeta)) [\phi] | \lesssim
  \RHS{\eqref{eq:unif-lac-size-domination}}\). Integrating this bound will
  yield the claim. Let \\\(\tilde{W}^- \eqd W^- \cup T_{\Theta} (0, \zeta, \sigma
  (\theta, \zeta))\); let \(\tilde{E}_-^{\ast} = E_-^{\ast} \cup \pi_T^{- 1}
  (T_{\Theta} (0, \zeta, \sigma (\theta, \zeta)))\) and let
  \(\mf{\tilde{b}_-^{\ast}} : \R^2 \rightarrow \R_+\) be the function given by
  \Cref{lem:geometry-of-boundary} such that its graphs is the boundary of the
  set \(\tilde{E}_-^{\ast}\). For any \(\left( \theta', \zeta',
  \widetilde{\mf{b}}_-^{\ast} (\theta', \zeta') \right) \in E^{\ast}_+
  \setminus \tilde{E}^{\ast}_-\) it holds that
  \[ F^{\ast} (\theta, \zeta, \sigma (\theta, \zeta)) [\phi] = F^{\ast}
     \left( \theta', \zeta', \widetilde{\mf{b}}_-^{\ast} (\theta', \zeta')
     \right) [\phi_{(\theta', \zeta')}] \]
  with \(\phi_{(\theta', \zeta')}\) defined by the relation
  \[ \Tr_{\zeta'} \Dil_{\beta \widetilde{\mf{b}}_-^{\ast} (\theta', \zeta')}
     \Mod_{- \theta_{\Gamma}} \phi_{(\theta', \zeta')} \eqd \Tr_{\zeta}
     \Dil_{\beta \sigma (\theta, \zeta)} \Mod_{- \theta_{\Gamma}} \phi . \]
  If \(| \zeta' - \zeta | <  \bar{r}^2 \beta \sigma (\theta, \zeta)\) and
  \(\theta' \in B_{\bar{r}^2} (\theta) \cap \Theta^{\tmop{in}} \) it holds that
  \(\widetilde{\mf{b}}_-^{\ast} (\theta', \zeta') < \mf{b}_{+ }^{\ast}
  (\theta', \zeta')\) given the continuity assumptions
  \eqref{eq:boundary-regularity-local} on \(\widetilde{\mf{b}}_-^{\ast}\) and
  \(\mf{b}_+^{\ast}\) and since \(\mf{b}_{+ }^{\ast} (\theta, \zeta) > (1 -
  \bar{r})^{- 1} \sigma (\theta, \zeta) = \widetilde{\mf{b}}_-^{\ast} (\theta,
  \zeta)\). Furthermore, we have that \(\phi_{(\theta', \zeta')} \in \Phi_{2
  \gr}^{\infty}\) with \(\| \phi_{(\theta', \zeta')} \|_{\Phi^N_{2 \gr}}
  \lesssim 1\). It follows that
  \[ \begin{aligned}[t]&
       | F^{\ast} (\theta, \zeta, \sigma (\theta, \zeta)) [\phi] |\\ &
       \qquad \lesssim \frac{1}{\beta \sigma (\theta, \zeta) \bar{r}^4}
       \sup_{\tilde{\phi}} {\int_{(B_{\bar{r}^2} (\theta) \cap
       \Theta^{\tmop{in}} ) \cap \bar{r}^2 \beta \sigma (\theta, \zeta)}} 
       \left| F^{\ast} \left( \theta', \zeta', \widetilde{\mf{b}}_-^{\ast}
       (\theta', \zeta') \right) [\tilde{\phi}] \right| \dd \theta' \dd
       \zeta'\\ &
       \lesssim \bigl\| \1_{(V^+ \cap W^+) \setminus (V^- \cup \tilde{W}^-)}
       \left( \Emb [f] \circ \Gamma \right) \bigr\|_{{\Gamma^{\ast}} 
       \SL^{(1, \infty)}_{(\Theta, \Theta^{\tmop{in}})} \dfct^{N - 2}_{4 \gr}
       (T')}
     \end{aligned} \]
  with \(T' = T_{\Theta} (0, \zeta, 10 \sigma (\theta, \zeta))\). Since
  \[ \begin{aligned}[t]&
       \bigl\| \1_{(V^+ \cap W^+) \setminus (V^- \cup \tilde{W}^-)} \left(
       \Emb [f] \circ \Gamma \right) \bigr\|_{{\Gamma^{\ast}}  \SL^{(1,
       \infty)}_{(\Theta, \Theta^{\tmop{in}})} \dfct^{N - 2}_{4 \gr} (T')}\\ &
       \qquad \lesssim \bigl\| \1_{(V^+ \cap W^+) \setminus (V^- \cup
       \tilde{W}^-)} \left( \Emb [f] \circ \Gamma \right)
       \bigr\|_{{\Gamma^{\ast}}  \SL^{(u, \infty)}_{(\Theta,
       \Theta^{\tmop{in}})} \dfct^{N - 2}_{4 \gr} (T')}
     \end{aligned} \]
  This concludes the proof. 
\end{proof}

\section{Uniform embedding bounds - bilinear case
}\label{sec:uniform-embeddings:bilinear}

In this section we prove the bilinear embedding bounds
\eqref{eq:embedding-bound:bilinear} of
\Cref{thm:uniform-embedding-bounds:bilinear}. It is noteworthy that all but
the last size \(\Gamma_{\times}^{\ast} \SL_{(\Theta , \Theta^{\tmop{in}})}^{(u
, 1)} \dfct_{\gr}^{N }\) appearing in the expression of \(\widetilde{\SF}^{u
}_{\Gamma_{\times}} \Phi_{\mf{r}}^N\) can be bound using sizes
\(\widetilde{\SF}^u_{\Gamma_j} \Phi_{\mf{r}}^N\), \(j \in \{ 2, 3 \}\). Thus, an
application of the outer Hölder inequality
(\Cref{cor:outer-holder-classical}) can yields bounds
\eqref{eq:embedding-bound:bilinear} with \(\SO^{u_{\times}}_{\Gamma_{\times}}
\Phi_{\gr}^N\) in place of \(\widetilde{\SF}^{u }_{\Gamma_{\times}}
\Phi_{\mf{r}}^N\) where
\begin{equation}
  \begin{aligned}[t]&
    \SO^{u_{\times}}_{\Gamma_{\times}} \Phi_{\gr}^N \eqd \begin{aligned}[t]&
      \beta^{\frac{1}{u_{\times}}} \Big( \Gamma_{\times}^{\ast} \SL_{(\Theta,
      \Theta^{\tmop{in}}) }^{\left( {u_{\times}} , 2 \right)} \left( \Phi
      \otimes \wpD \right)_{\gr}^N + \Gamma_{\times}^{\ast} \SL_{(\Theta,
      \Theta^{\tmop{in}}) }^{\left( {u_{\times}} , 2 \right)} \left( \wpD
    \otimes \Phi \right)_{\gr}^N\Big.
    \\ & \qquad \Big.+ \Gamma_{\times}^{\ast} {\SL 
      }^{(u_{\times}, \infty)}_{ (\Theta, \Theta^{\tmop{in}})} (\Phi \otimes
      \Phi)_{\gr}^{N } \Big)
      \hspace{2.0em} + \Gamma_{\times}^{\ast} {\SL  }^{(u_{\times},
      1)}_{\Theta} \left( \wpD \otimes \wpD \right)_{\gr}^{N } .
    \end{aligned}
  \end{aligned} \label{eq:uniform-embedding-full-size:bilinear-holder}
\end{equation}

We show this in \Cref{prop:bilinear-holder-bound}. It then remains to show the
validity of bound \eqref{eq:embedding-bound:bilinear} with the size
\(\Gamma_{\times}^{\ast} \SL_{(\Theta , \Theta^{\tmop{in}})}^{(u , 1)}
\dfct_{\gr}^{N }\) in place of \(\widetilde{\SF}^{u }_{\Gamma_{\times}}
\Phi_{\mf{r}}^N\). This is done in \Cref{prop:bilinear-bidefect-bound}.

\begin{proposition}
  \label{prop:bilinear-holder-bound}The bound
  \begin{equation}
    \| F_2 F_3 \|_{L^{p_{\times}}_{\nu_{\beta}} X^{q_{\times}, r_{\times},
    +}_{\mu_{\Theta}^1, \nu_{\beta}} \SO^{u_{\times}}_{\Gamma_{\times}}
    \Phi_{\gr}^N} \lesssim \prod_{j = 2}^3 \| F_j \|_{L^{p_j}_{\nu_{\beta }}
    X_{\mu_{\Theta}^1, \nu_{\beta }}^{q_j, r_j, +}
    \widetilde{\SF}^{u_j}_{\Gamma_j} \Phi_{\mf{r}}^N}
    \label{eq:bound:bilinear-holder}
  \end{equation}
  holds for any \(F_j \in L^{p_j}_{\nu_{\beta }} X_{\mu_{\Theta}^1, \nu_{\beta
  }}^{q_j, r_j, +} \widetilde{\SF}^{u_j}_{\Gamma_j} \Phi_{\mf{r}}^N\) with
  \[ \begin{aligned}[t]&
       p_{\times}^{- 1} = p_2^{- 1} + p_3^{- 1} \qquad q_{\times}^{- 1} =
       q_2^{- 1} + q_3^{- 1}\\ &
       r_{\times}^{- 1} = r_2^{- 1} + r_3^{- 1} \qquad u_{\times}^{- 1} =
       u_2^{- 1} + u_3^{- 1}
     \end{aligned} \]
  as long as \(p_j \in (1, \infty]\) and \(r_j \leq q_j\), \(j \in \{ 2, 3 \}\).
  As a consequence, the bound
  \begin{equation}
    \left\| \Emb [f_2] \circ \Gamma_2  \Emb [f_3] \circ \Gamma_3
    \right\|_{L^{p_{\times}}_{\nu_{\beta}} X^{q_{\times}, r_{\times},
    +}_{\mu_{\Theta}^1, \nu_{\beta}} \SO^{u_{\times}}_{\Gamma_{\times}}
    \Phi_{\gr}^N} \lesssim \| f_2 \|_{L^{p_2} (\R)} \| f_2
    \|_{L^{p_3} (\R)} .
    \label{eq:embedding-bound:bilinear-holder}
  \end{equation}
  holds for any \(p_j \in (1, \infty]\) as long as \(q_{\times}^{- 1} < \min
  \left( \frac{1}{2}, 1 - p_2^{- 1} \right) + \min \left( \frac{1}{2}, 1 -
  p_3^{- 1} \right)\) and as long as \(1 \leq u_{\times} < r_{\times} <
  q_{\times}\). The implicit constant is independent of \(\beta \in (0, 1]\),
  that determines \((\alpha_j, \beta_j, \gamma_j)\) and \(\Gamma_j\), \(j \in \{ 2,
  3 \}\), and of \(F_2, F_3\). 
\end{proposition}

\begin{proof}
  Bound \eqref{eq:embedding-bound:bilinear-holder} follows from bound
  \eqref{eq:bound:bilinear-holder} thanks to the linear embedding bounds
  \eqref{eq:embedding-bounds:uniform-iterated}. As a matter of fact, fix \(1
  \leq u_{\times} < r_{\times} < q_{\times}\) with \(q_{\times}^{- 1} < \min
  \left( \frac{1}{2}, 1 - p_2^{- 1} \right) + \min \left( \frac{1}{2}, 1 -
  p_3^{- 1} \right)\) and set
  \[ \begin{aligned}[t]&
       q_j^{- 1} = \min \left( \frac{1}{2}, 1 - p_j^{- 1} \right) -
       \varepsilon_q / 2,\\ &
       r_j^{- 1} = q_j^{- 1} - \varepsilon_r / 2,\\ &
       u_j = r_j - \varepsilon_u / 2.
     \end{aligned} \]
  with
  \[ \begin{aligned}[t]&
       \varepsilon_q \eqd q_{\times}^{- 1} - \min \left( \frac{1}{2}, 1 -
       p_2^{- 1} \right) - \min \left( \frac{1}{2}, 1 - p_3^{- 1} \right) >
       0,\\ &
       \varepsilon_r \eqd r_{\times}^{- 1} - q_{\times}^{- 1} > 0,\\ &
       \varepsilon_r \eqd u_{\times}^{- 1} - r_{\times}^{- 1} > 0.
     \end{aligned} \]
  For such a choice of \((q_j, r_j, u_j)\),
  \Cref{thm:uniform-embedding-bounds:linear} holds so \(\RHS{}
  \eqref{eq:bound:bilinear-holder} \lesssim
  \RHS{\eqref{eq:embedding-bound:bilinear-holder}}\) proving our claim.
  
  Let us now prove \eqref{eq:bound:bilinear-holder} by appropriately applying
  the outer Hölder inequality (\Cref{cor:outer-holder-classical}). For any
  \(F_j \in \Rad (\R^{3}_{+}) \otimes \Phi^{\infty}_{\gr}\), \(j \in \{
  2, 3 \}\), it holds that that
  \[ \nu_{\beta} (\| F_2 F_3 \|_{X^{q_{\times}, r_{\times},
     +}_{\mu_{\Theta}^1, \nu_{\beta}} \SO^{u_{\times}, N}_{\gr}} > 0) \leq
     \min_{j = 2, 3} \nu_{\beta} \bigl( \| F_j \|_{X_{\mu_{\Theta}^1,
     \nu_{\beta }}^{q_j, r_j, +} \widetilde{\SF}^{u_j}_{\Gamma_j}
     \Phi_{\mf{r}}^N} > 0 \bigr) \]
  since for any \(V^- \in \DD_{\beta}^{\cup}\) it holds that if \(\1_{\R^3_+
  \setminus V^-} F_j = 0\) for \(j = 2\) or \(j = 3\) then, trivially, \(\1_{\R^3_+
  \setminus V^-} F_2 F_3 = 0\). To apply the outer Hölder inequality we must
  check that for any \(F_j \in \Rad (\R^{3}_{+}) \otimes
  \Phi^{\infty}_{\gr}\), \(j \in \{ 2, 3 \}\) it holds that
  \[ \| F_2 F_3 \|_{X^{q_{\times}, r_{\times}, +}_{\mu_{\Theta}^1,
     \nu_{\beta}} \SO^{u_{\times}}_{\Gamma_{\times}} \Phi_{\gr}^N} \lesssim
     \prod_{j = 2}^3 \| F_j \|_{X_{\mu_{\Theta}^1, \nu_{\beta }}^{q_j, r_j, +}
     \widetilde{\SF}^{u_j}_{\Gamma_j} \Phi_{\mf{r}}^N} . \]
  The sizes \(X^{q_{\times}, r_{\times}, +}_{\mu_{\Theta}^1, \nu_{\beta}}
  \SO^{u_{\times}}_{\Gamma_{\times}} \Phi_{\gr}^N\) are themselves (localized)
  outer Lebesgue norms. According to \eqref{eq:X-size}, it is sufficient to
  check that for any \(V^+ \in \DD_{\beta}^{\cup}\) and any \(W^+ \in
  \TT_{\Theta}^{\cup}\) one has
  \[ \begin{aligned}[t]&
       \nu_{\beta} (V^+)^{- \frac{1}{r_{\times}}} \mu^{\infty}_{\Theta}
       (W^+)^{\frac{1}{q_{\times}} - \frac{1}{r_{\times}}} \left\| \1_{(V^+
       \cap W^+)} F_2 F_3 \right\|_{L^{r_{\times}}_{\mu_{\Theta}^1}
       \SO^{u_{\times}}_{\Gamma_{\times}} \Phi_{\gr}^N}\\ &
       \qquad \lesssim \prod_{j = 2}^3 \left( \nu_{\beta} (V^+)^{-
       \frac{1}{r_j}} \mu^{\infty}_{\Theta} (W^+)^{\frac{1}{q_j} -
       \frac{1}{r_j}} \left\| \1_{(V^+ \cap W^+)} F_j
       \right\|_{L^{r_j}_{\mu_{\Theta}^1} \widetilde{\SF}^{u_j}_{\Gamma_j}
       \Phi_{\mf{r}}^N} \right)
     \end{aligned} \]
  that boils down to showing
  \[ \left\| \1_{(V^+ \cap W^+)} F_2 F_3
     \right\|_{L^{r_{\times}}_{\mu_{\Theta}^1}
     \SO^{u_{\times}}_{\Gamma_{\times}} \Phi_{\gr}^N} \lesssim \prod_{j = 2}^3
     \left\| \1_{(V^+ \cap W^+)} F_j \right\|_{L^{r_j}_{\mu_{\Theta}^1}
     \widetilde{\SF}^{u_j}_{\Gamma_j} \Phi_{\mf{r}}^N} . \]
  This amounts to the outer Hölder inequality at the level of the outer
  Lebesgue norms \(L^{r_{\times}}_{\mu_{\Theta}^1}
  \SO^{u_{\times}}_{\Gamma_{\times}} \Phi_{\gr}^N\) and
  \(L^{r_j}_{\mu_{\Theta}^1} \widetilde{\SF}^{u_j}_{\Gamma_j} \Phi_{\mf{r}}^N\),
  \(j \in \{ 2, 3 \}\).
  
  Using that \(\mu_{\Theta}^1 (\| F_2 F_3
  \|_{\SO^{u_{\times}}_{\Gamma_{\times}} \Phi_{\gr}^N} > 0) \leq \min_{j = 2,
  3} \mu_{\Theta}^1 \bigl( \| F_j \|_{\widetilde{\SF}^{u_j}_{\Gamma_j}
  \Phi_{\mf{r}}^N} > 0 \bigr)\), we can apply the Hölder inequality to obtain
  the above bound as long as it holds that
  \[ \left\| \1_{(V^+ \cap W^+)} F_2 F_3
     \right\|_{\SO^{u_{\times}}_{\Gamma_{\times}} \Phi_{\gr}^N} \lesssim
     \prod_{j = 2}^3 \left\| \1_{(V^+ \cap W^+)} F_j
     \right\|_{\widetilde{\SF}^{u_j}_{\Gamma_j} \Phi_{\mf{r}}^N} . \]
  This however is a simple consequence of the classical Hölder inequality and
  the definitions \eqref{eq:product-tensor-sizes} of the sizes in play. We
  immediately have
  \[ \begin{aligned}[t]&
       \left\| \1_{(V^+ \cap W^+)} F_2 F_3
       \right\|_{\beta^{\frac{1}{{u_{\times}} }} \Gamma_{\times}^{\ast} {\SL 
       }^{(u_{\times}, \infty)}_{\Theta} (\Phi \otimes \Phi)_{\gr}^{N }}
       \lesssim \prod_{j = 2}^3 \beta^{\frac{1}{{u_j} }} \left\| \1_{(V^+ \cap
       W^+)} F_j \right\|_{\Gamma_j^{\ast} {\SL  }^{(u_j, \infty)}_{ (\Theta,
       \Theta^{\tmop{in}})} \Phi_{\mf{r}}^N},\\ &
       \left\| \1_{(V^+ \cap W^+)} F_2 F_3
       \right\|_{\beta^{\frac{1}{{u_{\times}} }} \Gamma_{\times}^{\ast} {\SL 
       }^{(u_{\times}, 1)}_{\Theta} \left( \Phi \otimes \wpD \right)_{\gr}^{N
       }} \lesssim \begin{aligned}[t]&
         \beta^{\frac{1}{u_2}} \left\| \1_{(V^+ \cap W^+)} F_2
         \right\|_{\Gamma_2^{\ast} {\SL  }^{(u_2, \infty)}_{ (\Theta,
         \Theta^{\tmop{in}})} \Phi_{\mf{r}}^N}\\ &
         \times \beta^{\frac{1}{{u_3} }} \left\| \1_{(V^+ \cap W^+)} F_3
         \right\|_{\Gamma_3^{\ast} {\SL  }^{(u_3, 2)}_{\Theta}
         \wpD_{\mf{r}}^N},
       \end{aligned}\\ &
       \left\| \1_{(V^+ \cap W^+)} F_2 F_3
       \right\|_{\beta^{\frac{1}{{u_{\times}} }} \Gamma_{\times}^{\ast} {\SL 
       }^{(u_{\times}, 1)}_{\Theta} \left( \Phi \otimes \wpD \right)_{\gr}^{N
       }} \lesssim \begin{aligned}[t]&
         \beta^{\frac{1}{u_2}} \left\| \1_{(V^+ \cap W^+)} F_2
         \right\|_{\Gamma_2^{\ast} {\SL  }^{(u_2, \infty)}_{\Theta}
         \wpD_{\mf{r}}^N}\\ &
         \times \beta^{\frac{1}{{u_3} }} \left\| \1_{(V^+ \cap W^+)} F_3
         \right\|_{\Gamma_3^{\ast} {\SL  }^{(u_3, \infty)}_{ (\Theta,
         \Theta^{\tmop{in}})} \Phi_{\mf{r}}^N},
       \end{aligned}\\ &
       \left\| \1_{(V^+ \cap W^+)} F_2 F_3 \right\|_{\Gamma_{\times}^{\ast}
       {\SL  }^{(u_{\times}, 1)}_{\Theta} \left( \wpD \otimes \wpD
       \right)_{\gr}^{N }} \lesssim \prod_{j = 2}^3 \left\| \1_{(V^+ \cap
       W^+)} F_j \right\|_{\Gamma_j^{\ast} {\SL  }^{(u_j, 2)}_{\Theta}
       \wpD_{\mf{r}}^N} .
     \end{aligned} \]
  
\end{proof}

\begin{proposition}
  \label{prop:bilinear-bidefect-bound}The bound
  \[ \left\| \Emb [f_2] \circ \Gamma_2  \Emb [f_3] \circ \Gamma_3
     \right\|_{L^{p_{\times}}_{\nu_{\beta}} X^{q_{\times}, r_{\times},
     +}_{\mu_{\Theta}^1, \nu_{\beta}} \Gamma_{\times}^{\ast} \SL_{(\Theta ,
     \Theta^{\tmop{in}})}^{\left( {u_{\times}} , 1 \right)} \dfct_{\gr}^{N }}
     \lesssim \| f_2 \|_{L^{p_2} (\R)} \| f_2 \|_{L^{p_3} (\R)} \]
  holds for any \(f_j \in \Sch (\R)\), and any  \(p_j \in (1,
  \infty]\), \(j \in \{ 2, 3 \}\), as long as \(p_{\times}^{- 1} = p_2^{- 1} +
  p_3^{- 1}\) and as long as \(q_{\times} \in (0, + \infty]\) with \(q_{\times}^{-
  1} < \min \left( \frac{1}{2}, 1 - p_2^{- 1} \right) + \min \left(
  \frac{1}{2}, 1 - p_3^{- 1} \right)\) and as long as \(1 \leq u_{\times} <
  r_{\times} \leq q_{\times} .\) The implicit constant is independent of \(\beta
  \in (0, 1]\), that determines \((\alpha_j, \beta_j, \gamma_j)\) and \(\Gamma_j\),
  \(j \in \{ 2, 3 \}\), and of \(f_2, f_3\). 
\end{proposition}

We believe that the bounds above cannot be obtained as in the case of
\Cref{prop:bilinear-holder-bound} from linear embeddings. This is because the
expressions \eqref{eq:product-space-defect} and
\eqref{eq:product-scale-defect} for the sizes \(\dfct_{\zeta, \gr }^{N }\) and
\(\dfct_{\sigma, \gr }^{N }\) are intrinsically bilinear and cannot be
controlled through a classical Hölder inequality. In the subsequent proof we
mimic the parts of \Cref{sec:uniform-embeddings:linear} related to bounding
the defect sizes.

\begin{proof}
  Using homogeneity let us assume that \(\| f_j \|_{L^{p_j} (\R)}
  = 1\). Fix \(q_j \in (\max (2, p_j'), \infty]\) such that \(q_{\times}^{- 1} =
  q_2^{- 1} + q_3^{- 1}\) and select \(u_j\) and \(r_j\), \(j \in \{ 2, 3 \}\), so
  that \(1 \leq u_j \leq r_j \leq q_j\) and \(r_{\times}^{- 1} = r_2^{- 2} +
  r_3^{- 2}\) and \(u_{\times}^{- 1} = u_2^{- 2} + u_3^{- 2}\). Under these
  conditions \Cref{thm:uniform-embedding-bounds:linear} holds for the
  exponents \((u_j, r_j, q_j, p_j)\) for \(j \in \{ 2, 3 \}\). According to the
  discussion in Remark \Cref{rmk:iterated-outer-bounds-explicit} this means
  that for \(j \in \{ 2, 3 \}\) there exists a choice of sets \(V^-_j (\lambda)
  \in \DD^{\cup}_{\beta}\) such that
  \[ \begin{aligned}[t]&
       \int_0^{\infty} \nu_{\beta} \left( {V^-_j}  (\lambda^{1 / p_j}) \right)
       \dd \lambda \lesssim 1, \qquad \left\| \1_{\R^3_+ \setminus V^-_j
       (\lambda)}  \Emb [f_j] \right\|_{X^{q_j, r_j, +}_{\mu_{\Theta}^1,
       \nu_{\beta}} \tilde{F}^{u_j}_{\Gamma_j} \Phi_{\gr}^{N }} \lesssim
       \lambda .
     \end{aligned} \]
  Unwrapping the definition of the size \(X^{q_j, r_j, +}_{\mu_{\Theta}^1,
  \nu_{\beta}} \tilde{F}^{u_j}_{\Gamma_j} \Phi_{\gr}^{N }\) this means that for
  any any \(\bar{r}_j \in \{ r_j, q_j \}\), any \(V^+ \in
  \mathbb{D}_{\beta}^{\cup}\), any \(W^+ \in \mathbb{T}_{\Theta}^{\cup}\), and
  any \(\tau > 0\) there exists a choice of \(W^-_j (\tau) \in
  \TT_{\Theta}^{\cup}\) such that
  \[ \begin{aligned}[t]&
       \int_0^{\infty} \mu^1_{\Theta} (W^-_j (\tau^{1 / \bar{r}_j})) \dd
       \tau \lesssim 1,\\ &
       \left\| \1_{(V^+ \cap W^+) \setminus (V^-_j (\lambda) \cup W^-_j
       (\tau))}  \Emb [f_j] \right\|_{\tilde{F}^{u_j}_{\Gamma_j}
       \Phi_{\gr}^{N }} \lesssim \tau \lambda \nu_{\beta} (V^+)^{1 /
       \bar{r}_j} \mu^{\infty}_{\Theta} (W^+)^{\frac{1}{\bar{r}_j} -
       \frac{1}{q_j}} .
     \end{aligned} \]
  Note that for any \(V^+, V^- \in \DD^{\cup}_{\beta}\) and \(W^+, W^- \in
  \TT_{\Theta}^{\cup}\) we have

  \[ \begin{aligned}[t]&
       \left\| \1_{(V^{_+} \cap W^{_+}) \setminus (V^- \cup W^-)} \Emb [f_2]
       \circ \Gamma_2  \Emb [f_3] \circ \Gamma_3
       \right\|_{\Gamma_{\times}^{\ast} \SL_{(\Theta ,
       \Theta^{\tmop{in}})}^{\left( {u_{\times}} , 1 \right)} \dfct_{\gr}^{N }
       (T)}\\ &
       \qquad \lesssim \begin{aligned}[t]&
         \left\| \left( \Emb [f_2] \circ \Gamma_2  \Emb [f_3] \circ \Gamma_3
         \right) (\eta, y, t) \right.\\ &
         \qquad \times \left. t \dd_t \left( \1_{(V^{_+} \cap W^{_+})
         \setminus (V^- \cup W^-)} (\eta, y, t) \right)
         \right\|_{\Gamma_{\times}^{\ast} \SL_{(\Theta ,
         \Theta^{\tmop{in}})}^{\left( {u_{\times}} , 1 \right)} \dfct_{\gr}^{N
         } (T)}
       \end{aligned}\\ &
       \qquad \lesssim \begin{aligned}[t]&
         \left\| \left( \Emb [f_2] \circ \Gamma_2  \Emb [f_3] \circ \Gamma_3
         \right) (\eta, y, t) \right.  \1_{(V^{_+} \cap W^{_+}) \setminus (V^-
         \cup W^-)} (\eta, y, t)\\ &
         \qquad \times \left.  \left( t \delta \left( t - \mf{b_+} (\eta, y)
         \right) + t \delta \left( t - \mf{b_-} (\eta, y) \right) \right. 
         \right\|_{\Gamma_{\times}^{\ast} \SL_{(\Theta ,
         \Theta^{\tmop{in}})}^{\left( {u_{\times}} , 1 \right)} (\Phi \otimes
         \Phi)_{\gr}^{N } (T)}
       \end{aligned}\\ &
       \qquad \lesssim \prod_{j = 2}^3 \begin{aligned}[t]&
         \left\| \left( \Emb [f_j] \circ \Gamma_j  \right) (\eta, y, t)
         \right.  \1_{(V^{_+} \cap W^{_+}) \setminus (V^- \cup W^-)} (\eta, y,
         t)\\ &
         \qquad \times \left.  \left( t \delta \left( t - \mf{b_+} (\eta, y)
         \right) + t \delta \left( t - \mf{b_-} (\eta, y) \right) \right. 
         \right\|_{\Gamma_j^{\ast} \SL_{(\Theta ,
         \Theta^{\tmop{in}})}^{\left( {u_j} , 1 \right)} \Phi_{\gr}^{N }} .
       \end{aligned}
     \end{aligned} \]
  On the other hand
  \[ \begin{aligned}[t]&
       \begin{aligned}[t]&
         \left\| \left( \Emb [f_j] \circ \Gamma_j  \right) (\eta, y, t)
         \right.  \1_{(V^{_+} \cap W^{_+}) \setminus (V^- \cup W^-)} (\eta, y,
         t)\\ &
         \hspace{4em} \times \left.  \left( t \delta \left( t - \mf{b_+}
         (\eta, y) \right) + t \delta \left( t - \mf{b_-} (\eta, y) \right)
         \right.  \right\|_{\Gamma_j^{\ast} \SL_{(\Theta ,
         \Theta^{\tmop{in}})}^{\left( {u_j} , 1 \right)} \Phi_{\gr}^{N }}
       \end{aligned}\\ &
       \qquad \lesssim \left\| \1_{(V^+ \cap W^+) \setminus (V^-  \cup W^-)} 
       \Emb [f_j] \right\|_{\tilde{F}^{u_j}_{\Gamma_j} \Phi_{\gr}^{N }} .
     \end{aligned} \]
  Based on the above two bounds, the rest of the proof is merely an unwrapping
  of definitions of iterated outer Lebesgue norms. We set \(V^-_{\times}
  (\lambda) = \bigcup_{j \in \{ 2, 3 \}} {V^-_j}  (\lambda^{p_{\times} /
  p_j})\) so that
  \[ \int_0^{\infty} \nu_{\beta} \left( {V^-_{\times}}  (\lambda^{1 /
     p_{\times}}) \right) \dd \lambda \leq \sum_{j \in \{ 2, 3 \}}
     \int_0^{\infty} \nu_{\beta} \left( {V^-_j}  (\lambda^{1 / p_j}) \right)
     \dd \lambda \lesssim 1. \]
  We claim that
  \[ \left\| \1_{\R^3_+ {\setminus V^-_{\times}}  (\lambda)} \Emb [f_2]
     \circ \Gamma_2  \Emb [f_3] \circ \Gamma_3 \right\|_{X^{q_j, r_j,
     +}_{\mu_{\Theta}^1, \nu_{\beta}} \Gamma_{\times}^{\ast} \SL_{(\Theta ,
     \Theta^{\tmop{in}})}^{\left( {u_{\times}} , 1 \right)} \dfct_{\gr}^{N }
     (T)} \lesssim \lambda \]
  i.e. that for \(\bar{r}_{\times} \in \{ r_{\times}, q_{\times} \}\) and for
  any \(V^+ \in \mathbb{D}_{\beta}^{\cup}\), any \(W^+ \in
  \mathbb{T}_{\Theta}^{\cup}\), and for any \(\tau > 0\) there exists
  \(W^-_{\times} (\tau) \in \TT_{\Theta}^{\cup}\) such that
  \[ \int_0^{\infty} \mu^1_{\Theta} (W^-_{\times} (\tau^{1 /
     \bar{r}_{\times}})) \dd \tau \lesssim 1 \]
  and
  \[ \begin{aligned}[t]&
       \left\| \1_{(V^{_+} \cap W^{_+}) \setminus \left( {V^-_{\times}} 
       (\lambda) \cup W^-_{\times} (\tau) \right)} \Emb [f_2] \circ \Gamma_2 
       \Emb [f_3] \circ \Gamma_3 \right\|_{\Gamma_{\times}^{\ast}
       \SL_{(\Theta , \Theta^{\tmop{in}})}^{\left( {u_{\times}} , 1 \right)}
       \dfct_{\gr}^{N } (T)}\\ &
       \qquad \lesssim \tau \lambda \nu_{\beta} (V^+)^{1 / \bar{r}_{\times}}
       \mu^{\infty}_{\Theta} (W^+)^{\frac{1}{\bar{r}_{\times}} -
       \frac{1}{q_{\times}}}
     \end{aligned} . \]
  It is sufficient to set \(W^-_{\times} (\tau) = \bigcup_{j \in \{ 2, 3 \}}
  {W^-_j}  (\tau^{\bar{r}_{\times} / \bar{r}_j})\) where \({W^-_j}  (\tau) \in
  \TT_{\Theta}^{\cup}\) is such that
  \[ \begin{aligned}[t]&
       \int_0^{\infty} \mu^1_{\Theta} (W^-_j (\tau^{1 / \bar{r}_j})) \dd \tau \lesssim 1, \\ &
       \begin{aligned}[t] &
       \left\| \1_{(V^+ \cap W^+) \setminus (V^-_j (\lambda^{p_{\times} /p_j}) \cup W^-_j (\tau))}  \Emb [f_j]\right\|_{\tilde{F}^{u_j}_{\Gamma_j} \Phi_{\gr}^{N }} \\ & \qquad \lesssim \tau
       \lambda^{p_{\times} / p_j} \nu_{\beta} (V^+)^{1 / \bar{r}_j}
       \mu^{\infty}_{\Theta} (W^+)^{\frac{1}{\bar{r}_j} - \frac{1}{q_j}} .
       \end{aligned}      
     \end{aligned} \]
  The claim follows.
\end{proof}

\printbibliography

@software{vanderhoevenGNUTeXmacsScientific1998,
  title = {{{GNU TeXmacs}}: A Scientific Editing Platform. {{https://www.texmacs.org}}},
  author = {van der Hoeven, Joris and {et al}},
  date = {1998},
  url = {https://www.texmacs.org}
}

@article{amentaBanachvaluedModulationInvariant2020,
  title = {Banach-valued modulation invariant {{Carleson}} embeddings and outer-\${{L}}\^p\$ spaces: the {{Walsh}} case},
  shorttitle = {Banach-{{Valued Modulation Invariant Carleson Embeddings}} and {{Outer-}}\$\${{L}}\^p\$\${{Spaces}}},
  author = {Amenta, Alex and Uraltsev, Gennady},
  date = {2020-06-22},
  journaltitle = {Journal of Fourier Analysis and Applications},
  shortjournal = {J Fourier Anal Appl},
  volume = {26},
  number = {4},
  pages = {53},
  issn = {1531-5851},
  doi = {10.1007/s00041-020-09768-0},
  langid = {english}
}

@article{amentaBilinearHilbertTransform2020,
  title = {The bilinear {{Hilbert}} transform in {{UMD}} spaces},
  author = {Amenta, Alex and Uraltsev, Gennady},
  date = {2020-12-01},
  journaltitle = {Mathematische Annalen},
  shortjournal = {Math. Ann.},
  volume = {378},
  number = {3},
  pages = {1129--1221},
  issn = {1432-1807},
  doi = {10.1007/s00208-020-02052-y},
  langid = {english}
}

@article{calderonCommutatorsSingularIntegral1965,
  title = {Commutators of singular integral operators},
  author = {Calder\'on, A. P.},
  date = {1965-05-01},
  journaltitle = {Proceedings of the National Academy of Sciences},
  volume = {53},
  number = {5},
  pages = {1092--1099},
  issn = {0027-8424, 1091-6490},
  doi = {10.1073/pnas.53.5.1092},
  langid = {english}
}

@article{carlesonConvergenceGrowthPartial1966,
  title = {On convergence and growth of partial sums of {{Fourier}} series},
  author = {Carleson, Lennart},
  date = {1966},
  journaltitle = {Acta Mathematica},
  shortjournal = {Acta Math.},
  volume = {116},
  number = {0},
  pages = {135--157},
  issn = {0001-5962},
  doi = {10.1007/BF02392815},
  langid = {english}
}

@article{demeterTwodimensionalBilinearHilbert2010,
  title = {On the two-dimensional bilinear {{Hilbert}} transform},
  author = {Demeter, Ciprian and Thiele, Christoph},
  date = {2010},
  journaltitle = {American Journal of Mathematics},
  shortjournal = {American Journal of Mathematics},
  volume = {132},
  number = {1},
  pages = {201--256},
  issn = {1080-6377},
  doi = {10.1353/ajm.0.0101},
  langid = {english}
}

@article{diplinioModulationInvariantCarleson2018,
  title = {A modulation invariant {{Carleson}} embedding theorem outside local {{L2}}},
  author = {di Plinio, Francesco and Ou, Yumeng},
  options = {useprefix=true},
  date = {2018-06-01},
  journaltitle = {Journal d'Analyse Math\'ematique},
  shortjournal = {JAMA},
  volume = {135},
  number = {2},
  pages = {675--711},
  issn = {1565-8538},
  doi = {10.1007/s11854-018-0049-4},
  langid = {english}
}

@unpublished{diplinioWeaktypeCarlesonTheorem2022,
  title = {The weak-type {{Carleson}} theorem via wave packet estimates},
  author = {Di Plinio, Francesco and Fragkos, Anastasios},
  date = {2022-04-17},
  eprint = {2204.08051},
  eprinttype = {arxiv},
  primaryclass = {math},
  publisher = {{arXiv}},
  doi = {10.48550/arXiv.2204.08051},
  archiveprefix = {arXiv}
}

@article{doTheoryOuterMeasures2015,
  title = {\${{L}}\^p\$ theory for outer measures and two themes of {{Lennart Carleson}} united},
  author = {Do, Yen and Thiele, Christoph},
  date = {2015},
  journaltitle = {Bulletin of the American Mathematical Society},
  volume = {52},
  number = {2},
  pages = {249--296}
}

@book{federerGeometricMeasureTheory1996,
  title = {Geometric {{Measure Theory}}},
  author = {Federer, Herbert},
  editor = {Eckmann, B. and van der Waerden, B. L.},
  options = {useprefix=true},
  date = {1996},
  series = {Classics in {{Mathematics}}},
  publisher = {{Springer Berlin Heidelberg}},
  location = {{Berlin, Heidelberg}},
  doi = {10.1007/978-3-642-62010-2},
  isbn = {978-3-540-60656-7 978-3-642-62010-2}
}

@article{feffermanPointwiseConvergenceFourier1973,
  title = {Pointwise convergence of {{Fourier}} series},
  author = {Fefferman, Charles},
  date = {1973-11},
  journaltitle = {The Annals of Mathematics},
  shortjournal = {The Annals of Mathematics},
  volume = {98},
  number = {3},
  pages = {551},
  issn = {0003486X},
  doi = {10.2307/1970917},
  langid = {english}
}

@article{grafakosUniformBoundsBilinear2004,
  title = {Uniform bounds for the bilinear {{Hilbert}} transforms, {{I}}},
  author = {Grafakos, Loukas and Li, Xiaochun},
  date = {2004-05-01},
  journaltitle = {Annals of Mathematics},
  shortjournal = {Ann. Math.},
  volume = {159},
  number = {3},
  pages = {889--933},
  issn = {0003-486X},
  doi = {10.4007/annals.2004.159.889},
  langid = {english}
}

@article{gressmanTrilinearSingularIntegral2016,
  title = {On a trilinear singular integral form with determinantal kernel},
  author = {Gressman, Philip and He, Danqing and Kova\v{c}, Vjekoslav and Street, Brian and Thiele, Christoph and Yung, Po-Lam},
  date = {2016},
  journaltitle = {Proceedings of the American Mathematical Society},
  volume = {144},
  number = {8},
  pages = {3465--3477},
  publisher = {{American Mathematical Society}},
  issn = {0002-9939},
  doi = {10.1090/proc/13007}
}

@article{kovacBoundednessTwistedParaproduct2012,
  ids = {KovacBoundednesstwistedparaproduct2012},
  title = {Boundedness of the twisted paraproduct},
  author = {Kova\v{c}, Vjekoslav},
  date = {2012},
  journaltitle = {Revista Matem\'atica Iberoamericana},
  volume = {28},
  number = {4},
  pages = {1143--1164},
  issn = {0213-2230},
  doi = {10.4171/RMI/707},
  langid = {english}
f}

@article{kovacDyadicTriangularHilbert2015,
  title = {Dyadic triangular hilbert transform of two general functions and one not too general function},
  author = {Kova\v{c}, Vjekoslav and Thiele, Christoph and Zorin-Kranich, Pavel},
  date = {2015-11},
  journaltitle = {Forum of Mathematics, Sigma},
  volume = {3},
  issn = {2050-5094},
  doi = {10.1017/fms.2015.25},
  langid = {english}
}

@article{laceyBilinearMaximalFunctions2000a,
  title = {The bilinear maximal functions map into \${{L}}\^p\$ for \$2/3 {$<$} p \textbackslash leq 1\$},
  author = {Lacey, Michael T.},
  date = {2000},
  journaltitle = {Annals of Mathematics},
  volume = {151},
  number = {1},
  pages = {35--57},
  publisher = {{Annals of Mathematics}},
  issn = {0003-486X},
  doi = {10.2307/121111}
}

@article{laceyCalderonConjectureBilinear1998,
  title = {On {{Calder\'on}}'s conjecture for the bilinear {{Hilbert}} transform},
  author = {Lacey, Michael T. and Thiele, Christoph M.},
  date = {1998},
  journaltitle = {Proceedings of the National Academy of Sciences of the United States of America},
  volume = {95},
  number = {9},
  pages = {4828},
  doi = {ttps://doi.org/10.1073/pnas.95.9.4828}
}

@article{laceyEstimatesBilinearHilbert1997,
  title = {Estimates on the bilinear {{Hilbert}} transform for 2{$<$} p{$<\infty$}},
  author = {Lacey, M. and Thiele, C. Lp},
  date = {1997},
  journaltitle = {Ann. Math},
  volume = {146},
  pages = {693--724}
}

@article{liUniformBoundsBilinear2006,
  title = {Uniform bounds for the bilinear {{Hilbert}} transforms, {{II}}},
  author = {Li, Xiaochun},
  date = {2006},
  journaltitle = {Revista Matem\'atica Iberoamericana},
  shortjournal = {Rev. Mat. Iberoamericana},
  volume = {22},
  number = {3},
  pages = {1069--1126},
  issn = {0213-2230},
  doi = {10.4171/RMI/483},
  langid = {english}
}

@article{muscaluCalderonCommutatorsCauchy2014,
  title = {Calder\'on commutators and the {{Cauchy}} integral on {{Lipschitz}} curves revisited {{I}}. {{First}} commutator and generalizations},
  author = {Muscalu, Camil},
  date = {2014},
  journaltitle = {Revista Matem\'atica Iberoamericana},
  shortjournal = {Rev. Mat. Iberoamericana},
  volume = {30},
  number = {2},
  pages = {727--750},
  issn = {0213-2230},
  doi = {10.4171/RMI/798},
  langid = {english}
}

@book{muscaluClassicalMultilinearHarmonic2013a,
  title = {Classical and multilinear harmonic analysis. {{Vol}}. {{II}}},
  author = {Muscalu, Camil and Schlag, Wilhelm},
  date = {2013},
  series = {Cambridge {{Studies}} in {{Advanced Mathematics}}},
  volume = {138},
  publisher = {{Cambridge University Press, Cambridge}},
  isbn = {978-1-107-03182-1},
  mrnumber = {3052499},
  pagetotal = {xvi+324}
}

@article{thieleUniformEstimate2002,
  title = {A {{Uniform Estimate}}},
  author = {Thiele, Christoph},
  date = {2002-09},
  journaltitle = {The Annals of Mathematics},
  shortjournal = {The Annals of Mathematics},
  volume = {156},
  number = {2},
  pages = {519--563},
  issn = {0003486X},
  doi = {10.2307/3597197}
}

@book{thieleWavePacketAnalysis2006,
  ids = {thieleWavePacketAnalysis2006a,thieleWavePacketAnalysis2006c},
  title = {Wave packet analysis},
  author = {Thiele, Christoph},
  date = {2006},
  series = {{{CBMS Regional Conference Series}} in {{Mathematics}}},
  volume = {105},
  publisher = {{Published for the Conference Board of the Mathematical Sciences, Washington, DC; by the American Mathematical Society, Providence, RI}},
  doi = {10.1090/cbms/105},
  isbn = {978-0-8218-3661-3},
  mrnumber = {2199086},
  pagetotal = {vi+86}
}

@unpublished{uraltsevVariationalCarlesonEmbeddings2016,
  title = {Variational {{Carleson}} embeddings into the upper 3-space},
  author = {Uraltsev, Gennady},
  date = {2016-10-24},
  eprint = {1610.07657},
  eprinttype = {arxiv},
  archiveprefix = {arXiv}
}

@thesis{warchalskiUniformEstimatesOneand2018,
type = {Thesis},
  title = {Uniform Estimates in One- and Two-Dimensional Time-Frequency Analysis},
  author = {Warchalski, Michał},
  date = {2019-02-22},
  institution = {{RFWU Bonn}},
  url = {https://bonndoc.ulb.uni-bonn.de/xmlui/handle/20.500.11811/7882},
  langid = {english}
  }

\end{document}